\documentclass[aos]{imsart}
\arxiv{2410.10704}

\usepackage{amsmath, amssymb,amsfonts,bm,bbm,verbatim,multirow,color,xcolor,mathtools}
\usepackage{epic,eepic,psfrag,epsfig}
\usepackage[sf,bf,SF,footnotesize]{subfigure}
\usepackage{graphicx}
\usepackage{amsthm,breakcites}
\usepackage{amscd}
\usepackage{epsfig}
\usepackage{algorithm}
\usepackage{algpseudocode}
\usepackage{aligned-overset}
\usepackage{tikz-cd}
\usepackage{upgreek}
\usepackage{mathrsfs}
\usepackage{fix-cm}

\usepackage{tikz}
\tikzstyle{edge} = [fill,opacity=.5,fill opacity=.5,line cap=round, line join=round, line width=50pt]
\pgfdeclarelayer{background}
\pgfsetlayers{background,main}
\usetikzlibrary{calc}
\usetikzlibrary{arrows.meta}
\usepackage[round]{natbib}

\usepackage{enumerate}
\usepackage{array}
\RequirePackage[colorlinks,citecolor=blue,urlcolor=blue,breaklinks]{hyperref}

\startlocaldefs
\theoremstyle{plain}
\newtheorem{theorem}{Theorem}
\newtheorem{lemma}[theorem]{Lemma}
\newtheorem{prop}[theorem]{Proposition}

\newtheorem{defn}[theorem]{Definition}
\theoremstyle{definition}
\newtheorem{example}[theorem]{Example}

\newtheorem*{remark*}{Remark}

\newcommand{\indep}{\perp \!\!\! \perp}

\DeclareMathOperator*{\sargmin}{sargmin}
\DeclareMathOperator*{\sargmax}{sargmax}
\DeclareMathOperator*{\argmin}{argmin}

\DeclareMathOperator*{\tr}{tr}
\DeclareMathOperator*{\Var}{Var}
\DeclareMathOperator*{\Cov}{Cov}

\makeatletter
\newcommand{\ostar}{\mathbin{\mathpalette\make@circled\star}}
\newcommand{\make@circled}[2]{%
  \ooalign{$\m@th#1\smallbigcirc{#1}$\cr\hidewidth$\m@th#1#2$\hidewidth\cr}%
}
\newcommand{\smallbigcirc}[1]{%
  \vcenter{\hbox{\scalebox{0.77778}{$\m@th#1\bigcirc$}}}%
}
\makeatother

\newlength{\widebarargwidth}
\newlength{\widebarargheight}
\newlength{\widebarargdepth}
\DeclareRobustCommand{\widebar}[1]{%
	\settowidth{\widebarargwidth}{\ensuremath{#1}}%
	\settoheight{\widebarargheight}{\ensuremath{#1}}%
	\settodepth{\widebarargdepth}{\ensuremath{#1}}%
	\addtolength{\widebarargwidth}{-0.3\widebarargheight}%
	\addtolength{\widebarargwidth}{-0.3\widebarargdepth}%
	\makebox[0pt][l]{\hspace{0.3\widebarargheight}%
		\hspace{0.3\widebarargdepth}%
		\addtolength{\widebarargheight}{0.3ex}%
		\rule[\widebarargheight]{0.95\widebarargwidth}{0.1ex}}%
	{#1}}

\newcommand{\overbar}[1]{\mkern 1.5mu\overline{\mkern-1.5mu#1\mkern-1.5mu}\mkern 1.5mu}

\newcommand\numberthis{\addtocounter{equation}{1}\tag{\theequation}}

\def\hat{\widehat}
\def\tilde{\widetilde}
\def\bar{\widebar}

\endlocaldefs

\begin{document}

\begin{frontmatter}
\title{Estimation beyond Missing (Completely) at Random}
\runtitle{Estimation beyond Missing (Completely) at Random}

\begin{aug}
\author[A]{\fnms{Tianyi}~\snm{Ma}\ead[label=e1]{tm681@cam.ac.uk}}
\author[B]{\fnms{Kabir A.}~\snm{Verchand}\ead[label=e2]{verchand@usc.edu}}
\author[C]{\fnms{Thomas B.}~\snm{Berrett}\ead[label=e3]{tom.berrett@warwick.ac.uk}}
\author[D]{\fnms{Tengyao}~\snm{Wang}\ead[label=e4]{t.wang59@lse.ac.uk}}
\author[A]{\fnms{Richard
J.}~\snm{Samworth}\ead[label=e5]{r.samworth@statslab.cam.ac.uk}}
\address[A]{Statistical Laboratory, University of
Cambridge\printead[presep={,\ }]{e1,e5}}

\address[B]{Department of Data Sciences and Operations, University of Southern California\printead[presep={,\ }]{e2}}

\address[C]{Department of Statistics, University of Warwick\printead[presep={,\ }]{e3}}

\address[D]{Department of Statistics, London School of Economics and Political Science\printead[presep={,\ }]{e4}}
\end{aug}

\begin{abstract}
We study the effects of missingness on the estimation of population parameters.  Moving beyond restrictive missing completely at random (MCAR) assumptions, we first formulate a missing data analogue of Huber's arbitrary $\epsilon$-contamination model.  For mean estimation with respect to squared Euclidean error loss, we show that the minimax quantiles decompose as a sum of the corresponding minimax quantiles under a heterogeneous, MCAR assumption, and a robust error term, depending on $\epsilon$, that reflects the additional error incurred by departure from MCAR.  

We next introduce natural classes of \emph{realisable $\epsilon$-contamination models}, where an MCAR version of a base distribution $P$ is contaminated by an arbitrary missing not at random (MNAR) version of $P$.  These classes are rich enough to capture various notions of biased sampling and sensitivity conditions, yet we show that they enjoy improved minimax performance relative to our earlier arbitrary contamination classes for both parametric and nonparametric classes of base distributions.  For instance, with a univariate Gaussian base distribution, consistent mean estimation over realisable $\epsilon$-contamination classes is possible even when $\epsilon$ and the proportion of missingness converge (slowly) to~1.  We extend our results to the setting of departures from missing at random (MAR) in normal linear regression with a realisable missing response, and also demonstrate that our methods can be made adaptive to the case of unknown $\epsilon$.
\end{abstract}
\end{frontmatter}

\section{Introduction} \label{sec:intro}

A major theme of modern statistical research concerns problems where we wish to make inference about (some aspect of) a target population, but do not have access to an independent sample of size $n$ from this distribution.  Departures from this idealised scenario may take many different forms:  spatial, temporal or some other form of dependence may be present \citep{cressie2015statistics,brockwell1991time}, or (some of) our data may be drawn from a source distribution that is different from, but related to, our target population, as in transfer learning \citep{cai2021transfer,reeve2021adaptive}.  In a similar vein, the field of robust statistics aims to draw reliable inference when some of our data may be contaminated~\citep{huber1964robust}. 

One of the most common ways in which observed data may fail to represent a sample from a target population is when components may be missing or unobserved.  Even the relatively benign setting where data are missing completely at random (MCAR)---that is, when the data generating and missingness mechanisms are independent---presents substantial challenges for practitioners and theoreticians alike.  A significant, ongoing research effort has therefore sought to introduce appropriate methodology under the MCAR hypothesis in several contemporary statistical problems, including sparse linear regression \citep{loh2012high,belloni2017linear}, classification \citep{tony2019high,sell2024nonparametric}, sparse or high-dimensional principal component analysis \citep{elsener2019sparse,zhu2022high,yan2024inference}, covariance and precision matrix estimation~\citep{lounici2014high,loh2018high} and high-dimensional changepoint estimation \citep{xie2012change,follain2022high}. 

Despite this progress, it is frequently argued that MCAR should be regarded very much as the exception rather than the rule in applications.  For instance, supporters of one political party may be less likely than other voters to respond to survey requests \citep{kennedy2018evaluation}, while in education, efforts to model the value added by teachers may be hindered by large numbers of students with incomplete records and the tendency for those students to be lower achieving \citep{mccaffrey2011missing}.  Likewise, in epidemiology, individuals with depression may be less likely to participate in a survey than those without depression \citep{prince2012core}, while metabolomic data are typically subject to a high proportion of non-MCAR missingness due to a metabolite-specific missingness mechanism in which more abundant analytes are more likely to be observed \citep{do2018characterization,mckennan2020estimation}.

The most well-studied alternative to MCAR is the missing at random (MAR) hypothesis~\citep{little2014statistical,seaman2013what,farewell2022missing}.  The main virtue of this assumption is that, in well-specified, identifiable parametric models, likelihood-based methods may retain parametric rates of convergence to population estimands.  On the other hand, it also has several drawbacks: first, it may well still be too restrictive as an appropriate missingness model for practical data sets (e.g.~in the examples of the previous paragraph).  Second, its links to likelihood-based methods and simple missingness patterns limit its applicability; third, even in simple parametric models, population parameters may be unidentifiable under MAR (see Section~\ref{sec:impossibility-MAR-mean}); and finally, it fails to measure proximity to the MCAR class in an appropriate, continuous fashion, and may therefore be unable to capture the essence of a given statistical challenge.

Our goal in this paper is to commence a line of work that seeks to understand the extent to which (non-MCAR) missingness affects our ability to estimate population parameters.  We primarily focus here on the most basic statistical problem of mean estimation, though we extend our results to regression settings where the response variable may be missing in Section~\ref{sec:regression-missing-response}.  In order to address the fundamental difficulty of the challenge, we introduce Huber-style models that interpolate between MCAR and larger classes that allow much more general dependence relationships between the data generating and missingness mechanisms.  We measure performance of estimators via their squared Euclidean error, but since this loss function is unbounded and we have a positive probability under our models of observing no data, the minimax risk is infinite (so uninformative for the purposes of comparing estimators).  Instead, we work with the recently-developed minimax quantile framework; see Section~\ref{sec:minimax-quantile}.

We begin in Section~\ref{sec:setup} by introducing a formal framework for studying missing data via extended measurable spaces, which allow for missing components.  Our main statistical models are what we refer to as \emph{arbitrary $\epsilon$-contamination} and \emph{realisable $\epsilon$-contamination} models.  In the former, we perturb a distribution $P$ that is subject to MCAR missingness by an additional mixture component (having corresponding mixture proportion $\epsilon$) that may be an arbitrary distribution on our extended measurable space; in particular, this latter mixture component may be viewed as an MNAR version of an arbitrary distribution $P'$.  On the other hand, in our realisable classes, although we again allow mixture perturbations of a base distribution $P$ subject to MCAR missingness, we now require the contamination component to be an MNAR version of $P$ itself.  Although we are not aware of previous studies of these realisable classes, we believe that in many practical settings, it is appropriate to regard our data (whether observed or not) as arising from a particular base distribution, and the missingness mechanism only playing a role thereafter (even if it is potentially dependent on the data).  Such a setting would result in our observed data as being from a realisable model, and these classes therefore form a natural way to restrict the vast array of different possible dependence relationships between data generating and missingness mechanisms.  A real data illustration of realisability from the income survey of \citet{bollinger2019trouble} is provided after Proposition~\ref{prop:univariate-realisability}.  As we establish in this work, realisable classes also offer the potential for improved performance guarantees relative to those available for arbitrary contamination models.

In Section~\ref{sec:mean-estimation-arbitrary-contamination}, we study the minimax quantiles of our squared Euclidean error loss function over arbitrary $\epsilon$-contamination models.  We introduce an \textproc{Iterative\_Robust\_Mean} algorithm that employs iterative imputation to convert a robust mean estimator designed for complete data into an estimator of the mean of the MCAR mixture component.   Theorem~\ref{thm:robust-descent-iterative-imputation-ub} provides an upper bound on the performance of this algorithm; this  decomposes as a sum of an MCAR term and a term quantifying the effect of contamination from MCAR.  A corresponding lower bound on the minimax quantile given in Theorem~\ref{thm:arbitrary-contamination-lb} reveals that, at least when the covariance matrix of our base distribution is diagonal, the upper bound in Theorem~\ref{thm:robust-descent-iterative-imputation-ub} is optimal up to multiplicative constants in terms of its behaviour under departures from MCAR, and is optimal up to logarithmic factors in the dimension and quantile level in the MCAR term. 

We turn our attention in Section~\ref{sec:gaussian-realisable-model} to realisable contamination of a Gaussian base distribution.  Focusing for now on the univariate case for simplicity of exposition, we introduce a minimum Kolmogorov distance estimator, and show in Theorems~\ref{thm:univariate-realisable-lb} and~\ref{thm:one-dim-kolmogorov-estimator} that it achieves the minimax optimal rate for both the MCAR and MCAR departure terms, except for a possible logarithmic dependence in an intermediate effective contamination level regime that vanishes with the effective sample size.  This latter result also reveals the surprising fact that consistent mean estimation is possible in this model even in settings where the proportion of missingness and the proportion of MNAR contamination converge (slowly) to 1.  Section~\ref{sec:nonparametric-realisable} concerns more general realisable models, where our base distribution is only required to satisfy moment or a $\psi_r$-Orlicz norm condition with $r \geq 1$.  Theorems~\ref{thm:one-dim-realisable-sample-mean-ub} and~\ref{thm:nonparametric-realisable-model-lb} provide upper and lower bounds on the minimax quantiles that match up to universal constants under both conditions.  For both our Gaussian and our nonparametric classes of base distributions, we also discuss multivariate extensions of these results.  

Table~\ref{table:summary} presents a selection of our findings for univariate mean estimation problems.  These illustrate the benefits in terms of improved worst-case performance of working with realisable, as opposed to arbitrary,  contamination.  It is interesting to see, for instance, that when the effective contamination level is small, the minimax quantile rate for a Gaussian distribution under arbitrary contamination agrees with the corresponding rate for a general distribution with finite variance under realisable contamination.  It is also notable that, while under arbitrary contamination the minimax quantiles are infinite as soon as $\epsilon \geq q/(1+q)$, where~$q$ denotes the MCAR observation proportion, under realisable contamination this threshold converges to 1 with the sample size.

{\renewcommand{\arraystretch}{1.3}
\setlength{\tabcolsep}{4pt}
\begin{table}[ht]
\small
\centering
\caption{\raggedright A comparison of minimax rates under arbitrary and realisable $\epsilon$-contamination for different univariate base distribution classes. Here, $\mathscr{M}_0 \coloneqq \frac{\sigma^2 \log(1/\delta)}{nq(1 - \epsilon)}$ denotes the MCAR minimax $(1-\delta)$th quantile rate for estimating the mean of a base distribution $P$ having variance (or squared sub-Gaussian norm) $\sigma^2$ based on $Z_1,\ldots,Z_n \overset{\mathrm{iid}}{\sim}\mathsf{MCAR}_{(q(1-\epsilon),P)}$ (see~\eqref{eq:MCAR-law} below), and $\kappa \coloneqq \frac{\epsilon}{q(1-\epsilon)}$ denotes the effective contamination level.  In the Gaussian realisable rate, we have ignored a potential logarithmic multiplicative factor in a regime where the overall rate remains polynomial in the effective sample size $nq(1-\epsilon)$. The results for arbitrary contamination are provided in Section~\ref{sec:univariate-arbitrary-contamination-lb} and Theorem~\ref{thm:arbitrary-contamination-lb} , while the results for realisable contamination are given in Section~\ref{sec:realisable-mean-est}.} \label{table:summary}
\begin{tabular}{|c|cc|cc|}
\hline
    & \multicolumn{2}{c|}{\bf{Arbitrary contamination}}                & \multicolumn{2}{c|}{\bf{Realisable contamination}}\\ \hline
    Base distribution &  \multicolumn{1}{c|}{Minimax rate} & $\epsilon$ condition & \multicolumn{1}{c|}{Minimax rate} & $\epsilon$ condition \\ \hline
    \rule{0pt}{4.3ex}\rule[-2.7ex]{0pt}{0pt} Gaussian & \multicolumn{1}{c|}{$\displaystyle \mathscr{M}_0 + \sigma^2\kappa^2$} & $\displaystyle \epsilon < \frac{q}{1+q}$ & \multicolumn{1}{c|}{$\displaystyle \mathscr{M}_0 + \frac{\sigma^2 \log^2(1 + \kappa)}{\log\{n q(1 - \epsilon)\}}$} & $\epsilon < 1-o_n(1)$ \\ \hline
    \rule{0pt}{4.3ex}\rule[-2.7ex]{0pt}{0pt} Sub-Gaussian & \multicolumn{1}{c|}{$\mathscr{M}_0 + \sigma^2\kappa^2\log\bigl(\frac{1}{\kappa}\bigr)$} & $\displaystyle \epsilon < \frac{q}{1+q}$ & \multicolumn{1}{c|}{ $\mathscr{M}_0 + \sigma^2\bigl(\kappa^2 \wedge \log(1 + \kappa)\bigr)$} & $\epsilon < 1-o_n(1)$ \\ \hline
    \rule{0pt}{4.3ex}\rule[-2.7ex]{0pt}{0pt} Finite variance & \multicolumn{1}{c|}{$\displaystyle \mathscr{M}_0 + \sigma^2\kappa$} & $\displaystyle \epsilon < \frac{q}{1+q}$ & \multicolumn{1}{c|}{$\displaystyle \mathscr{M}_0 + \sigma^2(\kappa^2 \wedge \kappa)$} & $\epsilon < 1-o_n(1)$ \\ \hline
\end{tabular} 
\end{table}}

Extensions of our results to normal linear regression models with realisable missing responses are discussed in Section~\ref{sec:regression-missing-response}.  Here, we show that even when the contamination proportion $\epsilon$ is allowed to grow slowly to 1, consistent estimation of the vector of regression coefficients remains achievable under a mild regularity assumption on the design.  Finally, in Section~\ref{sec:adaptation}, we show how our methodological proposals can be broadened to settings where the contamination proportion $\epsilon$ is unknown.  The price for adaptation turns out to be very small: for instance, in the setting of Theorem~\ref{thm:robust-descent-iterative-imputation-ub}, in one regime of the failure probability, we merely inflate the universal constant in the bound on the performance of the \textproc{Iterative\_Robust\_Mean} algorithm, while in the other regime, one term in our bound is inflated by a multiplicative factor of $O(\log \log n)$. 

All of our proofs are deferred to the supplement \citep{ma2024supplementary}, and results appearing in the supplement are prefaced with an `S'.

\subsection{Related work}

The $\epsilon$-contamination models that form the bedrock of our framework for the analysis of the effects of missing data are inspired by related models in the robust statistics literature \citep{huber1964robust}.  Recently, there has been a concentration of research effort attempting to argue that statistical procedures achieve the optimal dependence on $\epsilon$ in different statistical problems, thereby providing evidence of their robustness.  For instance, for fully observed data, \citet{chen2018robust} demonstrate the optimality in this sense of the Tukey median~\citep{tukey1975mathematics} for mean estimation, as well as a matrix depth estimator of a covariance matrix.  Since the Tukey median is computationally intractable, various alternatives have been considered in the both the statistics and theoretical computer science literature~\citep[see, e.g.,][and references therein]{diakonikolas_kane_2023}.  Other problems studied within this framework include linear regression \citep{bakshi2021robust,pensia2024robust}, nonparametric regression \citep{gao2020robust} and robust clustering \citep{liu2023robustly,jana2024general}.

Our realisable contamination models are related to several previous attempts to study restricted forms of missing not at random and biased sampling.  For instance, \citet{vardi1985empirical} introduced a biased sampling model and, under the assumption that the sampling mechanism is known, studied nonparametric estimation of the distribution function.  Under this oracle model,~\citet{gill1988large} studied classical asymptotics and efficiency guarantees for the nonparametric maximum likelihood estimator; see also~\citet{bickel1991large} for similar guarantees in a linear regression setting.  Later, \citet{aronow2013interval} and \citet{sahoo2022learning} introduced likelihood ratio constraints to perform estimation in situations where the sampling mechanism may be unknown.  In the causal inference literature, efforts to restrict unobserved confounding have led to the introduction of similar restrictions known as sensitivity conditions~\citep{rosenbaum87sensitivity,zhao2019sensitivity}.  As we show in our discussion following Proposition~\ref{prop:univariate-realisability}, our realisable contamination classes can be understood as generalisations of these notions.  In a different direction, and with a view towards computational efficiency,~\cite{daskalakis2018efficient} considered estimating population parameters in a biased sampling model induced by truncation to a known set;  \citet{kontonis2019efficient} and \citet{diakonikolas2024statistical} studied the computational and statistical consequences of the absence of knowledge of this truncation set.  Distributions obtained by truncation are missing not at random and as such can be captured when $\epsilon = 1$ by our realisable $\epsilon$-contamination classes.

\subsection{Notation} \label{sec:notation}
For $d\in\mathbb{N}$, we let $[d] \coloneqq \{1,\ldots,d\}$ and write $2^{[d]}$ for the power set of $[d]$. For $a,b\in\mathbb{R}$, we let $a \vee b \coloneqq \max\{a, b\}$ and $a \wedge b \coloneqq \min\{a, b\}$.  We also define $\log_{+}(x) \coloneqq \log(x) \vee 1$ for $x > 0$. If $I$ is an arbitrary index set, then for functions $f,g:I \rightarrow \mathbb{R}$, we write $f \gtrsim g$ if there exists a universal constant $c>0$ such that $f(i) \geq cg(i)$ for all $i \in I$, and write $f \lesssim g$ if there exists a universal constant $C>0$ such that $f(i) \leq Cg(i)$ for all $i \in I$.

For $S\subseteq [d]$, we define $\bm{1}_S \in \{0,1\}^d$ by $(\bm{1}_S)_j \coloneqq \mathbbm{1}_{\{j \in S\}}$; for $j\in[d]$, we write $e_j\in\mathbb{R}^d$ for the $j$th standard basis vector. We denote the unit Euclidean sphere in $\mathbb{R}^d$ by $\mathbb{S}^{d-1}$. 
The sets $\mathcal{S}^{d \times d}$, $\mathcal{S}^{d \times d}_+$ and $\mathcal{S}^{d \times d}_{++}$ denote the set of symmetric, symmetric positive semidefinite and symmetric positive definite matrices in $\mathbb{R}^{d\times d}$ respectively. For $A \in \mathbb{R}^{d \times d}$, we write $\| A \|_{\mathrm{op}}$ for its operator (spectral) norm and $\| A \|_{\infty}$ for its maximum absolute entry. Further, for $A\in\mathcal{S}_+^{d\times d}$, we let $\mathbf{r}(A) \coloneqq \tr(A)/\|A\|_{\mathrm{op}}$ denote the effective rank of $A$, with the convention that $0/0 \coloneqq 0$.
Given $(a_1, \ldots, a_d)^\top \in\mathbb{R}^d$, let $\mathrm{diag}(a_1, \ldots, a_d) \in \mathbb{R}^{d \times d}$ denote the diagonal matrix with entries $a_1, \ldots a_d$, and let $I_{d} \coloneqq \mathrm{diag}(1, \ldots, 1)$ denote the identity matrix in $\mathbb{R}^{d \times d}$. 

For a topological space $(\mathcal{X},\tau)$, we let $\mathcal{B}(\mathcal{X})$ denote the Borel $\sigma$-algebra of $\mathcal{X}$, and let $\mathcal{P}(\mathcal{X})$ denote the set of all probability measures on $\bigl(\mathcal{X},\mathcal{B}(\mathcal{X})\bigr)$.  For two measures $\mu_1,\mu_2$ on $\bigl(\mathcal{X},\mathcal{B}(\mathcal{X})\bigr)$, we write $\mu_1 \ll \mu_2$ if $\mu_1$ is absolutely continuous with respect to $\mu_2$.  We write $\lambda\in\mathcal{P}(\mathbb{R})$ for Lebesgue measure on $\mathbb{R}$. Given a collection $\mathcal{Q}$ of distributions, we define $\mathcal{Q}^{\otimes n} \coloneqq \{Q^{\otimes n}:\; Q \in \mathcal{Q}\}$. For a random variable $X$ taking values in $\mathcal{X}$, we let $\mathsf{Law}(X) \in \mathcal{P}(\mathcal{X})$ denote the distribution of~$X$, and $\mathrm{supp}(X) \subseteq \mathcal{X}$ denote the support of $X$, i.e.~the intersection of all closed sets $C \subseteq \mathcal{X}$ with $\mathbb{P}(X \in C) = 1$.  Given a topological space $(\mathcal{Z},\tau_\mathcal{Z})$, we write $C_{\mathrm{b}}(\mathcal{Z})$ for the set of bounded, continuous functions on $\mathcal{Z}$. 

For $\theta\in\mathbb{R}$ and $\sigma\in[0,\infty)$, we let $\Phi_{(\theta,\sigma)}(\cdot)$ and $\phi_{(\theta,\sigma)}(\cdot)$ denote the distribution and density functions of the $\mathsf{N}(\theta,\sigma^2)$ distribution respectively, with the shorthand that $\Phi \coloneqq \Phi_{(0,1)}$ and $\phi \coloneqq \phi_{(0,1)}$.

\section{Statistical setting} \label{sec:setup}

\subsection{The extended space \texorpdfstring{$\mathcal{X}_{\star}$}{X\_star} and classical models of missing data} \label{sec:extended-space-properties}

In this section, we introduce spaces that are convenient for models of missing data.  Let $d \in \mathbb{N}$ and, for $j \in [d]$, let $(\mathcal{X}_j, \tau_j)$ denote a topological space equipped with its Borel $\sigma$-algebra $\mathcal{B}(\mathcal{X}_j)$.  We use the symbol $\star$ to denote a missing element\footnote{When $\mathcal{X}_j = \mathbb{R}$, we adopt the conventions that $\star \cdot 0 \coloneqq 0 \eqqcolon 0 \cdot \star$, that $x \cdot \star \coloneqq \star \eqqcolon \star \cdot x$ for $x\in\mathcal{X}_\star \setminus \{0\}$ and that $\star + x = \star$ for $x \in \mathcal{X}_\star$.} and define, for each $j \in [d]$, the \emph{extended space} $\mathcal{X}_{j, \star} \coloneqq \mathcal{X}_j \cup \{\star\}$, equipped with the topology $\tau_{j,\star} \coloneqq \tau_j \cup \{A \cup \{\star\}:A \in \tau_j\}$ and corresponding Borel $\sigma$-algebra $\mathcal{B}(\mathcal{X}_{j, \star}) = \mathcal{B}(\mathcal{X}_j) \cup \{A \cup \{\star\} : A \in \mathcal{B}(\mathcal{X}_j)\}$.  Given a measure $\mu_j$ on $\bigl(\mathcal{X}_j,\mathcal{B}(\mathcal{X}_j)\bigr)$, we define the \emph{extended measure} $\mu_{j, \star}$ on $\bigl(\mathcal{X}_{j, \star},\mathcal{B}(\mathcal{X}_{j, \star})\bigr)$ by 
\[
\mu_{j, \star} (A) \coloneqq \mu_{j} (A) \quad \text{and} \quad \mu_{j, \star} (A \cup \{\star\}) \coloneqq \mu_j(A) + 1
\]
for $A \in \mathcal{B}(\mathcal{X}_j)$.  It is also convenient to define the product spaces $\mathcal{X} \coloneqq \prod_{j=1}^{d} \mathcal{X}_{j}$ and $\mathcal{X}_{\star} \coloneqq \prod_{j=1}^{d} \mathcal{X}_{j, \star}$, equipped with their product $\sigma$-algebras $\mathcal{B}(\mathcal{X}) \coloneqq \otimes_{j \in [d]} \mathcal{B}(\mathcal{X}_j)$ and $\mathcal{B}(\mathcal{X}_{\star}) \coloneqq \otimes_{j \in [d]} \mathcal{B}(\mathcal{X}_{j,\star})$ respectively.  

We will often reason about missing data via \emph{revelation vectors} $\omega \in \{0, 1\}^d$, which together with an element $x \in \mathcal{X}$ induce an element of the extended space $\mathcal{X}_{\star}$ through the binary operator $\ostar: \mathcal{X} \times \{0, 1\}^d \rightarrow \mathcal{X}_{\star}$, where the $j$th component of $x\ostar\omega$ is defined by
\begin{align*}
    (x\ostar\omega)_j \coloneqq \begin{cases}
        x_j \quad &\text{if } \omega_j = 1\\
        \star &\text{if } \omega_j = 0,
    \end{cases}
\end{align*}
for $j \in [d]$.  The following example gives a concrete illustration of the abstract notation.
\begin{example}
    Let $X \sim \mathsf{N}(0,1)$ and let $\Omega$ be a binary random variable satisfying $\mathbb{P}(\Omega = 1 \,|\, X = x) = g(x)$ for some Borel measurable function $g: \mathbb{R} \to [0,1]$. Then the $\mathbb{R}_\star$-valued random variable $X \ostar \Omega$ admits a density $f_\star: \mathbb{R}_{\star} \to [0,\infty)$ with respect to the extended Lebesgue measure $\lambda_{\star}$, where 
    \begin{flalign*}
        &&& f_\star(z) \coloneqq \begin{cases}
            g(z) \phi(z) &\text{if } z \in \mathbb{R}\\
            1 - \int_{\mathbb{R}} g(x) \phi(x) \, \mathrm{d}x \quad\qquad &\text{if } z = \star. 
        \end{cases} &&& \diamondsuit
    \end{flalign*}
\end{example}
An advantage of working with the extended measurable space $\mathcal{X}_{\star}$ is that it allows us to give succinct definitions of three classical models of missingness: missing completely at random (MCAR), missing at random (MAR) and missing not at random (MNAR).  For each definition, we will let $X \sim P \in \mathcal{P}(\mathcal{X})$ and $\pi \in \mathcal{P}(2^{[d]})$.  To define the \emph{MCAR distribution}, let~$\Omega$ be a random vector in $\{0,1\}^d$, independent of $X$, such that $\mathbb{P}(\Omega = \bm{1}_{S}) = \pi(S)$ for $S \subseteq [d]$, and define\footnote{When $d=1$, we may identify $\pi$ with $q \coloneqq \pi(\{1\})$, and write $\mathsf{MCAR}_{(q,P)}$ in place of $\mathsf{MCAR}_{(\pi,P)}$.}
\begin{align}\label{eq:MCAR-law}
    \mathsf{MCAR}_{(\pi, P)} \coloneqq \mathsf{Law}(X \ostar \Omega) \in \mathcal{P}(\mathcal{X}_{\star}).
\end{align}
Next, the \emph{family of MAR distributions} is the subset of $\mathcal{P}(\mathcal{X}_\star)$ given by\footnote{A formal definition of the conditional probabilities in~\eqref{eq:MAR-law} can be provided through the notion of disintegrations, whose existence is assumed here (and is guaranteed when $\mathcal{X}_j$ is a Polish space for each $j \in [d]$); see Section~\ref{sec:disintegration}.}
\begin{align}\label{eq:MAR-law}
    \mathsf{MAR}_{(\pi, P)} &\coloneqq \bigl\{\mathsf{Law}(X \ostar \Omega'): \; X \sim P, \, \mathbb{P}(\Omega' = \bm{1}_S) = \pi(S) \; \forall S\subseteq [d],\\
    &\quad \mathbb{P}(\Omega' = \omega\,|\, X=x) = \mathbb{P}(\Omega' = \omega\,|\, X\ostar\omega=x\ostar\omega) \;\forall \omega\in\{0,1\}^d, \text{$P$-a.e.~} x\in\mathcal{X}\bigr\}. \nonumber
\end{align}
The missing at random definition captures the intuitive idea that `missingness depends only on the observed variables' and is tailored toward likelihood-based methods for which it implies the so-called \emph{ignorability} of the missingness mechanism~\citep[][Section~5]{seaman2013what}.  Finally, we define the corresponding \emph{family of MNAR distributions} $\mathsf{MNAR}_{(\pi, P)} \subseteq \mathcal{P}(\mathcal{X}_{\star})$ as
\begin{align}\label{eq:MNAR-law}
        \mathsf{MNAR}_{(\pi, P)} &\coloneqq \bigl\{\mathsf{Law}(X \ostar \Omega'): X \sim P, \mathbb{P}(\Omega' = \bm{1}_S) = \pi(S) \; \forall S\subseteq [d]\bigr\} \subseteq \mathcal{P}(\mathcal{X}_{\star}),
\end{align}
According to these definitions, $\mathsf{MCAR}_{(\pi, P)} \in \mathsf{MAR}_{(\pi, P)} \subseteq \mathsf{MNAR}_{(\pi, P)}$.  When the distribution~$\pi$ is not fixed, we let 
\begin{align} \label{eq:def-MAR-MNAR}
\mathsf{MAR}_P \coloneqq \bigcup_{\rho \in \mathcal{P}(2^{[d]})} \mathsf{MAR}_{(\rho, P)} \quad \text{ and } \quad \mathsf{MNAR}_P \coloneqq \bigcup_{\rho \in \mathcal{P}(2^{[d]})} \mathsf{MNAR}_{(\rho, P)}.
\end{align}

\subsection{Models of departures from M(C)AR} \label{section:modeling-departures}

\subsubsection{Identifiability issues under the missing at random assumption} \label{sec:impossibility-MAR-mean}
The MAR assumption~\eqref{eq:MAR-law} is arguably the most widely adopted form of departure from the restrictive MCAR assumption in statistical practice.  However, in Example~\ref{Ex:Identifiability} below, we show that even in simple parametric scenarios, this can lead to identifiability issues that preclude consistent estimation; in particular, even though the MAR assumption holds, it is nonetheless impossible to identify the population mean. 

\begin{example}
\label{Ex:Identifiability} For $\theta = (\theta_1, \theta_2)^{\top} \in [0,1]^2$, define $P_{\theta} \in \mathcal{P}\bigl(\{0,1\}^3\bigr)$ such that for $X \coloneqq (X_1, X_2, X_3)^{\top} \sim P_{\theta}$, we have $X_1 \sim \mathsf{Ber}(\theta_1)$, $X_2 \sim \mathsf{Ber}(\theta_2)$ independently, and $X_3 = X_1 + X_2 \, (\mathrm{mod} \ 2)$.  We next specify a missingness mechanism via 
\[
\Omega \mid X = \begin{cases}
    (0, 0, 1) & \text{ if } X_3 = 1,\\
    (0, 1, 1) & \text{ with probability } 1/2 \text{ if } X_3 = 0, \\
    (1, 0, 1) & \text{ with probability } 1/2 \text{ if } X_3 = 0,
\end{cases}
\]
and note that if $X \sim P_{\theta}$, then $R_{\theta} \coloneqq \mathsf{Law}(X \ostar \Omega) \in \mathsf{MAR}_{P_\theta}$.  We have
\[
    R_{\theta}\bigl(\{(\star,\star,1)\}\bigr) = \theta_1(1-\theta_2) + (1-\theta_1)\theta_2,\quad R_{\theta}\bigl(\{(\star,1,0)\}\bigr) = R_{\theta}\bigl(\{(1,\star,0)\}\bigr) = \frac{\theta_1\theta_2}{2}
\]
and
\[
    R_{\theta}\bigl(\{(\star,0,0)\}\bigr) = R_{\theta}\bigl(\{(0,\star,0)\}\bigr) = \frac{(1-\theta_1)(1-\theta_2)}{2}.
\]
Thus, $R_{(\theta_1,\theta_2)} = R_{(\theta_2,\theta_1)}$ for all $(\theta_1,\theta_2)^\top \in [0,1]^2$, so that it is impossible to identify the parameter. This symmetry additionally implies that the population log-likelihood may not admit a unique global maximiser.  Indeed, letting $\theta^* \in [0,1]^2$ denote the true parameter, the population log-likelihood is given by $\mathcal{L}(\theta) \coloneqq \log\mathbb{E}_{\theta^*} \bigl\{R_{\theta}(\{X\ostar\omega\})\bigr\}$, which is symmetric in the components of $\theta$. 
\hfill $\diamondsuit$
\end{example}

\subsubsection{Huber-style models of departure from MCAR}\label{sec:departures-mcar}

Given the failure of the MAR assumption to ensure the tractability of the mean estimation problem, and in light of dual representation of the incompatibility index given by \citet[][Theorem~2]{berrett2023optimal}, it is natural to model departures from MCAR via a nonparametric, Huber-style contamination model.  In particular, given $P \in \mathcal{P}(\mathcal{X})$, $\epsilon \in [0,1]$ and $\pi \in \mathcal{P}(2^{[d]})$, we define the \emph{arbitrary $\epsilon$-contamination model}
\begin{align} \label{eq:P-arbitrary-contamination-model}
\mathcal{P}^{\mathrm{arb}}(P, \epsilon, \pi) \coloneqq \Bigl\{(1 - \epsilon)\mathsf{MCAR}_{(\pi, P)} +\epsilon Q: Q \in \mathcal{P}(\mathcal{X}_{\star}) \Bigr\}.
\end{align}
This family comprises mixture distributions in which one of the mixture components can be an arbitrary distribution on $\mathcal{X}_\star$.  One way to think about such distributions is via the following algorithm for drawing an observation $Z \sim (1 - \epsilon)\mathsf{MCAR}_{(\pi, P)} +\epsilon Q$, where $Q \in \mathcal{P}(\mathcal{X}_\star)$.  We first generate $W \sim \mathsf{Ber}(\epsilon)$; if $W=0$, we then draw $\Omega$ and $X$ independently, according to $\mathbb{P}(\Omega=\bm{1}_S \, | \, W=0) = \pi(S)$ for $S \subseteq [d]$ and $X \, | \,  \{W=0\} \sim P$, and finally set $Z \coloneqq X \ostar \Omega$.  On the other hand, if $W=1$, then we draw $Z\,|\, \{W=1\} \sim Q$.  The arbitrary $\epsilon$-contamination model allows us to interpolate in a continuous way between $\mathcal{P}^{\mathrm{arb}}(P, 0, \pi) = \mathsf{MCAR}_{(\pi, P)}$ and $\mathcal{P}^{\mathrm{arb}}(P, 1, \pi) = \mathcal{P}(\mathcal{X}_\star)$; see Figure~\ref{fig:arbitrary-contamination}.  One can also regard the model as a generalisation of the Huber contamination model to the missing data setting, where both missingness and outliers are present. This is common in many applications, including air quality data \citep{ottosen2019outlier,hua2024impact} and clinical data \citep{kwak2017statistical}, where missingness arises from communication failures or patient dropout respectively, while outliers occur due to sensor or equipment malfunction.

\definecolor{tikzblue}{RGB}{196, 211, 221}
\definecolor{tikzgreen}{RGB}{179, 255, 179}
\definecolor{tikzred}{RGB}{255, 120, 120}
\definecolor{tikzorange}{RGB}{255, 255, 120}
\begin{figure} 
      \centering
      \begin{tikzpicture}[scale=0.25]
        \draw[black] (0,0) ellipse (18cm and 12cm);
        \draw[black, dashed] (0,0) ellipse (9cm and 6cm);
        \draw[black, dashed] (0,0) ellipse (12cm and 8cm);
        \draw[black, dashed] (0,0) ellipse (15cm and 10cm);
        \draw[black] (6,0) ellipse (12cm and 5.7cm);
        \draw[black] (3,0) ellipse (15cm and 10cm);
        \draw[black] (0,0) ellipse (6cm and 4cm);
        \node[black,fill=white, inner sep=2pt] at (0,-5.8) {\footnotesize $\epsilon=0.25$};
        \node[black,fill=white, inner sep=2pt] at (0,-7.8) {\footnotesize $\epsilon=0.5$};
        \node[black,fill=white, inner sep=2pt] at (0,-9.8) {\footnotesize $\epsilon=0.75$};
        \node[black, fill=white, inner sep=2pt] at (0,-11.8) {\footnotesize $\epsilon=1$};


        \fill[black, tikzred, even odd rule, opacity=0.4] 
        (3,0) ellipse (15cm and 10cm) 
        (6,0) ellipse (12cm and 5.7cm);
        \fill[black, tikzorange, even odd rule, opacity=0.4] 
        (6,0) ellipse (12cm and 5.7cm) 
        (0,0) ellipse (6cm and 4cm);
        \draw[black, fill=tikzgreen, opacity=0.5] (0,0) ellipse (6cm and 4cm);
        \node[black] at (0,-0.3) {$\epsilon=0$};
        \shade[inner color=blue, outer color=white, even odd rule, opacity=0.5] ellipse (18cm and 12cm) ellipse (21.6cm and 14.4cm);
        \draw[-Latex] (12,12) -- (0,0.5) node[at start, above right] {$\mathsf{MCAR}_{(\pi, P)}$};
        \draw[-Latex] (19,6) -- (8,0.5) node[at start, above right] {$\mathsf{MAR}_P$};
        \draw[-Latex] (17.5,-10) -- (7,-8.5) node[at start, right] {$\mathsf{MNAR}_P$};
        \draw[-Latex] (-14,-10) -- (-10.5,-8.5) node[at start, left] {$\mathcal{P}(\mathcal{X}_\star)$};
       \end{tikzpicture}
       \caption{An illustration of the arbitrary $\epsilon$-contamination model $\mathcal{P}^{\mathrm{arb}}(P, \epsilon, \pi)$, which interpolates between $\mathsf{MCAR}_{(\pi,P)}$ and $\mathcal{P}(\mathcal{X}_\star)$.}
       \label{fig:arbitrary-contamination}
    \end{figure}
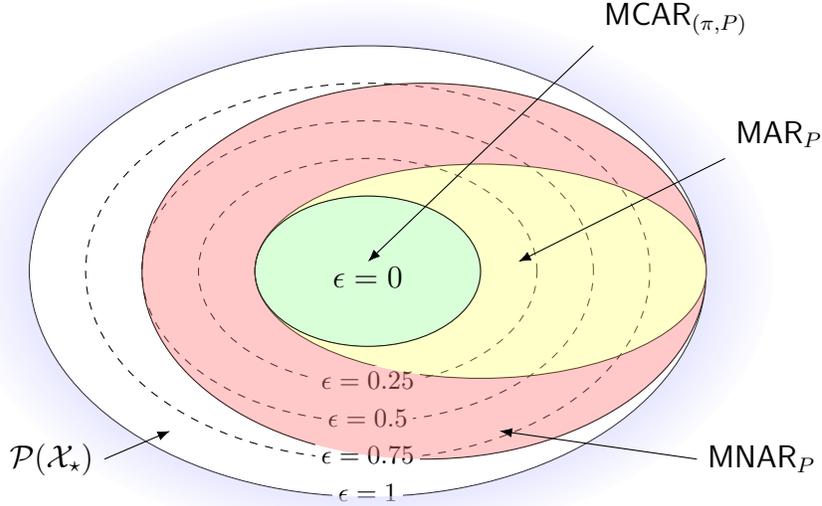

    \begin{figure} 
      \centering
      \begin{tikzpicture}[scale=0.25]
        \draw[black] (0,0) ellipse (18cm and 12cm);
        \draw[black, dashed] (0.75,0) ellipse (8.25cm and 5.5cm);
        \draw[black, dashed] (1.5,0) ellipse (10.5cm and 7cm);
        \draw[black, dashed] (2.25,0) ellipse (12.75cm and 8.5cm);
        \draw[black] (6,0) ellipse (12cm and 5.7cm);
        \draw[black] (3,0) ellipse (15cm and 10cm);
        \draw[black] (0,0) ellipse (6cm and 4cm);
        \node[black,fill=white, inner sep=2pt] at (0,-5.2) {\footnotesize $\epsilon=0.25$};
        \node[black,fill=white, inner sep=2pt] at (0,-6.7) {\footnotesize $\epsilon=0.5$};
        \node[black,fill=white, inner sep=2pt] at (0,-8.2) {\footnotesize $\epsilon=0.75$};
        \node[black, fill=white, inner sep=2pt] at (0,-9.7) {\footnotesize $\epsilon=1$};

        \fill[black, tikzred, even odd rule, opacity=0.4] 
        (3,0) ellipse (15cm and 10cm) 
        (6,0) ellipse (12cm and 5.7cm);
        \fill[black, tikzorange, even odd rule, opacity=0.4] 
        (6,0) ellipse (12cm and 5.7cm) 
        (0,0) ellipse (6cm and 4cm);
        \draw[black, fill=tikzgreen, opacity=0.5] (0,0) ellipse (6cm and 4cm);
        \node[black] at (0,-0.3) {$\epsilon=0$};
        \shade[inner color=blue, outer color=white, even odd rule, opacity=0.5] ellipse (18cm and 12cm) ellipse (21.6cm and 14.4cm);
        \draw[-Latex] (12,12) -- (0,0.5) node[at start, above right] {$\mathsf{MCAR}_{(\pi, P)}$};
        \draw[-Latex] (19,6) -- (8,0.5) node[at start, above right] {$\mathsf{MAR}_P$};
        \draw[-Latex] (17.5,-10) -- (7,-8.5) node[at start, right] {$\mathsf{MNAR}_P$};
        \draw[-Latex] (-14,-10) -- (-10.5,-8.5) node[at start, left] {$\mathcal{P}(\mathcal{X}_\star)$};
       \end{tikzpicture}
       \caption{An illustration of the realisable $\epsilon$-contamination model $\mathcal{R}(P, \epsilon, \pi)$, which interpolates between $\mathsf{MCAR}_{(\pi,P)}$ and $\mathsf{MNAR}_P$.}
       \label{fig:realisable-contamination}
    \end{figure}
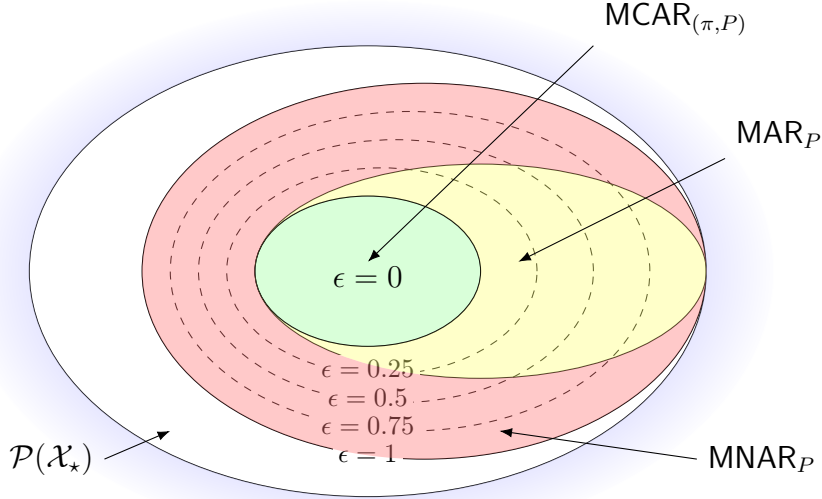

An attraction of the arbitrary contamination model is its generality.  Nevertheless, in many practical settings where one considers data arising from a particular distribution that are then subjected to some form of missingness, it may be preferable to seek classes to interpolate between $\mathsf{MCAR}_{(\pi, P)}$ and $\mathsf{MNAR}_P$.  To this end, a key definition in our framework is 
that of the \emph{realisable $\epsilon$-contamination model} 
\begin{align}
\label{eq:P-realisable-contamination-model}
 \mathcal{R}(P, \epsilon, \pi) \coloneqq \Bigl\{(1 - \epsilon)\mathsf{MCAR}_{(\pi, P)} +\epsilon \,Q : Q \in \mathsf{MNAR}_{P} \Bigr\};
\end{align}
see Figure~\ref{fig:realisable-contamination}.  In this model, the contamination mixture component is restricted to being a partially-observed version of $X \sim P$, where nevertheless the observation pattern may both be different from that in the uncontaminated component, and dependent on $X$.  Thus, the realisable contamination model $\mathcal{R}(P, \epsilon, \pi)$ represents a (still nonparametric) subclass of $\mathcal{P}^{\mathrm{arb}}(P, \epsilon, \pi)$, with the potential to yield improved rates of mean estimation.  On the other hand, noting that $\mathcal{R}(P, 1, \pi) = \mathsf{MNAR}_P$, the distribution $R_\theta$ in Example~\ref{Ex:Identifiability} belongs to $\mathsf{MNAR}_P$ but not to $\mathcal{R}(P, \epsilon, \pi)$ for any $\epsilon \in [0,1)$, so the inability to estimate its mean consistently does not contradict our minimax upper bounds established in Theorem~\ref{thm:one-dim-realisable-sample-mean-ub}.

\subsubsection{Regression with missing response}\label{sec:departures-mar}

A practical application of our framework for mean estimation with missingness is to regression problems where the response variable may be missing.  Here, we consider a $d$-dimensional (random) covariate vector and a real-valued response variable $Y$.  Given a disintegration $(P_{Y|x})_{x \in \mathbb{R}^d}$ of the joint distribution of $(X,Y)$ into conditional distributions on $\mathbb{R}$, and $q \in (0,1]$, we define the collection of \emph{missing at random (MAR) response distributions} as
\begin{align}
\label{Eq:MARRes}
\mathsf{MAR}^{\mathrm{Res}}_{(q, P_{Y|X})} \coloneqq \Bigl\{ \mathsf{Law}(Y\ostar\Omega\,|\,X) :\; Y\mid X \sim P_{Y \mid X},\, &\mathrm{supp}(\Omega) = \{0, 1\},\nonumber\\
& \Omega \indep Y \,|\, X \text{ and } \mathbb{P}(\Omega = 1 \,|\, X) \geq q \Bigr\},
\end{align}
and the corresponding collection of \emph{missing not at random (MNAR) response distributions} as
\begin{align}
\label{Eq:MNARRes}
\mathsf{MNAR}^{\mathrm{Res}}_{(P_{Y|X})} \coloneqq \Bigl\{ \mathsf{Law}(Y\ostar\Omega \, | \, X) :\;& Y \mid X \sim P_{Y \mid X} \text{ and }\,  \mathrm{supp}(\Omega) = \{0, 1\}\Bigr\}.
\end{align}
We note two crucial differences between the classes defined in~\eqref{Eq:MARRes} and~\eqref{Eq:MNARRes} above.  First, in the missing at random setting, we require that the missingness mechanism $\Omega$ is conditionally independent of the response $Y$ given the covariate vector $X$, whereas in the latter setting, the mechanism may depend arbitrarily on the response as well as the covariate.  Second, under the missing at random setting, we  require a lower bound on the probability of observing a particular response given its corresponding covariate vector $X$.  This latter condition implies a one-sided version of a so-called \emph{strict overlap condition}~\citep[see, e.g.,][Assumption 4(ii)]{hirano2003efficient}.  By contrast,  we impose no such assumption in the missing not at random setting.  
Given $\epsilon\in[0,1]$, we define our \emph{realisable $\epsilon$-contamination model with a missing response} as
\begin{align} \label{eq:realisable-response}
    \mathcal{R}^{\mathrm{Res}}(P_{Y \mid X},  \epsilon, q) \coloneqq \Bigl\{ (1 - \epsilon)\mathsf{MAR}^{\mathrm{Res}}_{(q, P_{Y \mid X})} + \epsilon Q :\; Q \in \mathsf{MNAR}^{\mathrm{Res}}_{(P_{Y \mid X})} \Bigr\}.
\end{align}

\subsection{Characterisation of realisability} \label{sec:realisability}

As we have argued previously, the realisable $\epsilon$-contamination model is often very natural in settings where our data are observed subject to missingness.  It is therefore of great interest to characterise distributions $R \in \mathcal{R}(P,\epsilon,\pi)$, and this is achieved in Theorem~\ref{cor:P-epsilon-pi-S-realisability-Farkas-form} below through integrals of bounded, continuous functions with respect to $R$.  For $f \in C_{\mathrm{b}}(\mathcal{X}_\star)$, we define $f_{\max}:\mathcal{X} \rightarrow \mathbb{R}$ by $f_{\max}(x) \coloneqq \max_{\omega\in \{0,1\}^d} f(x\ostar \omega)$.  Theorem~\ref{cor:P-epsilon-pi-S-realisability-Farkas-form} is stated for simplicity of exposition in the setting of Euclidean spaces, though the result in fact follows from the more general (and abstract) Theorem~\ref{Thm:AbstractVersion}.  

\begin{theorem}\label{cor:P-epsilon-pi-S-realisability-Farkas-form}
Fix $P \in \mathcal{P}(\mathbb{R}^d)$, $\epsilon\in (0,1]$, $\pi \in \mathcal{P}(2^{[d]})$. Let $R \in \mathcal{P}(\mathbb{R}^d_{\star})$, and define a signed measure on $\mathbb{R}^d_\star$ by $Q \coloneqq \epsilon^{-1}\{R - (1-\epsilon)\mathsf{MCAR}_{(\pi, P)}\}$. Then $R\in \mathcal{R}(P,\epsilon,\pi)$ if and only if $Q\in\mathcal{P}(\mathbb{R}^d_{\star})$ and 
    \begin{equation}
    \label{Eq:fmax}
    P(f_{\max}) \geq Q(f)
    \end{equation}
    for all $f \in C_{\mathrm{b}}(\mathbb{R}^d_\star)$.
\end{theorem}
Theorem~\ref{cor:P-epsilon-pi-S-realisability-Farkas-form} highlights the important fact, which inspires our minimum Kolmogorov distance estimator methodology in Section~\ref{sec:gaussian-realisable-model} below, that the tail behaviour of realisable distributions is controlled by the tails of $P$. Indeed, if we take $f(x \ostar \omega) = \mathbbm{1}_{\{\max_{j \in [d]} \omega_j |x_j| \geq K\}}$ for some $K \geq 0$, then $f_{\max}(x)=\mathbbm{1}_{\{\|x\|_\infty \geq K\}}$, and we see that
\begin{align*}
    \sum_{S \in 2^{[d]}} \mathbb{P}_R(\Omega=\bm{1}_S, \|X_S\|_\infty \geq K) &= (1-\epsilon) \sum_{S \in 2^{[d]}} \pi_S \mathbb{P}_P(\|X_S\|_\infty \geq K) + \epsilon Q(f)  \\
    &\leq (1-\epsilon) \sum_{S \in 2^{[d]}} \pi_S \mathbb{P}_P(\|X_S\|_\infty \geq K) + \epsilon P(f_{\max}) \\
    &\leq \mathbb{P}_P(\|X\|_\infty \geq K).
\end{align*}
(Although Theorem~\ref{cor:P-epsilon-pi-S-realisability-Farkas-form} is stated in terms of bounded, continuous functions, it is clear from the proof that the necessity of~\eqref{Eq:fmax} holds for any integrable function.)  

An important special case of Theorem~\ref{cor:P-epsilon-pi-S-realisability-Farkas-form}, and indeed the main content of its proof, concerns the setting where $\epsilon = 1$.  Here, the result states that a distribution $Q$ belongs to $\mathsf{MNAR}_P$ if and only if~\eqref{Eq:fmax} holds, and is a consequence of a generalised version of Farkas's lemma \citep{farkas1902theorie}, due to \citet{craven1977generalizations}.  An explanation of the relevance of this seemingly-unrelated lemma, which amounts to a proof of the theorem in the case where the underlying product space is finite, is provided before the proof of the full result in Section~\ref{sec:proof-general-realisable}. 

In the case where $d=1$, Proposition~\ref{prop:univariate-realisability} below provides a more explicit characterisation of realisability.  It is also convenient here to write $\mathcal{R}(P,\epsilon,q)$ in place of $\mathcal{R}(P,\epsilon,\pi)$ when $q \coloneqq \pi(\{1\})$. 
\begin{prop}\label{prop:univariate-realisability}
    Let $P \in \mathcal{P}(\mathbb{R})$ and assume that $P$ has density $p$ with respect to a Borel measure~$\mu$. Let $\epsilon \in [0,1]$, $\pi \in \mathcal{P}\bigl(\bigl\{\emptyset, \{1\}\bigr\}\bigr)$, and define $q \coloneqq \pi(\{1\})$. Then $R \in \mathcal{R}(P,\epsilon,q)$ if and only if $R \ll \mu_\star$ and there exists a Borel measurable function $m:\mathbb{R} \to [0,1]$ such that 
    \begin{align}\label{eq:radon-nikodym-realisable}
        \frac{\mathrm{d}R}{\mathrm{d}\mu_\star}(z) = \begin{cases}
            q(1-\epsilon) \cdot p(z) + \epsilon\cdot  m(z)p(z) \quad&\text{if }z\in\mathbb{R}\\
            1- q(1-\epsilon) - \epsilon\int_{\mathbb{R}} m(x)p(x) \,\mathrm{d}\mu(x) &\text{if }z=\star.
        \end{cases}
    \end{align}
\end{prop}
Proposition~\ref{prop:univariate-realisability} may be seen as a consequence of Theorem~\ref{cor:P-epsilon-pi-S-realisability-Farkas-form}, though we also give an independent proof in Section~\ref{sec:proof-prop-4}.  The function $m: \mathbb{R} \rightarrow [0,1]$ admits an interpretation as a missingness mechanism for the MNAR component.  More generally, Proposition~\ref{prop:univariate-realisability} reveals that univariate realisability is characterised via rejection sampling.  To see this in the extreme case when $\epsilon = 1$, consider a distribution $R \in \mathcal{P}(\mathbb{R}_{\star})$ such that $R \ll \mu_{\star}$, and for $Z \sim R$, let $g$ denote the conditional density with respect to $\mu$ of $Z$ given that $\{Z \neq \star\}$.  By Proposition~\ref{prop:univariate-realisability}, $R \in \mathsf{MNAR}_P$ if and only if $g(x)/p(x) \leq 1/R\bigl(\{\star\}\bigr)$ for $\mu$-almost all $x \in \mathbb{R}$.  Thus any $R \in \mathsf{MNAR}_P$ can be obtained via rejection sampling from~$P$.

From Proposition~\ref{prop:univariate-realisability}, we see that for $0\leq c_1 \leq c_2 \leq 1$, we have
\begin{align}
    c_1 \leq \mathbb{P}(Z\neq \star \,|\, X) \leq c_2 \quad \text{if and only if} \quad \mathsf{Law}(Z) \in \mathcal{R}\Bigl(P,\, c_2-c_1,\, \frac{c_1}{1+c_1-c_2}\Bigr). \label{Eq:KeyDisplay}
\end{align}
Thus, as long as $\mathbb{P}(Z\neq \star \,|\, X) > 0$ almost surely, the distribution of $Z$ belongs to a realisable contamination model for some $\epsilon<1$. As an example,~\eqref{Eq:KeyDisplay} can be used to provide a real data illustration of realisability in the context of income survey data, where non-responses typically exhibit U-shaped behaviour, i.e.~people with low or high income are less likely to respond to the survey \citep{rubin1983imputing, greenlees1982imputation,bollinger2019trouble}. In particular, \citet[][Figure~2]{bollinger2019trouble} plot the survey non-response rate against the income percentile, where social security administrative data provide true earnings even for survey non-respondents.  We can see from their figure that $0.6 \leq \mathbb{P}(Z\neq \star \,|\, X) \leq 0.85$, which implies from~\eqref{Eq:KeyDisplay} that their income data belong to a realisable contamination model with $\epsilon=0.25$ and $q=0.8$.  

Let us now compare our realisable class in~\eqref{eq:P-realisable-contamination-model} with related notions in the (primarily causal inference) literature, in the univariate case.  Let $Z \coloneqq X \ostar \Omega$, where $X \sim P \ll \mu$ and where $\Omega$ is a random variable taking values in $\{0,1\}$ that need not be independent of $X$.  Define $h: \mathbb{R} \rightarrow [0, 1]$ by
\[
h(x) \coloneqq \mathbb{P}(Z \neq \star \, \vert \, X = x).
\]
Proposition~\ref{prop:univariate-realisability} yields that $\mathsf{Law}(Z) \in \mathcal{R}(P,\epsilon,q)$ if and only if $q(1-\epsilon) \leq h(x) \leq q(1-\epsilon) + \epsilon$ for $\mu$-almost all $x\in\mathbb{R}$. The notion of \emph{$\Gamma$-biased sampling} of~\citet{sahoo2022learning} (see also~\citet{aronow2013interval}) can be stated as the condition on $h$ and $\overline{q} \coloneqq \mathbb{P}(Z \neq \star)$ that $\Gamma^{-1} \leq h(x)/\overline{q} \leq \Gamma$ for some $\Gamma \geq 1$ and $\mu$-almost all $x \in \mathbb{R}$.  In a similar spirit, the \emph{marginal sensitivity condition} of~\citet[Definition 1]{zhao2019sensitivity} asks that there exists $\Lambda \geq 1$ such that 
\[
\frac{1}{\Lambda} \leq \frac{h(x)}{1 - h(x)} \cdot \frac{1 - \overline{q}}{\overline{q}} \leq \Lambda
\]
for $\mu$-almost all $x \in \mathbb{R}$, while the classical \emph{sensitivity condition} of~\citet{rosenbaum87sensitivity} reads as 
\[
\frac{1}{\Lambda} \leq \frac{h(x_1)}{1 - h(x_1)} \cdot \frac{1 - h(x_2)}{h(x_2)} \leq \Lambda
\]
for $(\mu \otimes \mu)$-almost all $(x_1, x_2) \in \mathbb{R}^2$.  These classes all belong to $\mathcal{R}(P,\epsilon,q)$ for some $\epsilon \in [0,1)$ and $q \in (0,1]$.  Here, for simplicity of exposition, we have presented versions of these conditions without covariates.  Nevertheless the comparison remains valid when covariates are included; see Section~\ref{sec:regression-missing-response}.

\subsection{Minimax quantile framework} \label{sec:minimax-quantile}
In a traditional minimax analysis, the randomness in the loss function evaluated at our data is handled via a reduction to its expectation, namely the minimax risk.  As mentioned in the introduction, this minimax risk is infinite in the problems that we consider, so does not provide a meaningful way of comparing different statistical procedures.  We therefore adopt the minimax quantile framework of \citet{ma2024high}, which also offers the benefit of retaining all of the distributional information, e.g.~regarding tail behaviour, in the loss function.   

To introduce this paradigm in generality, we let $(\Theta, d)$ be a non-empty pseudo-metric space and for $\theta \in \Theta$, let $\mathcal{P}_{\theta}$ denote a family of probability measures on a measurable space $(\mathcal{Z}, \mathcal{C})$.  Further, let $g: [0, \infty) \rightarrow [0, \infty)$ denote an increasing function and define the loss $L: \Theta \times \Theta \rightarrow [0, \infty)$ by $L(\theta, \theta') \coloneqq g\bigl(d(\theta, \theta')\bigr)$.  Write $\widehat{\Theta}$ for the set of estimators of $\theta$, i.e.~the set of measurable functions from $\mathcal{Z}$ to $\Theta$. 
 For $\hat{\theta} \in \hat{\Theta}$, $P_\theta \in \mathcal{P}_\theta$ 
and a quantile level $\delta \in (0, 1]$, we write
\[
\mathrm{Quantile}\bigl(1 - \delta; P_{\theta}, L(\hat{\theta}, \theta)\bigr) \coloneqq \inf \Bigl\{ r \in [0, \infty):\; P_{\theta}\bigl\{L(\hat{\theta}, \theta) \leq r \bigr\} \geq 1 - \delta \Bigr\}, 
\]
and consider the \emph{minimax $(1 - \delta)$th quantile}, defined as
\begin{align}
    \label{eq:def-minimax-quantile}
    \mathcal{M}(\delta, \mathcal{P}_{\Theta}, L) \coloneqq \inf_{\hat{\theta} \in \widehat{\Theta}} \sup_{\theta \in \Theta} \sup_{P_{\theta} \in \mathcal{P}_{\theta}} \mathrm{Quantile}\bigl(1 - \delta; P_{\theta}, L(\hat{\theta}, \theta)\bigr),
\end{align}
where $\mathcal{P}_\Theta \coloneqq \{\mathcal{P}_\theta:\theta \in \Theta\}$. 
 If there exists $\hat{\theta} \in \hat{\Theta}$ such that with $P_{\theta}$-probability at least $1-\delta$, we have $L(\hat{\theta}, \theta) \leq \mathrm{UB}(\delta)$ for all $\theta\in\Theta$ and $P_{\theta}\in\mathcal{P}_{\theta}$, then $\mathcal{M}(\delta, \mathcal{P}_{\Theta}, L) \leq \mathrm{UB}(\delta)$.   For the squared Euclidean error loss $L(\hat{\theta},\theta) = \|\hat{\theta} - \theta\|_2^2$, we will slightly abuse notation by writing $\mathcal{M}\bigl(\delta, \mathcal{P}_{\Theta}, \| \cdot \|_2^2\bigr)$ in place of $\mathcal{M}(\delta, \mathcal{P}_{\Theta}, L)$.

\section{Mean estimation under arbitrary contamination} \label{sec:mean-estimation-arbitrary-contamination}

Throughout this section, we take $\mathcal{X} = \mathbb{R}^d$.  The set of distributions on $\mathbb{R}^d$ with mean vector $\theta \in \mathbb{R}^d$ and covariance matrix $\Sigma \in \mathcal{S}_+^{d \times d}$ is denoted 
 \begin{align}\label{eq:finite-covariance-set-of-distributions}
     \mathcal{P}(\theta, \Sigma) \coloneqq \Bigl\{ P \in \mathcal{P}(\mathbb{R}^d):\; \mathbb{E}_P(X) = \theta,\, \mathrm{Cov}_P (X) = \Sigma \Bigr\}.
 \end{align}
Given $\epsilon \in [0,1]$ and $\pi \in \mathcal{P}(2^{[d]})$, it is  convenient to define a specialised version of our arbitrary contamination model by
\begin{align}
\label{eq:arbitrary-contamination-model}
\mathcal{P}^{\mathrm{arb}}(\theta, \Sigma, \epsilon, \pi) \coloneqq \bigcup_{P \in \mathcal{P}(\theta, \Sigma)} \mathcal{P}^{\mathrm{arb}}(P,\epsilon, \pi).
\end{align}

The goal of this section is to provide upper and lower bounds on the minimax quantiles for multivariate mean estimation with respect to squared Euclidean error loss over the classes $\mathcal{P}^{\mathrm{arb}}(\theta, \Sigma, \epsilon, \pi)$.  With a view to the upper bound, the \textproc{Iterative\_Robust\_Mean} estimator of Algorithm~\ref{alg:robust-iterative-imputation} converts a robust mean estimation algorithm ALG for complete data into a robust mean estimator for data with missingness via iterative imputation.  In our missing data context, the behaviour of the \textproc{Iterative\_Robust\_Mean} estimator is governed by the matrix $\Sigma^{\mathrm{IPW}} \in \mathcal{S}^{d \times d}_{+}$, with entries 
\begin{align*}
    \bigl(\Sigma^{\mathrm{IPW}}\bigr)_{jk} \coloneqq  \frac{q_{jk}}{q_jq_k} \cdot \Sigma_{jk}, 
\end{align*}
for $j,k \in [d]$, where $q_{jk} \coloneqq \sum_{S \subseteq [d]: \{j, k\} \subseteq S} \pi(S)$ and $q_j \coloneqq q_{jj}$. Further define $q_{\min} \coloneqq \min_{j\in [d]} q_j$.  

Our performance guarantee for the \textproc{Iterative\_Robust\_Mean} estimator requires ALG to satisfy the following property: let $X_1,\ldots,X_n \overset{\mathrm{iid}}{\sim} (1-\epsilon)P+\epsilon Q$, where $P\in \mathcal{P}(\theta_0,\Sigma)$ and $Q\in \mathcal{P}(\mathbb{R}^d)$. Then there exist $\epsilon_{\max} \in (0,1/2)$, $a \in (0,1]$ and $C>0$ such that for any $\epsilon\in[0,\epsilon_{\max}]$, $\delta\in [e^{-an},1]$ and $n,d \in \mathbb{N}$, we have with probability at least $1-\delta$ that
\begin{align}
    \bigl\| \mathrm{ALG}(X_1,\ldots,X_n;\epsilon,\delta) - \theta_0\bigr\|_2^2 \leq C\biggl( \frac{\tr(\Sigma)}{n} + \frac{\|\Sigma\|_{\mathrm{op}}\log(1/\delta)}{n} + \|\Sigma\|_{\mathrm{op}}\epsilon\biggr). \label{eq:assumption-on-alg}
\end{align}
The above rate of convergence is known to be minimax optimal \citep[e.g.][]{lugosi21robust}.  The guarantee~\eqref{eq:assumption-on-alg} is satisfied, for instance, by the polynomial-time robust descent algorithm of \citet{depersin2022robust}, with $\epsilon_{\max} = 1/300$, $a = 1/180{,}000$ and $C = 3\times 10^{14}$; see also \citet{diakonikolas2020outlier}.

\begin{algorithm}[t]
\caption{\textproc{Iterative\_Robust\_Mean} for robust mean estimation with iterative imputation}\label{alg:robust-iterative-imputation}
\raggedright \textbf{Input:} Data $z_1,\ldots,z_n \in \mathbb{R}_\star^d$, contamination parameter $\epsilon \geq 0$, tolerance parameter $\delta > 0$, an algorithm ALG satisfying~\eqref{eq:assumption-on-alg} for some $\epsilon_{\max}\in(0,1/2)$, $a\in(0,1]$ and $C>0$\\
\textbf{Output:} An estimator $\hat{\theta}_{n}$ of $\theta_0$
\begin{algorithmic}[1]
\Function{Iterative\_Robust\_Mean}{$z_1,\ldots,z_n; \epsilon, \delta$}
\State $T \gets 1+\bigl\lceil \log_2\bigl\{\frac{C}{48C+2}\bigl(\mathbf{r}(\Sigma^{\mathrm{IPW}}) + 2\log(2ed/\delta) \bigr)\bigr\} \bigr\rceil$ and $M \gets \big\lceil \frac{2\lfloor n/T\rfloor\epsilon}{-\log(1-\epsilon_{\max})} \vee \log(2T/\delta) \big\rceil$
\For{$i \in [n]$ and $j \in [d]$}
   \State $\omega_{ij} \gets \mathbbm{1}_{\{z_{ij} \neq \star\}}$
\EndFor
\State Randomly select $T$ disjoint sets $(S^{(t)})_{t\in[T]} \subseteq [n]$ such that $|S^{(t)}| = \lfloor n/T \rfloor$ for $t\in[T]$
\For{$j \in [d]$}
    \State $I_j \gets \{i \in S^{(1)}:\omega_{ij} =1\}$
    \State $\hat{\theta}^{(1)}_j  \gets \mathrm{ALG}\bigl((z_{ij})_{i \in I_j};\frac{\epsilon}{q_j(1-\epsilon)},\, \frac{\delta}{4d}\bigr)$
\EndFor
\For{$t \in \{2,\ldots,T\}$}
\State Randomly select $M$ disjoint sets $(B_m^{(t)})_{m\in[M]} \subseteq S^{(t)}$ such that $|B_m^{(t)}| = \bigl\lfloor|S^{(t)}|/M\bigr\rfloor$
\For{$(m, j) \in [M] \times [d]$}
\State $\overbar{\omega}^{(t)}_{mj} \leftarrow \mathbbm{1}_{\{\sum_{i \in B_m^{(t)}} \omega_{ij} > 0\}}$
\vspace{0.5em}
\State $\begin{aligned}
    \bar{z}_{mj}^{(t)} \gets \overbar{\omega}^{(t)}_{mj} \cdot \frac{\sum_{i \in B_m^{(t)}} \omega_{ij} z_{ij}}{\sum_{i \in B_m^{(t)}} \omega_{ij}}  + (1 - \overbar{\omega}^{(t)}_{mj}) \cdot \hat{\theta}^{(t-1)}_j
\end{aligned}$
\EndFor
\State $\hat{\theta}^{(t)} \gets \mathrm{ALG}\bigl(\bar{z}^{(t)}_1,\ldots,\bar{z}^{(t)}_M; \epsilon_{\max}, \frac{\delta}{2T}\bigr)$
\EndFor
\State \textbf{return} $\hat{\theta}_n \gets \hat{\theta}^{(T)}$
\EndFunction
\end{algorithmic}
\end{algorithm}

\begin{theorem}\label{thm:robust-descent-iterative-imputation-ub}
    Let $\mathrm{ALG}$ satisfy~\eqref{eq:assumption-on-alg} for some $\epsilon_{\max} \in (0,1/2)$, $a \in (0,1]$, $C > 0$.
    Let $n\geq4$, $\epsilon \in [0,1/2)$, $\delta\in(0,1]$, $\pi\in\mathcal{P}(2^{[d]})$.  Let $Z_1, \ldots, Z_{n} \stackrel{\mathrm{iid}}{\sim} P \in \mathcal{P}^{\mathrm{arb}} \big(\theta_0, \Sigma, \epsilon, \pi \big)$ and define $\hat{\theta}_n \coloneqq \textproc{Iterative\_Robust\_Mean}(Z_1,\ldots,Z_n;\epsilon,\delta)$ from Algorithm~\ref{alg:robust-iterative-imputation}.  Let $T\coloneqq 1+\bigl\lceil \log_2\bigl\{\frac{C}{48C+2}\bigl(\mathbf{r}(\Sigma^{\mathrm{IPW}}) + 2\log(2ed/\delta) \bigr)\bigr\} \bigr\rceil$ as in the algorithm, and suppose that 
    \begin{align}
    \label{Eq:qmin}
        q_{\min} \geq \frac{4(192C+8)}{-\log(1-\epsilon_{\max})} \cdot \epsilon +  \Bigl(192C+\frac{16}{a}\Bigr) \frac{T\log\bigl(2e(T\vee d)/\delta\bigr)}{n}. 
    \end{align}
    Then, with probability at least $1 - \delta$, 
    \begin{align*}
        \|\hat{\theta}_n - \theta_0\|_2^2 
        \leq (288C+12)\biggl(\frac{T \tr(\Sigma^{\mathrm{IPW}})}{n} + \frac{T\|\Sigma^{\mathrm{IPW}}\|_{\mathrm{op}} \log(2T/\delta)}{n} + \|\Sigma^{\mathrm{IPW}}\|_{\mathrm{op}} \epsilon\biggr).
    \end{align*}
\end{theorem}
Although the choice of $T$ in Theorem~\ref{thm:robust-descent-iterative-imputation-ub} depends on the effective rank of $\Sigma^{\mathrm{IPW}}$, which will typically be unknown, a straightforward adaptive choice is available by replacing $\mathbf{r}(\Sigma^{\mathrm{IPW}})$ in the definition of $T$ with its upper bound of $d$.  Importantly, the dependence of $T$ on $\mathbf{r}(\Sigma^{\mathrm{IPW}})$ is only logarithmic, so the effect on the final upper bound is small.     

Since 
\[
T \lesssim \log d + \log_+ \log(1/\delta),
\]
Theorem~\ref{thm:robust-descent-iterative-imputation-ub} yields that, with $\mathcal{P}_{\Theta}^{\mathrm{arb}} \coloneqq \bigl\{ \mathcal{P}^{\mathrm{arb}}(\theta, \Sigma, \epsilon, \pi) : \theta\in\Theta \bigr\}$,
\begin{align}
\label{Eq:TwoTermDecomoposition}
\mathcal{M}\bigl(\delta, &\mathcal{P}_{\Theta}^{\mathrm{arb}}, \| \cdot \|_2^2\bigr) \nonumber \\
&\lesssim \bigl\{ \log{d} + \log_+\log(1/\delta)\bigr\} \biggl\{ \underbrace{ \frac{\tr(\Sigma^{\mathrm{IPW}})}{n}  + \frac{\| \Sigma^{\mathrm{IPW}} \|_{\mathrm{op}} \log(1/\delta)}{n}}_{\text{MCAR term}}\biggr\} + \underbrace{ \| \Sigma^{\mathrm{IPW}} \|_{\mathrm{op}} \epsilon}_{\text{MCAR departure}}. 
\end{align}
As indicated in~\eqref{Eq:TwoTermDecomoposition}, our upper bound decomposes into a sum of two distinct components: an MCAR term and an $\epsilon$-dependent term that captures the effect of departure from MCAR.  The first of these terms further decomposes as the sum of a risk component\footnote{Strictly speaking, this is a slight abuse of terminology, since the MCAR risk is infinite whenever $q_{\min} < 1$; however, it is the risk when $q_{\min} = 1$, and the terminology reflects the fact that the term does not depend on the quantile level $\delta$.} and a term that captures the dependence on the quantile level $\delta$.

In the strong contamination model, \citet[][Theorem~2]{hu2021robust} provide an estimator $\hat{\theta}^{\mathrm{HR}}$ satisfying the upper bound 
    \begin{align*}
        \|\hat{\theta}^{\mathrm{HR}} - \theta_0\|_2^2 \lesssim \frac{d\|\Sigma\|_{\mathrm{op}}\log d}{nq_{\min}} + \frac{d\|\Sigma\|_{\mathrm{op}}\log (1/\delta)}{nq_{\min}} + \frac{\|\Sigma\|_{\mathrm{op}}\epsilon}{q_{\min}}
    \end{align*}
with probability at least $1-\delta$, provided that $\delta\gtrsim de^{-cnq_{\min}}$ for an appropriately small $c > 0$.  If ALG satisfies~\eqref{eq:assumption-on-alg} under the strong contamination model (as the algorithms cited prior to the statement of Theorem~\ref{thm:robust-descent-iterative-imputation-ub} do), then our bound in Theorem~\ref{thm:robust-descent-iterative-imputation-ub} also holds in the strong contamination model, and this facilitates a comparison of our conclusion with that of \citet{hu2021robust}.  The improvements of our bound when $\delta \geq \exp(-e^d)$ arise from the facts that 
\[
\tr(\Sigma^{\mathrm{IPW}}) \leq d\|\Sigma^{\mathrm{IPW}}\|_{\mathrm{op}}, \quad \|\Sigma^{\mathrm{IPW}}\|_{\mathrm{op}} \leq \frac{\|\Sigma\|_{\mathrm{op}}}{q_{\min}}, 
\]
and 
\begin{align*}
\log\bigl(d\log(1/\delta)\bigr)\biggl\{\frac{\tr(\Sigma^{\mathrm{IPW}})}{n} &+ \frac{\| \Sigma^{\mathrm{IPW}} \|_{\mathrm{op}} \log(1/\delta)}{n}\biggr\} \\
&\leq 2\biggl\{\frac{d\|\Sigma\|_{\mathrm{op}}\log d}{nq_{\min}} + \frac{d\| \Sigma\|_{\mathrm{op}} \log(1/\delta)}{nq_{\min}}\biggr\}.
\end{align*}
These gains may be significant: for instance, when $\delta = e^{-d}$, we obtain that with probability at least $1-\delta$, 
\[
\| \hat{\theta}_n - \theta_0 \|_2^2 \lesssim  \frac{d\|\Sigma^{\mathrm{IPW}}\|_{\mathrm{op}}\log d}{n} + \| \Sigma^{\mathrm{IPW}} \|_{\mathrm{op}}\cdot \epsilon.
\]
By contrast, \citet{hu2021robust} obtain that with probability at least $1-\delta$,
\[
\|\hat{\theta}^{\mathrm{HR}} - \theta_0\|_2^2 \lesssim \frac{d^2\|\Sigma\|_{\mathrm{op}}}{nq_{\min}} + \frac{\| \Sigma \|_{\mathrm{op}}}{q_{\min}} \cdot \epsilon.
\]
As another example, to illustrate the effect of heterogeneous missingness across coordinates, if $d \geq 2$, $\Sigma = I_d + \bm{1}_{[d]} \bm{1}_{[d]}^\top$, $q_1 = 1/d$ and $q_j = 1$ for $j \geq 2$, then
\begin{align*}
\|\Sigma^{\mathrm{IPW}}\|_{\mathrm{op}} = \|I_d + \bm{1}_{[d]} \bm{1}_{[d]}^\top + (2d-2)e_1e_1^\top\|_{\mathrm{op}} &\leq 1 + \tr\bigl(\bm{1}_{[d]} \bm{1}_{[d]}^\top + (2d-2)e_1e_1^\top\bigr) \\
&= 3d-1 \leq d(d+1) = \frac{\|\Sigma\|_{\mathrm{op}}}{q_{\min}}.
\end{align*}
On the other hand, due to the sample splitting in Algorithm~\ref{alg:robust-iterative-imputation}, our condition on $\delta$ may be slightly stronger than that of \citet{hu2021robust}.  

The optimality of our procedure can be deduced from the following minimax lower bound. 
\begin{theorem}\label{thm:arbitrary-contamination-lb}
    Let $\Sigma \in \mathcal{S}^{d \times d}_{++}$ be diagonal, $\pi \in \mathcal{P}(2^{[d]}), \epsilon \in [0, 1]$, $\delta \in (0,1/4]$, $\Theta \coloneqq \mathbb{R}^d$ and $\mathcal{P}_\theta \coloneqq \mathcal{P}^{\mathrm{arb}}(\theta, \Sigma, \epsilon, \pi)^{\otimes n}$ for $\theta\in\Theta$. Then 
    \begin{align*}
        \mathcal{M}\bigl(\delta, \mathcal{P}_{\Theta} , \| \cdot \|_2^2\bigr) \begin{cases}
            \gtrsim \dfrac{\tr(\Sigma^{\mathrm{IPW}})}{n} + \dfrac{\| \Sigma^{\mathrm{IPW}}\|_{\mathrm{op}} \log(1/\delta)}{n} + \epsilon \| \Sigma^{\mathrm{IPW}} \|_{\mathrm{op}} \quad&\text{if } \epsilon < \frac{q_{\min}}{1+q_{\min}}\\
            = \infty \quad&\text{if } \epsilon \geq \frac{q_{\min}}{1+q_{\min}}.
        \end{cases}
    \end{align*}
\end{theorem}
From Theorem~\ref{thm:arbitrary-contamination-lb}, we see that up to multiplicative universal constants and when $\Sigma$ is diagonal, Algorithm~\ref{alg:robust-iterative-imputation} has optimal behaviour under departures from MCAR and, up to a logarithmic factor in $d$ and an iterated logarithmic factor in $1/\delta$, it adapts to the MCAR minimax quantile rate in settings where the MCAR term dominates.  In Theorem~\ref{thm:robust-descent-iterative-imputation-ub}, we imposed a condition of the form $q_{\min} \gtrsim \epsilon + T\log(d/\delta)/n$ for our upper bound, where the second part of this condition asks for the expected number of MCAR observations per coordinate to be at least poly-logarithmic in $d$ and $1/\delta$.  On the other hand, from Theorem~\ref{thm:arbitrary-contamination-lb}, we see that $q_{\min} > \epsilon/(1-\epsilon)$ is necessary to ensure finite error with high probability.

\section{Mean estimation under realisable contamination}\label{sec:realisable-mean-est}
  
In the arbitrary contamination setting of Section~\ref{sec:mean-estimation-arbitrary-contamination}, we saw that the contamination fraction $\epsilon$ has a severe effect on our ability to estimate a population mean.  The aim of this section, then, is to explore the potential benefits of restricting the form of contamination to MNAR observations from the same base distribution as our uncontaminated observations.  

\subsection{Gaussian realisable model} \label{sec:gaussian-realisable-model}

\subsubsection{Univariate case}

In this subsection, we consider Gaussian base distributions, and for $\theta \in \mathbb{R}$, as well as fixed $\sigma > 0$, $\epsilon \in [0,1)$ and $q \in (0,1]$, we write $\mathcal{R}(\theta) \coloneqq \mathcal{R}\bigl(\mathsf{N}(\theta, \sigma^2), \epsilon, q\bigr)$ as shorthand.  To gain intuition, recall the characterisation of univariate realisable distributions in Proposition~\ref{prop:univariate-realisability}: $R \in \mathcal{R}(\theta)$ if and only if both $R \ll \lambda_\star$ and the restriction $h: \mathbb{R} \rightarrow [0, \infty)$ of $\mathrm{d}R/\mathrm{d}\lambda_\star$ to $\mathbb{R}$ satisfies
\begin{align} 
\label{eq:definition-h-restriction}
h(x) \in \bigl[q(1 - \epsilon) \phi_{(\theta, \sigma)}(x),\, \{q (1 - \epsilon) + \epsilon\} \cdot \phi_{(\theta, \sigma)}(x)\bigr],
\end{align}
for $\lambda$-almost all $x$.  In Figure~\ref{fig:realisable-example-epsilon<1}(a), we plot a $\mathsf{N}(0,1)$-realisable $h$; on the other hand, in Figure~\ref{fig:realisable-example-epsilon<1}(b), we consider the same function $h$ and demonstrate that $h$ is not $\mathsf{N}(1/2, 1)$-realisable.  This suggests that it may be possible to identify the mean by checking whether the condition in~\eqref{eq:definition-h-restriction} is verified.  Indeed, if $R \in \mathcal{R}(\theta_0)$, then for any $\theta \neq \theta_0$ and for $|x-\theta_0|$ sufficiently large, we have
\begin{align*}
    h(x) \notin \bigl[q(1 - \epsilon) \phi_{(\theta, \sigma)}(x),\, \{q (1 - \epsilon) + \epsilon\} \cdot \phi_{(\theta, \sigma)}(x)\bigr].
\end{align*}

\begin{figure}[t]
\centering
\subfigure[\centering]
{{\includegraphics[width=0.49\textwidth]{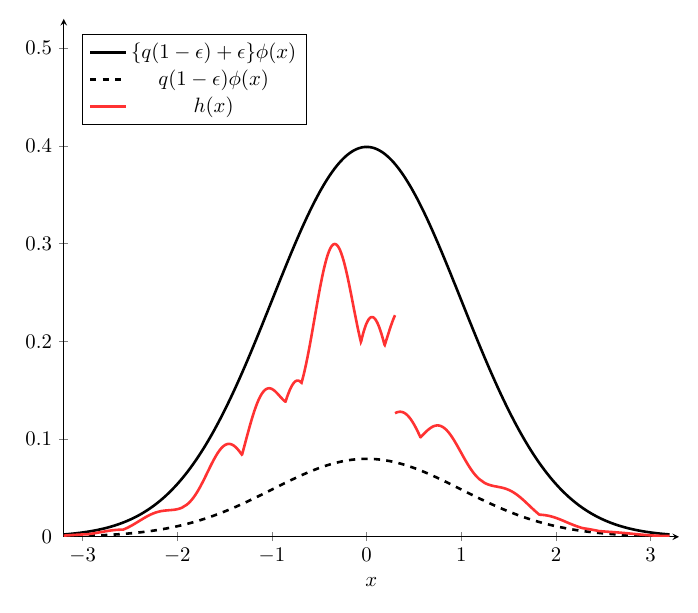}}}
\subfigure[\centering]{{\includegraphics[width=0.49\textwidth]{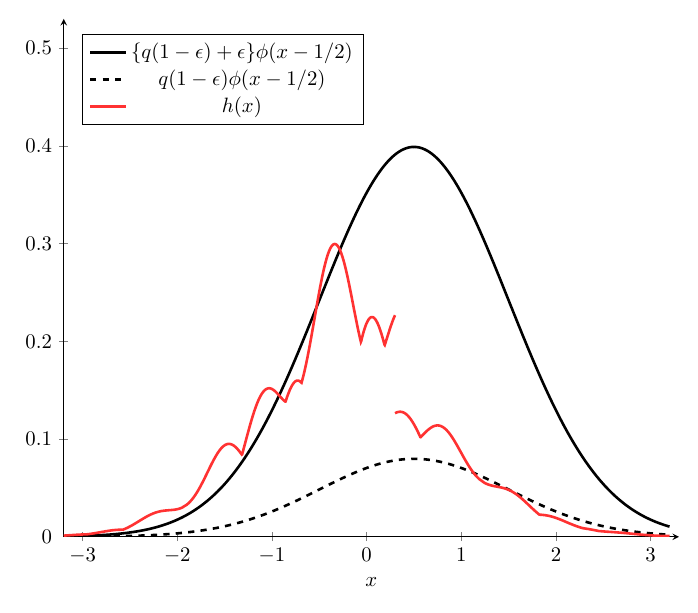}}}
\caption{An example of a Gaussian-realisable distribution.  Let $q = 1$ and $\epsilon = 0.8$.  Panel (a) plots (i) $\{q(1 - \epsilon) + \epsilon\} \cdot \phi(x)$ as a solid black curve, (ii) $q(1-\epsilon) \cdot \phi(x)$ as a dashed black curve and (iii) $\{q (1 - \epsilon) + m(x)\} \cdot \phi(x)$ as a solid red curve, for some $m:\mathbb{R} \to [0,1]$.  Note that the red curve is realisable by $\mathsf{N}(0, 1)$.  By contrast, panel (b) plots the red curve with no changes and uses $\phi(x - 1/2)$ in place of $\phi(x)$ for the two black curves.  In this case, the red curve is not realisable by $\mathsf{N}(1/2, 1)$.} 
\label{fig:realisable-example-epsilon<1}
\end{figure}
Motivated by this observation, and given data $Z_1, \ldots, Z_n \in \mathbb{R}_{\star}$, we define $\mathcal{D} \coloneqq \{i\in[n] : Z_i \neq \star\}$ and define an estimator $\hat{\theta}^{\mathrm{AE}}_n$ as
\begin{align}
\label{eq:def-theta-AE}
    \hat{\theta}^{\mathrm{AE}}_n(Z_1, \ldots, Z_n) \coloneqq \frac{1}{2} \cdot \Bigl(\max_{i\in\mathcal{D}} Z_i + \min_{i\in\mathcal{D}} Z_i\Bigr),
\end{align}
where we adopt the convention that $\hat{\theta}^{\mathrm{AE}}_n \coloneqq 0$ when $\mathcal{D} = \emptyset$.
Thus, $\hat{\theta}^{\mathrm{AE}}$ simply outputs the average of the extreme observed values.  In the realisable model $\mathcal{R}(\theta_0)$, its performance as an estimator of $\theta_0$ is summarised in the following theorem. 
\begin{theorem} 
\label{thm:univariate-gaussian-realisable-maxmin}
   Let $\theta_0\in\mathbb{R}$, $\epsilon \in [0,1)$, $q\in (0, 1]$, $\sigma > 0$, $n \in \mathbb{N}$, $\delta \in [4e^{-nq(1 -\epsilon)/8},1]$,  and consider $Z_1, \ldots, Z_n \overset{\mathrm{iid}}{\sim} R \in \mathcal{R}(\theta_0)$.   Then with probability at least $1 - \delta$, 
   \[
    \bigl( \hat{\theta}^{\mathrm{AE}}_n - \theta_0 \bigr )^2 \lesssim \frac{ \sigma^2 \log^2(8/\delta)}{\log{\bigl( nq(1-\epsilon) \bigr)}} + \frac{ \sigma^2 \log^2\bigl(1 + \frac{6\epsilon}{q(1 - \epsilon)}\bigr)}{\log{\bigl( nq(1-\epsilon) \bigr)}}.
   \]
\end{theorem}
We can interpret $nq(1-\epsilon)$ as the effective sample size from the MCAR component of~$R$: on average, a proportion $1-\epsilon$ of our observations come from this MCAR component, and a proportion $q$ of these are not missing.  The condition $\delta \geq 4e^{-nq(1 -\epsilon)/8}$ is therefore an effective sample size condition that asks for more MCAR observations for a higher confidence guarantee.   This condition is essentially necessary for the finiteness of the minimax quantile; see Theorem~\ref{thm:univariate-realisable-lb} below.  One of the interesting features of the conclusion of Theorem~\ref{thm:univariate-gaussian-realisable-maxmin} is that, if we consider $\epsilon$ and $q$ as fixed, then consistent mean estimation in the Gaussian realisable model $\mathcal{R}(\theta_0)$ is possible.  In fact, we can even achieve consistency when $\epsilon$ converges slowly to~1 and $q$ converges slowly to zero, a stark contrast with the conclusions drawn in the arbitrary contamination model of Section~\ref{sec:mean-estimation-arbitrary-contamination}.  Moreover, these results are achievable via a very simple estimator that does not require knowledge of $\epsilon$ (or $q$ or $\delta$).  On the other hand, the rate of convergence is only guaranteed to be logarithmic in the effective sample size (with the other problem parameters held fixed).  Nevertheless, our high-probability minimax lower bound in Theorem~\ref{thm:univariate-realisable-lb} below shows that this is the best that one can hope for within this model, at least when $\epsilon$ is a positive constant.
\begin{theorem}\label{thm:univariate-realisable-lb}
Let $\epsilon\in[0,1)$, $q\in(0,1]$, $\sigma>0$, $\delta \in (0,1/4]$ and $n\in\mathbb{N}$.
Suppose further that
\begin{align}
\label{eq:asm-eps-gaussian-realisable-lb}
    \log\biggl(1 + \frac{\epsilon}{q (1 - \epsilon)}\biggr) \leq \log\bigl( nq(1-\epsilon) \bigr).
\end{align}
Then, writing $\Theta \coloneqq \mathbb{R}$, as well as $\mathcal{P}_{\theta} \coloneqq \bigl\{R^{\otimes n}: \; R \in \mathcal{R}(\theta)\bigr\}$ for $\theta\in\Theta$, we have
\[
\mathcal{M}\bigl(\delta, \mathcal{P}_{\Theta}, | \cdot |^2\bigr)\begin{cases}
    \geq \dfrac{\sigma^2 \log(1/\delta)}{40nq(1-\epsilon)} + \dfrac{\sigma^2\log^2\bigl(1 + \frac{\epsilon}{q (1 - \epsilon)}\bigr)}{32 \log \bigl(nq(1-\epsilon)\bigr)} \quad&\text{if }\delta\geq\dfrac{\{1-q(1-\epsilon)\}^n}{2}\\
    = \infty \quad&\text{if }\delta<\dfrac{\{1-q(1-\epsilon)\}^n}{2}.
\end{cases} 
\]
\end{theorem}
Condition~\eqref{eq:asm-eps-gaussian-realisable-lb} is a mild effective sample size assumption.  When $q(1-\epsilon) \geq 1/2$, we have $\{1-q(1-\epsilon)\}^n \in [e^{-2nq(1-\epsilon)},e^{-nq(1-\epsilon)}]$, so the range of $\delta$ for which we have a finite minimax $(1-\delta)$th quantile guarantee in Theorem~\ref{thm:univariate-gaussian-realisable-maxmin} is almost optimal.  Comparing the bounds in Theorem~\ref{thm:univariate-realisable-lb} with those in  Theorem~\ref{thm:univariate-gaussian-realisable-maxmin}, we see that the second terms match up to a universal constant multiplicative factor.  On the other hand, the first term in the lower bound in Theorem~\ref{thm:univariate-realisable-lb} may be much smaller than the corresponding term in Theorem~\ref{thm:univariate-gaussian-realisable-maxmin}, both in terms of its dependence on the effective sample size and on the quantile level. 

To address the potential deficiency of the average of extremes estimator highlighted in the previous paragraph, we now introduce a minimum Kolmogorov distance estimator.  Let~$\hat{R}_n\in\mathcal{P}(\mathbb{R}_{\star})$ denote the empirical distribution of $Z_1,\ldots,Z_n \in \mathbb{R}_{\star}$, so that  
\begin{align*}
    \hat{R}_n(B) \coloneqq \frac{1}{n}\sum_{i=1}^n \mathbbm{1}_{\{Z_i \in B\}} \quad\text{for } B \in \mathcal{B}(\mathbb{R}_\star).
\end{align*}
Let $\mathcal{A} \coloneqq \{(-\infty,t] : t\in\mathbb{R}\}$ denote the set of all closed lower half intervals on $\mathbb{R}$. For $R_1,R_2 \in\mathcal{P}(\mathbb{R}_\star)$ and $\mathcal{Q} \subseteq \mathcal{P}(\mathbb{R}_\star)$, define
\begin{align*}
d_{\mathrm{K}}(R_1, R_2) \coloneqq \sup_{A \in \mathcal{A}}\; \bigl \lvert R_1(A) - R_2(A) \bigr \rvert \quad\text{and}\quad d_{\mathrm{K}}(R_1, \mathcal{Q}) \coloneqq \inf_{Q\in\mathcal{Q}}\; d_{\mathrm{K}}(R_1, Q)
\end{align*}
to be the Kolmogorov distance between $R_1$ and $R_2$, and the Kolmogorov distance between~$R_1$ and the set $\mathcal{Q}$ respectively. Then, the minimum Kolmogorov distance estimator $\hat{\theta}_n^{\mathrm{K}}$ for the Gaussian realisable class is defined as 
\begin{align*}
    \hat{\theta}^{\mathrm{K}}_n \coloneqq \sargmin_{\theta \in \mathbb{R}}\; d_{\mathrm{K}}\bigl(\hat{R}_n, \mathcal{R}(\theta)\bigr),
\end{align*}
where $\sargmin$ denotes the smallest element of the argmin set; this is well-defined since the function $\theta \mapsto d_{\mathrm{K}}\bigl(\hat{R}_n, \mathcal{R}(\theta)\bigr)$ is continuous with $d_{\mathrm{K}}\bigl(\hat{R}_n, \mathcal{R}(\theta)\bigr) \rightarrow 1$ as $|\theta| \rightarrow \infty$ and $d_{\mathrm{K}}\bigl(\hat{R}_n, \mathcal{R}(0)\bigr) < 1$.  We illustrate the Kolmogorov projection in Figure~\ref{fig:kolmogorov-projection}, and discuss its computation via a linear program in Section~\ref{sec:computation-of-kolmogorov-distance}.

\begin{figure}
    \begin{center}
        \begin{tikzpicture}
    \begin{scope}
        \clip (0,0) rectangle (4,4);
        \draw[draw=none, fill=gray!20] (0,0) circle(4);
    \end{scope}
    \begin{scope}
        \clip (4.1,0) rectangle (8.1,4);
        \draw[draw=none, fill=gray!30] (8.1,0) circle(4);
    \end{scope}

    \filldraw[black] (3,4) circle (2pt) node[anchor=south]{$\hat{R}_n$};
    \draw[dashed] (3,4) -- (2.398,3.2012);
    \draw[dashed] (3,4) -- (4.95285,2.4688995);
    \filldraw[black] (2.398,3.2012) circle (2pt);
    \filldraw[black] (4.95285,2.4688995) circle (2pt);
    \node at (1.7,1) {$\mathcal{R}(\theta_1)$};
    \node at (6.5,1) {$\mathcal{R}(\theta_2)$};
\end{tikzpicture}
    \end{center}
    \caption{\label{fig:kolmogorov-projection} Illustration of the Kolmogorov projection onto two distinct realisable sets.  The realisable sets are disjoint when $\theta_1 \neq \theta_2$, by Lemma~\ref{lemma:one-dim-kolmogorov-distance-realisable-sets}.}
\end{figure}
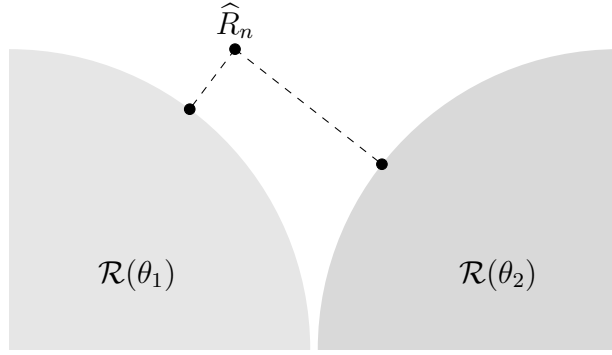

\begin{theorem}\label{thm:one-dim-kolmogorov-estimator}
    Let $\theta_0\in\mathbb{R}$, $\epsilon \in[0,1)$, $q\in(0,1]$, $\sigma > 0$, $\delta\in(0,1]$, $n\geq \frac{e}{q(1-\epsilon)}$ and consider $Z_1,\ldots, Z_n \overset{\mathrm{iid}}{\sim} R \in \mathcal{R}(\theta_0)$.  Let $\xi\in(0,1)$ be arbitrary and suppose that 
    \begin{align}
    \label{Eq:deltalowerbound}
        \delta \geq 4\exp\biggl\{ -\frac{\bigl\{ nq(1-\epsilon) \bigr\}^{1-\xi}}{6400} \biggr\},
    \end{align}
    and
    \begin{align}
    \label{Eq:bupperbound}
        \log\biggl( 1 + \frac{4\epsilon}{q(1-\epsilon)} \biggr) \leq \frac{7\xi}{128} \log\bigl( nq(1-\epsilon) \bigr).
    \end{align}
    Then with probability at least $1-\delta$,
    \begin{align}
    \label{Eq:MKDBound}
        (\hat{\theta}_n^{\mathrm{K}} - \theta_0)^2 \lesssim C_{n,q,\epsilon,\xi,\delta} \Biggl\{ \frac{\sigma^2\log(4/\delta)}{nq(1-\epsilon)} + \frac{\sigma^2\log^2\bigl( 1+\frac{4\epsilon}{q(1-\epsilon)} \bigr)}{\log\bigl(nq(1-\epsilon)\bigr)} \Biggr\},
    \end{align}
    where
    \begin{align*}
        C_{n,q,\epsilon,\xi,\delta} \coloneqq \begin{cases}
            1 \quad&\text{if } \log\bigl( 1 + \frac{4\epsilon}{q(1-\epsilon)} \bigr) \leq 2\sqrt{\frac{ \log(4/\delta)}{nq(1-\epsilon)}}\\
            \log\bigl( nq(1-\epsilon) \bigr) \quad&\text{if } 2\sqrt{\frac{ \log(4/\delta)}{nq(1-\epsilon)}} < \log\bigl( 1 + \frac{4\epsilon}{q(1-\epsilon)} \bigr) \leq 4\sqrt{\frac{\log(4/\delta)}{\{nq(1-\epsilon)\}^{1-\xi}}}\\
            1/\xi \quad&\text{if } \log\bigl( 1 + \frac{4\epsilon}{q(1-\epsilon)} \bigr) > 4\sqrt{\frac{\log(4/\delta)}{\{nq(1-\epsilon)\}^{1-\xi}}}.
        \end{cases}
    \end{align*}
\end{theorem}
The lower bound~\eqref{Eq:deltalowerbound} on $\delta$ and the effective sample size condition~\eqref{Eq:bupperbound} are both similar to those seen in Theorems~\ref{thm:univariate-gaussian-realisable-maxmin} and~\ref{thm:univariate-realisable-lb}.  Regarding $\xi$ as fixed, the main benefit of the minimum Kolmogorov distance estimator is that it is able to match both terms in the high-probability minimax lower bound of Theorem~\ref{thm:univariate-realisable-lb} up to a multiplicative universal constant, except in an intermediate parameter regime, where it may incur a multiplicative factor that is logarithmic in the effective sample size.  Even in this middle regime, which covers the phase transition where the two terms in the bound~\eqref{Eq:MKDBound} are equal, the rate remains polynomial in the effective sample size.

\subsubsection{Multivariate extension}\label{sec:kolmogorov-depth}
We now consider a simple multivariate extension of the Gaussian realisable model from the previous subsection, where for each observation, we either observe all coordinates simultaneously or none of them.  Thus, for $P \in \mathcal{P}(\mathbb{R}^d)$, $\epsilon \in [0,1)$ and $\pi \in \mathcal{P}(\{\emptyset,[d]\})$, we define 
\begin{align} \label{eq:realisable-all-or-nothing}
    \mathcal{R}_{\emptyset,[d]}(P,\epsilon,\pi) \coloneqq \Bigl\{ (1-\epsilon)\mathsf{MCAR}_{(\pi,P)} + \epsilon Q : Q\in\mathsf{MNAR}_{(\rho,P)},\, \rho \in\mathcal{P}\bigl(\{\emptyset,[d]\}\bigr) \Bigr\}.
\end{align}
For $\theta\in\mathbb{R}^d$, $\Sigma\in\mathcal{S}_{++}^{d\times d}$, we write $\mathcal{R}_{\emptyset,[d]}(\theta) \coloneqq \mathcal{R}_{\emptyset,[d]}\bigl(\mathsf{N}_d(\theta,\Sigma),\epsilon,\pi\bigr)$.  Given $Z_1,\ldots,Z_n \in \mathbb{R}_{\star}^d$ and $v\in\mathbb{S}^{d-1}$, let $\hat{\theta}_n^{\mathrm{K}}(v)$ denote the one-dimensional minimum Kolmogorov distance estimator based on $Z_1^{(v)},\ldots,Z_n^{(v)}$, where $Z_i^{(v)} \coloneqq v^\top Z_i \cdot \mathbbm{1}_{\{Z_i \in\mathbb{R}^d\}} + \star \cdot \mathbbm{1}_{\{Z_i \notin \mathbb{R}^d\}}$ for $i \in [n]$.  Let~$\mathcal{N}$ denote a $(1/4)$-net in Euclidean norm of $\mathbb{S}^{d-1}$ with $|\mathcal{N}|\leq 9^d$, which exists by, e.g.,~\citet[][Corollary~4.2.13]{vershynin2018high}.  We define the multivariate minimum Kolmogorov distance estimator $\hat{\theta}_n^{\mathrm{MK}}$ as 
\begin{align*}
    \hat{\theta}_n^{\mathrm{MK}} \coloneqq \sargmin_{\theta\in\mathbb{R}^d} \max_{v\in\mathcal{N}} \big( v^\top\theta - \hat{\theta}_n^{\mathrm{K}}(v) \big)^2,
\end{align*}
where sargmin here denotes the smallest element of the argmin set in the lexicographic ordering.
\begin{theorem}\label{thm:multivariate-kolmogorov-estimator}
    Fix $d, n \in \mathbb{N}$, $\theta_0 \in\mathbb{R}^d$, $\Sigma\in\mathcal{S}_{++}^{d\times d}$, $\epsilon \in [0,1)$, $\delta \in (0,1]$, $\pi \in \mathcal{P}(\{\emptyset,[d]\})$ and let $Z_1,\ldots,Z_n \stackrel{\mathrm{iid}}{\sim} R \in \mathcal{R}_{\emptyset,[d]}(\theta_0)$.  Let $\xi\in(0,1)$ be arbitrary, let $q\coloneqq \pi([d])$, and suppose that $nq(1-\epsilon) \geq e$,
    \begin{align*}
        \delta \geq 4\exp\biggl\{ d\log9 - \frac{\bigl\{nq(1-\epsilon)\bigr\}^{1-\xi}}{6400} \biggr\},
    \end{align*}
    and
    \begin{align*}
        \log\biggl( 1+\frac{4\epsilon}{q(1-\epsilon)} \biggr) \leq \frac{7\xi}{128}\log\bigl(nq(1-\epsilon)\bigr).
    \end{align*}
    Then with probability at least $1-\delta$,
    \begin{align*}
        \| \hat{\theta}_n^{\mathrm{MK}} - \theta_0\|_2^2 \lesssim C_{n,q,\epsilon,\xi,\delta/9^d}\biggl\{ \frac{\|\Sigma\|_{\mathrm{op}}\bigl(d+\log(4/\delta)\bigr)}{nq(1-\epsilon)} +  \frac{\|\Sigma\|_{\mathrm{op}} \log^2\bigl( 1+\frac{4\epsilon}{q(1-\epsilon)} \bigr)}{\log\bigl(nq(1-\epsilon)\bigr)} \biggr\},
    \end{align*}
    where $C_{n,q,\epsilon,\xi,\delta} > 0$ was defined in Theorem~\ref{thm:one-dim-kolmogorov-estimator}. 
\end{theorem}
Theorem~\ref{thm:multivariate-kolmogorov-estimator} reveals in particular that if we treat $q,\epsilon,\xi,\delta,\|\Sigma\|_{\mathrm{op}}$ as constants and if $\frac{d}{n^{1-\xi}} \to 0$ as $n \rightarrow \infty$, then $\hat{\theta}_n^{\mathrm{MK}}$ is a consistent estimator of $\theta_0$.  To facilitate comparisons with alternative estimators, we take $\Sigma = \sigma^2 I_d$ for simplicity.  A naive application of the univariate minimum Kolmogorov distance estimator in each coordinate would, via Theorem~\ref{thm:one-dim-kolmogorov-estimator} and a union bound, only yield a squared Euclidean error bound of order
\begin{align*}
    C_{n,q,\epsilon,\xi,\delta/d}\biggl\{\frac{d\sigma^2\log(4d/\delta)}{nq(1-\epsilon)} + \frac{d\sigma^2\log^2\bigl( 1+\frac{4\epsilon}{q(1-\epsilon)} \bigr)}{\log\bigl(nq(1-\epsilon)\bigr)}\biggr\}
\end{align*}
with probability at least $1-\delta$.  Thus, in the first term, the dimension and quantile terms would appear in a multiplicative as opposed to additive way, and the second term would be inflated by a factor of $d$.  Similarly, if we were to apply the average of extremes estimator in each coordinate, then Theorem~\ref{thm:univariate-gaussian-realisable-maxmin} and a union bound would give an upper bound on the minimax $(1-\delta)$th quantile for squared Euclidean error loss of order 
   \[
     \frac{ d \sigma^2 \log^2(8d/\delta)}{\log{\bigl( nq(1-\epsilon) \bigr)}} + \frac{ d \sigma^2\log^2\bigl(1 + \frac{6\epsilon}{q(1 - \epsilon)}\bigr)}{\log{\bigl( nq(1-\epsilon) \bigr)}}.
   \]
In fact, a high-probability minimax lower bound is available in this setting: letting $\mathcal{P}_{\theta} \coloneqq \bigl\{R^{\otimes n}:\; R \in \mathcal{R}_{\emptyset,[d]}\bigl(\mathsf{N}_d(\theta,\sigma^2 I_d),\epsilon,\pi\bigr)\bigr\}$, we have by combining Proposition~\ref{prop:arb-mean-MCAR-lb} and Theorem~\ref{thm:univariate-realisable-lb} that
\[
\mathcal{M}(\delta, \mathcal{P}_{\Theta}, \| \cdot \|_2^2) \gtrsim \frac{\sigma^2 \bigl(d + \log(1/\delta) \bigr)}{nq(1 - \epsilon)} + \frac{\sigma^2\log^2\bigl(1 + \frac{\epsilon}{q (1 - \epsilon)}\bigr)}{\log \bigl(nq(1-\epsilon)\bigr)}.
\]
Thus, up to the logarithmic factor in the effective sample size in the middle regime of the bound in Theorem~\ref{thm:multivariate-kolmogorov-estimator}, the multivariate minimum Kolmogorov distance estimator is minimax rate optimal in this setting.   

\subsubsection{Computing the Kolmogorov distance between the empirical distribution and a realisable set} \label{sec:computation-of-kolmogorov-distance}
   
Let $Z_1,\ldots,Z_n \in \mathbb{R}_{\star}$, let $m \coloneqq \sum_{i=1}^n \mathbbm{1}_{\{Z_i \neq \star\}}$ and let $-\infty\eqqcolon Z_{(0)} < Z_{(1)} \leq \cdots \leq Z_{(m)} < Z_{(m+1)} \coloneqq \infty$ denote the ordered observed data.  Further let $\hat{R}_n$ be the empirical distribution of $Z_1,\ldots,Z_n$ and let $P \in \mathcal{P}(\mathbb{R})$ be such that $P \ll \lambda$.  The following lemma and subsequent discussion provide an efficient way of computing the Kolmogorov distance between $\hat{R}_n$ and the realisable set $\mathcal{R}(P,\epsilon,q)$ via linear programming.

\begin{lemma} \label{lemma:compute-kolmogorov-distance}
Let $\epsilon \in [0,1)$ and $q \in [0,1]$.   Writing $V_0 \coloneqq 0$ and letting $\mathcal{V}$ denote the set of $(V_1,\ldots,V_{m+1})^\top \in [0,1]^{m+1}$ such that
    \begin{align}
        q(1-\epsilon)\cdot P\bigl((Z_{(i)},Z_{(i+1)})\bigr) \leq V_{i+1} - V_{i} \leq \{q(1-\epsilon)+\epsilon\} \cdot  P\bigl((Z_{(i)},Z_{(i+1)})\bigr) \label{eq:kolmogorov-projection-constraints}
    \end{align}
    for all $i\in \{0\} \cup [m]$, we have
    \begin{align} \label{eq:kolmogorov-distance-optimisation-prob}
    d_{\mathrm{K}}\bigl(\hat{R}_n,\mathcal{R}(P,\epsilon,q)\bigr) = \inf_{(V_1, \ldots, V_{m+1})^{\top} \in \mathcal{V}}\, \max_{i \in \{0\} \cup [m]}\; \Bigl\{\Bigl \lvert \frac{i}{n} - V_i \Bigr \rvert \vee \Bigl \lvert \frac{i}{n} - V_{i+1} \Bigr \rvert \Bigr\}.
    \end{align}
\end{lemma}

\noindent We can now rewrite the optimisation problem~\eqref{eq:kolmogorov-distance-optimisation-prob} as the following linear program:
\begin{align*}
    &\text{minimise} && \;\;t &&&\\
    &\text{subject to} && -t \leq \frac{i}{n} - V_i \leq t, \quad &&&i \in \{0\} \cup [m] \phantom{.}\\
    & &&-t \leq \frac{i}{n} - V_{i+1} \leq t, \quad &&&i \in \{0\} \cup [m] \phantom{.}\\
    & && q(1-\epsilon)\cdot P\bigl((Z_{(i)},Z_{(i+1)})\bigr) \leq V_{i+1} - V_{i} &&&\\
    & && \qquad \qquad \qquad \leq \{q(1-\epsilon)+\epsilon\} \cdot  P\bigl((Z_{(i)},Z_{(i+1)})\bigr), \quad &&&i\in\{0\}\cup [m].
\end{align*}
This can be solved efficiently using standard software, e.g.~\texttt{lpSolve} \citep{lpsolve} in \texttt{R}.

\subsection{Nonparametric realisable models} \label{sec:nonparametric-realisable}

\subsubsection{Univariate case}\label{sec:one-dim-nonparametric-realisable-model}

We now broaden our scope from the Gaussian realisable setting of Section~\ref{sec:gaussian-realisable-model} and seek to determine the minimax quantiles for mean estimation, again over realisable classes, but now with nonparametric families of base distributions, subject only to moment or tail decay conditions.   To this end, for $\theta \in \mathbb{R}$, $\sigma > 0$ and $r \geq 2$, we define the class of distributions $\mathcal{P}_{L^r}(\theta, \sigma^2)$ with a finite $r$th moment:
\begin{align}\label{eq:finite-moment-distribution}
    \mathcal{P}_{L^r}(\theta, \sigma^2) \coloneqq \Bigl\{ P\in \mathcal{P}(\mathbb{R}) : \mathbb{E}_P(X) = \theta,\,  \mathbb{E}_P \bigl(|X-\theta|^r\bigr) \leq \sigma^r \Bigr\}.
\end{align}
Similarly, for $r \geq 1$, we consider tail decay conditions specified by Orlicz norms with Orlicz functions $\psi_r: t \mapsto e^{t^r} - 1$, and define  
\begin{align}\label{eq:sub-weibull-distribution}
    \mathcal{P}_{\psi_r}(\theta, \sigma^2) \coloneqq \Bigl\{ P\in \mathcal{P}(\mathbb{R}) : \mathbb{E}_P(X) = \theta,\, \mathbb{E}_P \bigl\{\psi_r( |X-\theta|/\sigma)\bigr\} \leq 1 \Bigr\}.
\end{align}
Thus, if $X \sim P$ and we write $\| X \|_{\psi_r} \coloneqq \inf \bigl\{t > 0: \mathbb{E} \bigl( \psi_r(\lvert X \rvert/t)\bigr) \leq 1\bigr\}$, then $P \in \mathcal{P}_{\psi_r}(\theta, \sigma^2)$ if and only if $\mathbb{E}(X) = \theta$ and $\|X-\theta\|_{\psi_r} \leq \sigma$.  We also remark that $\mathcal{P}_{\psi_1}(\theta, \sigma^2)$ and $\mathcal{P}_{\psi_2}(\theta, \sigma^2)$ correspond to classes of sub-exponential and sub-Gaussian distributions with mean $\theta$ respectively.

Upper bounds on the minimax quantiles for mean estimation over realisable classes with base distributions belonging to the classes in~\eqref{eq:finite-moment-distribution} and~\eqref{eq:sub-weibull-distribution} are provided in Theorem~\ref{thm:one-dim-realisable-sample-mean-ub} below.  In the general setting of Theorem~\ref{thm:one-dim-realisable-sample-mean-ub}(a) where we only have a moment bound on the base distribution, and analogously to Section~\ref{sec:mean-estimation-arbitrary-contamination}, we assume the existence of an algorithm ALG2 for univariate mean estimation having the property that if $X_1,\ldots,X_n \overset{\mathrm{iid}}{\sim} P\in \mathcal{P}_{L^2}(\theta_0,\sigma^2)$, then there exist $a \in (0,1]$, $c \geq 1$ and $C>0$ such that for any $n \in \mathbb{N}$ and $\delta\in [ce^{-an},1]$, we have with probability at least $1-\delta$ that
\begin{align}
    \bigl(\mathrm{ALG2}(X_1,\ldots,X_n;\delta) - \theta_0\bigr)^2 \leq C\frac{\sigma^2\log(e/\delta)}{n}. \label{eq:assumption-on-alg-heavy-tail-univariate}
\end{align} 
For instance, when the median-of-means estimator is employed with $\lceil \log(1/\delta) \rceil$ blocks,~\eqref{eq:assumption-on-alg-heavy-tail-univariate} is satisfied with $a = 1/2$, $c=e$ and $C = 24e$ \citep[Proposition~1]{lerasle2011robust}.   Moreover, if $\sigma \in [\sigma_1,\sigma_2]$, for some $\sigma_2 \geq \sigma_1 > 0$, where $\sigma_2$ is known, then a $\delta$-independent Lepski-type version of the median-of-means estimator satisfies~\eqref{eq:assumption-on-alg-heavy-tail-univariate} with $a=1/2$, $c=4e$ and $C=600\sigma_2^2/\sigma_1^2$ \citep[Theorem~3.2(1)]{devroye2016subgaussian}; see also~\citet{minsker2021robust} for an alternative, asymptotically efficient estimator.  On the other hand, in the more specialised setting of Theorem~\ref{thm:one-dim-realisable-sample-mean-ub}(b), the sample mean of the observed data suffices to match the minimax lower bound in Theorem~\ref{thm:nonparametric-realisable-model-lb}(b) below, and this estimator is $\delta$-independent.


\begin{theorem}\label{thm:one-dim-realisable-sample-mean-ub}
    Let $\theta_0 \in \mathbb{R}$, $\epsilon\in[0,1)$, $q\in(0,1]$ and $\sigma > 0$.
    \begin{itemize}
        \item[(a)] Let $r \geq 2$, $P\in \mathcal{P}_{L^r}(\theta_0, \sigma^2)$ and $Z_1,\ldots,Z_n \overset{\mathrm{iid}}{\sim} R \in \mathcal{R}(P, \epsilon, q)$. 
        Let $\mathcal{D} \coloneqq \{i\in[n] : Z_i \neq \star\}$ and $\hat{\theta}_n \coloneqq \mathrm{ALG2}\bigl((Z_i)_{i\in\mathcal{D}};\delta\bigr)$, where $\mathrm{ALG2}$ satisfies~\eqref{eq:assumption-on-alg-heavy-tail-univariate}. Then, for any $n \in \mathbb{N}$ and $\delta \geq 2c\exp\bigl(-anq(1-\epsilon)/8\bigr)$, we have with probability at least $1-\delta$ that 
        \begin{align}
        \label{eq:MoM}
        \bigl(\hat{\theta}_n \!-\! \theta_0 \bigr)^2 \leq 4C \cdot \frac{\sigma^2 \log(2e/\delta) }{nq(1-\epsilon)} + (C\!+\!2)\sigma^2 \cdot \biggl\{\biggl( \frac{\epsilon}{q(1-\epsilon)} \biggr)^2 \wedge \biggl( \frac{\epsilon}{q(1-\epsilon)} \biggr)^{2/r} \biggr\}.
        \end{align}
        
        \item[(b)] Let $r\geq 1$, $P\in \mathcal{P}_{\psi_r}(\theta_0, \sigma^2)$, $R \in \mathcal{R}(P, \epsilon, q)$ and $Z_1,\ldots, Z_n \overset{\mathrm{iid}}{\sim} R$.  Let $\mathcal{D} \coloneqq \{i\in[n] : Z_i\neq\star\}$ and $\hat{\theta}_n \coloneqq |\mathcal{D}|^{-1}\sum_{i\in\mathcal{D}}Z_i$. Then, for any $n \in \mathbb{N}$ and $\delta\geq 8\exp\bigl(-nq(1-\epsilon)/8\bigr)$, we have with probability at least $1-\delta$ that
        \begin{align*}
        (\hat{\theta}_n - \theta_0)^2 \lesssim \frac{\sigma^2\log(8/\delta)}{nq(1-\epsilon)} + 
        \sigma^2 \cdot \biggl\{ \biggl( \frac{\epsilon}{q(1-\epsilon)} \biggr)^2 \;\wedge\; \log^{2/r} \biggl( 2 + \frac{2\epsilon}{q(1-\epsilon)} \biggr) \biggr\}.
        \end{align*}
    \end{itemize}
\end{theorem}
In both parts of Theorem~\ref{thm:one-dim-realisable-sample-mean-ub}, the first term in the bound reflects the error incurred in estimating the mean of a distribution $P$ belonging to either of the classes in~\eqref{eq:finite-moment-distribution} or~\eqref{eq:sub-weibull-distribution} based on a sample of size $\lceil n(1-\epsilon) \rceil$ from $\mathsf{MCAR}_{(q,P)}$.  The second terms arise from the contamination present in distributions $R \in \mathcal{R}(P,\epsilon,q)$.  When the \emph{effective contamination level} $\kappa \coloneqq \frac{\epsilon}{q(1-\epsilon)}$ is small in the sense that $\kappa \leq 1$, in both parts of the theorem, the error incurred from the contamination is at most of order $\sigma^2 \kappa^2$, which turns out to be minimax optimal; see Theorem~\ref{thm:nonparametric-realisable-model-lb} below.  Moreover, this rate is a substantial improvement on the corresponding term in the lower bound in Theorem~\ref{thm:arbitrary-contamination-lb} over arbitrary (non-realisable) contaminations of $\mathcal{P}_{L^2}(\theta_0,\sigma)$, which is of order $\sigma^2\kappa$.  On the other hand, when $\kappa > 1$, the contribution to the error from the contamination term depends on the tail behaviour of $P$: when $P \in \mathcal{P}_{L^r}(\theta_0,\sigma^2)$, it is at most of order $\sigma^2 \kappa^{2/r}$, whereas when $P \in \mathcal{P}_{\psi_r}(\theta_0,\sigma^2)$ the bound can be improved to order $\sigma^2\log^{2/r}(2+2\kappa)$.  Again, these bounds represent a stark contrast with the lower bound in the arbitrary contamination model setting of Theorem~\ref{thm:arbitrary-contamination-lb}, which is infinite in this regime.  Finally, we remark that an attractive feature of the methods employed in Theorem~\ref{thm:one-dim-realisable-sample-mean-ub} is that they do not require knowledge of $\kappa$. 


Theorem~\ref{thm:nonparametric-realisable-model-lb} provides a complementary lower bound on the minimax quantiles in this realisable setting:
\begin{theorem} \label{thm:nonparametric-realisable-model-lb}
    Let $\epsilon\in[0,1)$, $q\in(0,1]$, $\sigma>0$ and $\delta\in(0,1/4]$. 
    \begin{enumerate}
        \item[(a)] Let $r\geq 2$, $\Theta\coloneqq\mathbb{R}$ and $\mathcal{P}_{\theta} \coloneqq \bigl\{ R^{\otimes n} : R\in\mathcal{R}(P,\epsilon,q),\, P\in\mathcal{P}_{L^r}(\theta,\sigma^2) \bigr\}$ for $\theta\in\Theta$. Then
        \begin{align*}
            \mathcal{M}(\delta,\mathcal{P}_{\Theta},|\cdot|^2) \begin{cases}
                \gtrsim \dfrac{\sigma^2\log(1/\delta)}{nq(1-\epsilon)} + \sigma^2 \cdot \biggl\{ \biggl(\dfrac{ \epsilon}{q(1-\epsilon)} \biggr)^2 \;\wedge\; \biggl(\dfrac{ \epsilon}{q(1-\epsilon)} \biggr)^{2/r} \biggr\}\\
                \hspace{7.7cm}\text{if }\delta\geq e^{-nq(1-\epsilon)/2}\\
                =\infty \qquad\text{if }\delta<\frac{(1-q(1-\epsilon))^n}{2}.
            \end{cases}
        \end{align*}
        \item[(b)] Let $r\geq 1$, $\Theta\coloneqq\mathbb{R}$ and $\mathcal{P}_{\theta} \coloneqq \bigl\{ R^{\otimes n} : R\in\mathcal{R}(P,\epsilon,q),\, P\in\mathcal{P}_{\psi_r}(\theta,\sigma^2) \bigr\}$ for $\theta\in\Theta$. Then
        \begin{align*}
            \mathcal{M}(\delta,\mathcal{P}_{\Theta},|\cdot|^2) \begin{cases}
                \gtrsim \dfrac{\sigma^2\log(1/\delta)}{nq(1-\epsilon)} + \sigma^2 \cdot \biggl\{ \biggl(\dfrac{ \epsilon}{q(1-\epsilon)} \biggr)^2 \;\wedge\; \log^{2/r} \biggl( 2 + \dfrac{2\epsilon}{q(1-\epsilon)} \biggr) \biggr\}\\
                \hspace{7.7cm}\text{if }\delta\geq e^{-nq(1-\epsilon)/2}\\
                =\infty \qquad\text{if }\delta<\frac{(1-q(1-\epsilon))^n}{2}.
            \end{cases}
        \end{align*}
    \end{enumerate}
\end{theorem}
When $q(1-\epsilon) \leq c < 1$, we have $\{1 - q(1-\epsilon)\}^n > e^{-c'nq(1-\epsilon)}$ for some $c'$ depending only on~$c$, which reveals that the lower bound on $\delta$ in Theorem~\ref{thm:one-dim-realisable-sample-mean-ub} for which can control the $(1-\delta)$th minimax quantile is essentially optimal.  Thus, taken together, Theorems~\ref{thm:one-dim-realisable-sample-mean-ub} and~\ref{thm:nonparametric-realisable-model-lb} determine the minimax quantiles of the quadratic loss function for mean estimation over our realisable classes up to universal constants.  As a special case, when the effective contamination level $\kappa$ is less than 1, we see that the minimax quantile over realisable classes with base distribution $P \in \mathcal{P}_{L^2}(\theta_0,\sigma)$, which scales as $\sigma^2 \kappa^2$, coincides with the minimax rate of mean estimation over Gaussian classes in the arbitrary contamination model; see Section~\ref{sec:univariate-arbitrary-contamination-lb}.  Theorem~\ref{thm:nonparametric-realisable-model-lb}(b) further reveals that, while careful examination of the tails enabled consistent estimation in the Gaussian realisable setting of Section~\ref{sec:gaussian-realisable-model}, no such strategy can yield consistent estimation when the class is broadened to include all sub-Gaussian distributions with a fixed sub-Gaussian norm.  More generally, the optimal rates of Theorems~\ref{thm:one-dim-realisable-sample-mean-ub} and~\ref{thm:nonparametric-realisable-model-lb} provide a quantification of the benefits of realisable classes in terms of improved rates of mean estimation compared with the arbitrary contamination models of Section~\ref{sec:mean-estimation-arbitrary-contamination}.

\subsubsection{Multivariate extension} \label{sec:multivariate-nonparametric-realisable}

Here we show that the univariate results of Section~\ref{sec:one-dim-nonparametric-realisable-model} extend to the problem of estimating a multivariate mean, under the same simplifying assumption on the set of possible observation patterns as that considered in Section~\ref{sec:kolmogorov-depth}.  To this end, and by analogy with~\eqref{eq:finite-moment-distribution} and~\eqref{eq:sub-weibull-distribution}, for $d\in\mathbb{N}$, $\theta\in\mathbb{R}^d$, $\Sigma\in\mathcal{S}_{++}^{d\times d}$ and $r>0$, define
\begin{align*}
    \mathcal{P}_{L^r}(\theta,\Sigma) \coloneqq \bigl\{ P\in\mathcal{P}(\mathbb{R}^d) : &\;\mathbb{E}_P(X)=\theta, \mathsf{Law}_{X \sim P}(v^\top X) \in \mathcal{P}_{L^r}(v^\top\theta, v^\top\Sigma v) \, \forall v\in\mathbb{S}^{d-1} \bigr\}
\end{align*}
and
\begin{align*}
    \mathcal{P}_{\psi_r}(\theta,\Sigma) \coloneqq \bigl\{ P\in\mathcal{P}(\mathbb{R}^d) : &\;\mathbb{E}_P(X)=\theta, \mathsf{Law}_{X \sim P}(v^\top X) \in \mathcal{P}_{\psi_r}(v^\top\theta, v^\top\Sigma v) \, \forall  v\in\mathbb{S}^{d-1} \bigr\}.
\end{align*}
As for Theorem~\ref{thm:one-dim-realisable-sample-mean-ub}(a), we assume that we have access to an algorithm $\mathrm{ALG3}$ with property that if $X_1,\ldots,X_n \overset{\mathrm{iid}}{\sim} P \in \mathcal{P}_{L^2}(\theta_0,\Sigma)$, then there exist $a \in (0,1]$ and $C>0$ such that for $n\in\mathbb{N}$ and $\delta\in [e^{-an},1]$, we have with probability at least $1-\delta$ that
\begin{align}
    \bigl\| \mathrm{ALG3}(X_1,\ldots,X_n;\delta) - \theta_0\bigr\|_2^2 \leq C\biggl( \frac{\tr(\Sigma)}{n} + \frac{\|\Sigma\|_{\mathrm{op}}\log(1/\delta)}{n} \biggr). \label{eq:assumption-on-alg-heavy-tail-multivariate}
\end{align}
This is satisfied by, for example, the algorithms mentioned below~\eqref{eq:assumption-on-alg}, with the same values of $a$ and $C$ for the algorithm of \citet{depersin2022robust} (here, $\epsilon_{\max}$ is not needed since we do not require ALG3 to be robust to contamination).  Recall the definition of the realisable classes $\mathcal{R}_{\emptyset,[d]}(P, \epsilon, \pi)$ from~\eqref{eq:realisable-all-or-nothing}.

\begin{theorem} \label{thm:nonparametric-multivariate-realisable-mean-ub}
    Let $\theta_0 \in \mathbb{R}^d$, $\Sigma \in \mathcal{S}_{++}^{d\times d}$, $\epsilon\in[0,1)$, $\pi\in\mathcal{P}(\{\emptyset,[d]\})$ and $q\coloneqq \pi([d])$.
    \begin{itemize}
        \item[(a)] Let $r \geq 2$, let $P\in \mathcal{P}_{L^r}(\theta_0, \Sigma)$ and let $Z_1,\ldots,Z_n \overset{\mathrm{iid}}{\sim} R \in \mathcal{R}_{\emptyset,[d]}(P, \epsilon, q)$.  Further, let $\mathcal{D} \coloneqq \{i \in [n]:Z_i \in \mathbb{R}^d\}$ and let $\hat{\theta}_n \coloneqq \mathrm{ALG3}\bigl( (Z_i)_{i\in\mathcal{D}}; \delta \bigr)$, where $\mathrm{ALG3}$ satisfies~\eqref{eq:assumption-on-alg-heavy-tail-multivariate}.  If $nq(1-\epsilon) \geq \mathbf{r}(\Sigma)$, then for any $\delta\geq 2\exp\bigl(-anq(1-\epsilon)/8\bigr)$, we have with probability at least $1-\delta$ that
        \begin{align*}
        \|\hat{\theta}_n - \theta_0\|_2^2 \lesssim \frac{\tr(\Sigma) + \|\Sigma\|_{\mathrm{op}}\log(2/\delta)}{nq(1-\epsilon)} + \|\Sigma\|_{\mathrm{op}} \biggl\{\biggl( \frac{\epsilon}{q(1-\epsilon)} \biggr)^2 \wedge \biggl( \frac{\epsilon}{q(1-\epsilon)} \biggr)^{2/r} \biggr\}.
        \end{align*} 
        \item[(b)] Let $r\geq 1$, $P\in \mathcal{P}_{\psi_r}(\theta_0, \Sigma)$, $R \in \mathcal{R}_{\emptyset,[d]}(P, \epsilon, q)$ and $Z_1,\ldots, Z_n \overset{\mathrm{iid}}{\sim} R$. Let $\mathcal{D} \coloneqq \{i\in[n] : Z_i\neq\star\}$ and $\hat{\theta}_n \coloneqq |\mathcal{D}|^{-1}\sum_{i\in\mathcal{D}}Z_i$. If $nq(1-\epsilon) \geq \mathbf{r}(\Sigma)$, then for any $\delta\geq 8\exp\bigl(-nq(1-\epsilon)/8\bigr)$, we have with probability at least $1-\delta$ that
        \begin{align*}
        \|\hat{\theta}_n \!-\! \theta_0\|_2^2 \lesssim \frac{\tr(\Sigma) \!+\! \|\Sigma\|_{\mathrm{op}}\log(2/\delta)}{nq(1-\epsilon)} + \|\Sigma\|_{\mathrm{op}} \biggl\{\biggl( \frac{\epsilon}{q(1\!-\!\epsilon)} \biggr)^2 \wedge \log^{2/r}\biggl( 2 \!+\! \frac{2\epsilon}{q(1-\epsilon)} \biggr) \biggr\}.
        \end{align*}
    \end{itemize}
\end{theorem}
Thus, even in the multivariate setting considered in Theorem~\ref{thm:nonparametric-multivariate-realisable-mean-ub}, the realisable classes permit the same improvements in performance over fully-observed arbitrary contamination models as we saw in the univariate case in Theorem~\ref{thm:one-dim-realisable-sample-mean-ub}.  Moreover, the second (contamination) terms retain this improvement in a dimension-free manner.  By a small modification of the two-point construction of the proof of Theorem~\ref{thm:nonparametric-realisable-model-lb} so that the difference in means is in the direction of the leading eigenvector of $\Sigma$, together with \citet[][Proposition~10]{ma2024high} or \citet[][Theorem~4]{depersin2022optimal}, we see that both bounds in Theorem~\ref{thm:nonparametric-multivariate-realisable-mean-ub} are minimax rate-optimal.


\section{Extension: Linear regression with realisable missing response} 
\label{sec:regression-missing-response} 

The ideas of mean estimation with realisable missing observations developed in Section~\ref{sec:realisable-mean-est} extend to linear regression with a realisable missing response, as we now demonstrate.  Given $\theta_0 \in \mathbb{R}^d$ and $\sigma > 0$, consider a random design normal linear model 
\[
Y_i = X_i^\top \theta_0 + \zeta_i,
\]
where $(X_1,\zeta_1),\ldots,(X_n,\zeta_n)$ are independent with $\zeta_1,\ldots,\zeta_n \stackrel{\mathrm{iid}}{\sim} \mathsf{N}(0,\sigma^2)$.  In this section, we suppose that, instead of the desired responses $Y_1,\ldots,Y_n$, we are only able to observe corrupted versions that are subject to realisable missingness.  Thus, more precisely, we observe $Z_i \coloneqq (1 - B_i) \cdot Y_i \ostar \Omega^{(1)}_i + B_i \cdot Y_i \ostar \Omega^{(2)}_i$ for $i \in [n]$, where $B_1,\ldots,B_n \stackrel{\mathrm{iid}}{\sim} \mathsf{Ber}(\epsilon)$ are independent of the independent quadruples  $(X_i,\zeta_i,\Omega^{(1)}_i,\Omega^{(2)}_i)_{i \in [n]}$, where $\Omega^{(1)}_i|\{X_i=x\} \sim \mathsf{Ber}(q_x)$ with $\inf_{x \in \mathbb{R}^d} q_x \geq q > 0$ and where $\Omega^{(1)}_i \indep \zeta_i \,|\, X_i$.  We impose no restriction on the dependence between $\Omega^{(2)}_i$ and $(X_i,\zeta_i)$.  Thus, when $\epsilon = 0$, the pairs $(X_1,Z_1),\ldots,(X_n,Z_n)$ are missing at random (MAR).  We summarise the conditional distribution of the observed responses by writing $Z_1|\{X_1=x\}  \sim R_x \in \mathcal{R}^{\mathrm{Res}}\bigl(\mathsf{N}(x^{\top} \theta_0, \sigma^2), \epsilon, q\bigr)$.

Our main result in this section relies on what we will call a $(\beta,\gamma)$-regular design assumption, for $\beta \in (0, 1/2], \gamma > 0$.
\begin{defn}[$(\beta,\gamma)$-regular design]\label{asm:fixed-design-regularity}
    We say $(x_1,\ldots,x_n) \in (\mathbb{R}^d)^n$ is a $(\beta,\gamma)$-regular design if for all $v \in \mathbb{R}^d$, there exists a set $\mathcal{T} \subseteq [n]$ such that $\lvert \mathcal{T} \rvert \geq 2\beta n$ and $\lvert x_{i}^{\top} v \rvert > \gamma \|v\|_2$ for all $i \in \mathcal{T}$.  We write $\mathcal{C}_{(\beta,\gamma)}$ for the set of all $(\beta,\gamma)$-regular designs. 
\end{defn}
Thus, Definition~\ref{asm:fixed-design-regularity} asks for a relatively regular dispersion among $x_1,\ldots,x_n$ in the sense that for every hyperplane $H$, a non-trivial proportion of observations have distance greater than $\gamma$ from $H$.  The following lemma shows that a random design where the distribution is not supported on a hyperplane is $(\beta,\gamma)$-regular for some $\beta\in(0,1/2]$ and $\gamma>0$, with high probability when the sample size is sufficiently large.  
\begin{lemma} \label{lemma:beta-gamma-regular}
    Let $\delta\in(0,1]$, $\gamma>0$ and $X_1,\ldots,X_n \overset{\mathrm{iid}}{\sim} P \in \mathcal{P}(\mathbb{R}^d)$. Further define $\beta \coloneqq \frac{1}{3}\inf_{v \in \mathbb{S}^{d-1}} \mathbb{P}\bigl(|X_1^\top v| > \gamma\bigr)$.
    \begin{enumerate}[(a)]
    \item The infimum in the definition of $\beta$ is attained, and if $P(H) < 1$ for every hyperplane $H \subseteq \mathbb{R}^d$, then $\beta > 0$ for sufficiently small $\gamma > 0$.
    \item There exists a universal constant $c>0$ such that if $\frac{d + \log(1/\delta)}{n} \leq c\beta^2$, then, with probability at least $1-\delta$, the $n \times d$ matrix with $i$th row $X_i^\top$ is a $(\beta,\gamma)$-regular design.
    \end{enumerate}
\end{lemma}
To illustrate Lemma~\ref{lemma:beta-gamma-regular}(b) with a specific example, suppose that $X_1,\ldots,X_n \overset{\mathrm{iid}}{\sim} \mathsf{N}_d(0,\Sigma)$ for some $\Sigma\in\mathcal{S}_{++}^{d\times d}$, and let $\lambda_{\min}(\Sigma) > 0$ denote the minimum eigenvalue of $\Sigma$.  Then by Lemma~\ref{lemma:beta-gamma-regular}(b), there exists a universal constant $c_1>0$ such that if $\frac{d + \log(1/\delta)}{n} \leq c_1$, then with probability at least $1-\delta$, the $n \times d$ matrix with $i$th row $X_i^\top$ is a $\bigl( 2\Phi(-1)/3, \lambda_{\min}^{1/2}(\Sigma)\bigr)$-regular design.

Given $R_1,R_2 \in \mathcal{P}(\mathbb{R}_\star)$, we define their \emph{symmetrised Kolmogorov distance} by
\[
d_{\mathrm{K}}^{\mathrm{sym}}(R_1,R_2) \coloneqq \sup_{A \in \mathcal{A}^{\mathrm{sym}}} \lvert R_1(A) - R_2(A)\rvert,
\]
where $\mathcal{A}^{\mathrm{sym}} \coloneqq \bigl\{ (-\infty,t] : t\in\mathbb{R} \bigr\} \cup \bigl\{ [t,\infty) : t\in\mathbb{R} \bigr\}$.  This may be larger than $d_\mathrm{K}(R_1,R_2)$ when $R_1(\{\star\}) \neq R_2(\{\star\})$.  Now, for $\theta\in\mathbb{R}^d$, we define the empirical distribution $\hat{R}_{n,\theta}$ by $\hat{R}_{n,\theta}(B) \coloneqq n^{-1} \sum_{i=1}^n \mathbbm{1}_{\{Z_i-X_i^\top \theta \in B\}}$ for $B\in\mathcal{B}(\mathbb{R}_{\star})$.  
We let  
\[
\mathcal{R}_{0}^{\mathrm{Lin}} \coloneqq \mathcal{R}\bigl(\mathsf{N}(0, \sigma^2),1 - q(1 - \epsilon), 1\bigr),
\]
and define the Kolmogorov distance estimator as\begin{align}
    \hat{\theta}^{\mathrm{K}}_n \coloneqq \argmin_{\theta\in \mathbb{R}^d}\; d_{\mathrm{K}}^{\mathrm{sym}}\bigl(\hat{R}_{n,\theta} , \mathcal{R}_0^{\mathrm{Lin}}\bigr).
\end{align}
The realisable set $\mathcal{R}_0^{\mathrm{Lin}}$ ensures that the selection probability $h_x \coloneqq \mathbb{P}(Z_1 \neq \star \mid X_1 = x)$ 
satisfies the sandwich relation $q(1 - \epsilon) \leq q_x(1-\epsilon) \leq h_x \leq 1$ for all $x \in \mathbb{R}^d$.  Writing $\tilde{R}_{i,\theta} \coloneqq \mathsf{Law}(Z_i - x_i^\top \theta)$ for $i \in [n]$ and $\theta \in \mathbb{R}^d$, as well as $R_{n,\theta} \coloneqq \frac{1}{n} \sum_{i=1}^n \tilde{R}_{i,\theta}$, it then follows from Proposition~\ref{prop:univariate-realisability} that $R_{n,\theta_0} \in \mathcal{R}_{0}^{\mathrm{Lin}}.$
\begin{theorem} \label{thm:gaussian-realisable-response}
    Let $n ,d \in \mathbb{N}, \epsilon \in [0, 1), q \in (0, 1], \xi \in (0,1), \delta\in(0,1]$ and $\theta_0 \in \mathbb{R}^d$.  Suppose that $(x_1, \ldots, x_n) \in \mathcal{C}_{(\beta,\gamma)}$ for some $\beta \in (0,1/2]$ and $\gamma > 0$, and let $Z_1,\ldots,Z_n$ be independent with $Z_i\,|\,\{X_i=x_i\} \sim R_{x_i} \in \mathcal{R}^{\mathrm{Res}}\bigl(\mathsf{N}(x_i^\top\theta_0,\sigma^2), \epsilon, q\bigr)$.  There exists a universal constant $C_1>0$ such that if 
    \begin{align} \label{Eq:SampleSizeCondition}
        n^{1-\xi} \geq C_1\bigl\{d+\log(1/\delta)\bigr\} \quad\text{and}\quad \log\biggl( 1+ \frac{4(1-\beta q(1-\epsilon))}{\beta q(1-\epsilon)} \biggr) \leq \frac{\xi\log n}{15},
    \end{align}
    then with probability at least $1 - \delta$, conditional on $X_1=x_1,\ldots,X_n=x_n$, 
    \begin{align*}
    \bigl\| \hat{\theta}^{\mathrm{K}}_n - \theta_0 \bigr\|_2^2 \lesssim \sigma^2\cdot \frac{\log^2\bigl( 1+ \frac{4(1-\beta q(1-\epsilon))}{\beta q(1-\epsilon)} \bigr)}{\gamma^2 \xi \log \bigl( nq(1-\epsilon)\bigr)}.
    \end{align*}
\end{theorem}
One of the main consequences of Theorem~\ref{thm:gaussian-realisable-response} is that if we consider $\beta,\gamma,q, \epsilon, \xi$ and $\delta$ as constants, but allow $d$ to grow subject to $\frac{d}{n^{1-\xi}} \to 0$ as $n\to\infty$, then under the assumptions of Theorem~\ref{thm:gaussian-realisable-response}, we have that $\hat{\theta}_n^{\mathrm{K}}$ is a consistent estimator of $\theta_0$ in squared Euclidean norm as $n \to \infty$.  In fact, similar to Section~\ref{sec:gaussian-realisable-model}, this conclusion continues to hold even if we allow~$\epsilon$ to converge slowly to 1 and $q$ to converge slowly to zero.  This lies in stark contrast to the complete-case arbitrary contamination setting in which for any constant $\epsilon > 0$, consistent estimation is impossible~\citep[see, e.g., the discussion following][Theorem 3.2]{gao2020robust}.  Moreover, when the parameters $\beta, \gamma, q, \epsilon$ and $\delta$ are positive constants, the optimality of the rate $1/\log n$ follows from our mean estimation lower bound (Theorem~\ref{thm:univariate-realisable-lb}).

\section{Adaptation to unknown contamination proportion} \label{sec:adaptation}

The aim of this section is to provide a general strategy for extending the algorithms and theory presented in Sections~\ref{sec:mean-estimation-arbitrary-contamination}, and~\ref{sec:realisable-mean-est} to cases where the contamination proportion $\epsilon$ is unknown.  This is facilitated via a Lepski-type argument \citep[similar to][\S 6.2]{dalalyan2022all}.  

To describe our setting and algorithm, suppose that the true contamination fraction $\epsilon_{\star}$ lies in an interval $[0,\epsilon_{\max}]$, where $\epsilon_{\max}\in(0,1)$ is a known upper bound.  Let $\Delta\subseteq[0,1]$, let $\hat{\theta}_n(\epsilon,\delta)$ be an estimator of $\theta_0 \in \Theta \subseteq \mathbb{R}^d$ and let $\phi:[0,\epsilon_{\max}]\times\Delta \to (0,\infty)$ be such that $\epsilon \mapsto \phi(\epsilon,\delta)$ is decreasing in $\epsilon$ for every $\delta \in \Delta$.  Assume further that for each $\delta\in\Delta$ and $\epsilon \in [\epsilon_{\star},\epsilon_{\max}]$, we have with probability at least $1-\delta$ that
\begin{align}
    \|\hat{\theta}_n(\epsilon,\delta) - \theta_0\|_2 \leq \phi(\epsilon,\delta). \label{eq:phi-def}
\end{align}
For $x \in \mathbb{R}^d$ and $r \geq 0$, define the closed Euclidean ball of radius $r$ about $x$ by $B_2(x,r) \coloneqq \{y \in \mathbb{R}^d:\|y-x\|_2 \leq r\}$.  Define $\delta'\coloneqq \frac{\delta}{1+\lceil\log_2 n\rceil}$, 
\begin{align*}
    \mathcal{C} \coloneqq \bigl\{ 2^{-\ell}\epsilon_{\max} : \ell=0,1,\ldots,\lceil\log_2 n\rceil \bigr\}
\end{align*}
and 
\begin{align*}
    \epsilon_0 \coloneqq \min\biggl\{ \epsilon\in \mathcal{C} : \bigcap_{\epsilon'\in\mathcal{C}:\epsilon'\geq\epsilon} B_2\bigl(\hat{\theta}_n(\epsilon',\delta'), \phi(\epsilon',\delta')\bigr) \neq \emptyset \biggr\}.
\end{align*}
The definition of $\epsilon_0$ does not involve $\epsilon_{\star}$, and our final estimator is $\tilde{\theta}_n(\delta)\coloneqq \hat{\theta}_n(\epsilon_0,\delta')$.

\begin{prop}\label{prop:adaptation}
Suppose that $\delta'\in\Delta$ and let $C \geq 1$ be such that
\begin{align*}
    \phi(\epsilon_{\max}/n,\delta') \leq C\phi(0,\delta').
\end{align*}
Then with probability at least $1-\delta$, we have
    \begin{align*}
        \|\tilde{\theta}_n(\delta) - \theta_0\|_2 \leq 3C\phi(2\epsilon_{\star}\wedge\epsilon_{\max},\delta').
    \end{align*}
\end{prop}
The conditions on $\phi$ in Proposition~\ref{prop:adaptation} should be regarded as mild, and are satisfied by the \textproc{Iterative\_Robust\_Mean} estimator in  Theorem~\ref{thm:robust-descent-iterative-imputation-ub} (with $C=2$), the minimum Kolmogorov distance estimator in Theorem~\ref{thm:one-dim-kolmogorov-estimator} (with $C=3$) and its multivariate analogue in Theorem~\ref{thm:multivariate-kolmogorov-estimator} (with $C=3$).  The conclusion of Proposition~\ref{prop:adaptation} is essentially that, up to universal constant factors, the price of adaptation is the replacement of $\delta$ in the original, non-adaptive bound, with $\delta' = \delta/(1 + \lceil \log_2 n \rceil)$.  For instance, applying Proposition~\ref{prop:adaptation} to the \textproc{Iterative\_Robust\_Mean} estimator in Theorem~\ref{thm:robust-descent-iterative-imputation-ub} yields that, if the conditions of that result hold with $\delta$ replaced with $\delta'$, then with probability at least $1 - \delta$, 
    \begin{align*}
        \|\hat{\theta}_n - \theta_0\|_2^2 
        \lesssim \frac{T \tr(\Sigma^{\mathrm{IPW}})}{n} + \frac{T\|\Sigma^{\mathrm{IPW}}\|_{\mathrm{op}} \bigl\{\log(2T/\delta) + \log \log n\bigr\}}{n} + \|\Sigma^{\mathrm{IPW}}\|_{\mathrm{op}} \epsilon_\star.
    \end{align*}
In other words, if $\delta \leq 2 T/\log n$, then we can adapt to the unknown contamination fraction $\epsilon_\star \in [0,\epsilon_{\max}]$ at the cost only of an increased universal constant in the bound.  Moreover, even when $\delta > 2 T/\log n$, the price for adaptation is at most an inflation of the middle term in the bound by a multiplicative factor of order $O(\log\log n)$.

\medskip

\textbf{Acknowledgements:} The research of TM, KAV and RJS was supported by RJS's European Research Council (ERC) Advanced Grant 101019498.  The work of KAV was supported in part by National Science Foundation grant DMS-2210734. The work of TBB was supported by Engineering and Physical Sciences Research Council (EPSRC) New Investigator Award EP/W016117/1 and ERC Starting Grant 101163546.  The research of TW was supported by EPSRC New Investigator Award EP/T02772X/1.

{
\bibliographystyle{imsart-nameyear.bst}
\bibliography{bibliography}
}

\clearpage
\setcounter{section}{0}
\setcounter{equation}{0}
\setcounter{theorem}{0}
\def\theequation{S\arabic{equation}}
\def\thesection{S\arabic{section}}
\def\thetheorem{S\arabic{theorem}}
\def\thefigure{S\arabic{figure}}
\def\thealgorithm{S\arabic{algorithm}}

This is the supplementary material for \cite{ma2024estimation}.

\section{Notation used in proofs}\label{sec:notation-proofs}

For a measurable space $(\mathcal{Z}, \mathcal{C})$ and probability measures $P, Q \in \mathcal{P}(\mathcal{Z})$, we write $P \perp Q$ if $P$ and $Q$ are singular.  The Lebesgue decomposition theorem yields the unique decomposition $P = P_{\mathrm{ac}} + P_{\mathrm{sing}}$ where $P_{\mathrm{ac}} \ll Q$ and where $P_{\mathrm{sing}} \perp Q$.  For a convex function $f: (0, \infty) \rightarrow \mathbb{R}$, we let $M_f \coloneqq \lim_{x \rightarrow \infty} f(x)/x \in (-\infty, \infty]$ denote its \emph{maximal slope}.   We then define the \emph{$f$-divergence} between $P$ and $Q$ to be 
\begin{align} \label{eq:f-divergence}
\mathrm{Div}_f(P, Q) \coloneqq \int_{\mathcal{Z}} f\biggl(\frac{\mathrm{d} P_{\mathrm{ac}}}{\mathrm{d}Q}\biggr) \, \mathrm{d}Q + M_f \cdot P_{\mathrm{sing}}(\mathcal{Z}).
\end{align}
As important examples, if $f(x) = |x-1|/2$, then we obtain the total variation distance $\mathrm{TV}(P, Q) \coloneqq \sup_{A \in \mathcal{C}} \lvert P(A) - Q(A) \rvert$, while if $f(x) = x \log x$, then the resulting $f$-divergence is the Kullback--Leibler divergence
\[
\mathrm{KL}(P,Q) \coloneqq \begin{cases} 
\int_{\mathcal{Z}} \log\bigl(\frac{\mathrm{d} P}{\mathrm{d} Q}\bigr) \, \mathrm{d} Q & \text{ if } P \ll Q\\
\infty & \text { otherwise}.
\end{cases}
\]
Finally, if $f(x) = (x - 1)^2$, then we obtain the $\chi^2$-divergence
\[
\chi^2(P,Q) \coloneqq \begin{cases} 
\int_{\mathcal{Z}} \bigl(\frac{dP}{dQ} - 1\bigr)^2 \, dQ & \text{ if } P \ll Q \\
\infty & \text{ otherwise}.
\end{cases}
\]
Recalling the spaces $\mathcal{X}_1,\ldots,\mathcal{X}_d$ from Section~\ref{sec:extended-space-properties}, for a set $S \in 2^{[d]} \setminus \{\emptyset\}$, let $\mathcal{X}_{S} \coloneqq \prod_{j \in S} \mathcal{X}_j$, and also define $\mathcal{X}_\emptyset \coloneqq \{\star\}$ and $\mathcal{X} \coloneqq \prod_{j=1}^d \mathcal{X}_j$.  Given $x = (x_1,\ldots,x_d) \in \mathcal{X}$ and $S \in 2^{[d]} \setminus \{\emptyset\}$, we define $x_S \coloneqq (x_j)_{j \in S}$, with $x_\emptyset \coloneqq \star$.  For $S \subseteq [d]$, we define $\mathcal{X}_j^{(S)} \coloneqq \mathcal{X}_j$ if $j\in S$ and $\mathcal{X}_j^{(S)} \coloneqq \{\star\}$ if $j \notin S$, and also set $\mathcal{X}^{(S)} \coloneqq \prod_{j=1}^d \mathcal{X}_j^{(S)}$.  Next, we let
\[
\mathcal{B}^{(S)}(\mathcal{X}_\star) \coloneqq \bigl\{A \in \mathcal{B}(\mathcal{X}_\star): \forall z = (z_1,\ldots,z_d) \in A,\, z_j \neq \star, \, \forall j \in S \ \text{and} \ z_k = \star,\, \forall k \notin S \bigr\}.
\]
Given $S \subseteq [d]$, we write $\mathcal{G}_S$ for the set of real-valued functions on $\mathcal{X}_S$, and also write $\mathcal{G}_\star$ for the set of real-valued functions on $\mathcal{X}_\star$.  A function $f \in \mathcal{G}_\star$ may be identified with the sequence of functions $(f_S:S \subseteq [d])$, where $f_S \in \mathcal{G}_S$ for each $S$.  Formally, this identification is via the bijection $\psi: \prod_{S \subseteq [d]} \mathcal{G}_S \rightarrow \mathcal{G}_\star$ given by $\psi\bigl((f_{S'}:S' \subseteq [d])\bigr)(z) \coloneqq f_S(z_S)$ for $z \in \mathcal{X}^{(S)}$ and $S \subseteq [d]$.  In other words, we evaluate $f \in \mathcal{G}_\star$ at $z \in \mathcal{X}_\star$ by setting $S$ to be the coordinates in $z$ that are not equal to $\star$, and then computing $f_S(z_S)$.

\section{Proofs from Section~\ref{sec:setup}}\label{sec:proofs-setup}

\subsection{Proof of Theorem~\ref{cor:P-epsilon-pi-S-realisability-Farkas-form}}\label{sec:proof-general-realisable}

Theorem~\ref{cor:P-epsilon-pi-S-realisability-Farkas-form} follows immediately from Theorem~\ref{Thm:AbstractVersion} below, which is stated in greater generality, encompassing both continuous and discrete spaces.  In fact, we begin with a sketch of the proof of this general result in the setting where $\mathcal{X}$ is finite, both to explain the relevance of (a generalisation of) Farkas's lemma in this context, and to provide intuition for the more technical arguments that follow.  Let $X \sim P \in \mathcal{P}(\mathcal{X})$ and let $Q \coloneqq \mathsf{Law}(X \ostar \Omega)$ for some random vector $\Omega$ taking values in $\{0,1\}^d$.  We write $M = (M_{S,x})_{S\subseteq [d],x\in\mathcal{X}} \coloneqq \bigl(\mathbb{P}(\Omega = \bm{1}_S \, | \, X = x)\bigr)_{S \subseteq [d], x \in \mathcal{X}} \in [0,1]^{2^{[d]} \times \mathcal{X}}$ to summarise the missingness mechanism.  Now write $\mathbb{A} \in [0,1]^{\mathcal{X}_\star \times (2^{[d]} \times \mathcal{X})}$ for the matrix with
\[
    \mathbb{A}_{z,(S,x)} \coloneqq P(\{x\}) \mathbbm{1}_{\{z_S=x_S\}} \prod_{j \in S^{c}}\mathbbm{1}_{\{ z_{j} = \star \}},
\]
so that each column of $\mathbb{A}$ has at most one non-zero entry.  Then
\[
    (\mathbb{A} M)_{z} = \sum_{S \subseteq [d]} \sum_{x \in \mathcal{X}} P(\{x\}) M_{S, x} \mathbbm{1}_{\{z_S=x_S\}}\prod_{j \in S^{c}}\mathbbm{1}_{\{ z_{j} = \star \}} = Q(\{z\}).
\]
Now, for $x \in \mathcal{X}$, write $\sigma_x \in \{0,1\}^{2^{[d]} \times \mathcal{X}}$ for the vector with $(\sigma_x)_{(S,x')} \coloneqq \mathbbm{1}_{\{x=x'\}}$, so that $\sigma_x^\top M = \sum_{S \in 2^{[d]}} M_{S,x}$, and form the matrix $\mathbb{B} \coloneqq (\sigma_x^\top)_{x \in \mathcal{X}} \in \{0,1\}^{\mathcal{X} \times (2^{[d]} \times \mathcal{X})}$.  We can then define $\mathcal{J} \coloneqq \{M \in [0,1]^{2^{[d]} \times \mathcal{X}}:\mathbb{B} M = \bm{1}_{\mathcal{X}}\}$ to denote the set of valid mechanisms.  We deduce that $Q \in \mathsf{MNAR}_P$ if and only if there exists $M \in \mathcal{J}$ such that $\mathbb{A} M = Q$.  By Farkas's lemma, this latter condition is equivalent to the statement that there does not exist $(y,w) = \bigl((y_z)_{z \in \mathcal{X}_\star},(w_x)_{x \in \mathcal{X}}\bigr) \in \mathbb{R}^{\mathcal{X}_\star} \times \mathbb{R}^{\mathcal{X}}$ such that $\sum_{z \in \mathcal{X}_\star} Q(\{z\}) y_z + \sum_{x \in \mathcal{X}} w_x < 0$ and $0 \leq (\mathbb{A}^\top y + \mathbb{B}^\top w)_{(S,x)} = P(\{x\}) y_{x \ostar \bm{1}_{S}} + w_x$ for each $S \subseteq [d]$ and $x \in \mathcal{X}$. 
The search for such a pair $(y,w)$ amounts to a constrained optimisation problem, whose solution for each fixed~$y$ is to take $w_x = -P(\{x\}) \min_{S \subseteq [d]} y_{x \ostar \bm{1}_S}$ for $x \in \mathcal{X}$.  Then 
\[
\sum_{z \in \mathcal{X}_\star} Q(\{z\}) y_z + \sum_{x \in \mathcal{X}} w_x = \sum_{z \in \mathcal{X}_\star} Q(\{z\}) y_z - \sum_{x \in \mathcal{X}} P(\{x\}) \min_{S \subseteq [d]} y_{x \ostar \bm{1}_S},
\]
so the condition that there does not exist $(y,w)$ for which this quantity is negative corresponds to~\eqref{Eq:fmax} after identifying $y$ with~$-f$.

Moving now to the proof of the full theorem, we require several preliminary topological results that are stated and proved in Section~\ref{sec:auxiliary}.  We will also use the generalisation of Farkas's lemma below.  Recall that if $X$ is a real vector space, then the \emph{algebraic dual} of $X$, denoted $X^*$, is the vector space of linear functions $f:X \rightarrow \mathbb{R}$.  Whenever $X'$ is a subspace of this algebraic dual, we say $X'$ \emph{separates points} if for every $x_1,x_2 \in X$ with $x_1 \neq x_2$, there exists $f \in X'$ with $f(x_1) \neq f(x_2)$.  The \emph{weak topology} on $X$ generated by $X'$ is the coarsest topology such that $f^{-1}(U)$ is open in $X$ for every $f\in X'$ and open set $U\subseteq \mathbb{R}$.  Now let $Y$ be another real vector space and let $Y'$ be a subspace of its algebraic dual.  A linear map $T:X \rightarrow Y$ is \emph{$(X',Y')$-weakly continuous} if it is continuous when $X$ and $Y$ are equipped with the weak topologies generated by $X'$ and $Y'$ respectively.  Where $X'$ and $Y'$ are clear from context, we will abbreviate this terminology by simply referring to $T$ as weakly continuous.
\begin{theorem}[{\citealp[Theorem~2]{craven1977generalizations}}]
\label{Thm:GeneralisedFarkas}
    Let $X$ and $Y$ be real vector spaces, and let $X'$ and $Y'$ be subspaces of the algebraic duals of $X$ and $Y$, respectively, that separate points.  Given $y\in Y$, a weakly continuous linear map $T:X\to Y$, and a convex cone $K\subseteq X$ such that $T(K)$ is weakly closed in $Y$, the following are equivalent:
    \begin{enumerate}
        \item[(a)] $Tx=y$ has a solution $x\in K$;
        \item[(b)] If $g \in Y'$ satisfies $g(Tx)\geq 0$ for all $x\in K$, then $g(y)\geq 0$.
    \end{enumerate}
\end{theorem}

For any topological space $\mathcal{Z}$, we write $C_{\mathrm{b}}(\mathcal{Z})$ for the space of bounded continuous real-valued functions on $\mathcal{Z}$.  Let $\mathcal{M}(\mathcal{Z})$ denote the space of finite, signed Borel measures on $\mathcal{Z}$ and let $\mathcal{M}_+(\mathcal{Z})$ be the subspace of (non-negative) finite Borel measures.  We call $\mathcal{Z}$ a \emph{Hausdorff space} if, given any distinct $z_1,z_2 \in \mathcal{Z}$, we can find disjoint open subsets $V_1,V_2$ such that $z_1 \in V_1$, $z_2 \in V_2$.  The space $\mathcal{Z}$ is \emph{locally compact} if every point in $\mathcal{Z}$ has a compact neighbourhood, i.e.~if for every $z \in \mathcal{Z}$, we can find an open set $U \subseteq \mathcal{Z}$ and a compact set $K \subseteq \mathcal{Z}$ such that $z \in U \subseteq K$.

The main content of the proof of Theorem~\ref{cor:P-epsilon-pi-S-realisability-Farkas-form} is Proposition~\ref{thm:general-realisable} below.  Observe that the restriction of the bijection $\psi$ in Section~\ref{sec:notation-proofs} to the set $\{(f_S:S \subseteq [d]) : f_S \in C_{\mathrm{b}}(\mathcal{X}_S) \;\forall S\subseteq[d]\}$ has image $C_{\mathrm{b}}(\mathcal{X}_\star)$.  This identifies $C_{\mathrm{b}}(\mathcal{X}_\star)$ with $\bigl(C_{\mathrm{b}}(\mathcal{X}_S) : S \in 2^{[d]}\bigr)$, but henceforth we will not be explicit about this identification, and will simply write $f = (f_S:S \in 2^{[d]}) \in C_{\mathrm{b}}(\mathcal{X}_\star)$.  Given such an $f = (f_S:S \in 2^{[d]}) \in C_{\mathrm{b}}(\mathcal{X}_\star)$, we can express the function $f_{\max}$ from Section~\ref{sec:realisability} as $f_{\max}(x) \coloneqq \max_{S\in 2^{[d]}} f_S(x_S)$ for $x \in \mathcal{X}$.
  
\begin{prop}
\label{thm:general-realisable}
Let $\mathcal{X}_1,\ldots,\mathcal{X}_d$ be locally compact Hausdorff spaces, and let $\mathcal{X} \coloneqq \prod_{j=1}^d \mathcal{X}_j$.  Assume that every open set in $\mathcal{X}$ is $\sigma$-compact.  If $P \in \mathcal{P}(\mathcal{X})$ and $Q \in \mathcal{P}(\mathcal{X}_{\star})$, then $Q \in \mathsf{MNAR}_P$ if and only if 
    \[
    P (f_{\max}) \geq Q(f) 
    \]
    for all $f\in C_{\mathrm{b}}(\mathcal{X}_\star)$.
\end{prop}
\begin{proof}
Recall the definition of $\phi_{\mathcal{Z}}: C_{\mathrm{b}}(\mathcal{Z}) \to \mathcal{M}(\mathcal{Z})^*$ before Lemma~\ref{Lemma:DualPair}. We endow $\mathcal{M}(\mathcal{Z})$ with the weak topology generated by  $\phi_{\mathcal{Z}}\bigl(C_{\mathrm{b}}(\mathcal{Z})\bigr)$, for $\mathcal{Z} \in \{\mathcal{X},\mathcal{X}_\star,\mathcal{X} \times 2^{[d]}\}$.  This ensures that $\phi_{\mathcal{Z}}(g)$ is weakly continuous for every $g \in C_{\mathrm{b}}(\mathcal{Z})$.

Let $h: \mathcal{X}\times 2^{[d]} \to \mathcal{X}_\star$ be the continuous function defined by $h(x, S) \coloneqq x \ostar \bm{1}_S$. Then $h$ induces a linear map $h_*: \mathcal{M}(\mathcal{X}\times 2^{[d]}) \to \mathcal{M}(\mathcal{X}_\star)$ given by $h_*(\mu)(B) \coloneqq \mu\bigl(h^{-1}(B)\bigr)$ (see Figure~\ref{Fig:CD} below).
Similarly, let $j:\mathcal{X}\times 2^{[d]}\to\mathcal{X}$ be the projection map $j(x,S) \coloneqq x$, and define its induced map $j_*: \mathcal{M}(\mathcal{X}\times 2^{[d]}) \to \mathcal{M}(\mathcal{X})$.
We have
$\{g\circ h: g\in C_{\mathrm{b}}(\mathcal{X}_{\star})\} \subseteq C_{\mathrm{b}}(\mathcal{X}\times 2^{[d]})$ and similarly $\{g \circ j: g\in C_{\mathrm{b}}(\mathcal{X})\}\subseteq C_{\mathrm{b}}(\mathcal{X}\times 2^{[d]})$, we have by \citet[Theorem~IV.2.1]{schaefer1971} that both $h_*$ and $j_*$ are weakly continuous. By Lemma~\ref{Lemma:ProductContinuity}, the linear map $T = (h_*,j_*): \mathcal{M}(\mathcal{X}\times 2^{[d]}) \to \mathcal{M}(\mathcal{X}_\star)\times \mathcal{M}(\mathcal{X})$ is continuous when we endow the image space with the product topology, which by Lemma~\ref{Lemma:ProductWeakTopology} is the same as the weak topology on $\mathcal{M}(\mathcal{X}_\star)\times \mathcal{M}(\mathcal{X})$ generated by $\phi_{\mathcal{X}_\star}\bigl(C_{\mathrm{b}}(\mathcal{X}_{\star})\bigr) \times \phi_{\mathcal{X}}\bigl(C_{\mathrm{b}}(\mathcal{X})\bigr)$.

\begin{figure}[htbp]
\begin{center}
  \begin{tikzcd}
    \mathcal{X}\times 2^{[d]} \arrow{r}{h} \arrow[swap]{dr}{g\circ h\in C_{\mathrm{b}}(\mathcal{X}\times 2^{[d]}) } & \mathcal{X}_{\star} \arrow{d}{g\in C_{\mathrm{b}}(\mathcal{X}_\star)} \\
     & \mathbb{R}
  \end{tikzcd}
  \hspace{1cm}
  \begin{tikzcd}
    \mathcal{M}(\mathcal{X}\times 2^{[d]}) \arrow{r}{h_*} \arrow[swap]{dr}{\phi_{\mathcal{X}\times 2^{[d]}}(g\circ h)\in \mathcal{M}(\mathcal{X}\times 2^{[d]})^*} & \mathcal{M}(\mathcal{X}_{\star}) \arrow{d}{\phi_{\mathcal{X}_\star}(g)\in \mathcal{M}(\mathcal{X}_\star)^*} \\
     & \mathbb{R}
  \end{tikzcd}
  \end{center}
  \caption{\label{Fig:CD}Schematic diagrams of various maps defined in the proof. The fact that the maps in the right panel commute follows from the fact that $h_*(\mu)(g) = \mu(g\circ h)$ for all $g\in C_{\mathrm{b}}(\mathcal{X}_\star)$.}
\end{figure}
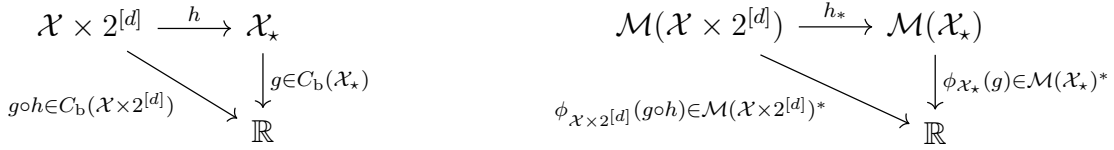

Define $K$ to be the convex cone $\mathcal{M}_+(\mathcal{X}\times 2^{[d]})$. We claim that $h_*(K) = \mathcal{M}_+(\mathcal{X}_{\star})$. It is clear that $h_*(K) \subseteq \mathcal{M}_+(\mathcal{X}_{\star})$ since for any $\mu\in K$ and any $g\in C_{\mathrm{b}}(\mathcal{X}_{\star})$ such that $g\geq 0$, we have by \citet[Proposition~10.1]{folland1999real} that $h_*(\mu)(g) = \mu(g\circ h) \geq 0$. For the surjectivity, define $i: \mathcal{X}_{\star} \to \mathcal{X}\times2^{[d]}$ by $i(z) \coloneqq (z\odot \bm{1}_{\{j:z_j\neq \star\}} , \{j:z_j\neq \star\})$ and let $i_*: \mathcal{M}(\mathcal{X}_{\star})\to\mathcal{M}(\mathcal{X}\times2^{[d]})$ be its induced linear map. By the same argument as above, we have $i_*(\mathcal{M}_+(\mathcal{X}_\star)) \subseteq K$. For $\nu$ on $\mathcal{M}_+(\mathcal{X}_{\star})$, we have $\nu  = h_*\bigl(i_*(\nu)\bigr)$, and the surjectivity is established since $i_*(\nu)\in K$. Consequently, 
\begin{align*} 
h_*(K) = \mathcal{M}_+(\mathcal{X}_\star) &= \bigcap_{g\in C_{\mathrm{b}}(\mathcal{X}_{\star}): g \geq 0}\{\nu\in\mathcal{M}(\mathcal{X}_{\star}): \nu(g) \geq 0\} \\
&= \bigcap_{g\in C_{\mathrm{b}}(\mathcal{X}_{\star}): g \geq 0} \bigl(\phi_{\mathcal{X}_\star}(g)\bigr)^{-1}\bigl([0,\infty)\bigr)
\end{align*}
is a weakly closed set.  A similar argument shows that $j_*(K)=\mathcal{M}_+(\mathcal{X})$ is a weakly closed set. Thus, $T(K)$ is weakly closed set in $\mathcal{M}(\mathcal{X}_{\star})\times \mathcal{M}(\mathcal{X})$, by Lemma~\ref{Lemma:ProductWeakTopology}.

By definition, $Q\in \mathsf{MNAR}_P$ if and only if there exists $\mu_0\in K$ such that $T(\mu_0) = (Q, P)$. Therefore, by Lemma~\ref{Lemma:DualPair}, we can apply the generalised Farkas' lemma (Lemma~\ref{Thm:GeneralisedFarkas}) to obtain that
\begin{align*}
\label{Eq:FarkasEquivalence2}
    Q\in\mathsf{MNAR}_P &\iff \bigcap_{\mu\in K}\bigl\{(f,g)\in C_{\mathrm{b}}(\mathcal{X}_\star)\times C_{\mathrm{b}}(\mathcal{X}): h_*(\mu)(f) + j_*(\mu)(g) \geq 0\bigr\}\nonumber\\
    &\hspace{3cm}\subseteq \bigl\{(f,g) \in C_{\mathrm{b}}(\mathcal{X}_\star)\times C_{\mathrm{b}}(\mathcal{X}): Q(f) + P(g) \geq 0\bigr\}.
\end{align*}
Now, for any $(f,g) \in C_{\mathrm{b}}(\mathcal{X}_\star)\times C_{\mathrm{b}}(\mathcal{X})$ and $\mu \in K$, we have
\[
    h_*(\mu)(f) + j_*(\mu)(g) = \sum_{S\in2^{[d]}} \int_{\mathcal{X}} \{(f\circ h)(x,S) + g(x)\}\,d\mu(x,S).
\]
Hence,  $(f,g)$ satisfies $h_*(\mu)(f) + j_*(\mu)(g) \geq 0$ for all $\mu\in K$ if and only if $(f\circ h)(x,S)+g(x)\geq 0$ for all $x\in\mathcal{X}$ and $S\in2^{[d]}$. Since $P(g)$ is increasing in $g$, it therefore suffices to check that for each $f \in C_{\mathrm{b}}(\mathcal{X}_\star)$ the function $g_f \in C_{\mathrm{b}}(\mathcal{X})$ given by $g_f(x) \coloneqq -\min_{S\in2^{[d]}} (f\circ h)(x,S) = -\min_{S\in2^{[d]}} f_S(x_S)$ satisfies $Q(f) + P(g_f) \geq 0$.  Substituting $f' \coloneqq -f$, we have 
\begin{align*}
    Q\in\mathsf{MNAR}_P &\iff \text{$Q(f) + P(g_f)\geq 0$ for all $f\in C_{\mathrm{b}}(\mathcal{X}_\star)$} \\
    &\iff \text{$Q(f') \leq P(f'_{\max})$ for all $f'\in C_{\mathrm{b}}(\mathcal{X}_\star)$}
\end{align*}
as desired.
\end{proof}
We are now in a position to state and prove the more general version of Theorem~\ref{cor:P-epsilon-pi-S-realisability-Farkas-form}.
\begin{theorem}
    \label{Thm:AbstractVersion}
    Let $\mathcal{X}_1,\ldots,\mathcal{X}_d$ be locally compact Hausdorff spaces and let $\mathcal{X} \coloneqq \prod_{j=1}^d \mathcal{X}_j$.  Assume that every open set in $\mathcal{X}$ is $\sigma$-compact.  Fix $P \in \mathcal{P}(\mathcal{X})$, $\epsilon\in (0,1]$, $\pi \in \mathcal{P}(2^{[d]})$. Let $R \in \mathcal{P}(\mathcal{X}_{\star})$, and define a signed measure on $\mathcal{X}_\star$ by $Q \coloneqq \epsilon^{-1}\{R - (1-\epsilon)\mathsf{MCAR}_{(\pi, P)}\}$. Then $R\in \mathcal{R}(P,\epsilon,\pi)$ if and only if $Q\in\mathcal{P}(\mathcal{X}_{\star})$ and 
    \begin{equation*}
    P(f_{\max}) \geq Q(f)
    \end{equation*}
    for all $f \in C_{\mathrm{b}}(\mathcal{X}_\star)$.
\end{theorem}
\begin{proof}
From the definition, $R \in \mathcal{R}(P,\epsilon,\pi)$ if and only if $Q \in \mathsf{MNAR}_P$, which by Proposition~\ref{thm:general-realisable} occurs if and only if $Q \in \mathcal{P}(\mathcal{X}_\star)$ and $P(f_{\max})\geq Q(f)$ for all $f\in C_{\mathrm{b}}(\mathcal{X}_{\star})$. 
\end{proof}

\subsection{Proof of Proposition~\ref{prop:univariate-realisability}} \label{sec:proof-prop-4}
\begin{proof}[Proof of Proposition~\ref{prop:univariate-realisability}]
Suppose that $R \in \mathcal{R}(P,\epsilon,q)$ and let $A \in \mathcal{B}(\mathbb{R}_{\star})$ be such that $\mu_{\star}(A) = 0$.  Recall that if $X \sim P$, $B \sim \mathsf{Bern}(\epsilon)$, $\Omega^{(1)} \sim \mathsf{Bern}(q)$ and $\Omega^{(2)} \sim \mathsf{Bern}(q_2)$ for some $q_2 \in [0,1]$ with $B \indep (X,\Omega^{(1)},\Omega^{(2)})$ and $\Omega^{(1)} \indep X$, then we can generate $Z \sim R$ via $Z \coloneqq (1-B)\cdot (X \ostar \Omega^{(1)}) + B \cdot (X \ostar \Omega^{(2)})$.  Then by definition of $\mu_{\star}$, we must have $A \in \mathcal{B}(\mathbb{R})$ and $\mu(A) = 0$.  Since $P \ll \mu$, it follows that
\begin{align*}
    0 = P(A) = \mathbb{P}(X\in A) \geq \mathbb{P}(Z\in A) = R(A).
\end{align*}
This proves that $R \ll \mu_{\star}$.  Now define $m: \mathbb{R} \to [0,1]$ by $m(x) \coloneqq \mathbb{P}(\Omega^{(2)} = 1 \,|\, X = x)$.  Then for any $A \in \mathcal{B}(\mathbb{R})$,
\begin{align*}
    \mathbb{P}(Z \in A) &= (1-\epsilon)\cdot \mathbb{P}(X \in A,\, \Omega^{(1)} = 1) + \epsilon \cdot \mathbb{P}(X \in A,\, \Omega^{(2)} = 1)\\
    &= q(1-\epsilon) \cdot \int_A p(x)\; \mathrm{d}\mu(x) + \epsilon \cdot \int_A m(x)p(x)\; \mathrm{d}\mu(x).
\end{align*}
Hence, $\frac{\mathrm{d}R}{\mathrm{d}\mu_\star}(x) = q(1-\epsilon) \cdot p(x) + \epsilon\cdot  m(x)p(x)$ for $x\in\mathbb{R}$, and $\frac{\mathrm{d}R}{\mathrm{d}\mu_\star}(\star) = \mathbb{P}(Z = \star) = 1- q(1-\epsilon) - \epsilon\int_{\mathbb{R}} m(x)p(x) \,\mathrm{d}\mu(x)$.

Conversely, suppose that $R \in \mathcal{P}(\mathbb{R}_{\star})$ satisfies $R \ll \mu_\star$, and there exists a Borel measurable function $m:\mathbb{R} \to [0,1]$ such that $\mathrm{d}R/\mathrm{d}\mu_{\star}$ satisfies~\eqref{eq:radon-nikodym-realisable}.  Given $X \sim P$, define a random variable $\Omega^{(2)}$ taking values in $\{0,1\}$ such that $\mathbb{P}(\Omega^{(2)} = 1 \,|\, X=x) = m(x)$ for $x\in\mathbb{R}$.  Let $B \sim \mathsf{Bern}(\epsilon)$ and $\Omega^{(1)} \sim \mathsf{Bern}(q)$ be such that $B \indep (X,\Omega^{(1)},\Omega^{(2)})$ and $\Omega^{(1)} \indep X$.  Then $Z \coloneqq (1-B) \cdot (X \ostar \Omega^{(1)}) + B \cdot (X \ostar \Omega^{(2)}) \sim R$ and hence by construction $R \in \mathcal{R}(P,\epsilon,q)$.

This completes the proof, but we also provide an alternative proof of the converse statement using Theorem~\ref{cor:P-epsilon-pi-S-realisability-Farkas-form}.
Again suppose that $R \in \mathcal{P}(\mathbb{R}_{\star})$ satisfies $R \ll \mu_\star$, and that $\mathrm{d}R/\mathrm{d}\mu_{\star}$ satisfies~\eqref{eq:radon-nikodym-realisable}.  Define $Q \coloneqq \epsilon^{-1}\{R - (1-\epsilon)\mathsf{MCAR}_{(\pi, P)}\} \in \mathcal{M}(\mathbb{R}_\star)$ as in Theorem~\ref{cor:P-epsilon-pi-S-realisability-Farkas-form}, and let $f = (f_{\{1\}}, f_{\emptyset}) \in C_{\mathrm{b}}(\mathbb{R}_{\star})$.  Note that by definition, $f_{\emptyset} \in \mathbb{R}$ is a constant and $f_{\max}(x) = f_{\{1\}}(x) \vee f_{\emptyset}$ for all $x \in \mathbb{R}$.  Moreover, since $\mathsf{MCAR}_{(\pi, P)} \in \mathcal{R}(P,0,\pi)$, we have by the argument in the direct part of the proof that $\mathsf{MCAR}_{(\pi, P)} \ll \mu_\star$ with $\frac{\mathrm{d}\mathsf{MCAR}_{(\pi, P)}}{\mathrm{d}\mu_\star}(x) = q \cdot p(x)$ for $x \in \mathbb{R}$ and $\frac{\mathrm{d}\mathsf{MCAR}_{(\pi, P)}}{\mathrm{d}\mu_\star}(\star) = 1-q$, so
\begin{align*}
    \frac{\mathrm{d}Q}{\mathrm{d}\mu_\star}(z) = \begin{cases}
        m(z)p(z) \quad&\text{if }z\in\mathbb{R}\\
        1- \int_{\mathbb{R}} m(x)p(x) \,\mathrm{d}\mu(x) &\text{if }z=\star.
    \end{cases}
\end{align*}
Hence $Q \in \mathcal{P}(\mathbb{R}_\star)$, and
\begin{align*}
    P(f_{\max}) &= \int_{\mathbb{R}} \bigl( f_{\{1\}}(x) \vee f_{\emptyset} \bigr) p(x) \,\mathrm{d}\mu(x)\\
    &\geq \int_{\mathbb{R}} \bigl\{ m(x)f_{\{1\}}(x) + \bigl(1-m(x)\bigr)f_{\emptyset} \bigr\} p(x) \,\mathrm{d}\mu(x)
    = Q(f),
\end{align*}
where the inequality follows from the fact that $\max(a,b)$ is at least as large as any convex combination of $a$ and $b$, for $a,b \in \mathbb{R}$.  We conclude that $R \in \mathcal{R}(P,\epsilon,q)$, by Theorem~\ref{cor:P-epsilon-pi-S-realisability-Farkas-form}.
\end{proof}

\section{Proofs from Section~\ref{sec:mean-estimation-arbitrary-contamination}} \label{sec:proofs-mean-estimation-arbitrary}

\subsection{Proof of Theorem~\ref{thm:robust-descent-iterative-imputation-ub}}
We begin with some lemmas.  Recalling the way that we can generate $Z_1, \ldots, Z_{n} \stackrel{\mathrm{iid}}{\sim} P \in \mathcal{P}^{\mathrm{arb}} \bigl(\theta_0, \Sigma, \epsilon, \pi \bigr)$ from Section~\ref{sec:departures-mcar}, we let $\mathcal{I}_n \subseteq [n]$ denote the `inliers', or the indices of the uncontaminated observations.  Likewise, we denote by $\mathcal{O}_n \subseteq [n]$ the `outliers', or the indices of the contaminated observations so that $\mathcal{I}_n \cup \mathcal{O}_n = [n]$.  

\begin{lemma}
\label{lemma:covariance-of-imputed-block-means} 
Let $n,M \in \mathbb{N}$ be such that $n/M \geq 4$.  Suppose that $Z_1, \ldots, Z_{n} \stackrel{\mathrm{iid}}{\sim} P \in \mathcal{P}^{\mathrm{arb}} \big(\theta_0, \Sigma, \epsilon, \pi \big)$, with corresponding observation patterns $\Omega_1, \ldots, \Omega_{n} \in \{0, 1\}^{d}$.  Randomly select $M$ disjoint sets $(B_m)_{m \in [M]} \subseteq [n]$ (independent of $Z_1,\ldots,Z_n$) such that $\lvert B_m \rvert = \lfloor n/M \rfloor$, and for $\theta = (\theta_1,\ldots,\theta_d)^\top \in \mathbb{R}^d$, $m \in [M]$ and $j \in [d]$, define
    \begin{align} \label{eq:z-bar-definition}
        \overbar{\Omega}_{mj} \coloneqq \mathbbm{1}_{\{\sum_{i\in B_m} \Omega_{ij} > 0\}} \qquad\text{ and }\qquad \bar{Z}_{mj} \coloneqq \frac{\sum_{i \in B_m} \Omega_{ij}Z_{ij}}{\sum_{i \in B_m} \Omega_{ij}} \cdot \overbar{\Omega}_{mj} + \theta_j \cdot (1 - \overbar{\Omega}_{mj}).
    \end{align}
    Let $\bar{Z}_m \coloneqq (\bar{Z}_{m1},\ldots,\bar{Z}_{md})^\top$.
    Then for all $m \in [M]$, we have 
    \begin{subequations}
    \begin{align} \label{ineq:bias-of-block-means}
        \|\mathbb{E}(\bar{Z}_m \,|\, B_m \subseteq \mathcal{I}_n) - \theta_0\|_2^2 \leq \frac{\|\theta-\theta_0\|_2^2}{e|B_m| q_{\min}},
    \end{align}
    \begin{align} 
        \tr\bigl(\Cov(\bar{Z}_m \,|\, B_m \subseteq \mathcal{I}_n)\bigr) &\leq \tr \Bigl(\mathbb{E} \bigl\{ (\bar{Z}_m - \theta_0)(\bar{Z}_m - \theta_0)^\top \,\big|\, B_m \subseteq \mathcal{I}_n \bigr\}\Bigr) \nonumber\\
        &\leq \frac{2}{|B_m|} \cdot \tr\bigl( \Sigma^{\mathrm{IPW}} \bigr) +  \frac{\|\theta - \theta_0\|_2^2}{e|B_m| q_{\min}} \label{ineq:trace-bound-iterative-imputation}
    \end{align}
    and 
    \begin{align} 
    \bigl\|\Cov(\bar{Z}_m \,|\, B_m \subseteq \mathcal{I}_n)\bigr\|_{\mathrm{op}} &\leq \bigl\| \mathbb{E} \bigl\{ (\bar{Z}_m - \theta_0)(\bar{Z}_m - \theta_0)^\top \,\big|\, B_m \subseteq \mathcal{I}_n \bigr\} \bigr\|_{\mathrm{op}} \nonumber\\
    &\leq \frac{6}{|B_m|} \cdot \big\| \Sigma^{\mathrm{IPW}} \big\|_{\mathrm{op}} +  \frac{\|\theta - \theta_0\|_2^2}{e|B_m| q_{\min}}. \label{ineq:op-norm-bound-iterative-imputation}
    \end{align}
    \end{subequations}
\end{lemma}

\begin{proof} 
Write $\theta_0 = (\theta_{01},\ldots,\theta_{0d})^\top \in \mathbb{R}^d$.  For $m \in [M]$ with $B_m \subseteq \mathcal{I}_n$, and for $j \in [d]$, we have
\begin{align*}
    \bigl(\mathbb{E}(\bar{Z}_{mj}) - \theta_{0j}\bigr)^2 &= \bigl(\mathbb{P}(\bar{\Omega}_{mj}=1) \theta_{0j} + \mathbb{P}(\bar{\Omega}_{mj}=0)\theta_j - \theta_{0j}\bigr)^2 \\
    &= (1-q_j)^{2|B_m|}(\theta_j - \theta_{0j})^2\\
    &\leq e^{-2|B_m|q_{\min}}(\theta_j - \theta_{0j})^2 \leq \frac{(\theta_j - \theta_{0j})^2}{e|B_m|q_{\min}}.
\end{align*}
This proves~\eqref{ineq:bias-of-block-means}.

For~\eqref{ineq:trace-bound-iterative-imputation} and~\eqref{ineq:op-norm-bound-iterative-imputation}, we compute the entries of the matrix $(\bar{Z}_m - \theta_0)(\bar{Z}_m - \theta_0)^\top$, beginning with those on the diagonal.  For $j \in [d]$, let
\begin{align} 
\label{ineq:A_jj-bound}
    A_{jj} \coloneqq \mathbb{E} \biggl( \frac{|B_m|q_j}{\sum_{i \in B_m} \Omega_{ij}} \cdot \mathbbm{1}_{\{\sum_{i\in B_m} \Omega_{ij} > 0\}}  \biggr) \leq 2,
\end{align} 
where the inequality follows by the first part of Lemma~\ref{lem:inverse-binomial-bounds}.  Further, let $E_{jj} \coloneqq (1 - q_j)^{|B_m|}$.  For $i \in \mathcal{I}_n$, we can write $Z_i = X_i \ostar \Omega_i$, where $\mathbb{E}(X_i) = \theta_0$, $\mathrm{Cov}(X_i) = \Sigma$ and $X_i \indep \Omega_i$.  Hence, for any $m \in [M]$ such that $B_m \subseteq \mathcal{I}_n$ and any $j \in [d]$, 
\begin{align*}
    \mathbb{E} \bigl\{ (\bar{Z}_{mj} &- \theta_{0,j})^2 \bigr\} = \mathbb{E} \bigl[ \bigl\{ \overbar{\Omega}_{mj}(\bar{Z}_{mj} - \theta_{0,j}) + (1 - \overbar{\Omega}_{mj})(\theta_j - \theta_{0,j}) \bigr\}^2 \bigr]\\
    &= \mathbb{E} \bigl\{ \bigl(\overbar{\Omega}_{mj}(\bar{Z}_{mj} - \theta_{0,j}) \bigr)^2 \bigr\} + \mathbb{E} \bigl\{ (1 - \overbar{\Omega}_{mj})^2 (\theta_j - \theta_{0,j})^2 \bigr\} \\
    &= \mathbb{E} \biggl\{ \biggl( \frac{\sum_{i \in B_m} \Omega_{ij}(X_{ij} \!-\! \theta_{0,j})}{\sum_{i \in B_m} \Omega_{ij}} \cdot \mathbbm{1}_{\{\sum_{i\in B_m} \Omega_{ij} > 0\}} \biggr)^2 \biggr\} + \mathbb{P}(\overbar{\Omega}_{mj} = 0)(\theta_j - \theta_{0,j})^2 \\
    &= \mathbb{E} \biggl( \frac{\Sigma_{jj}}{\sum_{i \in B_m} \Omega_{ij}} \cdot \mathbbm{1}_{\{\sum_{i\in B_m} \Omega_{ij} > 0\}}  \biggr) + (1 - q_j)^{|B_m|} (\theta_j - \theta_{0,j})^2 \\
    &= A_{jj} \cdot \frac{\Sigma_{jj}^{\mathrm{IPW}}}{|B_m|} + E_{jj} \cdot (\theta_j - \theta_{0,j})^2. \numberthis \label{eq:diagonal-entries}
\end{align*}
Turning to the off-diagonal entries, for any $m \in [M]$ such that $B_m \subseteq \mathcal{I}_n$ and any distinct $j,k \in [d]$,
\begin{align*}
    \mathbb{E} \bigl\{ (\bar{Z}_{mj} &- \theta_{0,j}) (\bar{Z}_{mk} - \theta_{0,k}) \bigr\} \\
    =\; &\mathbb{E} \bigl[ \bigl\{ \overbar{\Omega}_{mj}(\bar{Z}_{mj} - \theta_{0,j}) + (1 - \overbar{\Omega}_{mj})(\theta_j - \theta_{0,j}) \bigr\} \\
    &\hspace{4cm}\cdot \bigl\{ \overbar{\Omega}_{mk}(\bar{Z}_{mk} - \theta_{0,k}) + (1 - \overbar{\Omega}_{mk})(\theta_k - \theta_{0,k}) \bigr\} \bigr]\\
    =\; & \mathbb{E} \bigl\{ \bigl(\overbar{\Omega}_{mj}(\bar{Z}_{mj} - \theta_{0,j}) \bigr) \bigl(\overbar{\Omega}_{mk}(\bar{Z}_{mk} - \theta_{0,k}) \bigr) \big\} \\
    &\hspace{4cm}+ \mathbb{E} \bigl\{ (1 - \overbar{\Omega}_{mj}) (1 - \overbar{\Omega}_{mk}) (\theta_j - \theta_{0,j}) (\theta_k - \theta_{0,k}) \big\},
\end{align*}
where in the final step, the cross-terms vanish as $\mathbb{E}(X_i) = \theta_0$.  
Without loss of generality, we assume that $1\in B_m$.  For the first term, we first define
\begin{align}
    A_{jk} &\coloneqq \mathbb{E} \Biggl\{ \frac{ (|B_m|q_j)(|B_m|q_k) }{\big(1 + \sum_{i \in B_m\setminus\{1\}} \Omega_{ij} \big) \cdot \big(1 + \sum_{i \in B_m\setminus\{1\}} \Omega_{i k} \big)} \Biggr\} \nonumber\\
    &\hspace{0.09cm}\leq \mathbb{E} \Biggl\{ \frac{ (|B_m|q_j)^2}{\bigl(1 + \sum_{i \in B_m\setminus\{1\}} \Omega_{ij} \bigr)^2} \Biggr\}^{1/2} \mathbb{E} \Biggl\{ \frac{ (|B_m|q_k)^2}{\bigl(1 + \sum_{i \in B_m\setminus\{1\}} \Omega_{ik} \bigr)^2} \Biggr\}^{1/2} \nonumber \\
    &\hspace{0.09cm}\leq \frac{2|B_m|^2}{(|B_m|-1)^2} \leq 4, \label{ineq:A_jk-bound}
\end{align}
where the first inequality follows from Cauchy--Schwarz and the second inequality follows from the second part of Lemma~\ref{lem:inverse-binomial-bounds}, and the final inequality uses the fact that $|B_m| \geq 4$. We then have
\begin{align*}
    \mathbb{E} &\bigl\{ \overbar{\Omega}_{mj}(\bar{Z}_{mj} - \theta_{0,j}) \cdot \overbar{\Omega}_{mk}(\bar{Z}_{mk} - \theta_{0,k}) \bigr\} \\
    &= \Sigma_{jk} \cdot \mathbb{E} \Biggl\{ \frac{ \bigl(\sum_{i \in B_{m}}\Omega_{ij} \Omega_{i k} \bigr)\overbar{\Omega}_{mj} \overbar{\Omega}_{mk}}{\big(\sum_{i \in B_m} \Omega_{ij} \big) \cdot \big(\sum_{i \in B_m} \Omega_{i k} \big)}  \Biggr\}\\
    &=  \Sigma_{jk} \cdot |B_m| \cdot \mathbb{E} \Biggl\{ \frac{ \Omega_{1j} \Omega_{1 k} }{\big(\sum_{i \in B_m} \Omega_{ij} \big) \cdot \big(\sum_{i \in B_m} \Omega_{i k} \big)}  \Biggr\} \\
    &= \Sigma_{jk} \cdot |B_m| \cdot \mathbb{P}(\Omega_{1j} = \Omega_{1 k} = 1) \cdot \mathbb{E} \Biggl\{ \frac{ 1 }{\big(1 + \sum_{i \in B_m\setminus\{1\}} \Omega_{ij} \big) \cdot \big(1 + \sum_{i \in B_m\setminus\{1\}} \Omega_{i k} \big)} \Biggr\} \\
    &= A_{jk} \cdot \frac{\Sigma_{jk}q_{jk}}{|B_m|q_j q_k}, 
\end{align*}
where the first equality follows from substituting the definition of $\bar{Z}_{mj}$ on the event $\{\overbar{\Omega}_{mj} = 1\}$ (and similarly for $k$) and the second equality follows by symmetry.
For the second term, we have 
\begin{align*}
    \mathbb{E} \bigl\{ (1 - \overbar{\Omega}_{mj}) (1 - \overbar{\Omega}_{mk}) (\theta_j - \theta_{0,j}) &(\theta_k - \theta_{0,k}) \bigr\} \\
    &= \mathbb{P}(\overbar{\Omega}_{mj} = \overbar{\Omega}_{mk} = 0) \cdot (\theta_j - \theta_{0,j}) (\theta_k - \theta_{0,k}) \\
    &= (1 - q_j - q_k + q_{jk})^{|B_m|} \cdot (\theta_j - \theta_{0,j}) (\theta_k - \theta_{0,k})\\
    &\eqqcolon E_{jk}\cdot (\theta_j - \theta_{0,j}) (\theta_k - \theta_{0,k}).
\end{align*}
Combining these two equalities then yields 
\begin{align} \label{eq:off-diagonal-entries}
    \mathbb{E} \big[(\bar{Z}_{mj} - \theta_{0,j}) (\bar{Z}_{mk} - \theta_{0,k})  \big] = A_{jk} \cdot \frac{1}{|B_m|} \cdot \Sigma^{\mathrm{IPW}}_{jk} + E_{jk}\cdot (\theta_j - \theta_{0,j}) (\theta_k - \theta_{0,k}).
\end{align}
Therefore, by \eqref{eq:diagonal-entries} and \eqref{eq:off-diagonal-entries},
\begin{align*}
    \mathbb{E} \bigl\{ (\bar{Z}_m - \theta_0)(\bar{Z}_m - \theta_0)^\top \bigr\} = \frac{1}{|B_m|} \cdot A\odot \Sigma^{\mathrm{IPW}} + E\odot \bigl\{ (\theta - \theta_0)(\theta - \theta_0)^\top \bigr\},
\end{align*}
where $A \coloneqq (A_{jk})_{j,k\in[d]}$ and $E \coloneqq (E_{jk})_{j,k\in[d]}$.  The desired inequality~\eqref{ineq:trace-bound-iterative-imputation} then follows as
\begin{align*}
    \tr\bigl( \mathbb{E} \bigl\{ (\bar{Z}_m - \theta_0)(\bar{Z}_m - \theta_0)^\top \bigr\} \bigr) &= \frac{1}{|B_m|} \cdot \sum_{j=1}^d A_{jj} \Sigma^{\mathrm{IPW}}_{jj} + \sum_{j=1}^d E_{jj} (\theta_j - \theta_{0,j})^2\\
    &\leq \frac{2}{|B_m|}\cdot \tr(\Sigma^{\mathrm{IPW}}) + \frac{\|\theta - \theta_0\|_2^2}{e|B_m| q_{\min}},
\end{align*}
where the inequality follows by~\eqref{ineq:A_jj-bound} and Lemma~\ref{lemma:controlling-matrix-E}.  

For inequality~\eqref{ineq:op-norm-bound-iterative-imputation}, we define a matrix $A' = (A_{jk}') \in \mathbb{R}^{d \times d}$ by $A'_{jk} \coloneqq A_{jk}$ for $j\neq k$ and 
\begin{align} \label{ineq:A'_jj-bound}
    A'_{jj} \coloneqq \mathbb{E} \Biggl\{ \frac{ (|B_m|q_j)^2}{\bigl(1 + \sum_{i \in B_m\setminus\{1\}} \Omega_{ij} \bigr)^2} \Biggr\} \leq 2,
\end{align}
where the inequality follows from the second part of Lemma~\ref{lem:inverse-binomial-bounds} and the assumption that $|B_{m}| \geq 4$.
Note that $A'$ is a positive semi-definite matrix, as it is the expectation of a positive semi-definite matrix. Now 
\begin{align*}
    \bigl\| \mathbb{E} &\bigl\{ (\bar{Z}_m - \theta_0)(\bar{Z}_m - \theta_0)^\top \bigr\} \bigr\|_{\mathrm{op}}\\
    &= \biggl\| \frac{1}{|B_m|} \cdot A\odot \Sigma^{\mathrm{IPW}} + E\odot \bigl\{ (\theta - \theta_0)(\theta - \theta_0)^\top \bigr\} \biggr\|_{\mathrm{op}}\\
    &\leq \frac{\bigl\|A' \odot \Sigma^{\mathrm{IPW}} \bigr\|_{\mathrm{op}}}{|B_m|} + \frac{\bigl\|(A-A') \odot \Sigma^{\mathrm{IPW}} \bigr\|_{\mathrm{op}}}{|B_m|} + \bigl\|E\odot \bigl\{ (\theta - \theta_0)(\theta - \theta_0)^\top \bigr\} \bigr\|_{\mathrm{op}}\\
    \overset{(i)}&{\leq}  \frac{\|A'\|_{\infty} \|\Sigma^{\mathrm{IPW}} \bigr\|_{\mathrm{op}}}{|B_m|}+ \frac{\|A-A'\|_{\infty} \bigl\|\Sigma^{\mathrm{IPW}} \bigr\|_{\mathrm{op}}}{|B_m|} + \bigl\|E\odot \bigl\{ (\theta - \theta_0)(\theta - \theta_0)^\top \bigr\} \bigr\|_{\mathrm{op}}\\
    \overset{(ii)}&{\leq} \frac{6}{|B_m|} \bigl\|\Sigma^{\mathrm{IPW}} \bigr\|_{\mathrm{op}} + \frac{\|\theta - \theta_0\|_2^2}{e|B_m| q_{\min}},
\end{align*}
where the first term in step $(i)$ follows from Lemma~\ref{lemma:operator-norm-of-hadamard-product} since $A'$ is positive semidefinite, the second term in step $(i)$ follows since $A-A'$ is diagonal, and step $(ii)$ follows from the inequalities~\eqref{ineq:A_jj-bound},~\eqref{ineq:A_jk-bound} and~\eqref{ineq:A'_jj-bound}, as well as Lemma~\ref{lemma:controlling-matrix-E}.
\end{proof}

\begin{lemma}\label{lemma:controlling-matrix-E}
    Under the set up in the proof of Lemma~\ref{lemma:covariance-of-imputed-block-means}, we have \begin{align*}
        \|E\|_{\infty} \leq \frac{1}{e|B_m|q_{\min}} \quad\text{and}\quad \|E\odot \bigl\{ (\theta - \theta_0)(\theta - \theta_0)^\top \bigr\}\|_{\mathrm{op}} \leq \frac{\|\theta - \theta_0\|_2^2}{e|B_m|q_{\min}}.
    \end{align*}
\end{lemma}
\begin{proof}
We will make use of the following inequality
\begin{align} \label{eq:block-means-simple-ineq}
    (1-x)^k \leq \frac{1}{ekx} \quad\text{for all }  x \in (0,1] \text{ and } k \in \mathbb{N}.
\end{align}
To see this, note that $k\log(1-x) \leq -kx \leq -\log(kx) - 1$.  Hence, for each $j \in [d]$,
\begin{align*}
    E_{jj} = (1-q_j)^{|B_m|} \leq \frac{1}{e|B_m|q_{\min}},
\end{align*}
and for each $j,k \in [d]$, 
\begin{align*}
    E_{jk} = (1-q_j-q_k + q_{jk})^{|B_m|} \leq \frac{1}{e|B_m|(q_j+q_k - q_{jk})} \leq \frac{1}{e|B_m|q_{\min}},
\end{align*}
where the final inequality follows since $q_{jk}\leq q_k$, so that $q_j+q_k-q_{jk}\geq q_j \geq q_{\min}$.  This establishes the first inequality.

For the second bound, we have 
\begin{align*}
    \bigl| \bigl[E\odot \bigl\{ (\theta - \theta_0)(\theta - \theta_0)^\top\bigr\}\bigr]_{jk} \bigr| \leq \frac{1}{e|B_m|q_{\min}} \cdot |\theta_j - \theta_{0,j}| \cdot |\theta_k - \theta_{0,k}|.
\end{align*}
Hence
\begin{align*}
    \bigl\|E\odot \bigl\{ (\theta - \theta_0)(\theta - \theta_0)^\top \bigr\}\bigr\|_{\mathrm{op}} \leq \frac{1}{e|B_m|q_{\min}} \bigl\| |\theta - \theta_0| \cdot |\theta - \theta_0|^\top \bigr\|_{\mathrm{op}} = \frac{\|\theta - \theta_0\|_2^2}{e|B_m|q_{\min}},
\end{align*}
where $|\theta - \theta_0|$ denotes the entrywise absolute value, and the inequality follows from the fact\footnote{To see this, observe that $v^\top A v \leq |v|^\top |A| |v| \leq |v|^\top B |v|$ for all $v\in\mathbb{R}^d$, where $|A|$ denotes the entrywise absolute value of $A$.} that if $A = (A_{jk}),B = (B_{jk}) \in \mathcal{S}^{d\times d}$ are such that $|A_{jk}| \leq B_{jk}$ for all $j,k\in[d]$, then $\|A\|_{\mathrm{op}} \leq \|B\|_{\mathrm{op}}$.
\end{proof}

\begin{lemma}\label{lemma:error-per-iteration}
    Let $\mathrm{ALG}$ satisfy~\eqref{eq:assumption-on-alg} for some $\epsilon_{\max} \in (0,1/2)$, $a \in (0,1]$, $C > 0$, and let $n\geq4$, $\epsilon\in\bigl[0,\frac{-\log(1-\epsilon_{\max})}{16}\bigr]$, $\delta\in[e^{-an/8},1]$ and $M\coloneqq \bigl\lceil \frac{2n\epsilon}{-\log(1-\epsilon_{\max})} \vee \log(1/\delta)\bigr\rceil$. Let $Z_1, \ldots, Z_{n} \stackrel{\mathrm{iid}}{\sim} P \in \mathcal{P}^{\mathrm{arb}} \big(\theta_0, \Sigma, \epsilon, \pi \big)$, let $\bar{Z}_1,\ldots,\bar{Z}_M$ be defined as in Lemma~\ref{lemma:covariance-of-imputed-block-means} for some $\theta \in \mathbb{R}^d$ and let
    \begin{align*}
        \tilde{\theta}_n \coloneqq \mathrm{ALG}(\bar{Z}_1,\ldots,\bar{Z}_M;\epsilon_{\max},\delta).
    \end{align*}
    Then, writing $C'\coloneqq 48C+2$, we have with probability at least $1-\delta$ that 
    \begin{align*}
        \|\tilde{\theta}_n - \theta_0\|_2^2 \leq C'\biggl( \frac{\tr(\Sigma^{\mathrm{IPW}})}{n} + \frac{\|\Sigma^{\mathrm{IPW}}\|_{\mathrm{op}}\log(1/\delta)}{n} + \|\Sigma^{\mathrm{IPW}}\|_{\mathrm{op}}\epsilon + \frac{M\|\theta-\theta_0\|_2^2}{nq_{\min}}\biggr).
    \end{align*}
\end{lemma}
\begin{proof}
    By our assumptions on $\epsilon$ and $\delta$, we have $M\leq n/4$. Moreover, for $m\in[M]$,
    \begin{align*}
        \mathbb{P}(B_m \subseteq \mathcal{I}_n) = (1-\epsilon)^{|B_m|} \geq (1-\epsilon)^{\frac{-\log(1-\epsilon_{\max})}{2\epsilon}} \geq (1-\epsilon)^{\frac{-\log(1-\epsilon_{\max})}{-\log(1-\epsilon)}} = 1-\epsilon_{\max}.
    \end{align*}
    Let $\mu \coloneqq \mathbb{E}(\bar{Z}_m \, | \, B_m \subseteq \mathcal{I}_n) \in \mathbb{R}^d$ and $\Gamma  \coloneqq \mathrm{Cov}(\bar{Z}_m \, | \, B_m \subseteq \mathcal{I}_n) \in \mathbb{R}^{d \times d}$ denote respectively the mean vector and covariance matrix of $\bar{Z}_m$ given that it is uncontaminated.  Then  $\bar{Z}_1,\ldots,\bar{Z}_M \overset{\mathrm{iid}}{\sim} (1-\epsilon_{\max})\bar{P} + \epsilon_{\max}\bar{Q}$, where $\bar{P}\in\mathcal{P}(\mu,\Gamma)$ and $\bar{Q} \in \mathcal{P}(\mathbb{R}^d)$. Thus, by~\eqref{eq:assumption-on-alg} and Lemma~\ref{lemma:covariance-of-imputed-block-means}, we have
    \begin{align*}
        \|\tilde{\theta}_n - \theta_0\|_2^2 &\leq 2\|\tilde{\theta}_n - \mu\|_2^2 + 2\|\mu - \theta_0\|_2^2 \\
        &\leq 2C\biggl(\frac{\tr(\Gamma)}{M} + \frac{\|\Gamma\|_{\mathrm{op}}\log(1/\delta)}{M} + \epsilon_{\max}\|\Gamma\|_{\mathrm{op}}\biggr) + 2\|\mu - \theta_0\|_2^2\\
        &\leq C'\biggl(\frac{\tr(\Sigma^{\mathrm{IPW}})}{n} + \frac{\|\Sigma^{\mathrm{IPW}}\|_{\mathrm{op}}\log(1/\delta)}{n} + \|\Sigma^{\mathrm{IPW}}\|_{\mathrm{op}}\epsilon + \frac{M\|\theta-\theta_0\|_2^2}{nq_{\min}}\biggr),
    \end{align*}
    as required.
\end{proof}


\begin{lemma} \label{lemma:error-of-initialisation}
    Let $\mathrm{ALG}$ satisfy~\eqref{eq:assumption-on-alg} for some $\epsilon_{\max} \in (0,1/2)$, $a \in (0,1]$, $C > 0$. Let $\epsilon\in \bigl[0, \frac{\epsilon_{\max}}{1+\epsilon_{\max}}q_{\min}\bigr]$, $\delta\in\bigl[2de^{-anq_{\min}/8}, 1\bigr]$ and $Z_1,\ldots,Z_n \overset{\mathrm{iid}}{\sim} P\in\mathcal{P}^{\mathrm{arb}}(\theta_0,\Sigma,\epsilon,\pi)$. For $j\in[d]$, let $I_j \coloneqq \{i\in[n] : Z_{ij} \neq \star\}$, $\tilde{\theta}_{n,j} \coloneqq \mathrm{ALG}\bigl((Z_{ij})_{i\in I_j}; \frac{\epsilon}{q_j(1-\epsilon)}, \frac{\delta}{2d}\bigr)$ and $\tilde{\theta}_n \coloneqq (\tilde{\theta}_{n,1}, \ldots, \tilde{\theta}_{n,d})^\top$. Then, with probability at least $1-\delta$,
    \begin{align*}
        \|\tilde{\theta}_n - \theta_0\|_2^2 \leq C\biggl(\frac{2\tr(\Sigma^{\mathrm{IPW}})\log(2ed/\delta)}{n} + \frac{\epsilon}{1-\epsilon}\tr(\Sigma^{\mathrm{IPW}})\biggr).
    \end{align*}
\end{lemma}
\begin{proof}
    By Bayes' theorem, for $i\in[n]$ and $j \in [d]$,
    \begin{align*}
        \mathbb{P}(i\in\mathcal{O}_n \,|\, Z_{ij} \neq \star) = \frac{\mathbb{P}( Z_{ij} \neq \star \,|\, i\in\mathcal{O}_n) \mathbb{P}(i\in\mathcal{O}_n)}{\mathbb{P}(Z_{ij} \neq \star)} \leq \frac{\epsilon}{q_j(1-\epsilon)}\eqqcolon\kappa_j,
    \end{align*}
    and $\kappa_j\leq \epsilon_{\max}$ by our assumption on $\epsilon$.
    Thus, writing $\theta_0 = (\theta_{0,1},\ldots,\theta_{0,d})^\top$ and $\Sigma = (\Sigma_{ij})_{i,j \in [d]}$, we have conditional on $I_j$ that  $(Z_{ij})_{i\in I_j} \overset{\mathrm{iid}}{\sim} (1-\kappa_j)P_j'+\kappa_j Q_j'$ where $P_j'\in\mathcal{P}(\theta_{0,j},\Sigma_{jj})$ and $Q_j'\in\mathcal{P}(\mathbb{R})$. Therefore, by~\eqref{eq:assumption-on-alg},
    \begin{align*}
        \mathbb{P}\biggl\{(\tilde{\theta}_{n,j}-\theta_{0,j})^2 \leq C\biggl(\frac{2\Sigma_{jj}\log(2ed/\delta)}{nq_j} + \kappa_j\Sigma_{jj}\biggr) \,\bigg|\, |I_j| \geq \frac{nq_j}{2}\biggr\} \geq 1-\frac{\delta}{2d}.
    \end{align*}
    Moreover, by Lemma~\ref{lemma:binomial-tail}(b) and since $\delta\geq 2de^{-nq_{\min}/8}$, we have 
    \begin{align*}
        \mathbb{P}\biggl(|I_j| \geq \frac{nq_j}{2}\biggr) \geq 1-\frac{\delta}{2d}.
    \end{align*}
    Hence,
    \begin{align*}
        \mathbb{P}\biggl\{&(\tilde{\theta}_{n,j}-\theta_{0,j})^2 \leq C\biggl(\frac{2\Sigma_{jj}\log(2ed/\delta)}{nq_j} + \kappa_j\Sigma_{jj}\biggr)\biggr\}\\
        &\geq \mathbb{P}\biggl\{(\tilde{\theta}_{n,j}-\theta_{0,j})^2 \leq C\biggl(\frac{2\Sigma_{jj}\log(2ed/\delta)}{nq_j} + \kappa_j\Sigma_{jj}\biggr) \,\bigg|\, |I_j| \geq \frac{nq_j}{2}\biggr\} \mathbb{P}\biggl(|I_j| \geq \frac{nq_j}{2}\biggr) \\
        &\geq 1-\frac{\delta}{d}.
    \end{align*}
    The final result now follows by a union bound.
\end{proof}

\begin{proof}[Proof of Theorem~\ref{thm:robust-descent-iterative-imputation-ub}]
    Let $C'\coloneqq 48C+2$, and recall the definition of $M$ from Algorithm~\ref{alg:robust-iterative-imputation}.  By Lemma~\ref{lemma:error-per-iteration}, for $t\in[T-1]$, we have with probability at least $1-\delta/(2T)$ that
    \begin{align*}
        \|\hat{\theta}^{(t+1)} &- \theta_0\|_2^2 \\
        &\leq 2C'\biggl( \frac{T\tr(\Sigma^{\mathrm{IPW}})}{n} + \frac{T\|\Sigma^{\mathrm{IPW}}\|_{\mathrm{op}}\log(2T/\delta)}{n} + \|\Sigma^{\mathrm{IPW}}\|_{\mathrm{op}}\epsilon + \frac{TM\|\hat{\theta}^{(t)}-\theta_0\|_2^2}{nq_{\min}}\biggr)\\
        &\eqqcolon \alpha + \beta \|\hat{\theta}^{(t)} \!-\!\theta_0\|_2^2.
    \end{align*}
    By assumption, we have $(192C+8)TM \leq nq_{\min}$, so $\beta\leq 1/2$. Therefore, by a union bound, with probability at least $1-\delta/2$, we have
    \begin{align}
        \|\hat{\theta}^{(T)} - \theta_0\|_2^2 \leq \alpha \sum_{\ell=0}^{T-2} \beta^{\ell} + \beta^{T-1} \|\hat{\theta}^{(1)} - \theta_0\|_2^2 \leq 2\alpha + \frac{\|\hat{\theta}^{(1)} - \theta_0\|_2^2}{2^{T-1}}. \label{eq:theta^T-theta_0}
    \end{align}
    Moreover, by Lemma~\ref{lemma:error-of-initialisation}, we have with probability at least $1-\delta/2$ that
    \begin{align}
        \|\hat{\theta}^{(1)} - \theta_0\|_2^2 \leq C\biggl(\frac{4T\tr(\Sigma^{\mathrm{IPW}})\log(2ed/\delta)}{n} + \frac{\epsilon}{1-\epsilon}\tr(\Sigma^{\mathrm{IPW}})\biggr). \label{eq:theta^1-theta_0}
    \end{align}
    Our choice of $T$ ensures that on combining~\eqref{eq:theta^T-theta_0} and~\eqref{eq:theta^1-theta_0} we obtain that with probability at least $1-\delta$,
    \begin{align*}
        \|\hat{\theta}^{(T)} - \theta_0\|_2^2 &\leq 3\alpha \\
        &= (288C+12)\biggl( \frac{T\tr(\Sigma^{\mathrm{IPW}})}{n} + \frac{T\|\Sigma^{\mathrm{IPW}}\|_{\mathrm{op}}\log(2T/\delta)}{n} + \|\Sigma^{\mathrm{IPW}}\|_{\mathrm{op}}\epsilon\biggr),
    \end{align*}
    as required.
\end{proof}

\subsection{Proof of Theorem~\ref{thm:arbitrary-contamination-lb}} 
\begin{proof}[Proof of Theorem~\ref{thm:arbitrary-contamination-lb}]
First, note that when $\epsilon = 0$, by Proposition~\ref{prop:arb-mean-MCAR-lb}, we have
\begin{align}\label{ineq:mcar-lb1}
\mathcal{M}\bigl(\delta, \mathcal{P}_{\Theta}, \| \cdot \|_2^2\bigr) \gtrsim \frac{\tr(\Sigma^{\mathrm{IPW}})}{n} + \frac{\| \Sigma^{\mathrm{IPW}}\|_{\mathrm{op}} \log(1/\delta)}{n}. 
\end{align}
    Now we consider the case $\epsilon \in \bigl(0,\frac{q_{\min}}{1+q_{\min}} \bigr)$.  
    Without loss of generality, assume that $\Sigma^{\mathrm{IPW}}_{11} = \max_{j \in [d]} \Sigma^{\mathrm{IPW}}_{jj}$, and let $a \coloneqq (\alpha+\alpha^2)/2$ and $b \coloneqq (3\alpha+\alpha^2)/2$ 
    for some $\alpha \in (0,1/3]$ to be chosen later.  Define random vectors $X^{(1)} = (X^{(1)}_1, \ldots, X^{(1)}_d)^\top \sim P^{(1)}$ and $X^{(2)} = (X^{(2)}_1, \ldots, X^{(2)}_d)^\top \sim P^{(2)}$ with independent components satisfying
    \begin{align*}
        X^{(1)}_1 \coloneqq \begin{cases}
            -\sqrt{\frac{\Sigma_{11}}{2\alpha}} \quad &\text{with prob. } \alpha\\
            0 &\text{with prob. } 1-2\alpha,\\
            \sqrt{\frac{\Sigma_{11}}{2\alpha}} \quad &\text{with prob. } \alpha
        \end{cases} \quad 
        X^{(2)}_1 \coloneqq \begin{cases}
            -\sqrt{\frac{\Sigma_{11}}{2\alpha}} \quad &\text{with prob. } a\\
            0 &\text{with prob. } 1 - a - b \\
            \sqrt{\frac{\Sigma_{11}}{2\alpha}} \quad &\text{with prob. } b,
        \end{cases}
    \end{align*}
    and $X^{(1)}_j \overset{d}{=} X^{(2)}_j \sim \mathsf{N}(0, \Sigma_{jj})$ for $j \in \{2, \ldots, d\}$.  Then 
    \[
    \Var(X^{(2)}_1) = \frac{(a + b + 2ab - a^2 - b^2)\Sigma_{11}}{2\alpha} = \Sigma_{11}.
    \]
    Thus $\Cov(X^{(1)}) = \Cov(X^{(2)}) = \Sigma$, so $P^{(\ell)} \in \mathcal{P}\bigl( \mathbb{E}(X^{(\ell)}), \Sigma \bigr)$ for $\ell \in \{1,2\}$, and
    \begin{align*}
        \bigl\| \mathbb{E}(X^{(1)}) - \mathbb{E}(X^{(2)}) \bigr\|_2^2 = \frac{\alpha \Sigma_{11}}{2}.
    \end{align*} 
    Moreover, by Lemma~\ref{lem:equivalence-of-f-af}, we have 
    \begin{align*}
        \mathrm{TV}\bigl(\mathsf{MCAR}_{(\pi, P^{(1)})},\mathsf{MCAR}_{(\pi, P^{(2)})}\bigl) &= \mathrm{ATV}(P^{(1)}, P^{(2)}, \pi) = \sum_{S: 1\in S} \pi(S) \cdot \mathrm{TV}\bigl(P^{(1)}_S, P^{(2)}_S\bigr)\\
        &= q_1 \cdot \mathrm{TV}\bigl(P^{(1)}_1, P^{(2)}_1\bigr) = \frac{q_1}{2}\biggl( \frac{\alpha \!-\! \alpha^2}{2} + \alpha^2 + \frac{\alpha \!+\! \alpha^2}{2} \biggr) \\
        &\leq q_1 \alpha.
    \end{align*}
    We then pick $\alpha = \epsilon / (3 q_1) < 1/3$ since $\epsilon < q_{\min}$ so that
    \[
    \mathrm{TV}\bigl(\mathsf{MCAR}_{(\pi, P^{(1)})},\mathsf{MCAR}_{(\pi, P^{(2)})}\bigl) \leq \epsilon \leq \frac{\epsilon}{1 - \epsilon}, \quad \text{and} \quad \bigl\| \mathbb{E}(X^{(1)}) - \mathbb{E}(X^{(2)}) \bigr\|_2^2 = \frac{\epsilon \Sigma_{11}^{\mathrm{IPW}}}{6}.
    \]
    Consequently, by~\citet[Theorem 4 and Lemma 25]{ma2024high}, we have
    \begin{align}\label{ineq:arb-lb-epsilon-term}
        \mathcal{M}\bigl(\delta, \mathcal{P}_{\Theta}, \| \cdot \|_2^2\bigr) \geq \frac{\bigl\| \mathbb{E}(X^{(1)}) - \mathbb{E}(X^{(2)}) \bigr\|_2^2}{4} = \frac{\epsilon \Sigma_{11}^{\mathrm{IPW}}}{24}.
    \end{align}
    Combining~\eqref{ineq:mcar-lb1} and~\eqref{ineq:arb-lb-epsilon-term} yields the desired result.  
    
    Next, we consider the case where $\epsilon \geq \frac{q_{\min}}{1+q_{\min}}$. Without loss of generality, assume that $q_1 = q_{\min}$.  Let $\theta^{(1)} \coloneqq (2t^{1/2}, 0, \ldots, 0)^\top \in \mathbb{R}^d$ for some $t>0$ and let $\theta^{(2)} \coloneqq 0 \in \mathbb{R}^d$.  Writing $P^{(1)} \coloneqq \mathsf{N}(\theta^{(1)}, \Sigma)$ and $P^{(2)} \coloneqq \mathsf{N}(\theta^{(2)}, \Sigma)$, we have by Lemma~\ref{lem:equivalence-of-f-af} that
    \begin{align*}
        \mathrm{TV}\bigl(\mathsf{MCAR}_{(\pi, P^{(1)})},\mathsf{MCAR}_{(\pi, P^{(2)})}\bigl) &= \mathrm{ATV}(P^{(1)}, P^{(2)}; \pi) = \sum_{S: 1\in S} \pi(S) \cdot \mathrm{TV}\bigl(P^{(1)}_S, P_S^{(2)}\bigr)\\
        &= q_1 \mathrm{TV}\bigl(P^{(1)}_{\{1\}}, P^{(2)}_{\{1\}}\bigr) \leq q_1 \leq \frac{\epsilon}{1-\epsilon}.
    \end{align*}
    Hence, by~\citet[Theorem 4 and Lemma 25]{ma2024high}, we see that $\mathcal{M}\bigl(\delta, \mathcal{P}_{\Theta}, \| \cdot \|_2^2\bigr) \geq t$.  Since $t > 0$ was arbitrary, the result follows.
\end{proof}

\subsection{Univariate arbitrary contamination lower bounds} \label{sec:univariate-arbitrary-contamination-lb}

The lower bounds in Proposition~\ref{Prop:Univariate-arb-contam-lb} are presented primarily to ensure the completeness of Table~\ref{table:summary}.  Corresponding upper bounds are attained by the median in the Gaussian case \citep[][Theorem~2.1]{chen2018robust}, and a trimmed mean \citep[][Theorem~1 and the subsequent remark]{lugosi21robust} in the sub-Gaussian case, in both cases applied to the observed data.
\begin{prop}
\label{Prop:Univariate-arb-contam-lb}
    Let $\epsilon\in[0,1)$, $q\in(0,1]$, $\sigma>0$, $\delta\in(0,1/4]$ and $\kappa \coloneqq \frac{\epsilon}{q(1-\epsilon)}$. 
    \begin{enumerate}
        \item[(a)] Let $\Theta\coloneqq\mathbb{R}$ and let $\mathcal{P}_{\theta} \coloneqq \bigl\{ P_0^{\otimes n} : P_0\in\mathcal{P}^{\mathrm{arb}}\bigl(\mathsf{N}(\theta,\sigma^2),\epsilon,q\bigr)\bigr\}$ for $\theta\in\Theta$. Then
        \begin{align*}
            \mathcal{M}(\delta,\mathcal{P}_{\Theta},|\cdot|^2) \begin{cases}
                \gtrsim \dfrac{\sigma^2\log(1/\delta)}{nq(1-\epsilon)} + \sigma^2 \kappa^2 \quad&\text{if } \epsilon < \frac{q}{1+q}\\
                =\infty &\text{if }\epsilon \geq \frac{q}{1+q}.
            \end{cases}
        \end{align*}
        \item[(b)] Let $\Theta\coloneqq\mathbb{R}$ and let $\mathcal{P}_{\theta} \coloneqq \bigl\{ P_0^{\otimes n} : P_0\in\mathcal{P}^{\mathrm{arb}}(P,\epsilon,q),\, P\in\mathcal{P}_{\psi_2}(\theta,\sigma^2) \bigr\}$ for $\theta\in\Theta$. Then
        \begin{align*}
            \mathcal{M}(\delta,\mathcal{P}_{\Theta},|\cdot|^2) \begin{cases}
                \gtrsim \dfrac{\sigma^2\log(1/\delta)}{nq(1-\epsilon)} + \sigma^2 \kappa^2 \log(1/\kappa) \quad&\text{if } \epsilon < \frac{q}{1+q}\\
                =\infty &\text{if }\epsilon \geq \frac{q}{1+q}.
            \end{cases}
        \end{align*}
    \end{enumerate}    
\end{prop}
\begin{proof}
    (a) First consider the case where $\epsilon < \frac{q}{1+q}$. Let $X_1\sim \mathsf{N}(0,\sigma^2) \eqqcolon P_1$ and $X_2 \sim \mathsf{N}(2\sigma\kappa,\sigma^2) \eqqcolon P_2$. By Pinsker's inequality, $\mathrm{TV}(P_1,P_2) \leq \sqrt{\frac{1}{2}\mathrm{KL}(P_1,P_2)} = \kappa$, so that by Lemma~\ref{lem:equivalence-of-f-af},
    \begin{align*}
        \mathrm{TV}\bigl( \mathsf{MCAR}_{(q,P_1)}, \mathsf{MCAR}_{(q,P_2)} \bigr) = q\cdot\mathrm{TV}(P_1,P_2) \leq \frac{\epsilon}{1-\epsilon}.
    \end{align*}
    Hence, by \citet[Lemma~25]{ma2024high}, we have
    \begin{align}
        \mathcal{M}(\delta,\mathcal{P}_{\Theta},|\cdot|^2) \geq \frac{(\mathbb{E}X_1 - \mathbb{E}X_2)^2}{4} = \sigma^2\kappa^2. \label{eq:uni-arb-gaussian-lb-1}
    \end{align}
    Further, by choosing the contamination distribution $Q\in\mathcal{P}(\mathbb{R}_{\star})$ such that $Q\bigl( \{\star\} \bigr)=1$ and applying Proposition~\ref{prop:univariate-mcar-lb}(a), we deduce that
    \begin{align}
        \mathcal{M}(\delta,\mathcal{P}_{\Theta},|\cdot|^2) \gtrsim \frac{\sigma^2\log(1/\delta)}{nq(1-\epsilon)}. \label{eq:uni-arb-gaussian-lb-2}
    \end{align}
    Combining~\eqref{eq:uni-arb-gaussian-lb-1} and~\eqref{eq:uni-arb-gaussian-lb-2} yields the lower bound for $\epsilon < \frac{q}{1+q}$.

    Next consider the case where $\epsilon \geq \frac{q}{1+q}$. Let $a>0$, $X_1\sim P_1 \coloneqq \mathsf{N}(0,\sigma^2)$ and $X_2 \sim P_2 \coloneqq \mathsf{N}(a\sigma,\sigma^2)$. By Lemma~\ref{lem:equivalence-of-f-af},
    \begin{align*}
        \mathrm{TV}\bigl( \mathsf{MCAR}_{(q,P_1)}, \mathsf{MCAR}_{(q,P_2)} \bigr) = q\cdot\mathrm{TV}(P_1,P_2) \leq q \leq \frac{\epsilon}{1-\epsilon}.
    \end{align*}
    Hence, by \citet[Lemma~25]{ma2024high}, we have
    \begin{align*}
        \mathcal{M}(\delta,\mathcal{P}_{\Theta},|\cdot|^2) \geq \frac{(\mathbb{E}X_1 - \mathbb{E}X_2)^2}{4} = \frac{\sigma^2 a^2}{4}.
    \end{align*}
    Since this holds for all $a>0$, we deduce that $\mathcal{M}(\delta,\mathcal{P}_{\Theta},|\cdot|^2) = \infty$ in this case.
    
    \medskip
    (b) Let $c_1>0$ be a universal constant that will be specified later. Define $P_1 \in \mathcal{P}(\mathbb{R})$ by $P_1\bigl((t,\infty)\bigr) \coloneqq e^{-t^2/(c_1\sigma)^2}$ for $t\geq 0$. Define $P_2 \in \mathcal{P}(\mathbb{R})$ by
    \begin{align*}
        P_2(\{0\}) \coloneqq \kappa, \quad P_2\bigl((t,\infty)\bigr) \coloneqq 
        \begin{cases}
            e^{-t^2/(c_1\sigma)^2} \;&\text{if } 0\leq t \leq c_1\sigma\sqrt{\log(1/\kappa)}\\
            0 \;&\text{if } t> c_1\sigma\sqrt{\log(1/\kappa)}.
        \end{cases}
    \end{align*}
    Let $X_1\sim P_1$ and $X_2\sim P_2$. Since $\mathbb{P}(|X_\ell| \geq t) \leq e^{-t^2/(c_1\sigma)^2}$ for $t\geq 0$ and $\ell\in\{1,2\}$, we have by \citet[Proposition~2.5.2]{vershynin2018high} that $\|X_\ell\|_{\psi_2} \leq c_1C_2\sigma$ for $\ell\in\{1,2\}$, where $C_2>0$ is a universal constant. Thus by \citet[Lemma~2.6.8]{vershynin2018high}, there exists a universal constant $C_3>0$ such that $\|X_\ell - \mathbb{E}X_\ell\|_{\psi_2} \leq c_1C_2C_3\sigma$. Hence, taking $c_1\coloneqq (C_2C_3)^{-1}$, we have $P_\ell \in \mathcal{P}_{\psi_2}\bigl(\mathbb{E}(X_\ell), \sigma^2\bigr)$ for $\ell\in\{1,2\}$.
    Moreover, $\mathrm{TV}(P_1,P_2) = P_2(\{0\}) = \kappa$, so that by Lemma~\ref{lem:equivalence-of-f-af},
    \begin{align*}
        \mathrm{TV}\bigl( \mathsf{MCAR}_{(q,P_1)}, \mathsf{MCAR}_{(q,P_2)} \bigr) = q\cdot\mathrm{TV}(P_1,P_2) = \frac{\epsilon}{1-\epsilon}.
    \end{align*}
    Now, integrating by parts yields
    \begin{align*}
        \mathbb{E}X_1 - \mathbb{E}X_2 &= \int_{c_1\sigma\sqrt{\log(1/\kappa)}}^\infty x\cdot \frac{2x}{(c_1\sigma)^2}e^{-x^2/(c_1\sigma)^2} \,\mathrm{d}x\\
        &= \kappa \cdot c_1\sigma\sqrt{\log(1/\kappa)} + \int_{c_1\sigma\sqrt{\log(1/\kappa)}}^\infty e^{-x^2/(c_1\sigma)^2} \,\mathrm{d}x \gtrsim \sigma\kappa\sqrt{\log(1/\kappa)}. 
    \end{align*}
    Hence, by \citet[Lemma~25]{ma2024high}, 
    \begin{align}
        \mathcal{M}(\delta,\mathcal{P}_{\Theta},|\cdot|^2) \geq \frac{(\mathbb{E}X_1 - \mathbb{E}X_2)^2}{4} \gtrsim \sigma^2\kappa^2\log(1/\kappa). \label{eq:uni-arb-sub-gaussian-lb-1}
    \end{align}
    By~\eqref{eq:uni-arb-sub-gaussian-lb-1} and applying Proposition~\ref{prop:univariate-mcar-lb}(a) with contamination distribution $Q\in\mathcal{P}(\mathbb{R}_{\star})$ satisfying $Q(\{\star\}) = 1$ as in~(a), we obtain the desired lower bound for $\epsilon < \frac{q}{1+q}$. Finally, the lower bound for $\epsilon \geq \frac{q}{1+q}$ follows from part~(a).
\end{proof}

\section{Proofs from Section~\ref{sec:realisable-mean-est}}\label{sec:proofs-realisable-mean-est}

\subsection{Proofs from Section~\ref{sec:gaussian-realisable-model}} \label{sec:proofs-one-dim-gaussian}

\subsubsection{Proof of Theorem~\ref{thm:univariate-gaussian-realisable-maxmin}}
\begin{proof}[Proof of Theorem~\ref{thm:univariate-gaussian-realisable-maxmin}]
Given $R\in \mathcal{R}(\theta_0)$, we can simulate from $R$ as follows: let $(X_0,\Omega_0^{(1)},\Omega_0^{(2)}, W_0)$ denote a random vector taking values in $\mathbb{R} \times \{0,1\}^{3}$ such that $W_0 \indep (X_0, \Omega_0^{(1)}, \Omega_0^{(2)})$, $W_0 \sim \mathsf{Ber}(\epsilon)$, $\Omega_0^{(1)} \indep (X_0, \Omega_0^{(2)})$, $\Omega_0^{(1)} \sim \mathsf{Ber}(q)$, $X_0 \sim \mathsf{N}(\theta_0, \sigma^2)$ and 
\[
R = \mathsf{Law}\bigl((1 - W_0) \cdot X_0 \ostar \Omega_{0}^{(1)} + W_0 \cdot X_0 \ostar \Omega_{0}^{(2)} \bigr).
\]
Note that if $Z_0 \coloneqq (1 - W_0) \cdot X_0 \ostar \Omega_{0}^{(1)} + W_0 \cdot X_0 \ostar \Omega_{0}^{(2)}$, then $Z_0 \mid \{W_0 = 0\} \sim \mathsf{MCAR}_{(\mathsf{N}(\theta_0, \sigma^2), q)}$.  We then generate $(X_i, \Omega_i^{(1)}, \Omega_i^{(2)}, W_i)_{i=1}^{n} \overset{\mathrm{iid}}{\sim} \mathsf{Law}(X_0, \Omega_0^{(1)}, \Omega_0^{(2)}, W_0)$, and set $Z_i \coloneqq (1 - W_i) \cdot X_i \ostar \Omega_{i}^{(1)} + W_i \cdot X_i \ostar \Omega_{i}^{(2)}$ for $i \in [n]$, so that $Z_1,\ldots,Z_n \stackrel{\mathrm{iid}}{\sim} R$.

    Now define the inliers as $\mathcal{I} \coloneqq \{i \in [n]: W_i = 0\}$, the outliers as $\mathcal{O} \coloneqq \{i \in [n] : W_i = 1\}$, and the observed indices as $\mathcal{D} \coloneqq \{i \in [n]: Z_i \neq \star\}$.  Equipped with this notation, we note the following pair of structural properties
    \[
    \max_{i \in \mathcal{I} \cap \mathcal{D}}\, X_i \leq \max_{i \in \mathcal{D}}\, Z_i \leq \max_{i \in ( \mathcal{I} \cap \mathcal{D}) \cup \mathcal{O}}\, X_i \quad \text{ and } \quad \min_{i \in ( \mathcal{I} \cap \mathcal{D}) \cup \mathcal{O}}\, X_i \leq \min_{i \in \mathcal{D}}\, Z_i \leq \min_{i \in \mathcal{I} \cap \mathcal{D}}\, X_i.
    \]
    We deduce the sandwich relation
    \begin{align} \label{ineq:sandwich-max-min}
    \frac{1}{2} \cdot \Bigl(\max_{i \in \mathcal{I} \cap \mathcal{D}}\, X_i + \min_{i \in ( \mathcal{I} \cap \mathcal{D}) \cup \mathcal{O}}\, X_i\Bigr) \leq \hat{\theta}^{\mathrm{AE}} \leq \frac{1}{2} \cdot \Bigl(\max_{i \in ( \mathcal{I} \cap \mathcal{D}) \cup \mathcal{O}}\, X_i + \min_{i \in \mathcal{I} \cap \mathcal{D}}\, X_i\Bigr).
    \end{align}
    Now $X_1,\ldots,X_n$ and $\mathcal{I} \cap \mathcal{D}$ are independent, and similarly $X_1,\ldots,X_n$ and $\mathcal{O}$ are independent, so $(X_i)_{i \in \mathcal{I} \cap \mathcal{D}}|(\mathcal{I} \cap \mathcal{D}) \overset{\mathrm{iid}}{\sim} \mathsf{N}(\theta_0, \sigma^2)$ and $(X_i)_{i \in (\mathcal{I} \cap \mathcal{D}) \cup \mathcal{O}}|\bigl((\mathcal{I} \cap \mathcal{D}) \cup \mathcal{O}\bigr) \overset{\mathrm{iid}}{\sim} \mathsf{N}(\theta_0, \sigma^2)$.  We let $N_1 \coloneqq \lvert \mathcal{I} \cap \mathcal{D} \rvert$ and $N_2 \coloneqq \lvert (\mathcal{I} \cap \mathcal{D}) \cup \mathcal{O} \rvert$, and define 
    \[
    B_{\ell} \coloneqq \sigma \sqrt{2\log{N_{\ell}}} + \frac{\sigma}{2} \cdot \frac{\log{\log{N_{\ell}}} + \log(4\pi)}{\sqrt{2 \log{N_{\ell}}}},
    \]
    for $\ell \in \{1,2\}$.  Let $\mathcal{E}_{1}\coloneqq \{ N_{1} \geq nq(1-\epsilon)/2\}$ and $\mathcal{E}_2 \coloneqq \{|\mathcal{O}|\leq 3n\epsilon\}$. By~\citet[Theorem 3]{tanguy2015some}, there exists a universal constant $C_1' > 0$ such that for $\ell \in \{1, 2\}$ and $\delta > 0$, 
    \begin{align*}
    &\mathbb{P}\biggl(\biggl\{\Bigl|\max_{i \in [N_{\ell}]} X_i - \theta - B_{\ell} \Bigr| \geq \frac{C_1'\sigma \log(8/\delta)}{\log^{1/2}N_{\ell}} \biggr\} \bigcap \mathcal{E}_1 \; \bigg \vert \; N_{\ell}\biggr) \leq \frac{\delta}{4}
    \end{align*}
    and
    \begin{align*}
    &\mathbb{P}\biggl( \biggl\{\Bigl|\min_{i \in [N_{\ell}]} X_i - \theta +  B_{\ell}\Bigr| \geq \frac{C_1' \sigma \log(8/\delta)}{\log^{1/2} N_{\ell}}\biggr\} \bigcap \mathcal{E}_1 \;\bigg \vert \; N_{\ell}\biggr) \leq \frac{\delta}{4}.
    \end{align*}
    Combining these inequalities with the sandwich relation~\eqref{ineq:sandwich-max-min} yields
    \begin{align}\label{ineq:deviation-bound-thetaAE}
    \mathbb{P}\biggl(\biggl\{\bigl \lvert \hat{\theta}^{\mathrm{AE}}_n - \theta_0 \bigr \rvert \geq B_2 - B_1 +  \frac{2C_1'\sigma \log(8/\delta)}{\log^{1/2}N_1}\biggr\} \bigcap \mathcal{E}_1 \; \bigg \vert \; N_1, N_2 \biggr) \leq \frac{2\delta}{3}.
    \end{align}
    Using the inequality $\sqrt{a} - \sqrt{b} \leq (a - b)/\sqrt{b}$ for $0 < b < a$ and the fact that $x \mapsto \frac{\log \log x + \log(4\pi)}{\log^{1/2} x}$ is decreasing for $x\geq \exp\bigl(\frac{e^2}{4\pi}\bigr)$, we deduce that on $\mathcal{E}_1 \cap \mathcal{E}_2$,
    \begin{align} \label{ineq:bound-diff-B1-B2}
    B_2 - B_1 \leq \frac{2\sigma \log(N_2/N_1)}{\log^{1/2}N_1} &\leq \frac{2 \sigma \log\bigl(1 + 3n\epsilon/N_1\bigr)}{\log^{1/2}N_1} \nonumber\\
    &\leq \frac{2 \sigma \log\bigl(1 + \frac{6\epsilon}{q(1 - \epsilon)}\bigr)}{\log^{1/2}{\bigl(nq(1 - \epsilon)/2\bigr)}} \leq \frac{3 \sigma \log\bigl(1 + \frac{6\epsilon}{q(1 - \epsilon)}\bigr)}{\log^{1/2}\bigl(nq(1-\epsilon)\bigr)}.
    \end{align}
    Now, we first assume that $\epsilon \geq n^{-1}\log(4/\delta)$. By Lemma~\ref{lemma:binomial-tail}, we have $\mathbb{P}(\mathcal{E}_1 \cap \mathcal{E}_2) \geq 1 - \delta/2$, since by assumption, $nq(1 -\epsilon) \geq 8 \log(4/\delta)$.  Moreover, combining the inequalities~\eqref{ineq:deviation-bound-thetaAE} and~\eqref{ineq:bound-diff-B1-B2} yields that on $\mathcal{E}_1 \cap \mathcal{E}_2$,
    \begin{align}
    \bigl \lvert \hat{\theta}^{\mathrm{AE}}_n - \theta_0 \bigr \rvert \leq \frac{3 \sigma \log\bigl(1 + \frac{6\epsilon}{q(1 - \epsilon)}\bigr)}{\log^{1/2}\bigl(nq(1-\epsilon)\bigr)} + \frac{2C_1' \sigma \log(8/\delta)}{\log^{1/2} N_1} \leq  C_1 \sigma \cdot \frac{ \log\bigl(1 + \frac{6\epsilon}{q(1 - \epsilon)}\bigr) + \log(8/\delta)}{\log^{1/2}\bigl(nq(1-\epsilon)\bigr)}, \label{eq:average-min-max}
    \end{align}
    where $C_1 \coloneqq 3(1+C_1')$. Hence, \eqref{eq:average-min-max} holds with probability at least $1-\delta$ when $\epsilon\geq n^{-1}\log(4/\delta)$. Finally, consider the case in which $\epsilon < n^{-1}\log(4/\delta)$. Then, since we have $\mathcal{R}\bigl(\mathsf{N}(\theta_0,\sigma^2),\epsilon, q\bigr) \subseteq \mathcal{R}\bigl(\mathsf{N}(\theta_0,\sigma^2),n^{-1}\log(4/\delta), q\bigr)$, it follows by~\eqref{eq:average-min-max} that
    \begin{align*}
        \bigl \lvert \hat{\theta}^{\mathrm{AE}}_n - \theta_0 \bigr \rvert &\leq C_1 \sigma \cdot \frac{ \log\bigl(1 + \frac{6\log(4/\delta)}{nq(1 - \epsilon)}\bigr) + \log(8/\delta)}{\log^{1/2}\bigl(nq(1-\epsilon)\bigr)} \leq C_1 \sigma \cdot \frac{ \frac{6\log(4/\delta)}{nq(1 - \epsilon)} + \log(8/\delta)}{\log^{1/2}\bigl(nq(1-\epsilon)\bigr)} \\
        &\leq C_1 \sigma \cdot \frac{2\log(8/\delta)}{\log^{1/2}\bigl(nq(1-\epsilon)\bigr)} \leq 2C_1 \sigma \cdot \frac{ \log\bigl(1 + \frac{6\epsilon}{q(1 - \epsilon)}\bigr) + \log(8/\delta)}{\log^{1/2}\bigl(nq(1-\epsilon)\bigr)},
    \end{align*}
    with probability at least $1-\delta$.
\end{proof}

\subsubsection{Proof of Theorem~\ref{thm:univariate-realisable-lb}}

\begin{figure}[ht]
    \centering
    \subfigure[\centering The two black curves are $\{q(1-\epsilon) + \epsilon\}\phi_{(-a,\sigma)}$ and $\{q(1-\epsilon) + \epsilon\}\phi_{(a,\sigma)}$ respectively, as labelled in the figure. The blue curve is $q(1-\epsilon)\phi_{(-a,\sigma)}$, and the orange curve is $q(1-\epsilon)\phi_{(a,\sigma)}$.]{{\includegraphics[width=0.9\textwidth]{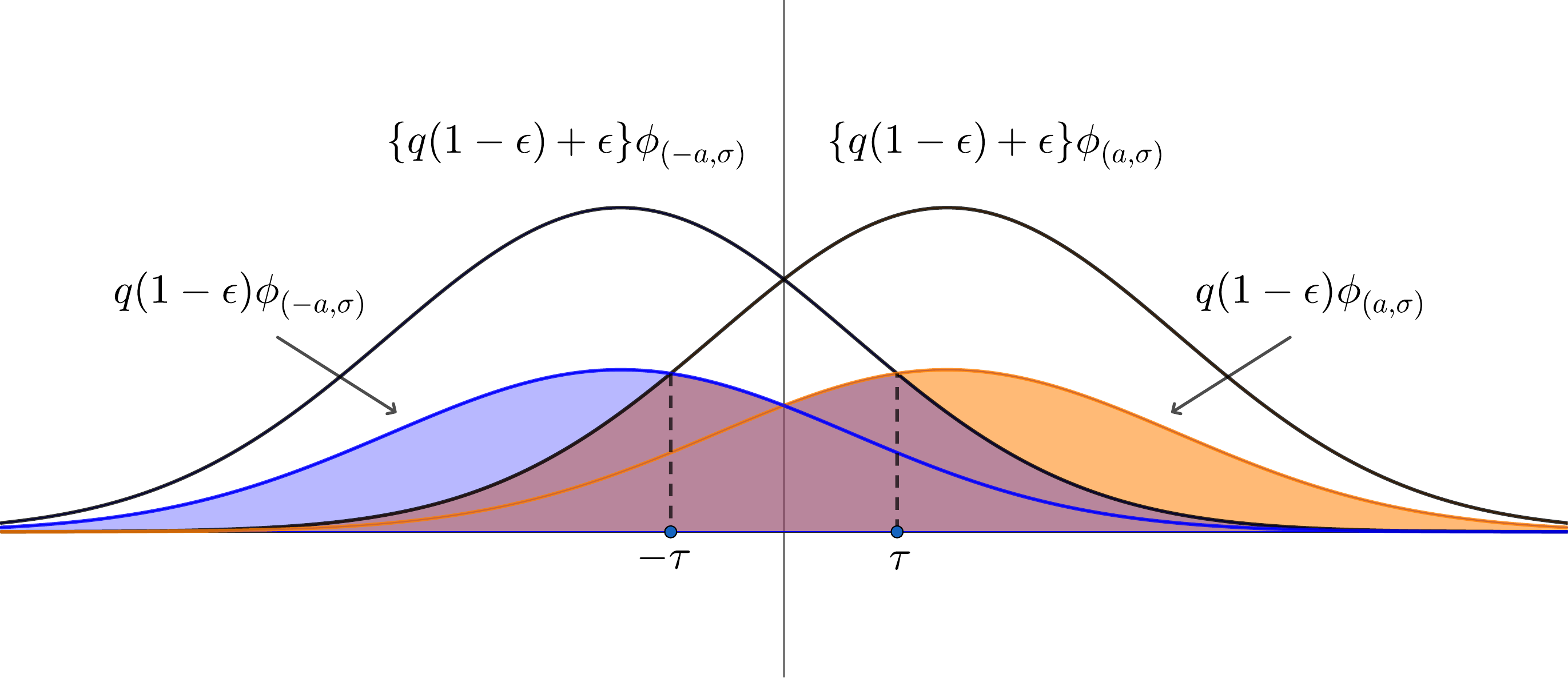}}}
    
    \subfigure[\centering The curve above the blue region illustrates the function $f_1$ in~\eqref{eq:f1-definition-univariate-lb}.]{{\includegraphics[width=0.49\textwidth]{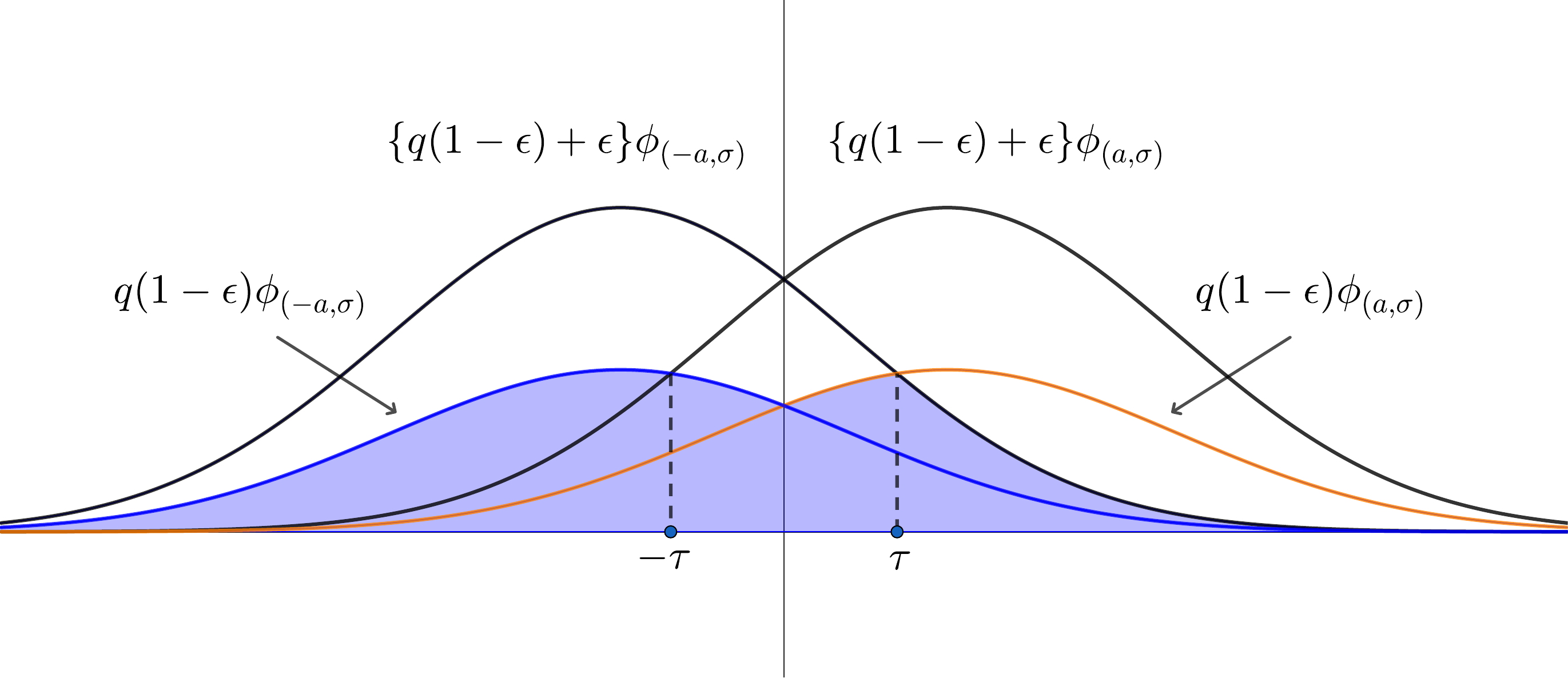}}}
    \subfigure[\centering The curve above the orange region illustrates the function $f_2$ in~\eqref{eq:f2-definition-univariate-lb}.]{{\includegraphics[width=0.49\textwidth]{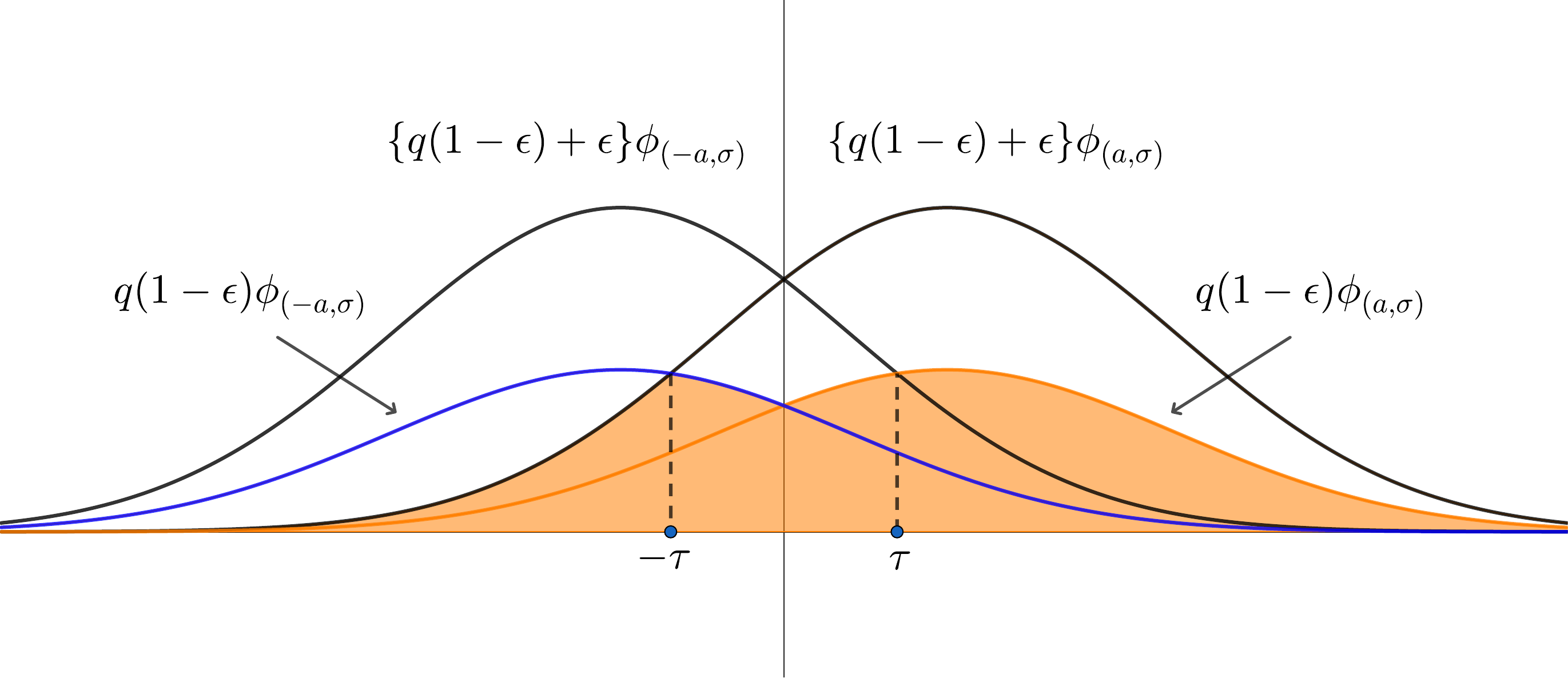}}}
    \caption{Construction of the lower bound in Theorem \ref{thm:univariate-realisable-lb}.}\label{fig:gaussian-realisable-lb}
\end{figure}

\begin{proof}[Proof of Theorem~\ref{thm:univariate-realisable-lb}]
Consider the construction illustrated in Figure~\ref{fig:gaussian-realisable-lb}.  For $a>0$ to be specified later, let 
\begin{align} \label{eq:def-t-eps-q-a}
    \tau \coloneqq \frac{\sigma^2}{2a} \cdot  \log \bigg( 1+\frac{\epsilon}{q(1-\epsilon)} \bigg)
\end{align}
denote the unique point in $\mathbb{R}$ where $\{q(1-\epsilon) + \epsilon\}\phi_{(-a,\sigma)}(\tau) = q(1 - \epsilon) \phi_{(a,\sigma)}(\tau)$.  Next, define the function $f_1: \mathbb{R} \rightarrow \mathbb{R}$ as
    \begin{align} \label{eq:f1-definition-univariate-lb}
    f_1(x) \coloneqq \begin{cases}
        q(1 - \epsilon)  \phi_{(-a,\sigma)}(x) & \text{ if } x \leq 0\\
        q(1 - \epsilon)  \phi_{(a,\sigma)}(x) & \text{ if } 0 < x \leq \tau\\
        \bigl\{q ( 1- \epsilon) + \epsilon\bigr\} \cdot \phi_{(-a,\sigma)}(x) & \text{ if } x > \tau.
    \end{cases} 
    \end{align}
    Similarly, we note that $-\tau$ is the unique point satisfying $\{q(1-\epsilon) + \epsilon\} \phi_{(a,\sigma)}(-\tau) = q (1 - \epsilon) \phi_{(-a,\sigma)}(-\tau)$ and define the function $f_2: \mathbb{R} \rightarrow \mathbb{R}$ as 
    \begin{align} \label{eq:f2-definition-univariate-lb}
    f_2(x) \coloneqq \begin{cases}
        \bigl\{q ( 1- \epsilon) + \epsilon\bigr\} \cdot\phi_{(a,\sigma)}(x) & \text{ if } x \leq -\tau\\
        q(1 - \epsilon)  \phi_{(-a,\sigma)}(x) & \text{ if } -\tau < x \leq 0\\
        q (1 - \epsilon)\phi_{(a,\sigma)}(x) & \text{ if } x > 0.
    \end{cases}
    \end{align}
    Note that $\int_{\mathbb{R}} f_\ell(x) \, \mathrm{d}x \leq q(1-\epsilon) + \epsilon \leq 1$ for $\ell \in \{1,2\}$, so we may construct $P_1, P_2 \in \mathcal{P}(\mathbb{R}_{\star})$ with Radon--Nikodym derivatives
    \[
    \frac{\mathrm{d}P_{\ell}}{\mathrm{d} \lambda_{\star}}(z) \coloneqq f_{\ell}(z) \mathbbm{1}_{\{z \in \mathbb{R}\}} + \biggl(1 - \int_{\mathbb{R}} f_{\ell}(x)\, \mathrm{d}x\biggr) \mathbbm{1}_{\{z = \star\}} \quad \text{for} \quad \ell \in \{1, 2\},
    \]
    where $\lambda_{\star}$ denotes the extension of the Lebesgue measure to $\mathbb{R}_{\star}$ as defined in Section~\ref{sec:notation}.  Then, by Proposition~\ref{prop:univariate-realisability}, $P_{1} \in \mathcal{R}\bigl(\mathsf{N}(-a, \sigma^2), \epsilon, q\bigr)$ and $P_{2} \in \mathcal{R}\bigl(\mathsf{N}(a, \sigma^2), \epsilon, q\bigr)$.  Since $P_{1}(\{\star\}) = P_{2}(\{\star\})$ and $f_1(x) = f_2(x)$ for $x \in [-\tau, \tau]$, we compute
    \begin{align*} 
        &\mathrm{KL}(P_1, P_2) = \int_{-\infty}^{-\tau} q(1 - \epsilon) \phi_{(-a,\sigma)}(x) \log\biggl(\frac{q(1 - \epsilon) \phi_{(-a,\sigma)}(x)}{\bigl\{q ( 1- \epsilon) + \epsilon\bigr\}\phi_{(a,\sigma)}(x)}\biggr)\, \mathrm{d}x  \\
        & \qquad \qquad+ \int_{\tau}^{\infty} \bigl\{q ( 1- \epsilon) + \epsilon\bigr\} \phi_{(-a,\sigma)}(x) \log\biggl(\frac{\bigl\{q ( 1- \epsilon) + \epsilon\bigr\} \phi_{(-a,\sigma)}(x)}{q(1 - \epsilon) \phi_{(a,\sigma)}(x)}\biggr)\, \mathrm{d} x \\
        &= q (1 - \epsilon) \biggl\{ \frac{2a^2}{\sigma^2} - \log\biggl(1+ \frac{\epsilon}{q(1-\epsilon)} \biggr) \biggr\} \bigl\{1 - \Phi_{(0,\sigma)}(\tau-a)\bigr\}\\
        &\qquad\qquad + \bigl\{q (1 - \epsilon) + \epsilon\bigr\} \biggl\{ \frac{2a^2}{\sigma^2} + \log\biggl(1+ \frac{\epsilon}{q(1-\epsilon)} \biggr) \biggr\}\bigl\{1-\Phi_{(0,\sigma)}(\tau+a)\bigr\} \nonumber\\
        &\qquad\qquad + 2a\bigl[ q(1-\epsilon) \phi_{(0,\sigma)}(\tau - a) - \bigl\{q(1-\epsilon) + \epsilon\bigr\} \phi_{(0,\sigma)}(\tau+a) \bigr] \\
        &= \frac{2aq(1-\epsilon)}{\sigma^2}(a-\tau)\bigl\{1 - \Phi_{(0,\sigma)}(\tau-a)\bigr\} + \frac{2a\{q(1-\epsilon)+\epsilon\}}{\sigma^2}(a+\tau)\bigl\{1-\Phi_{(0,\sigma)}(\tau+a)\bigr\} \\
        &\qquad\qquad + 2a\bigl[ q(1-\epsilon) \phi_{(0,\sigma)}(\tau - a) - \bigl\{q(1-\epsilon) + \epsilon\bigr\} \phi_{(0,\sigma)}(\tau+a) \bigr].
    \end{align*}
    Next, set 
    \begin{align*} 
    a \coloneqq \frac{\sigma}{4} \cdot \log\biggl(1 + \frac{\epsilon}{q (1 - \epsilon)}\biggr) \cdot \log^{-1/2}{\bigl(nq(1-\epsilon)\bigr)} > 0,
    \end{align*}
    so that by substituting this definition into~\eqref{eq:def-t-eps-q-a}, we obtain
    \[
    \tau=2\sigma\log^{1/2}\bigl( nq(1-\epsilon) \bigr) \quad\text{and}\quad a\leq \frac{\tau}{8},
    \]
    where the inequality follows from our assumption~\eqref{eq:asm-eps-gaussian-realisable-lb}.  Hence, by the Mills ratio bound $1 - \Phi_{(0,\sigma)}(x) \leq \sigma^2 \phi_{(0,\sigma)}(x)/x$ for $x > 0$, we have
    \begin{align*}
    \mathrm{KL}(P_1, P_2)
    &\leq 2a\{q (1 - \epsilon) + \epsilon \} \phi_{(0,\sigma)}(\tau+a) \\
    & \qquad \qquad + 2a\bigl[ q(1-\epsilon) \phi_{(0,\sigma)}(\tau - a) - \bigl\{q(1-\epsilon) + \epsilon\bigr\} \phi_{(0,\sigma)}(\tau+a) \bigr]\\
    &= 2aq (1 - \epsilon) \phi_{(0,\sigma)}(\tau-a) \\
    &= \frac{\sigma}{2} \cdot \log\biggl(1 + \frac{\epsilon}{q (1 - \epsilon)}\biggr) \cdot \log^{-1/2}\bigl(nq(1-\epsilon)\bigr) \cdot q (1 - \epsilon) \cdot \phi_{(0,\sigma)}(\tau-a)\\
    &\leq \frac{\sigma}{2} \cdot \log^{1/2}\bigl(nq(1-\epsilon)\bigr) \cdot q (1 - \epsilon) \cdot \phi_{(0,\sigma)}(7\tau/8) \\
    &= \frac{q(1-\epsilon)}{2\sqrt{2\pi}} \cdot \log^{1/2}\bigl(nq(1-\epsilon)\bigr) \cdot \exp\biggl\{-\frac{1}{2}\cdot \Bigl(\frac{7}{8}\Bigr)^2 \cdot 4\log\bigl(nq(1-\epsilon)\bigr)\biggr\} \\
    &\leq \frac{q(1-\epsilon)}{2\sqrt{2\pi}} \cdot \log^{1/2}\bigl(nq(1-\epsilon)\bigr) \cdot \bigl\{nq(1-\epsilon)\bigr\}^{-3/2} \leq \frac{1}{5n}.
    \end{align*} 
    Thus, $\mathrm{KL}(P_1^{\otimes n}, P_2^{\otimes n}) \leq 1/5 < \log\bigl( \frac{1}{4\delta(1-\delta)} \bigr)$ for $\delta\in(0,1/4]$, so by \citet[Theorem~4 and Corollary~6]{ma2024high}, we deduce that for $\delta\in(0,1/4]$,
    \begin{align} \label{eq:gaussian-realisable-proof-lb1}
        \mathcal{M}\bigl(\delta, \mathcal{P}_{\Theta}, | \cdot |^2\bigr) \geq a^2 = \frac{\sigma^2 \log^2\bigl(1 + \frac{\epsilon}{q (1 - \epsilon)}\bigr)}{16 \log \bigl(nq(1-\epsilon)\bigr)}.
    \end{align}
   Finally, note that $\mathsf{MCAR}_{(q(1-\epsilon), \mathsf{N}(\theta,\sigma^2))} \in \mathcal{R}(\mathsf{N}(\theta, \sigma^2), \epsilon, q)$ for all $\theta \in \mathbb{R}$, since we can choose the contamination distribution $Q$ such that $Q(\{\star\})=1$. Therefore, by Proposition~\ref{prop:univariate-mcar-lb}(a), we have that for $\delta\in(0,1/4]$,
   \begin{align}\label{eq:gaussian-realisable-proof-lb2}
       \mathcal{M}\bigl(\delta, \mathcal{P}_{\Theta}, | \cdot |^2\bigr) 
       \begin{cases}
         \geq \dfrac{\sigma^2 \log(1/\delta)}{20nq(1-\epsilon)} \quad&\text{if }\delta\geq \dfrac{\{1-q(1-\epsilon)\}^n}{2}\\
         = \infty \quad&\text{if }\delta< \dfrac{\{1-q(1-\epsilon)\}^n}{2}.
    \end{cases}
   \end{align}
   Combining \eqref{eq:gaussian-realisable-proof-lb1} and \eqref{eq:gaussian-realisable-proof-lb2} yields the desired result.
\end{proof}

\subsubsection{Proof of Theorem~\ref{thm:one-dim-kolmogorov-estimator}}
In order to prove Theorem~\ref{thm:one-dim-kolmogorov-estimator}, we require a preliminary lemma.
\begin{lemma}\label{lemma:one-dim-kolmogorov-distance-realisable-sets}
    Let $\theta_1, \theta_2 \in \mathbb{R}$ be distinct, and set $a \coloneqq \lvert \theta_1 - \theta_2 \rvert/2$. Then, writing $b \coloneqq \frac{1}{2}\log\bigl( 1+ \frac{4\epsilon}{q(1-\epsilon)}\bigr)$, there exists a continuous and strictly increasing function $f_{\mathrm{K},b}:(0,\infty) \to (0,1]$ such that
    \begin{align*}
        d_{\mathrm{K}}\bigl(\mathcal{R}(\theta_1),\mathcal{R}(\theta_2)\bigr) \geq f_{\mathrm{K},b}(a).
    \end{align*}    
    Moreover, 
    \begin{align*}
        f_{\mathrm{K},b}(a) \geq q(1-\epsilon) \cdot \frac{a}{\sigma} \cdot \phi\biggl(\frac{a}{\sigma}+\frac{\sigma b}{a}\biggr) \quad\text{when }b\leq 1/2,
    \end{align*}
    and
    \begin{align*}
        f_{\mathrm{K},b}(a) \geq q(1-\epsilon) \cdot \Phi\biggl(\frac{a}{\sigma}-\frac{2\sigma b}{a}\biggr) - \{q(1-\epsilon) + \epsilon\}\cdot \Phi\biggl(-\frac{a}{\sigma}-\frac{2\sigma b}{a}\biggr) \quad\text{when }b> 1/2.
    \end{align*}
\end{lemma}
\begin{proof}
    Since $d_\mathrm{K}$ is translation invariant, we may assume without loss of generality that $\theta_1 = -a$ and $\theta_2 = a$. By Proposition~\ref{prop:univariate-realisability}, if $R_\ell \in \mathcal{R}(\theta_\ell)$ for $\ell\in\{1,2\}$, then each admits a density $h_{\ell} :\mathbb{R}_\star \to \mathbb{R}$ with respect to the extended Lebesgue measure $\lambda_\star$ such that $h_\ell(x) / \phi_{(\theta_\ell,\sigma)}(x) \in [q(1-\epsilon),\, q(1-\epsilon) + \epsilon]$ for all $x\in\mathbb{R}$. Let $\tau \coloneqq \frac{\sigma^2}{2a} \cdot \log\bigl( 1 + \frac{\epsilon}{q(1-\epsilon)} \bigr) \leq  \frac{\sigma^2 b}{a}$,
    so that $q(1-\epsilon)\phi_{(-a,\sigma)}(-\tau) = \{q(1-\epsilon)+\epsilon\} \phi_{(a,\sigma)}(-\tau)$, see Figure~\ref{fig:gaussian-realisable-lb}.
    
    When $b\leq 1/2$,
    \begin{align*}
        d_{\mathrm{K}}\bigl(\mathcal{R}(\theta_1),\mathcal{R}(\theta_2)\bigr) &= \inf_{R_1\in\mathcal{R}(\theta_1),\, R_2\in\mathcal{R}(\theta_2)} \sup_{A\in\mathcal{A}} |R_1(A)-R_2(A)|\\
        &\geq
        \inf_{R_1\in\mathcal{R}(\theta_1),\, R_2\in\mathcal{R}(\theta_2)} \bigl\{R_1\bigl((-\infty,-\sigma^2b/a]\bigr) - R_2\bigl((-\infty,-\sigma^2b/a]\bigr)\bigr\}\\
        &\geq q(1-\epsilon) \cdot \Phi_{(\theta_1,\sigma)}(-\sigma^2b/a) - \{q(1-\epsilon) + \epsilon\}\cdot \Phi_{(\theta_2,\sigma)}(-\sigma^2b/a)\\
        &= q(1-\epsilon) \cdot \Phi\biggl(\frac{a}{\sigma}-\frac{\sigma b}{a}\biggr) - \{q(1-\epsilon) + \epsilon\}\cdot \Phi\biggl(-\frac{a}{\sigma}-\frac{\sigma b}{a}\biggr) \eqqcolon f_{\mathrm{K},b}(a).
    \end{align*}
    Now $f_{\mathrm{K},b}$ is continuously differentiable, with
    \begin{align*}
    f_{\mathrm{K},b}'(a) &= q(1 \!-\! \epsilon) \cdot \biggl(\frac{1}{\sigma} \!+\! \frac{\sigma b}{a^2}\biggr) \phi\biggl(\frac{a}{\sigma} \!-\! \frac{\sigma b}{a}\biggr) - \bigl\{q(1 - \epsilon) + \epsilon \bigr\}\cdot \Bigl(-\frac{1}{\sigma} + \frac{\sigma b}{a^2}\Bigr) \phi\biggl(-\frac{a}{\sigma}-\frac{\sigma b}{a}\biggr)\\
    &> \biggl(\frac{1}{\sigma} + \frac{\sigma b}{a^2}\biggr) \cdot \sigma \biggl\{ q(1 \!-\! \epsilon)  \cdot \phi_{(-a,\sigma)}\biggl(-\frac{\sigma^2 b}{a}\biggr) - \bigl\{q(1 - \epsilon) + \epsilon \bigr\}\cdot \phi_{(a,\sigma)}\biggl(-\frac{\sigma^2 b}{a}\biggr) \biggr\} \\
    &\geq \biggl(\frac{1}{\sigma} + \frac{\sigma b}{a^2}\biggr) \cdot \sigma \Bigl( q(1 - \epsilon)  \cdot \phi_{(-a,\sigma)}(-\tau) - \bigl\{q(1 - \epsilon) + \epsilon \bigr\}\cdot \phi_{(a,\sigma)}(-\tau) \Bigr) = 0,
    \end{align*}
    so that $f_{\mathrm{K},b}$ is strictly increasing as well.
    Moreover,
    \begin{align*}
        f_{\mathrm{K},b}(a) &= q(1-\epsilon) \cdot \biggl\{ \Phi\biggl(\frac{a}{\sigma}-\frac{\sigma b}{a}\biggr) - \Phi\biggl(-\frac{a}{\sigma}-\frac{\sigma b}{a}\biggr) \biggr\} - \epsilon\cdot \Phi\biggl(-\frac{a}{\sigma}-\frac{\sigma b}{a}\biggr)\\
        &\geq q(1-\epsilon) \cdot \frac{2a}{\sigma} \cdot \phi\biggl(\frac{a}{\sigma}+\frac{\sigma b}{a}\biggr) - \epsilon\cdot \Phi\biggl(-\frac{a}{\sigma}-\frac{\sigma b}{a}\biggr), \numberthis \label{eq:kolmogorov-distance-lb2}
    \end{align*}
    where the final inequality follows from the mean value theorem $\Phi\bigl(\frac{a}{\sigma}-\frac{\sigma b}{a}\bigr) - \Phi\bigl(-\frac{a}{\sigma}-\frac{\sigma b}{a}\bigr) = \frac{2a}{\sigma}\cdot\phi(x') \geq \frac{2a}{\sigma} \cdot \phi\bigl(\frac{a}{\sigma}+\frac{\sigma b}{a}\bigr)$, where $x'\in \bigl[-\frac{a}{\sigma}-\frac{\sigma b}{a},\, \frac{a}{\sigma}-\frac{\sigma b}{a}\bigr]$. Next notice that \begin{align*}
        \epsilon\cdot \Phi\biggl(-\frac{a}{\sigma}-\frac{\sigma b}{a}\biggr) \leq \frac{\epsilon}{a/\sigma+ \sigma b/a} \cdot \phi\biggl(\frac{a}{\sigma}+\frac{\sigma b}{a}\biggr) &\leq \frac{\epsilon a}{\sigma b} \cdot \phi\biggl(\frac{a}{\sigma}+\frac{\sigma b}{a}\biggr)\\
        &\leq q(1-\epsilon) \cdot \frac{a}{\sigma} \cdot \phi\biggl(\frac{a}{\sigma}+\frac{\sigma b}{a}\biggr),\numberthis \label{eq:kolmogorov-distance-lb3}
    \end{align*}
    where the first inequality follows from the Mills ratio bound $\Phi(-x) \leq \phi(x)/x$ for $x > 0$, and the final inequality follows from the fact that $\log(1+x) \geq x/2$ for $x \in [0,2]$, so that $b=\frac{1}{2}\log\bigl(1+\frac{4\epsilon}{q(1-\epsilon)}\bigr) \geq \frac{\epsilon}{q(1-\epsilon)}$. Therefore, by~\eqref{eq:kolmogorov-distance-lb2} and~\eqref{eq:kolmogorov-distance-lb3} we deduce that \begin{align*}
         f_{\mathrm{K},b}(a) \geq q(1-\epsilon) \cdot \frac{a}{\sigma} \cdot \phi\biggl(\frac{a}{\sigma}+\frac{\sigma b}{a}\biggr),
    \end{align*}
    when $b\leq 1/2$.

    On the other hand, when $b> 1/2$,
    \begin{align*}
        d_{\mathrm{K}}\bigl(\mathcal{R}(\theta_1),&\mathcal{R}(\theta_2)\bigr)    \geq
        \inf_{R_1\in\mathcal{R}(\theta_1),\, R_2\in\mathcal{R}(\theta_2)} \bigl\{R_1\bigl((-\infty,-2\sigma^2b/a]\bigr) - R_2\bigl((-\infty,-2\sigma^2b/a]\bigr)\bigr\}\\
        &\geq q(1-\epsilon) \cdot \Phi_{(\theta_1,\sigma)}(-2\sigma^2b/a) - \{q(1-\epsilon) + \epsilon\}\cdot \Phi_{(\theta_2,\sigma)}(-2\sigma^2b/a)\\
        &= q(1-\epsilon) \cdot \Phi\biggl(\frac{a}{\sigma}-\frac{2\sigma b}{a}\biggr) - \{q(1-\epsilon) + \epsilon\}\cdot \Phi\biggl(-\frac{a}{\sigma}-\frac{2\sigma b}{a}\biggr) \eqqcolon f_{\mathrm{K},b}(a).
    \end{align*}
    Similarly to the previous case, $f_{\mathrm{K},b}$ is continuously differentiable and strictly increasing.
\end{proof}

\begin{proof}[Proof of Theorem~\ref{thm:one-dim-kolmogorov-estimator}]
    We first derive an upper bound on $d_{\mathrm{K}} \bigl(\hat{R}_n, \mathcal{R}(\theta_0)\bigr)$. Let $\mathcal{D} \coloneqq \{i \in [n] : Z_i \neq \star\}$ and $\bar{q} \coloneqq \mathbb{P}(Z_1 \neq \star)$, so that with the convention that $0/0 \coloneqq 0$,
    \begin{align*}
        \sup_{A\in\mathcal{A}} |\hat{R}_n(A) - R(A)| &= \sup_{A\in\mathcal{A}} \biggl| \frac{|\mathcal{D}|}{n} \cdot \frac{1}{\lvert \mathcal{D} \rvert}\sum_{i\in\mathcal{D}} \mathbbm{1}_{\{Z_i \in A\}}- \bar{q} \cdot \mathbb{P}(Z_1 \in A | Z_1 \neq \star) \biggr|\\
        &\leq \frac{|\mathcal{D}|}{n} \cdot \sup_{A\in\mathcal{A}} \biggl| \frac{1}{\lvert \mathcal{D} \rvert}\sum_{i\in\mathcal{D}} \mathbbm{1}_{\{Z_i \in A\}} - \mathbb{P}(Z_1 \in A | Z_1 \neq \star) \biggr| + \biggl| \frac{|\mathcal{D}|}{n} - \bar{q} \biggr|. \numberthis \label{eq:dkw-decomposition}
    \end{align*}
Now, since $\bar{q} \geq q(1-\epsilon)$, we have by our lower bound on $\delta$ that
\[
\log\Bigl(\frac{4}{\delta}\Bigr) \leq \frac{\bigl\{ nq(1-\epsilon) \bigr\}^{1-\xi}}{6400} \leq \frac{nq(1-\epsilon)}{6400} \leq \frac{n\bar{q}}{6400}.
\]
Hence, by Bernstein's inequality~\citep[Theorem 2.8.4]{vershynin2018high}, with probability at least $1 - \delta/2$ that
    \begin{align}
        \biggl| \frac{|\mathcal{D}|}{n} - \bar{q} \biggr| 
        \leq \sqrt{\frac{4\bar{q}\log(4/\delta)}{n}} <\bar{q}. \label{eq:upper-bound-observed-proportion}
    \end{align}
    Furthermore, by the Dvoretzky--Kiefer--Wolfowitz--Massart--Reeve inequality \citep{massart1990tight, reeve2024short}, 
    \begin{align}
        \sup_{A\in\mathcal{A}} \biggl| \frac{1}{\lvert \mathcal{D} \rvert} \sum_{i\in\mathcal{D}} \mathbbm{1}_{\{Z_i \in A\}}- \mathbb{P}(Z_1 \in A \,\vert\, Z_1 \neq \star) \biggr| \leq \sqrt{\frac{\log(4/\delta)}{2|\mathcal{D}|}}, \label{eq:dkw-observed-part}
    \end{align}
    with probability at least $1 - \delta/2$.  Combining~\eqref{eq:dkw-decomposition}, \eqref{eq:upper-bound-observed-proportion} and~\eqref{eq:dkw-observed-part} we deduce that, with probability at least $1 - \delta$, 
    \begin{align}
        d_{\mathrm{K}} \bigl(\hat{R}_n, \mathcal{R}(\theta_0)\bigr) \leq \sup_{A\in\mathcal{A}} |\hat{R}_n(A) - R(A)| &\leq \sqrt{\frac{\lvert \mathcal{D} \rvert}{n}} \cdot \sqrt{\frac{\log(4/\delta)}{2 n}} + \sqrt{\frac{4\bar{q}\log(4/\delta)}{n}}\nonumber\\
        &\leq \sqrt{\frac{\bar{q} \log(4/\delta)}{n}} + \sqrt{\frac{4\bar{q}\log(4/\delta)}{n}}\nonumber\\
        &\leq 3\sqrt{\frac{\{q(1-\epsilon)+\epsilon\}\log(4/\delta)}{n}} \eqqcolon r_n. \label{eq:kolmogorov-distance-1-dim}
    \end{align}

    We now work on the event $\mathcal{E} \coloneqq \bigl\{ d_{\mathrm{K}} \bigl(\hat{R}_n, \mathcal{R}(\theta_0)\bigr) \leq r_n \bigr\}$, which occurs with probability at least $1-\delta$ by~\eqref{eq:kolmogorov-distance-1-dim}.
    If $\theta\in\mathbb{R}$ satisfies $d_{\mathrm{K}}\bigl(\mathcal{R}(\theta), \mathcal{R}(\theta_0)\bigr) > 2r_n$, then on the event $\mathcal{E}$, 
    \begin{align*}
        d_{\mathrm{K}}\bigl(\hat{R}_n, \mathcal{R}(\theta)\bigr) \geq d_{\mathrm{K}}\bigl(\mathcal{R}(\theta), \mathcal{R}(\theta_0)\bigr) - d_{\mathrm{K}} \bigl(\hat{R}_n, \mathcal{R}(\theta_0)\bigr) > r_n \geq d_{\mathrm{K}} \bigl(\hat{R}_n, \mathcal{R}(\theta_0)\bigr),
    \end{align*}
    so $\hat{\theta}_n^{\mathrm{K}} \neq \theta$.  Therefore, with $f_{\mathrm{K}, b}$ as defined in Lemma~\ref{lemma:one-dim-kolmogorov-distance-realisable-sets} and $b \coloneqq \frac{1}{2}\log\bigl( 1+ \frac{4\epsilon}{q(1-\epsilon)}\bigr)$, we deduce that on $\mathcal{E}$,
    \begin{align}
        |\hat{\theta}_n^{\mathrm{K}} - \theta_0| &\leq \sup \bigl\{ |\theta-\theta_0| : \theta\in\mathbb{R},\, d_{\mathrm{K}}\bigl(\mathcal{R}(\theta), \mathcal{R}(\theta_0)\bigr) \leq 2r_n \bigr\} \nonumber \\
        &\leq 2\sup \bigl\{ a\geq 0 : f_{\mathrm{K},b}(a) \leq 2r_n \bigr\} 
        = 2\inf \bigl\{ a\geq 0 : f_{\mathrm{K},b}(a) \geq 2r_n \bigr\}, \label{eq:one-dim-gaussian-realisable-proof-ub1}
    \end{align}
    where the second inequality follows since by Lemma~\ref{lemma:one-dim-kolmogorov-distance-realisable-sets}, $d_{\mathrm{K}}\bigl(\mathcal{R}(\theta), \mathcal{R}(\theta_0)\bigr) \geq f_{\mathrm{K},b}\bigl( \frac{|\theta-\theta_0|}{2} \bigr)$, and the final equality follows since $f_{\mathrm{K},b}$ is a strictly increasing and continuous function.
    
    When $b\leq 1/2$, we have by~\eqref{eq:one-dim-gaussian-realisable-proof-ub1} and Lemma~\ref{lemma:one-dim-kolmogorov-distance-realisable-sets} that on $\mathcal{E}$,
    \begin{align*}
        |\hat{\theta}_n^{\mathrm{K}} - \theta_0| 
        &\leq 2\inf \bigl\{a\geq 0 : f_{\mathrm{K},b}(a) \geq 2r_n \bigr\}\\
        &\leq 2\inf \biggl\{a\geq 0: q(1-\epsilon) \cdot \frac{a}{\sigma} \cdot \phi\biggl(\frac{a}{\sigma}+\frac{\sigma b}{a}\biggr) \geq 6\sqrt{\frac{\{q(1-\epsilon)+\epsilon\}\log(4/\delta)}{n}} \biggr\}\\
        &= 2\sigma \inf \Biggl\{ a\geq 0 : a \cdot \phi\biggl(a+\frac{b}{a}\biggr) \geq \sqrt{\biggl(1+\frac{\epsilon}{q(1-\epsilon)}\biggr) \cdot \frac{36\log(4/\delta)}{nq(1-\epsilon)}} \Biggr\}. \numberthis \label{eq:one-dim-gaussian-realisable-proof-ub2}
    \end{align*}

    Now suppose further that $b\leq \sqrt{\frac{\log(4/\delta)}{nq(1-\epsilon)}}$. The assumption on $\delta$ means that $b\leq \sqrt{\frac{\log(4/\delta)}{nq(1-\epsilon)}} \leq 1/80$ and thus $1+\frac{4\epsilon}{q(1-\epsilon)} < 5/4$. Let $a \coloneqq 20 \sqrt{\frac{\log(4/\delta)}{nq(1-\epsilon)}}$, so that $a\leq 1/4$.  Moreover, $b/a \leq 1/20$,  so $a + b/a \leq 3/10$.  Therefore, 
    \begin{align*}
        a\cdot \phi(a+b/a) \geq 20 \sqrt{\frac{\log(4/\delta)}{nq(1-\epsilon)}} \cdot \phi(3/10)
        &\geq \sqrt{\frac{5}{4} \cdot \frac{36\log(4/\delta)}{nq(1-\epsilon)}}\\
        &\geq \sqrt{\biggl(1+\frac{\epsilon}{q(1-\epsilon)}\biggr) \cdot \frac{36\log(4/\delta)}{nq(1-\epsilon)}}.
    \end{align*}
    Hence, by~\eqref{eq:one-dim-gaussian-realisable-proof-ub2}, we have on $\mathcal{E}$ that $|\hat{\theta}_n^{\mathrm{K}} - \theta_0| \leq 40\sigma \sqrt{\frac{\log(4/\delta)}{nq(1-\epsilon)}}$ when $b\leq \sqrt{\frac{\log(4/\delta)}{nq(1-\epsilon)}}$.

    Next, we consider the case $\sqrt{\frac{\log(4/\delta)}{nq(1-\epsilon)}} < b \leq 2\sqrt{\frac{\log(4/\delta)}{(nq(1-\epsilon))^{1-\xi}}}$.  Then $b\leq 1/40$ and we again have $1+\frac{4\epsilon}{q(1-\epsilon)} < 5/4$.  Let $a \coloneqq 20b$, so that $a \leq 1/2$. Then
    \begin{align*}
        a\cdot \phi(a+b/a) > 20\sqrt{\frac{\log(4/\delta)}{nq(1-\epsilon)}} \cdot \phi\biggl( \frac{1}{2} + \frac{1}{20} \biggr)
        &\geq \sqrt{\frac{5}{4} \cdot \frac{36\log(4/\delta)}{nq(1-\epsilon)}}\\
        &\geq \sqrt{\biggl(1+\frac{\epsilon}{q(1-\epsilon)}\biggr) \cdot \frac{36\log(4/\delta)}{nq(1-\epsilon)}}.
    \end{align*}
    Hence, by~\eqref{eq:one-dim-gaussian-realisable-proof-ub2}, when $\sqrt{\frac{\log(4/\delta)}{nq(1-\epsilon)}} < b \leq 2\sqrt{\frac{\log(4/\delta)}{(nq(1-\epsilon))^{1-\xi}}}$ we have on $\mathcal{E}$ that $|\hat{\theta}_n^{\mathrm{K}} - \theta_0| \leq 40\sigma b$. 

    As our third case, assume that $2\sqrt{\frac{ \log(4/\delta)}{\{nq(1-\epsilon)\}^{1-\xi}}} < b \leq 1/2$.  We define $a \coloneqq \frac{16b}{\sqrt{\xi \log(nq(1-\epsilon))}}$, so that $a\geq 32\sqrt{\frac{\log(4/\delta)}{(nq(1-\epsilon))^{1-\xi/2}}}$ using the fact that $x^{\xi/2} \geq \xi\log x$ for $x\in(0,\infty)$. By the assumption~\eqref{Eq:bupperbound}, we have $b\leq \frac{7\xi}{256}\log\bigl(nq(1-\epsilon)\bigr)$, so $a\leq 7b/a$. Therefore, 
    \begin{align*}
        a\cdot \phi(a+b/a) &\geq 32 \sqrt{\frac{\log(4/\delta)}{\bigl\{nq(1-\epsilon)\bigr\}^{1-\xi/2}}} \cdot \phi(8b/a)\\
        &= 32\sqrt{\frac{\log(4/\delta)}{ \bigl\{nq(1-\epsilon)\bigr\}^{1-\xi/2}}} \cdot \frac{1}{\sqrt{2\pi}} \cdot \exp\biggl\{ -\frac{\xi\log\bigl(nq(1-\epsilon)\bigr)}{8} \biggr\}\\
        &= \sqrt{\frac{512\log(4/\delta)}{\pi \bigl\{nq(1-\epsilon)\bigr\}^{1-\xi/4}}} \geq \sqrt{\biggl(1+\frac{\epsilon}{q(1-\epsilon)}\biggr) \cdot \frac{36\log(4/\delta)}{nq(1-\epsilon)}},
    \end{align*}
    where the final inequality holds since $\log\bigl(1+\frac{\epsilon}{q(1-\epsilon)}\bigr)\leq \frac{7\xi\log(nq(1-\epsilon))}{128}$. 
    Hence, by~\eqref{eq:one-dim-gaussian-realisable-proof-ub2}, we have on $\mathcal{E}$ that $|\hat{\theta}_n^{\mathrm{K}} - \theta_0| \leq \frac{32\sigma b}{\sqrt{\xi \log(nq(1-\epsilon))}}$ when $2\sqrt{\frac{\log(4/\delta)}{(nq(1-\epsilon))^{1-\xi}}} < b \leq 1/2$.

    Finally, consider the case where $b > 1/2$ (when this interval is not vacuous).  Then by~\eqref{eq:one-dim-gaussian-realisable-proof-ub1} and Lemma~\ref{lemma:one-dim-kolmogorov-distance-realisable-sets} we have that on $\mathcal{E}$,
    \begin{align*}
        |\hat{\theta}_n^{\mathrm{K}} - \theta_0| 
        &\leq 2\inf \bigl\{ a \geq 0 : f_{\mathrm{K},b}(a) \geq 2r_n \bigr\}\\
        &\leq 2\inf \biggl\{ a \geq 0: q(1-\epsilon) \cdot \Phi\biggl(\frac{a}{\sigma}-\frac{2\sigma b}{a}\biggr) - \{q(1-\epsilon) + \epsilon\}\cdot \Phi\biggl(-\frac{a}{\sigma}-\frac{2\sigma b}{a}\biggr)\\
        &\hspace{7cm}\geq \sqrt{\frac{36\{q(1-\epsilon)+\epsilon\}\log(4/\delta)}{n}} \biggr\}. \numberthis \label{eq:one-dim-gaussian-realisable-proof-ub3}
    \end{align*}
    Letting $a \coloneqq \frac{3\sigma b}{\sqrt{\xi\log(nq(1-\epsilon))}}$, we have
    \begin{align*}
        q(1 &-\epsilon) \cdot \Phi\biggl(\frac{a}{\sigma}-\frac{2\sigma b}{a}\biggr) - \{q(1-\epsilon) + \epsilon\}\cdot \Phi\biggl(-\frac{a}{\sigma}-\frac{2\sigma b}{a}\biggr)\\
        \overset{(i)}&{\geq} \frac{q(1-\epsilon)}{\bigl(-\frac{a}{\sigma}+\frac{2\sigma b}{a}\bigr) + \bigl(-\frac{a}{\sigma}+\frac{2\sigma b}{a}\bigr)^{-1}} \cdot \phi\biggl(-\frac{a}{\sigma}+\frac{2\sigma b}{a}\biggr) - \frac{q(1-\epsilon)+\epsilon}{\frac{a}{\sigma}+\frac{2\sigma b}{a}} \cdot \phi\biggl(\frac{a}{\sigma}+\frac{2\sigma b}{a}\biggr)\\
        \overset{(ii)}&{\geq} \biggl( \frac{a}{\sigma}+\frac{2\sigma b}{a} \biggr)^{-1} \frac{1}{\sqrt{2\pi}} \cdot \biggl\{ q(1-\epsilon) \exp\biggl( -\frac{a^2}{2\sigma^2} - \frac{2\sigma^2 b^2}{a^2} + 2b \biggr) \\
        &\hspace{6cm}  - \bigl\{q(1-\epsilon)+\epsilon\bigr\} \exp\biggl( -\frac{a^2}{2\sigma^2} - \frac{2\sigma^2 b^2}{a^2} - 2b \biggr) \biggr\}\\
        \overset{(iii)}&{\geq} \biggl( \frac{a}{\sigma}+\frac{2\sigma b}{a} \biggr)^{-1} \frac{4\epsilon}{\sqrt{2\pi}} \cdot \exp\biggl( -\frac{a^2}{2\sigma^2} - \frac{2\sigma^2 b^2}{a^2}\biggr)\\
        \overset{(iv)}&{\geq} \frac{1}{\sqrt{\xi\log\bigl( nq(1-\epsilon) \bigr)}} \cdot \frac{4\epsilon}{\sqrt{2\pi}} \cdot \bigl( nq(1-\epsilon) \bigr)^{-\xi/4}\\
        \overset{(v)}&{\geq} \sqrt{\frac{36\{q(1-\epsilon)+\epsilon\}\log(4/\delta)}{n}} = 2r_n.
    \end{align*}
    Here, $(i)$ follows from the Mills ratio bound $\phi(x)/(x + x^{-1}) \leq \Phi(-x) \leq \phi(x)/x$ for $x > 0$; $(ii)$ follows since  $\bigl(-\frac{a}{\sigma}+\frac{2\sigma b}{a}\bigr) + \bigl(-\frac{a}{\sigma}+\frac{2\sigma b}{a}\bigr)^{-1} \leq \frac{a}{\sigma}+\frac{2\sigma b}{a}$ whenever $1/2 < b \leq \frac{\xi\log(nq(1-\epsilon))}{9}$; $(iii)$ follows by substituting the definition of $b$; $(iv)$ follows since, by assumption $b\leq \frac{7\xi\log(nq(1-\epsilon))}{256}$, so $\frac{a^2}{\sigma^2} \leq \xi\log\bigl(nq(1-\epsilon)\bigr)/100$; and $(v)$ follows from the assumptions that $b > 1/2$ so $q(1-\epsilon) < 3\epsilon$, the fact that $x^{\xi/2} \geq \xi\log x$ for $x\in(0,\infty)$ and the assumption~\eqref{Eq:deltalowerbound}. Hence, by~\eqref{eq:one-dim-gaussian-realisable-proof-ub3}, we have on $\mathcal{E}$ that $|\hat{\theta}_n^{\mathrm{K}} - \theta_0| \leq \frac{6\sigma b}{\sqrt{\xi\log(nq(1-\epsilon))}}$ when $b > 1/2$. Combining all four cases yields the desired result.
\end{proof}

\subsubsection{Proof of Theorem~\ref{thm:multivariate-kolmogorov-estimator}}

\begin{lemma} \label{lemma:realisability-of-projection}
    Let $\epsilon\in[0,1)$, $\pi\in\mathcal{P}\bigl(\{\emptyset,[d]\}\bigr)$, $P\in\mathcal{P}(\mathbb{R}^d)$, $R\in\mathcal{R}_{\emptyset,[d]}(P,\epsilon,\pi)$ and $v\in\mathbb{R}^d$. Suppose that $X\sim P$, $Z\sim R$ and define $Z^{(v)} \coloneqq v^\top Z \cdot \mathbbm{1}_{\{Z\in\mathbb{R}^d\}} + \star \cdot \mathbbm{1}_{\{Z\notin\mathbb{R}^d\}}$ for $v\in\mathbb{R}^d$. Then, writing $P^{(v)} \coloneqq \mathsf{Law}(v^\top X)$ and $R^{(v)}\coloneqq \mathsf{Law}(Z^{(v)})$, we have $R^{(v)} \in \mathcal{R}\bigl( P^{(v)}, \epsilon, \pi([d]) \bigr)$.
\end{lemma}
\begin{proof}
    Let $q \coloneqq \pi([d])$.  We have $\mathsf{Law}(Z) = (1-\epsilon)\mathsf{Law}(X \ostar \Omega^{(1)}) + \epsilon \mathsf{Law}(X \ostar \Omega^{(2)})$ where $\Omega^{(1)} \indep X$ and $\mathbb{P}(\Omega^{(1)} = \bm{1}_{[d]}) = q = 1 - \mathbb{P}(\Omega^{(1)} = 0)$ and where $\Omega^{(2)}$ takes values in $\{0,\bm{1}_{[d]}\}$.  By properties of disintegrations (see Section~\ref{sec:disintegration}), we may define $m^{(v)}:\mathbb{R} \rightarrow [0,1]$ by $m^{(v)}(y) \coloneqq \mathbb{P}(\Omega^{(2)} = \bm{1}_{[d]} \,|\, v^\top X=y)$.  We also let $\mu^{(v)}$ be a $\sigma$-finite measure on $\mathbb{R}$ such that $P^{(v)}\ll\mu^{(v)}$ and let $p^{(v)}\coloneqq \mathrm{d}P^{(v)}/\mathrm{d}\mu^{(v)}$.  Finally, define $g:\mathbb{R}_\star \rightarrow [0,\infty)$ by
    \[
    g(z) \coloneqq \begin{cases}
    q(1 - \epsilon) p^{(v)}(z) + \epsilon m^{(v)}(z) p^{(v)}(z) & \text{ if } z \in \mathbb{R}\\
    1 - q(1-\epsilon) - \epsilon \int_{\mathbb{R}} m^{(v)}(y) p^{(v)}(y)\, \mathrm{d}\mu^{(v)}(y) & \text{ if } z = \star. 
    \end{cases}
    \]
    Then, for $A\in\mathcal{B}(\mathbb{R})$, we have
    \begin{align*}
    \int_A g(z) \, \mathrm{d}\mu^{(v)}_\star(z) &= q(1-\epsilon)\mathbb{P}(v^\top X \in A) + \epsilon \mathbb{P}\bigl(\{\Omega^{(2)} = \bm{1}_{[d]}\} \cap \{ v^\top X \in A\}\bigr)\\
    &= (1-\epsilon)\mathbb{P}\bigl( v^\top(X\ostar\Omega^{(1)}) \in A \bigr) + \epsilon\mathbb{P}\bigl( v^\top(X\ostar\Omega^{(2)}) \in A \bigr)\\
    &= \mathbb{P}(Z^{(v)} \in A) = R^{(v)}(A).
    \end{align*}
    It follows that $R^{(v)} \ll \mu^{(v)}_\star$, with Radon--Nikodym derivative $g$.  Hence, by Proposition~\ref{prop:univariate-realisability}, we have $R^{(v)} \in \mathcal{R}(P^{(v)}, \epsilon, q)$.
\end{proof}

\begin{proof}[Proof of Theorem~\ref{thm:multivariate-kolmogorov-estimator}]
    By Lemma~\ref{lemma:realisability-of-projection}, Theorem \ref{thm:one-dim-kolmogorov-estimator} and a union bound,
    \begin{align} \label{eq:union-bound-for-one-dim-Kolmogorov-estimators}
        \max_{v\in\mathcal{N}} \bigl(\hat{\theta}_n^{\mathrm{K}}(v) - v^\top \theta_0 \bigr)^2 \lesssim C_{n,q,\epsilon,\xi,\delta/9^d} \biggl\{ \frac{\|\Sigma\|_{\mathrm{op}}\bigl(d + \log(4/\delta)\bigr)}{nq(1-\epsilon)} +  \frac{\|\Sigma\|_{\mathrm{op}} \log^2\bigl( 1+\frac{4\epsilon}{q(1-\epsilon)} \bigr)}{\log\bigl(nq(1-\epsilon)\bigr)} \biggr\}
    \end{align}
    with probability at least $1-\delta$.  Next, since any $v \in \mathbb{S}^{d-1}$ can be written as $v = v_1 + v_2$, where $v_1 \in \mathcal{N}$ and $\|v_2\|_2 \leq 1/4$, we have 
    \begin{align*}
        \|\hat{\theta}_n^{\mathrm{MK}} - \theta_0\|_2 = \sup_{v\in\mathbb{S}^{d-1}} |v^\top\hat{\theta}_n^{\mathrm{MK}} - v^\top\theta_0 | \leq \max_{v\in\mathcal{N}} |v^\top\hat{\theta}_n^{\mathrm{MK}} - v^\top\theta_0 | + \frac{1}{4} \cdot \|\hat{\theta}_n^{\mathrm{MK}} - \theta_0\|_2,
    \end{align*}
    so
    \begin{align*}
         \|\hat{\theta}_n^{\mathrm{MK}} - \theta_0\|_2 \leq \frac{4}{3} \cdot \max_{v\in\mathcal{N}} |v^\top\hat{\theta}_n^{\mathrm{MK}} - v^\top\theta_0|.
    \end{align*}
    Hence, 
    \begin{align*}
        \|\hat{\theta}_n^{\mathrm{MK}} - \theta_0\|_2^2 &\leq 2 \max_{v\in\mathcal{N}}\, \bigl( v^\top \hat{\theta}_n^{\mathrm{MK}} -  \hat{\theta}_n^{\mathrm{K}}(v) +  \hat{\theta}_n^{\mathrm{K}}(v) - v^\top \theta_0 \bigr)^2\\
        &\leq 4\max_{v\in\mathcal{N}}\, \bigl( v^\top \hat{\theta}_n^{\mathrm{MK}} -  \hat{\theta}_n^{\mathrm{K}}(v) \bigr)^2 + 4\max_{v\in\mathcal{N}}\, \big( v^\top \theta_0 -  \hat{\theta}_n^{\mathrm{K}}(v) \big)^2\\
        &\leq 8\max_{v\in\mathcal{N}}\, \big( v^\top \theta_0 -  \hat{\theta}_n^{\mathrm{K}}(v) \big)^2\\
        &\lesssim C_{n,q,\epsilon,\xi,\delta/9^d} \biggl\{ \frac{\|\Sigma\|_{\mathrm{op}}\bigl(d+\log(4/\delta)\bigr)}{nq(1-\epsilon)} +  \frac{\|\Sigma\|_{\mathrm{op}} \log^2\bigl( 1+\frac{4\epsilon}{q(1-\epsilon)} \bigr)}{\log\bigl(nq(1-\epsilon)\bigr)} \biggr\},
    \end{align*}
    with probability at least $1-\delta$, where the third inequality follows from the definition of $\hat{\theta}_n^{\mathrm{MK}}$, and the last inequality follows from~\eqref{eq:union-bound-for-one-dim-Kolmogorov-estimators}.
\end{proof}

\subsubsection{Proof of Lemma~\ref{lemma:compute-kolmogorov-distance}}

\begin{proof}[Proof of Lemma~\ref{lemma:compute-kolmogorov-distance}]
    We first show that, for any $R\in\mathcal{R}(P, \epsilon, q)$, we have
    \begin{align*}
        d_{\mathrm{K}}(\hat{R}_n,R) = \max_{i\in \{0\} \cup [m]}\; \Bigl\{\Bigl|\frac{i}{n} - R\bigl((-\infty,Z_{(i)})\bigr) \Bigr| \vee \Bigl|\frac{i}{n} - R\bigl((-\infty,Z_{(i+1)})\bigr)\Bigr| \Bigr\}.
    \end{align*}
    To this end, fix $i \in \{0\} \cup [m]$. Then, since $\hat{R}_n\big((-\infty,t)\bigr) = i/n$ for $t \in [Z_{(i)},Z_{(i+1)}) \cap \mathbb{R}$, $t \mapsto R\bigl((-\infty,t]\bigr)$ is increasing on this interval and since $R \ll \lambda_{\star}$ by Proposition~\ref{prop:univariate-realisability}, we have
    \begin{align*}
    \sup_{t \in [Z_{(i)},Z_{(i+1)}) \cap \mathbb{R}} \bigl|\hat{R}_n\bigl((-\infty,t]\bigr) &- R\bigl((-\infty,t]\bigr)\bigr| \\
    &= \Bigl|\frac{i}{n} - R\bigl((-\infty,Z_{(i)})\bigr)\Bigr| \vee \lim_{t \nearrow Z_{(i+1)}} \Bigl|\frac{i}{n} \!-\! R\bigl((-\infty,t)\bigr)\Bigr| \\
    &= \Bigl|\frac{i}{n} - R\bigl((-\infty,Z_{(i)})\bigr)\Bigr| \vee \Bigl|\frac{i}{n} - R\bigl((-\infty,Z_{(i+1)})\bigr)\Bigr|.
    \end{align*}
    Hence
    \begin{align*}
    \sup_{t\in\mathbb{R}} \bigl|\hat{R}_n\bigl((-\infty, t]\bigr) &- R\bigl((-\infty,t]\bigr) \bigr|\\
    &= \max_{i\in \{0\} \cup [m]} \Bigl\{\Bigl |\frac{i}{n} - R\bigl((-\infty,Z_{(i)})\bigr) \Bigr| \vee \Bigl|\frac{i}{n} - R\bigl((-\infty, Z_{(i+1)})\bigr) \Bigr| \Bigr\}.
    \end{align*}
    Now, by Proposition~\ref{prop:univariate-realisability}, for $0 \leq V_1 \leq \ldots \leq V_{m+1} \leq 1$, there exists $R \in \mathcal{R}(P, \epsilon, q)$ such that $V_i = R\bigl( (-\infty,Z_{(i)}] \bigr)$ for $i\in[m]$ and $V_{m+1}= R\bigl((-\infty,\infty)\bigr)$ if and only if $(V_1,\ldots,V_{m+1})^\top \in \mathcal{V}$. The claim then follows.
\end{proof}

\subsection{Proofs from Section~\ref{sec:nonparametric-realisable}}

\subsubsection{Proof of Theorem~\ref{thm:one-dim-realisable-sample-mean-ub}}

The proof of Theorem~\ref{thm:one-dim-realisable-sample-mean-ub} relies on the following preliminary result, which controls the bias.
\begin{prop}\label{thm:bias-of-mean-one-dim-realisable-case}
    Let $\theta_0\in\mathbb{R}$, $\epsilon\in[0,1)$, $q\in(0,1]$ and $\sigma>0$.
    \begin{itemize}
        \item[(a)] Let $r \geq 2$, $P \in \mathcal{P}_{L^r}(\theta_0, \sigma^2)$ and $Z \sim R \in \mathcal{R}(P, \epsilon, q)$. Then 
        \begin{align*}
        \bigl\{ \mathbb{E}( Z \,|\, Z \neq \star) - \theta_0 \bigr\}^2 \leq \sigma^2 \cdot \biggl\{ \biggl(\frac{\epsilon}{q(1-\epsilon)}\biggr)^2 \wedge \biggl(\frac{\epsilon}{q(1-\epsilon)}\biggr)^{2/r} \biggr\}.
        \end{align*}
        \item[(b)] Let $r\geq 1$, $P \in \mathcal{P}_{\psi_r}(\theta_0, \sigma^2)$ and $Z\sim R \in \mathcal{R}(P, \epsilon, q)$. Then 
        \begin{align*}
        \bigl\{ \mathbb{E}( Z \,|\, Z \neq \star) - \theta_0 \bigr\}^2 \leq \sigma^2 \cdot \biggl\{ 4\biggl(\frac{\epsilon}{q(1-\epsilon)}\biggr)^2 \;\wedge\;  \log^{2/r} \biggl( 2 + \frac{2\epsilon}{q(1-\epsilon)} \biggr) \biggr\}.
        \end{align*}
        
    \end{itemize}
\end{prop}
\begin{proof}
    Let $\kappa \coloneqq \frac{\epsilon}{q(1-\epsilon)}$. By translation invariance, we may assume without loss of generality that $\theta_0 = 0$ throughout the proof.

    (a) Let $\mu$ be a measure on $\mathbb{R}$ such that $P \ll \mu$ and let $p\coloneqq \frac{\mathrm{d}P}{\mathrm{d}\mu}$, then by Proposition~\ref{prop:univariate-realisability}, we have \begin{align}
    \label{Eq:RRNDeriv}
        \frac{\mathrm{d}R}{\mathrm{d}\mu_\star}(z) = \begin{cases}
            q(1-\epsilon) \cdot p(z) + \epsilon\cdot  m(z)p(z) \quad&\text{if }z\in\mathbb{R}\\
            1- q(1-\epsilon) - \epsilon\int_{\mathbb{R}} m(x)p(x) \,\mathrm{d}\mu(x) &\text{if }z=\star,
        \end{cases}
    \end{align}
    for some Borel measurable function $m:\mathbb{R} \to [0,1]$. Therefore,
    \begin{align*}
        \big| \mathbb{E} (Z \,|\, Z \neq \star) \big| &= \frac{\bigl| q(1-\epsilon) \cdot \int_{\mathbb{R}} xp(x) \,\mathrm{d}\mu(x) + \epsilon\cdot \int_{\mathbb{R}} xm(x)p(x) \,\mathrm{d}\mu(x) \bigr|}{q(1-\epsilon) + \epsilon\int_{\mathbb{R}} m(x)p(x) \,\mathrm{d}\mu(x)}\\
        &= \frac{\epsilon \cdot \bigl|  \mathbb{E}_{P}\{Xm(X)\} \bigr|}{q(1-\epsilon) + \epsilon \cdot \mathbb{E}_{P}\{m(X)\}} \leq \frac{\epsilon \cdot \sigma \cdot \bigl\{\mathbb{E}_{P} \bigl(m^{r/(r-1)}(X)\bigr)\bigr\}^{1-1/r}}{q(1-\epsilon) + \epsilon \cdot \mathbb{E}_{P}\{m(X)\}},
    \end{align*}
    where the second equality follows from the assumption that $\theta_0=0$, and where the inequality follows from H\"{o}lder's inequality and the fact that $\mathbb{E}_{P}(|X|^r)^{1/r} \leq \sigma$. On the one hand, since $\bigl\{\mathbb{E}_{P} \bigl(m^{r/(r-1)}(X)\bigr)\bigr\}^{1-1/r} \leq 1$ and $\mathbb{E}_{P}\{m(X)\} \geq 0$, we have
    \begin{align} \label{ineq:holder-m-ineq1}
        \frac{\epsilon \cdot \bigl\{\mathbb{E}_{P} \bigl(m^{r/(r-1)}(X)\bigr)\bigr\}^{1-1/r}}{q(1-\epsilon) + \epsilon \cdot \mathbb{E}_{P}\{m(X)\}} \leq \kappa.
    \end{align}
    On the other hand, since $m(x)\in[0,1]$, we have $m^{r/(r-1)}(x) \leq m(x)$ for all $x\in\mathbb{R}$ and thus $\bigl\{\mathbb{E}_{P} \bigl(m^{r/(r-1)}(X)\bigr)\bigr\}^{1-1/r} \leq \bigl\{\mathbb{E}_{P}\bigl(m(X)\bigr)\bigr\}^{1-1/r} \eqqcolon t$. Therefore,
    \begin{align*}
        &\frac{\epsilon  \cdot \bigl\{\mathbb{E}_{P} \bigl(m^{r/(r-1)}(X)\bigr)\bigr\}^{1-1/r}}{q(1-\epsilon) + \epsilon \cdot \mathbb{E}_{P}\{m(X)\}} \leq \frac{\epsilon t}{q(1-\epsilon) + \epsilon t^{r/(r-1)}}\\
        &\qquad\qquad \leq \sup_{t'\geq 0} \frac{\epsilon t'}{q(1-\epsilon) + \epsilon (t')^{r/(r-1)}} \overset{(i)}{=} \frac{\epsilon \cdot \{(r-1)q(1-\epsilon)/\epsilon \}^{1-1/r}}{q(1-\epsilon)+(r-1)q(1-\epsilon)}\\
        &\qquad\qquad \leq (r-1)^{-1/r} \kappa^{1/r} \leq \kappa^{1/r}, \numberthis \label{ineq:holder-m-ineq2}
    \end{align*}
where $(i)$ follows from the fact that the function $t' \mapsto \frac{\epsilon t'}{q(1-\epsilon) + \epsilon (t')^{r/(r-1)}}$ is maximised when $t'= \{(r-1)q(1-\epsilon)/\epsilon \}^{1-1/r}$. Combining~\eqref{ineq:holder-m-ineq1} and~\eqref{ineq:holder-m-ineq2}, we deduce that
    \begin{align*}
        \big| \mathbb{E} (Z \,|\, Z \neq \star) \big| \leq \frac{\epsilon \cdot \sigma \cdot \bigl\{\mathbb{E}_{P} \bigl(m^{r/(r-1)}(X)\bigr)\bigr\}^{1-1/r}}{q(1-\epsilon) + \epsilon \cdot \mathbb{E}_{P}\{m(X)\}} \leq \sigma (\kappa \wedge \kappa^{1/r}),
    \end{align*}
    as desired.

    (b) Let $Q \in \mathcal{P}(\mathbb{R})$ such that $Q\ll P$. 
 By the variational characterisation of Kullback--Leibler divergence \citep[e.g.][Corollary 4.15]{boucheron2003concentration},
 \begin{align}\label{eq:variational-principle-KL}
        \mathbb{E}_{X\sim Q}\bigl( g(X) \bigr) \leq \mathrm{KL}(Q,P) + \log \mathbb{E}_{X\sim P}\bigl( e^{g(X)} \bigr),
    \end{align}
    for all Borel measurable functions $g:\mathbb{R} \to [0,\infty)$.  Now take $Q$ to be the conditional distribution of $Z$ given $\{Z\neq\star\}$.  Let $\mu$ and $p$ be as in the proof of~(a), so that~\eqref{Eq:RRNDeriv} holds for some Borel measurable function $m:\mathbb{R} \to [0,1]$.
    Therefore, for all $x \in \mathbb{R}$,
    \begin{align*}
        \frac{\mathrm{d}Q}{\mathrm{d}\mu}(x) = 
        \frac{q(1-\epsilon) \cdot p(x) + \epsilon\cdot  m(x)p(x)}{q(1-\epsilon) + \epsilon \cdot \int_{\mathbb{R}}m(y)p(y)\,\mathrm{d}\mu(y)}.
    \end{align*}
    Hence $Q\ll P$ and
    \begin{align}
        \frac{\mathrm{d}Q}{\mathrm{d}P}(x) \in \biggl[ 1-\frac{\epsilon}{q(1-\epsilon)+\epsilon},\, 1+\frac{\epsilon}{q(1-\epsilon)} \biggr], \label{Eq:RRNDerivBound}
    \end{align}
    for all $x \in \mathbb{R}$, from which we deduce that 
    \begin{align}
    \label{eq:KLbound}
        \mathrm{KL}(Q,P) = \int_{\mathbb{R}} \log\biggl( \frac{\mathrm{d}Q}{\mathrm{d}P}\biggr) \, \mathrm{d}Q \leq \log ( 1 + \kappa).
    \end{align}
    Taking $g(\cdot) = |\cdot|^r/\sigma^r$ and combining~\eqref{eq:variational-principle-KL} and~\eqref{eq:KLbound} yields
    \begin{align}
        \mathbb{E} \bigl( |Z|^r/\sigma^r \,\big|\, Z \neq \star \bigr) &\leq \log( 1+\kappa) + \log \mathbb{E}_{X\sim P} \bigl\{ \exp\bigl( |X|^r/\sigma^r \bigr) \bigr\} \nonumber\\
        &\leq \log( 1+\kappa) + \log2 = \log( 2+2\kappa), \label{eq:expectation-of-g(Z)}
    \end{align}
    where the second inequality follows since $P \in \mathcal{P}_{\psi_r}(\theta_0, \sigma^2)$ and since $\theta_0 = 0$ by assumption. Thus, 
    \begin{align}
        \bigl| \mathbb{E} (Z \,|\, Z \neq \star) \bigr| \leq \mathbb{E} \bigl( |Z| \,\big|\, Z \neq \star \bigr) \leq \mathbb{E} \bigl( |Z|^r \,\big|\, Z \neq \star \bigr)^{1/r} \leq \sigma\log^{1/r}( 2+2\kappa), \label{eq:psi-r-bias-1}
    \end{align}
    where the second inequality follows from the conditional version of Jensen's inequality and the third inequality follows from \eqref{eq:expectation-of-g(Z)}. Moreover, by \citet[Lemma A.2]{gotze2021concentration}, we have $\Var_{X\sim P}(X)^{1/2} \leq 2\bigl(\frac{2}{re}\bigr)^{1/r}\sigma \leq 2\sigma$ for $r\geq 1$. Hence $P \in \mathcal{P}_{L^2}(0,4\sigma^2)$, so we can apply part (a) of the theorem to obtain 
    \begin{align}
        \big| \mathbb{E} \bigl( Z \,\big|\, Z \neq \star \bigr) \big| \leq 2\sigma\kappa. \label{eq:psi-r-bias-2}
    \end{align}
    Combining~\eqref{eq:psi-r-bias-1} and~\eqref{eq:psi-r-bias-2} proves part (b). \medskip
\end{proof}
\begin{proof}[Proof of Theorem~\ref{thm:one-dim-realisable-sample-mean-ub}]
    Let $\kappa \coloneqq \frac{\epsilon}{q(1-\epsilon)}$.

    (a) Let $\mu$ and $p$ be as in the proof of Proposition~\ref{thm:bias-of-mean-one-dim-realisable-case}(a), so that~\eqref{Eq:RRNDeriv} holds for some Borel measurable function $m:\mathbb{R} \to [0,1]$.  
    On the one hand, since $m(X) \in [0,1]$, we have
    \begin{align*}
        \Var(Z_1 \,|\, Z_1 \neq \star) &= \Var(Z_1 - \theta_0 \,|\, Z_1 \neq \star) \leq \mathbb{E}\bigl\{ (Z_1 - \theta_0)^2 \,|\, Z_1 \neq \star \bigr\} \nonumber \\
        &= \frac{\int_{\mathbb{R}} (x-\theta_0)^2 \{q(1-\epsilon)p(x) + \epsilon m(x)p(x)\} \, \mathrm{d}\mu(x)}{q(1-\epsilon) + \epsilon \int_{\mathbb{R}} m(x)p(x) \, \mathrm{d}\mu(x)} \\
        &\leq \sigma^2 + \frac{\epsilon \cdot \mathbb{E}_P\{(X - \theta_0)^2 m(X)\}}{q(1 - \epsilon) + \epsilon \cdot \mathbb{E}_P\{m(X)\}} \leq (1+\kappa) \sigma^2. \numberthis \label{ineq:holder-m-ineq3}
    \end{align*}
    On the other hand, for $r > 2$, we have by H{\"o}lder's inequality that
    \begin{align*}
        \Var(Z_1 \,&|\, Z_1 \neq \star) \\
        &\leq \sigma^2 + \frac{\epsilon \cdot \mathbb{E}_P\{(X - \theta_0)^2 m(X)\}}{q(1 - \epsilon) + \epsilon \cdot \mathbb{E}_P\{m(X)\}} \leq \biggl[1 + \frac{\epsilon \cdot \bigl\{\mathbb{E}_P\bigr(m^{r/(r-2)}(X)\bigl)\bigr\}^{1 - 2/r}}{q(1 - \epsilon) + \epsilon \cdot \mathbb{E}_P\{m(X)\}}\biggr] \sigma^2\\
        &\leq \biggl[1 + \frac{\epsilon \cdot \bigl\{\mathbb{E}_P\bigr(m(X)\bigl)\bigr\}^{1 - 2/r}}{q(1 - \epsilon) + \epsilon \cdot \mathbb{E}_P\{m(X)\}}\biggr] \sigma^2 \leq \sup_{t'\geq 0} \biggl(1 + \frac{\epsilon t'}{q(1 - \epsilon) + \epsilon (t')^{r/(r-2)}}\biggr) \sigma^2\\
        &= \biggl\{1 + \frac{2}{r} \biggl( \frac{r-2}{2} \biggr)^{1-2/r} \cdot \kappa^{2/r}\biggr\}\sigma^2 \leq \biggl\{1 + \biggl( \frac{2}{r} \biggr)^{2/r} \kappa^{2/r}\biggr\}\sigma^2 \leq (1 + \kappa^{2/r}) \sigma^2, \numberthis \label{ineq:holder-m-ineq4}
    \end{align*}
    where the equality follows since the supremum is attained when $t' = \bigl( \frac{(r-2)q(1-\epsilon)}{2\epsilon}\bigr)^{1-2/r}$. Combining~\eqref{ineq:holder-m-ineq3} and~\eqref{ineq:holder-m-ineq4} yields that for $r\geq 2$,
    \begin{align}
        \Var(Z_1 \,|\, Z_1 \neq \star) \leq (1+\kappa^{2/r})\sigma^2. \label{Eq:CondVarBound}
    \end{align}
    By Lemma~\ref{lemma:binomial-tail}(b), the event $\mathcal{E}_0 \coloneqq \{|\mathcal{D}| \geq nq(1-\epsilon)/2\}$ has $\mathbb{P}(\mathcal{E}_0) \geq 1-\delta/2$, since $nq(1-\epsilon)\geq (8/a)\log(2c/\delta) \geq 8\log(2/\delta)$.  By~\eqref{eq:assumption-on-alg-heavy-tail-univariate} and~\eqref{Eq:CondVarBound},
    \begin{align*}
        \mathbb{P}\biggl(\bigl(\hat{\theta}_n - \mathbb{E}(Z_1 \, | \, Z_1\neq \star)\bigr)^2 \leq 2C(1+\kappa^{2/r}) \frac{\sigma^2 \log(2e/\delta) }{nq(1-\epsilon)} \,\bigg|\, \mathcal{E}_0 \biggr) \geq 1-\frac{\delta}{2}.
    \end{align*}
    Thus, 
    \begin{align}
        \mathbb{P}\biggl(\bigl(\hat{\theta}_n &- \mathbb{E}(Z_1|Z_1\neq \star)\bigr)^2 \leq 2C(1+\kappa^{2/r}) \frac{\sigma^2 \log(2e/\delta) }{nq(1-\epsilon)} \biggr)\nonumber\\
        &\geq \mathbb{P}\biggl(\bigl(\hat{\theta}_n - \mathbb{E}(Z_1|Z_1\neq \star)\bigr)^2 \leq 2C(1+\kappa^{2/r}) \frac{\sigma^2 \log(2e/\delta) }{nq(1-\epsilon)} \,\bigg|\, \mathcal{E}_0 \biggr)\mathbb{P}(\mathcal{E}_0) \geq 1-\delta. \label{Eq:E0E1bound}
    \end{align}
    On the event that $\bigl\{\bigl(\hat{\theta}_n - \mathbb{E}(Z_1|Z_1\neq \star)\bigr)^2 \leq 2C(1+\kappa^{2/r}) \frac{\sigma^2 \log(2e/\delta) }{nq(1-\epsilon)}\bigr\}$, we have by Proposition~\ref{thm:bias-of-mean-one-dim-realisable-case}(a) that
    \begin{align*}
        \bigl(\hat{\theta}_n - \theta_0 \bigr)^2 &\leq 2\bigl\{ \hat{\theta}_n - \mathbb{E}(Z_1 \,|\, Z_1\neq\star)\bigr\}^2 + 2\bigl\{ 
        \mathbb{E}(Z_1 \,|\, Z_1\neq\star) - \theta_0\bigr\}^2 \nonumber\\
        &\leq 4C\cdot \frac{\sigma^2 \log(2e/\delta) }{nq(1-\epsilon)} + 4C\kappa^{2/r}\cdot \frac{\sigma^2 \log(2e/\delta) }{nq(1-\epsilon)} + 2\sigma^2(\kappa^2 \wedge \kappa^{2/r}) \nonumber\\
        &\leq 8C\cdot \frac{\sigma^2 \log(2e/\delta) }{nq(1-\epsilon)} + (2C+2)\sigma^2(\kappa^2 \wedge \kappa^{2/r}),
    \end{align*}
    where the final inequality follows by considering separately the cases $\kappa\leq 1$ and $\kappa>1$, and in the second case noting that $\frac{\log(2e/\delta)}{nq(1-\epsilon)} \leq \frac{\{1+\log^{-1}(2)\} \log(2/\delta)}{nq(1-\epsilon)} \leq 1/2$ by assumption.

    \medskip
    (b) Let $\mu$ and $p$ be as in the proof of Proposition~\ref{thm:bias-of-mean-one-dim-realisable-case}(a), so that~\eqref{Eq:RRNDeriv} holds for some Borel measurable function $m:\mathbb{R} \to [0,1]$.  Then, for integers $\ell \geq 2$, we have 
    \begin{align*}
        \bigl\{ \mathbb{E}\bigl( |Z_1 - \theta_0|^{\ell} \,|\, Z_1\neq\star \bigr) \bigr\}^{1/\ell} &= \biggl(\frac{\int_{\mathbb{R}} |x-\theta_0|^\ell \{q(1-\epsilon)p(x) + \epsilon m(x)p(x)\} \, \mathrm{d}\mu(x)}{q(1-\epsilon) + \epsilon \int_{\mathbb{R}} m(x)p(x) \, \mathrm{d}\mu(x)} \biggr)^{1/\ell}\\
        &\leq \biggl( \frac{\{q(1-\epsilon) + \epsilon\}\mathbb{E}_{X\sim P}(|X-\theta_0|^\ell)}{q(1-\epsilon)} \biggr)^{1/\ell} 
        \overset{(i)}{\lesssim} \sqrt{1 + \kappa} \cdot \sigma \ell,
    \end{align*}
    where $(i)$ is true since $X\sim P \in \mathcal{P}_{\psi_r}(\theta_0, \sigma^2) \subseteq \mathcal{P}_{\psi_1}(\theta_0, \sigma^2)$ by Lemma~\ref{lemma:inclusion-of-psi-r-class}, so $\bigl(\mathbb{E}_{X\sim P} |X-\theta_0|^\ell\bigr)^{1/\ell} \lesssim \sigma\ell$ by \citet[Proposition 2.7.1]{vershynin2018high}. Moreover, by the Cauchy--Schwarz inequality and~\eqref{ineq:holder-m-ineq3}, $\mathbb{E}\bigl( |Z_1 - \theta_0| \,|\, Z_1\neq\star \bigr) \leq \bigl\{ \mathbb{E}\bigl( |Z_1 - \theta_0|^{2} \,|\, Z_1\neq\star \bigr) \bigr\}^{1/2} \lesssim \sigma\sqrt{1+\kappa}$.  Hence, by \citet[Proposition 2.7.1]{vershynin2018high} again, conditional on $\{Z_1\neq\star\}$, we have $\|Z_1 - \theta_0\|_{\psi_1} \lesssim \sigma\sqrt{1+\kappa}$. Then, by \citet[Lemma 2.7.10]{vershynin2018high}, we have, conditional on $\{Z_1\neq\star\}$, that
    \begin{align}
        \bigl\|Z_1 - \mathbb{E}(Z_1\,|\, Z_1\neq\star)\bigr\|_{\psi_1} \lesssim \sigma\sqrt{1+\kappa}. \label{eq:sub-exponential-norm-bound}
    \end{align}  
    Recall the definition of the event $\mathcal{E}_0$ from the proof of~(a), and observe that $\mathbb{P}(\mathcal{E}_0) \geq 1- \delta/4$ by Lemma~\ref{lemma:binomial-tail}(b).  Now, similarly to~\eqref{Eq:E0E1bound}, by Bernstein's inequality \citep[][Corollary~2.8.3]{vershynin2018high} and since $\frac{\log(8/\delta)}{nq(1-\epsilon)} \leq 1/8$, there exists a universal constant $C_2 > 0$ such that the event 
    \[
    \mathcal{E}_2 \coloneqq \biggl\{ \biggl( \frac{\sum_{i\in\mathcal{D}} \{Z_i - \mathbb{E}(Z_1\,|\, Z_1\neq\star)\}}{|\mathcal{D}|} \biggr)^2 \leq C_2 ( 1 + \kappa) \frac{\sigma^2\log(8/\delta)}{nq(1-\epsilon)} \biggr\}
    \]
    satisfies $\mathbb{P}(\mathcal{E}_0 \cap \mathcal{E}_2) \geq 1 - \delta/2$.  Moreover, on $\mathcal{E}_0 \cap \mathcal{E}_2$, by Proposition~\ref{thm:bias-of-mean-one-dim-realisable-case}(b),
   \begin{align*}
        \bigl( \hat{\theta}_n - \theta_0 \bigr)^2 &\lesssim \biggl( \frac{\sum_{i\in\mathcal{D}} \{Z_i - \mathbb{E}(Z_1\,|\, Z_1\neq\star)\}}{|\mathcal{D}|} \biggr)^2 + \bigl\{ \mathbb{E}(Z_1\,|\, Z_1\neq\star) - \theta_0 \bigr\}^2\\
        &\lesssim ( 1 + \kappa) \cdot \frac{\sigma^2\log(8/\delta)}{nq(1-\epsilon)} + \sigma^2 \kappa^2 \lesssim \frac{\sigma^2\log(8/\delta)}{nq(1-\epsilon)} + \sigma^2\biggl( \frac{\log(8/\delta)}{nq(1-\epsilon)} \biggr)^2 + \sigma^2 \kappa^2\\
        &\lesssim \frac{\sigma^2\log(8/\delta)}{nq(1-\epsilon)} + \sigma^2 \kappa^2, \numberthis \label{eq:exponential-tail-bound-1}
    \end{align*}
    where the penultimate inequality follows from the inequality $ab \leq \frac{a^2+b^2}{2}$ for $a,b\in\mathbb{R}$, and the final inequality follows from the assumption $\frac{\log(8/\delta)}{nq(1-\epsilon)} \leq 1/8$.
    
    Next, let $Q\in\mathcal{P}(\mathbb{R})$ be such that $Q\ll P$. Then the variational characterisation of $\chi^2$-divergence \citep[e.g.,][Example~7.4]{polyanskiy2024information} yields that
    \begin{align}
        2\mathbb{E}_{X\sim Q}\{g(X)\} \leq 1 + \chi^2(Q,P) + \mathbb{E}_{X\sim P}\{g^2(X)\}, \label{eq:variational-form-chi-squared-divergence}
    \end{align}
    for all Borel measurable $g:\mathbb{R} \to [0,\infty)$. We first consider the case $r>1$.  Now take $Q$ to be the conditional distribution of $Z_1$ given $\{Z_1\neq\star\}$, so that, by the representation in~\eqref{Eq:RRNDerivBound}, we have
    \[
    \chi^2(Q,P) = \int_{\mathbb{R}} \biggl(\frac{\mathrm{d}Q}{\mathrm{d}P} - 1\biggr)^2 \, \mathrm{d}P \leq \kappa^2.
    \]
    Thus, taking $g : x \mapsto \exp\bigl\{\lambda (x-\theta_0)\bigr\}$ in~\eqref{eq:variational-form-chi-squared-divergence} and applying Lemma~\ref{lemma:MGF-bound} yields that
    \begin{align*}
        2\mathbb{E}\bigl[ \exp\bigl\{\lambda (Z_1-\theta_0)\bigr\} \bigm| Z_1\neq\star \bigr] \leq 1 + \kappa^2 + 2\exp\bigl\{(2\sigma\lambda)^{r/(r-1)}\bigr\},
    \end{align*}
    for all $\lambda>0$. Hence, for $s \in [n]$,
    \begin{align*}
        \log\mathbb{E}\biggl\{ \exp\biggl(\frac{\lambda}{|\mathcal{D}|}\sum_{i\in\mathcal{D}}(Z_i -\theta_0)\biggr) \,&\biggm|\, |\mathcal{D}| = s \biggr\} \\
        &\leq s\log\Biggl\{ \frac{1}{2} \biggl[ 1 + \kappa^2 + 2\exp\biggl\{\biggl(\frac{2\sigma\lambda}{s}\biggr)^{r/(r-1)}\biggr\} \biggr] \Biggr\}\\
        &\leq s \biggl\{ \log( 1 + \kappa^2) + \log 2 + \biggl(\frac{2\sigma\lambda}{s}\biggr)^{r/(r-1)}  \biggr\}, \numberthis\label{eq:log-mgf-bound}
    \end{align*}
    where the final inequality follows from the fact that $\log\bigl(\frac{a+b}{2}\bigr) \leq \log a + \log b$ for all $a,b\geq 1$.
    Then, applying a Chernoff bound gives that for every $t \geq 0$ and $s \in [n]$, 
    \begin{align*}
        \mathbb{P}\bigl(\hat{\theta}_n - \theta_0 \geq t \bigm| |\mathcal{D}| = s\bigr) = \mathbb{P}\biggl( \frac{1}{|\mathcal{D}|}\sum_{i\in\mathcal{D}} (Z_i-\theta_0) \geq t \biggm| |\mathcal{D}| = s\biggr) \leq \exp\bigl(-\psi^*(t)\bigr),
    \end{align*}
    where
    \begin{align*}
        \psi^*(t) &\coloneqq \sup_{\lambda>0}\, \biggl\{\lambda t - s\log( 1 + \kappa^2) - s\log 2 - \frac{(2\sigma\lambda)^{r/(r-1)}}{s^{1/(r-1)}} \biggr\}\\
        &= \frac{st^r}{(2\sigma)^r} \cdot \biggl(\frac{r-1}{r}\biggr)^{r-1} \cdot \frac{1}{r} - s\log( 2 + 2\kappa^2) \geq \frac{st^r}{(2\sigma)^r} \cdot \frac{1}{er} - s\log( 2 + 2\kappa^2),
    \end{align*}
    since the supremum over $\lambda \in (0,\infty)$ is attained at $\lambda^* \coloneqq \bigl( \frac{r-1}{r} \cdot \frac{ts^{1/(r-1)}}{(2\sigma)^{r/(r-1)}} \bigr)^{r-1}$.
    By replacing $Z_i - \theta_0$ with $-(Z_i - \theta_0)$ for $i \in [n]$, we deduce that for every $t \geq 0$,
    \begin{align*}
        \mathbb{P}\bigl(|\hat{\theta}_n - \theta_0| \geq t \bigm| |\mathcal{D}|=s\bigr) \leq 2\exp\biggl\{ -\frac{st^r}{(2\sigma)^r} \cdot \frac{1}{er} + s\log( 2 + 2\kappa^2) \biggr\}.
    \end{align*}
    Hence, defining the event
    \[
    \mathcal{E}_3 \coloneqq \biggl\{ (\hat{\theta}_n - \theta_0)^2 \leq \biggl( \frac{(2\sigma)^r er\log(8/\delta)}{nq(1-\epsilon)/2} + (2\sigma)^r er \log( 2 + 2\kappa^2) \biggr)^{2/r}\biggr\},
    \]
    and proceeding in a similar fashion to~\eqref{Eq:E0E1bound}, we deduce that $\mathbb{P}(\mathcal{E}_0 \cap \mathcal{E}_3) \geq 1- \delta/2$.  Moreover, on $\mathcal{E}_0 \cap \mathcal{E}_3$, we have
    \begin{align*}
        (\hat{\theta}_n - \theta_0)^2 &\leq \biggl\{ \frac{(2\sigma)^r er\log(8/\delta)}{nq(1-\epsilon)/2} + (2\sigma)^r er \log( 2 + 2\kappa^2) \biggr\}^{2/r}\\
        &\leq \biggl\{ \frac{(2\sigma)^r er}{4} + (2\sigma)^r er \log( 2 + 2\kappa^2) \biggr\}^{2/r}\leq \biggl\{ \frac{3}{2} \cdot (2\sigma)^r er \log(2 + 2\kappa^2) \biggr\}^{2/r}\\
        &\leq \bigl\{ 3 \cdot (2\sigma)^r er \log( 2 + 2\kappa) \bigr\}^{2/r} \leq (9e\sigma)^2 \log^{2/r}( 2 + 2\kappa). \numberthis \label{eq:exponential-tail-bound-2}
    \end{align*}
    Thus, on $\mathcal{E}_0 \cap \mathcal{E}_2 \cap \mathcal{E}_3$, which satisfies $\mathbb{P}(\mathcal{E}_0 \cap \mathcal{E}_2 \cap \mathcal{E}_3) \geq 1 - \delta$, we combine~\eqref{eq:exponential-tail-bound-1} and~\eqref{eq:exponential-tail-bound-2} to obtain the desired result for $r>1$. 

    Finally, we consider the case where $r=1$. By \citet[Lemma~2.5]{zhivotovskiy2024dimension}, \eqref{eq:variational-form-chi-squared-divergence} yields, with the same choice of $Q$ and $g$, that
    \begin{align*}
        2\mathbb{E}\bigl[ \exp\bigl\{\lambda (Z_1-\theta_0)\bigr\} \bigm| Z_1\neq\star \bigr] \leq 1 + \kappa^2 + \exp\bigl\{(2\sigma\lambda)^2\bigr\},
    \end{align*}
    for all $|\lambda|\leq \frac{1}{2\sigma}$. Hence, by a similar argument to the $r > 1$ case,
    \begin{align}
        \log\mathbb{E}\biggl\{ \exp\biggl(\frac{\lambda}{|\mathcal{D}|}\sum_{i\in\mathcal{D}}(Z_i-\theta_0)\biggr) \,\biggm|\, |\mathcal{D}| = s \biggr\} \leq s \biggl\{ \log(1 + \kappa^2) + \biggl(\frac{2\sigma\lambda}{s}\biggr)^2 \biggr\}, \label{eq:log-mgf-bound-r=1}
    \end{align}
    for $s\in[n]$ and $|\lambda|\leq \frac{s}{2\sigma}$. Then, applying a Chernoff bound yields
    \begin{align*}
        \mathbb{P}\bigl(|\hat{\theta}_n - \theta_0| \geq t \bigm| |\mathcal{D}|=s\bigr) = \mathbb{P}\biggl( \Bigl| \frac{1}{|\mathcal{D}|} \sum_{i\in\mathcal{D}} (Z_i-\theta_0) \Bigr| \geq t \biggm| |\mathcal{D}| = s \biggr) \leq 2\exp\bigl(-\psi^*(t)\bigr),
    \end{align*}
    where
    \begin{align*}
        \psi^*(t) &\coloneqq \sup_{0<\lambda\leq \frac{s}{2\sigma}}\, \biggl\{\lambda t - s\log( 1 + \kappa^2) - \frac{(2\sigma\lambda)^2}{s}\biggr\}.
    \end{align*}
    Taking $t \coloneqq 8\sigma \log( 2 + 2\kappa^2)$ and $\lambda = s/(2\sigma)$ yields, for $s\geq nq(1-\epsilon)/2$, that
    \begin{align*}
        \mathbb{P}\bigl(|\hat{\theta}_n - \theta_0| \geq t \bigm| |\mathcal{D}| = s\bigr) &\leq 2\exp\bigl\{ -3s \log( 2 + 2\kappa^2) + s \bigr\} \leq 2\exp(-s) \leq \frac{\delta}{4},
    \end{align*}
    where the final inequality follows from the assumption $nq(1-\epsilon)\geq 8\log(8/\delta)$. Therefore, letting $\mathcal{E}_4 \coloneqq \bigl\{ |\hat{\theta}_n - \theta_0| < 8\sigma \log( 2 + 2\kappa^2) \bigr\}$, we have $\mathbb{P}(\mathcal{E}_0 \cap \mathcal{E}_4) \geq 1-\delta/2$. Moreover, on $\mathcal{E}_0 \cap \mathcal{E}_4$,
    \begin{align}
        (\hat{\theta}_n - \theta_0)^2 < 64\sigma^2 \log^2(2 + 2\kappa^2) \leq 256\sigma^2 \log^2( 2 + 2\kappa). \label{eq:exponential-tail-bound-3}
    \end{align}
    Thus, on $\mathcal{E}_0 \cap \mathcal{E}_2 \cap \mathcal{E}_4$, which satisfies $\mathbb{P}(\mathcal{E}_0 \cap \mathcal{E}_2 \cap \mathcal{E}_4) \geq 1 - \delta$, we combine~\eqref{eq:exponential-tail-bound-1} and~\eqref{eq:exponential-tail-bound-3} to obtain the desired result for $r=1$. 
\end{proof}

\subsubsection{Proof of Theorem~\ref{thm:nonparametric-realisable-model-lb}}
For $\theta\in\mathbb{R}$ and $K>0$, define
\begin{align}\label{eq:distributions-with-bounded-support}
    \mathcal{P}_{\mathrm{b}}(\theta,K) \coloneqq \Bigl\{ P \in \mathcal{P}(\mathbb{R}) : \mathbb{E}_P(X) = \theta,\text{$P$ is supported on an interval of length at most $K$} \Bigr\}.
\end{align}
\begin{proof}[Proof of Theorem~\ref{thm:nonparametric-realisable-model-lb}]    
    (a) Define $a \coloneqq \frac{q(1-\epsilon)}{q(1-\epsilon)+\epsilon} \in (0,1]$ and $b\coloneqq \frac{\sigma}{2} \cdot a^{-1/r} > 0$.  Let $X_1 \sim P_1$ and $X_2 \sim P_2$ be random variables satisfying
    \begin{align*}
        X_1 = \begin{cases}
            -b \quad&\text{with probability }\frac{1}{a+1}\\
            b &\text{with probability }\frac{a}{a+1}
        \end{cases}
        \quad\text{and}\quad X_2 = \begin{cases}
            -b \quad&\text{with probability }\frac{a}{a+1}\\
            b &\text{with probability }\frac{1}{a+1}.
        \end{cases}
    \end{align*}
    Then $\theta_1 \coloneqq \mathbb{E}(X_1) = -\frac{(1-a)b}{a+1}$ and $\theta_2 \coloneqq \mathbb{E}(X_2) = \frac{(1-a)b}{a+1}$. Moreover,
    \begin{align*}
        \mathbb{E}\bigl( |X_1 - \theta_1|^r \bigr) &= \frac{(2ab)^r + a(2b)^r}{(a+1)^{r+1}} \leq a \cdot (2b)^r = \sigma^r,
    \end{align*}
    where we have used the fact that $a^r + a \leq a(a+1) \leq a(a+1)^{r+1}$; by symmetry, $\mathbb{E}\bigl( |X_2 - \theta_2|^r \bigr) \leq \sigma^r$.  Consequently, $P_1 \in \mathcal{P}_{L^r}(\theta_1, \sigma^2)$ and $P_2 \in \mathcal{P}_{L^r}(\theta_2, \sigma^2)$. Now define $R_0 \in \mathcal{P}(\mathbb{R}_{\star})$ by
    \begin{align*}
        R_0(\{-b\}) \coloneqq \frac{q(1-\epsilon)}{a+1} \eqqcolon R_0(\{b\}) \quad \text{ and } \quad R_0(\{\star\}) \coloneqq 1- R_0(\{-b\}) - R_0(\{b\}) \in [0,1).
    \end{align*}
    By Proposition~\ref{prop:univariate-realisability}, $R_0 \in \mathcal{R}(P_1,\epsilon,q) \cap \mathcal{R}(P_2,\epsilon,q)$.  Therefore, by \citet[][Theorem~4 and Lemma~5]{ma2024high},
    \begin{align*}
        \mathcal{M}(\delta,\mathcal{P}_{\Theta},|\cdot|^2) \geq \frac{(\theta_2 - \theta_1)^2}{4} = \biggl\{ \frac{(1-a)b}{a+1} \biggr\}^2
        &= \frac{\sigma^2}{4} \biggl(\frac{\epsilon}{2q(1-\epsilon)+\epsilon}\biggr)^2  \biggl(\frac{q(1-\epsilon) + \epsilon}{q(1-\epsilon)}\biggr)^{2/r} \\
        &\geq \frac{\sigma^2}{36} \cdot \biggl\{ \biggl(\frac{\epsilon}{q(1-\epsilon)}\biggr)^2 \wedge \biggl(\frac{\epsilon}{q(1-\epsilon)}\biggr)^{2/r} \biggr\},
    \end{align*}
    where the final bound is obtained by considering separately the cases $\epsilon \leq q(1-\epsilon)$ and $\epsilon > q(1-\epsilon)$.  This proves the second term in the lower bound.

    For the first term in the lower bound, we observe that $\mathcal{P}_{\mathrm{b}}(\theta,\sigma) \subseteq \mathcal{P}_{L^r}(\theta,\sigma^2)$ for all $r \geq 2$.  We therefore obtain the desired conclusion by choosing the contamination distribution $Q\in\mathcal{P}(\mathbb{R}_{\star})$ such that $Q\bigl( \{\star\} \bigr)=1$ and applying Proposition~\ref{prop:univariate-mcar-lb}(b).

    \medskip(b) Define $P_1, P_2 \in \mathcal{P}(\mathbb{R})$ with Lebesgue densities $p_1, p_2$ respectively as in Lemma~\ref{lemma:psi-r-orlicz-norm-of-mixture}, so that $P_1 \in \mathcal{P}_{\psi_r}\bigl(\mathbb{E}_{P_1}(X_1),\sigma^2\bigr)$ and $P_2 \in \mathcal{P}_{\psi_r}\bigl(\mathbb{E}_{P_2}(X_2),\sigma^2\bigr)$.  Further, with $b > 0$ defined as in Lemma~\ref{lemma:psi-r-orlicz-norm-of-mixture}, define $R_1 \in \mathcal{P}(\mathbb{R}_{\star})$ through its Radon--Nikodym derivative
    \begin{align*}
        \frac{\mathrm{d}R_1}{\mathrm{d}\lambda_{\star}}(z) \coloneqq \begin{cases}
            q(1-\epsilon) \cdot p_1(z) \quad&\text{if }z\in(-\infty,b)\\
            \bigl\{q(1-\epsilon) + \epsilon \bigr\} \cdot p_1(z) \quad&\text{if }z\in[b,\infty)\\
            1- q(1-\epsilon) \cdot \int_{-\infty}^b p_1(x)\,\mathrm{d}x - \bigl\{q(1-\epsilon) + \epsilon \bigr\} \cdot \int_{b}^{\infty} p_1(x)\,\mathrm{d}x \quad&\text{if }z=\star,
        \end{cases}
    \end{align*}
    so that, by Proposition~\ref{prop:univariate-realisability}, $R_1 \in\mathcal{R}(P_1,\epsilon,q) \cap \mathcal{R}(P_2,\epsilon,q)$. Therefore, by \citet[][Theorem~4 and Lemma~5]{ma2024high},
    \begin{align}
        \mathcal{M}(\delta,\mathcal{P}_{\Theta},|\cdot|^2) \geq \frac{\bigl\{ \mathbb{E}_{P_2}(X_2) - \mathbb{E}_{P_1}(X_1) \bigr\}^2}{4}. \label{eq:minimax-quantile-psi-r-lb}
    \end{align}
    Now, writing $\sigma_0 \coloneqq \sigma/C_0$,
    \begin{align*}
        \mathbb{E}_{P_2}(X_2) &- \mathbb{E}_{P_1}(X_1) = \frac{\epsilon}{q(1-\epsilon)} \int_b^{\infty} xp_1(x)\,\mathrm{d}x - \frac{\epsilon}{q(1-\epsilon)+\epsilon} \int_0^b xp_1(x)\,\mathrm{d}x\\
        \overset{(i)}&{=} \frac{\epsilon}{q(1-\epsilon)} \biggl\{ be^{-(b/\sigma_0)^r} + \int_b^{\infty} e^{-(x/\sigma_0)^{r}} \,\mathrm{d}x \biggr\} \\
        &\qquad - \frac{\epsilon}{q(1-\epsilon)+\epsilon} \biggl\{ -be^{-(b/\sigma_0)^r} + \int_0^b e^{-(x/\sigma_0)^r} \,\mathrm{d}x \biggr\}\\
        \overset{(ii)}&{=} \biggl( \frac{\epsilon}{q(1-\epsilon)} + \frac{\epsilon}{q(1-\epsilon)+\epsilon} \biggr) \cdot \frac{q(1-\epsilon)}{2q(1-\epsilon) + \epsilon} \cdot \sigma_0\log^{1/r}\biggl( 2 + \frac{\epsilon}{q(1-\epsilon)} \biggr)\\
        &\qquad +  \frac{\epsilon}{q(1-\epsilon)} \int_b^{\infty} e^{-(x/\sigma_0)^{r}} \,\mathrm{d}x - \frac{\epsilon}{q(1-\epsilon)+\epsilon} \int_0^b e^{-(x/\sigma_0)^{r}}\,\mathrm{d}x, \numberthis \label{eq:difference-of-mean}
    \end{align*}
    where $(i)$ follows from integration by parts, $(ii)$ follows by substituting the definition of~$b$.
    Now let $h(t) \coloneqq \frac{1}{t} \int_0^t e^{-(x/\sigma_0)^{r}}\,\mathrm{d}x$. Then 
    \begin{align*}
        h'(t) = \frac{te^{-(t/\sigma_0)^r} - \int_0^t e^{-(x/\sigma_0)^{r}}\,\mathrm{d}x}{t^2} \leq 0,
    \end{align*}
    so $h$ is a decreasing function.   

    First consider the case where $\frac{\epsilon}{q(1-\epsilon)} \geq e^{2^r} - 2$ or equivalently $\epsilon \geq \frac{\{\exp(2^r)-2\}q}{1+\{\exp(2^r)-2\}q}$, so that $\log^{1/r}\bigl( 2+\frac{\epsilon}{q(1-\epsilon)} \bigr) \geq 2$ and
    \begin{align*}
        h(b) \leq h(2\sigma_0) = \frac{\int_0^2 e^{-x^r} \,\mathrm{d}x}{2} = \frac{\int_0^1 e^{-x^r} \,\mathrm{d}x + \int_1^2 e^{-x^r} \,\mathrm{d}x}{2} \leq \frac{1+e^{-1}}{2}.
    \end{align*}
    Hence, by~\eqref{eq:difference-of-mean},
    \begin{align*}
        \mathbb{E}_{P_2}(X_2) \! - \! \mathbb{E}_{P_1}(X_1) &\geq \biggl( \frac{\epsilon}{q(1\!-\!\epsilon)} + \frac{\epsilon}{q(1\!-\!\epsilon)\!+\!\epsilon} \biggr) \cdot \frac{q(1\!-\!\epsilon)}{2q(1\!-\!\epsilon) \!+\! \epsilon} \cdot \sigma_0\log^{1/r}\biggl( 2 \!+\! \frac{\epsilon}{q(1\!-\!\epsilon)} \biggr)\\
        &\qquad - \frac{\epsilon}{q(1-\epsilon)+\epsilon} \cdot \frac{1+e^{-1}}{2} \cdot \sigma_0\log^{1/r}\biggl( 2 + \frac{\epsilon}{q(1-\epsilon)} \biggr)\\
        &= \frac{1-e^{-1}}{2} \cdot \frac{\epsilon}{q(1-\epsilon)+\epsilon} \cdot \sigma_0\log^{1/r}\biggl( 2 + \frac{\epsilon}{q(1-\epsilon)} \biggr)\\
        &\geq \frac{1-e^{-1}}{8} \cdot \frac{\sigma}{C_0} \cdot \log^{1/r}\biggl( 2 + \frac{2\epsilon}{q(1-\epsilon)} \biggr).
    \end{align*}
    Therefore, by~\eqref{eq:minimax-quantile-psi-r-lb}, when $\epsilon \geq \frac{\{\exp(2^r)-2\}q}{1+\{\exp(2^r)-2\}q}$,
    \begin{align}
        \mathcal{M}(\delta,\mathcal{P}_{\Theta},|\cdot|^2) \gtrsim \sigma^2 \log^{2/r}\biggl( 2 + \frac{2\epsilon}{q(1-\epsilon)} \biggr). \label{eq:minimax-quantile-psi-r-lb-1}
    \end{align}
    Next consider the case where $\frac{\epsilon}{q(1-\epsilon)} \leq 1$, or equivalently $\epsilon \leq q/(1+q)$. Define $P_3,P_4 \in \mathcal{P}(\mathbb{R})$ by
    \begin{align*}
        P_3\biggl(\biggl\{ -\frac{\sigma}{4} \biggr\}\biggr) \coloneqq \frac{q(1-\epsilon)}{2q(1-\epsilon)+\epsilon} &\eqqcolon P_4\biggl(\biggl\{ \frac{\sigma}{4} \biggr\}\biggr), \quad \text{ and }\\
        P_3\biggl(\biggl\{ \frac{\sigma}{4} \biggr\}\biggr) \coloneqq \frac{q(1-\epsilon) + \epsilon}{2q(1-\epsilon)+\epsilon} &\eqqcolon P_4\biggl(\biggl\{ -\frac{\sigma}{4} \biggr\}\biggr).
    \end{align*}
    Thus $P_3\in\mathcal{P}_{\psi_r}\bigl(\mathbb{E}_{P_3}(X_3), \sigma^2\bigr)$ and $P_4\in\mathcal{P}_{\psi_r}\bigl(\mathbb{E}_{P_4}(X_4), \sigma^2\bigr)$. Further define $R_2 \in\mathcal{P}(\mathbb{R}_{\star})$ by
    \begin{align*}
        R_2\biggl(\biggl\{ -\frac{\sigma}{4} \biggr\}\biggr) &\coloneqq \frac{q(1-\epsilon)\{q(1-\epsilon) + \epsilon\}}{2q(1-\epsilon)+\epsilon} \eqqcolon R_2\biggl(\biggl\{\frac{\sigma}{4} \biggr\}\biggr) , \\
        R_2(\{\star\}) &\coloneqq 1- \frac{2q(1-\epsilon)\{q(1-\epsilon) + \epsilon\}}{2q(1-\epsilon)+\epsilon}.
    \end{align*}
    By Proposition~\ref{prop:univariate-realisability}, $R_2 \in \mathcal{R}(P_3,\epsilon,q) \cap \mathcal{R}(P_4,\epsilon,q)$. Therefore, by \citet[][Theorem~4 and Lemma~5]{ma2024high}, when $\epsilon \leq \frac{q}{1+q}$,
    \begin{align}
        \mathcal{M}(\delta,\mathcal{P}_{\Theta},|\cdot|^2) \geq \frac{\bigl\{ \mathbb{E}_{P_3}(X_3) - \mathbb{E}_{P_4}(X_4) \bigr\}^2}{4} \geq \frac{\sigma^2}{16}\biggl( \frac{\epsilon}{2q(1-\epsilon)+\epsilon} \biggr)^2 \geq \frac{\sigma^2}{144} \biggl( \frac{\epsilon}{q(1-\epsilon)} \biggr)^2. \label{eq:minimax-quantile-psi-r-lb-2}
    \end{align}
    Combining~\eqref{eq:minimax-quantile-psi-r-lb-1} and~\eqref{eq:minimax-quantile-psi-r-lb-2} yields that when $\epsilon \leq \frac{q}{1+q}$ or $\epsilon \geq \frac{\{\exp(2^r)-2\}q}{1+\{\exp(2^r)-2\}q}$,
    \begin{align}
        \mathcal{M}(\delta,\mathcal{P}_{\Theta},|\cdot|^2) \gtrsim \sigma^2\cdot \biggl\{ \biggl( \frac{\epsilon}{q(1-\epsilon)} \biggr)^2 \wedge \log^{2/r}\biggl( 2 + \frac{2\epsilon}{q(1-\epsilon)} \biggr) \biggr\}. \label{eq:minimax-quantile-psi-r-lb-3}
    \end{align}
    Further observe that $\mathcal{R}(P,\frac{q}{1+q},q) \subseteq \mathcal{R}(P,\epsilon,q)$ for all $P\in\mathcal{P}(\mathbb{R})$ when $\epsilon > \frac{q}{1+q}$. Thus, by~\eqref{eq:minimax-quantile-psi-r-lb-3}, we deduce that when $\frac{q}{1+q} < \epsilon < \frac{\{\exp(2^r)-2\}q}{1+\{\exp(2^r)-2\}q}$,
    \begin{align*}
        \mathcal{M}(\delta,\mathcal{P}_{\Theta},|\cdot|^2) &\gtrsim \sigma^2\cdot \biggl\{ \biggl( \frac{\frac{q}{1+q}}{q(1-\frac{q}{1+q})} \biggr)^2 \wedge \log^{2/r}\biggl( 2 + \frac{\frac{2q}{1+q}}{q(1-\frac{q}{1+q})} \biggr) \biggr\}\\
        &= \sigma^2 \gtrsim \sigma^2\cdot \biggl\{ \biggl( \frac{\epsilon}{q(1-\epsilon)} \biggr)^2 \wedge \log^{2/r}\biggl( 2 + \frac{2\epsilon}{q(1-\epsilon)} \biggr) \biggr\}, \numberthis \label{eq:minimax-quantile-psi-r-lb-4}
    \end{align*}
    where the last inequality follows from the fact that when $\epsilon < \frac{\{\exp(2^r)-2\}q}{1+\{\exp(2^r)-2\}q}$, 
    \begin{align*}
    \log^{2/r}\biggl( 2 + \frac{2\epsilon}{q(1-\epsilon)} \biggr) < \log^{2/r}\bigl(2+2(e^{2^r}-2)\bigr) &\leq \log^{2/r}\bigl(e^{2^r + \log 2}\bigr)\\
    &\leq (2 \cdot 2^r)^{2/r} = 2^{2/r}\cdot 2 \leq 8.
    \end{align*}
    Combining~\eqref{eq:minimax-quantile-psi-r-lb-3} and~\eqref{eq:minimax-quantile-psi-r-lb-4} yields the second term in the lower bound.

    For the first term in the lower bound, we observe that $\mathcal{P}_{\mathrm{b}}(\theta,\sigma/2) \subseteq \mathcal{P}_{\psi_r}(\theta,\sigma^2)$ for all $r \geq 1$.  We therefore obtain the desired conclusion by choosing the contamination distribution $Q\in\mathcal{P}(\mathbb{R}_{\star})$ such that $Q\bigl( \{\star\} \bigr)=1$ and applying Proposition~\ref{prop:univariate-mcar-lb}(b).
\end{proof}

\begin{lemma}\label{lemma:psi-r-orlicz-norm-of-mixture}
    Let $\epsilon \in [0,1)$, $q \in (0,1]$, $\sigma > 0$ and $r \geq 1$.  There exists a universal constant $C_0 > 0$ such that if $X_1 \sim P_{1} \in \mathcal{P}(\mathbb{R})$ and $X_2 \sim P_{2} \in \mathcal{P}(\mathbb{R})$ have Lebesgue densities $p_1$ and~$p_2$ respectively, where $p_1(x) \coloneqq \frac{rx^{r-1}}{(\sigma/C_0)^r}e^{-(C_0x/\sigma)^r}\mathbbm{1}_{\{x \geq 0\}}$ and 
    \begin{align*}
        p_2(x) \coloneqq \begin{cases}
            \frac{q(1-\epsilon)}{q(1-\epsilon) + \epsilon} \cdot p_1(x) \quad&\text{if }x<b\\
            \frac{q(1-\epsilon)+\epsilon}{q(1-\epsilon)} \cdot p_1(x) \quad&\text{if }x\geq b
        \end{cases}
        \quad\text{with}\quad b\coloneqq \frac{\sigma}{C_0}\log^{1/r}\biggl( 2 + \frac{\epsilon}{q(1-\epsilon)} \biggr),
    \end{align*}
    then $\|X_1 - \mathbb{E}X_1\|_{\psi_r} \vee \|X_2 - \mathbb{E}X_2\|_{\psi_r} \leq \sigma$.  
\end{lemma}
\begin{proof}
    Since $\mathbb{P}(|X_1| \geq x) = e^{-(C_0 x/\sigma)^{r}}$ for all $x\geq 0$, we have by \citet[Proposition 2.7.1]{vershynin2018high} that $\|X_1^r\|_{\psi_1} \leq C_1(\sigma/C_0)^r$ for some universal constant $C_1 > 0$, so $\|X_1\|_{\psi_r} \leq C_1^{1/r} \sigma/C_0 \leq (C_1 \vee 1)\sigma/C_0$. Then, by \citet[Lemma A.3]{gotze2021concentration}, we have 
    \[
    \|X_1-\mathbb{E}X_1\|_{\psi_r} \leq \biggl\{ 1+ \biggl(\frac{2}{(re)^{1/r}\log 2} \biggr)^{1/r} \biggr\}(C_1 \vee 1)\frac{\sigma}{C_0} \leq 4(C_1 \vee 1)\frac{\sigma}{C_0}.
    \]
    Turning to $X_2$, first observe that $p_2$ is a Lebesgue density, since
    \[
    \int_{\mathbb{R}} p_2(x) \, \mathrm{d}x = \frac{q(1-\epsilon)}{q(1-\epsilon) + \epsilon}\{1 - e^{-(C_0b/\sigma)^r}\} + \frac{q(1-\epsilon)+\epsilon}{q(1-\epsilon)} e^{-(C_0b/\sigma)^r} = 1.
    \]
    Now, for $x \geq 0$, we have
    \begin{align*}
        \mathbb{P}(X_2-b \geq x) = \frac{q(1-\epsilon)+\epsilon}{q(1-\epsilon)} \cdot \mathbb{P}(X_1 \geq b+x) &\leq \frac{q(1-\epsilon)+\epsilon}{q(1-\epsilon)} \cdot e^{-(C_0 b/\sigma)^r- (C_0 x/\sigma)^r} \\
        &= \frac{q(1-\epsilon)+\epsilon}{2q(1-\epsilon) + \epsilon} \cdot e^{-(C_0 x/\sigma)^r} \leq e^{-(C_0x/\sigma)^r}.
    \end{align*}
    Define $a \coloneqq \frac{q(1-\epsilon)}{q(1-\epsilon)+\epsilon} \in (0,1]$, so that $b = (\sigma/C_0) \log^{1/r}\bigl( \frac{1+a}{a} \bigr)$.
    For $x\in[0,b]$, we have 
    \begin{align*}
        \mathbb{P}(X_2-b \leq -x) &= a \cdot \mathbb{P}(X_1\leq b-x) \leq a \leq \frac{2a}{1+a} = 2e^{-(C_0b/\sigma)^r} \leq 2e^{-(C_0x/\sigma)^r}.
    \end{align*}
    For $x>b$, we have $\mathbb{P}(X_2-b \leq -x) = 0$.
    Combining these inequalities, we obtain $\mathbb{P}(|X_2-b| \geq x) \leq 3e^{-(C_0x/\sigma)^r}$ for all $x\geq 0$. Therefore, by~\citet[Proposition 2.7.1]{vershynin2018high}, we deduce\footnote{Note that in \citet[Proposition 2.7.1(a)]{vershynin2018high}, the condition is that $\mathbb{P}(|X|\geq t) \leq 2\exp(-t/K_1)$ for all $t \geq 0$. However, the result is still true if we replace the factor $2$ by $3$. See for example, \citet[Proposition 2.5.2]{vershynin2018high} for the proof strategy.} that $\|(X_2-b)^r\|_{\psi_1} \leq C_2(\sigma/C_0)^r$ for some universal constant $C_2>0$, so 
    \[
    \|X_2-b\|_{\psi_r} \leq C_2^{1/r}\frac{\sigma}{C_0} \leq (C_2 \vee 1)\frac{\sigma}{C_0}.
    \]
    By \citet[Lemma A.3]{gotze2021concentration} again, $\|X_2-\mathbb{E}X_2\|_{\psi_r} \leq 4(C_2 \vee 1)\sigma/C_0$. Finally, taking $C_0 \coloneqq 4C_1 \vee 4C_2 \vee 4$ completes the proof.
\end{proof}

\subsubsection{Proof of Theorem~\ref{thm:nonparametric-multivariate-realisable-mean-ub}}

The following proposition, which is analogous to Proposition~\ref{thm:bias-of-mean-one-dim-realisable-case} in the univariate case, will be used in the proof of Theorem~\ref{thm:nonparametric-multivariate-realisable-mean-ub}.
\begin{prop} \label{prop:bias-of-multivariate-mean}
    Let $\theta_0 \in \mathbb{R}^d$, $\Sigma \in \mathcal{S}_{++}^{d\times d}$, $\epsilon\in[0,1)$, $\delta\in(0,1]$, $\pi\in\mathcal{P}\bigl(\{\emptyset,[d]\}\bigr)$ and $q\coloneqq \pi([d])$.
    \begin{itemize}
        \item[(a)] Let $r\geq 2$, $P\in \mathcal{P}_{L^r}(\theta_0, \Sigma)$ and $Z \sim R \in \mathcal{R}_{\emptyset,[d]}(P, \epsilon, \pi)$. Then
        \begin{align*}
            \bigl\| \mathbb{E}(Z\,|\,Z\in\mathbb{R}^d) - \theta_0 \bigr\|_2^2 \leq \|\Sigma\|_{\mathrm{op}} \biggl\{ \biggl(\frac{\epsilon}{q(1-\epsilon)}\biggr)^2 \wedge \biggl(\frac{\epsilon}{q(1-\epsilon)}\biggr)^{2/r} \biggr\}.
        \end{align*}
        \item[(b)] Let $r\geq 1$, $P\in \mathcal{P}_{\psi_r}(\theta_0, \Sigma)$, and $Z \sim R \in \mathcal{R}_{\emptyset,[d]}(P, \epsilon, \pi)$. Then
        \begin{align*}
            \bigl\| \mathbb{E}(Z\,|\,Z\in\mathbb{R}^d) - \theta_0 \bigr\|_2^2 \leq \|\Sigma\|_{\mathrm{op}} \biggl\{ 4\biggl(\frac{\epsilon}{q(1-\epsilon)}\biggr)^2 \;\wedge\;  \log^{2/r} \biggl( 2 + \frac{2\epsilon}{q(1-\epsilon)} \biggr) \biggr\}.
        \end{align*}
    \end{itemize}
\end{prop}
\begin{proof}
    Let $\kappa \coloneqq \frac{\epsilon}{q(1-\epsilon)}$, $X\sim P$, $v \in \mathbb{S}^{d-1}$, $Z^{(v)} \coloneqq v^\top Z \cdot \mathbbm{1}_{\{Z\in\mathbb{R}^d\}} + \star \cdot \mathbbm{1}_{\{Z\notin\mathbb{R}^d\}}$, $R^{(v)} \coloneqq \mathsf{Law}(Z^{(v)})$ and $P^{(v)} \coloneqq \mathsf{Law}(v^\top X)$. By Lemma~\ref{lemma:realisability-of-projection}, we have $R^{(v)} \in \mathcal{R}(P^{(v)},\epsilon,q)$.

    (a) Since $P^{(v)} \in \mathcal{P}_{L^r}(v^\top\theta_0, v^\top\Sigma v)$ we have by Proposition~\ref{thm:bias-of-mean-one-dim-realisable-case}(a) that
    \begin{align*}
        \bigl\| \mathbb{E}(Z\,|\,Z\in\mathbb{R}^d) - \theta_0 \bigr\|_2^2 &= \sup_{v\in\mathbb{S}^{d-1}} \bigl\{ v^\top \mathbb{E}(Z | Z\in\mathbb{R}^d) - v^\top\theta_0 \bigr\}^2 \\
        &= \sup_{v\in\mathbb{S}^{d-1}} \bigl\{ \mathbb{E}(Z^{(v)} | Z^{(v)} \neq \star) - v^\top\theta_0 \bigr\}^2\\
        &\leq \sup_{v\in\mathbb{S}^{d-1}} v^\top\Sigma v \cdot (\kappa^2 \wedge \kappa^{2/r})\\
        &= \|\Sigma\|_{\mathrm{op}} (\kappa^2 \wedge \kappa^{2/r}),
    \end{align*}
    as required.
    
    (b) We now have $P^{(v)} \in \mathcal{P}_{\psi_r}(v^\top\theta_0, v^\top\Sigma v)$, so the proof is the same as part (a), except that we use Proposition~\ref{thm:bias-of-mean-one-dim-realisable-case}(b) instead.
\end{proof}

\begin{proof}[Proof of Theorem~\ref{thm:nonparametric-multivariate-realisable-mean-ub}]
    Let $\kappa \coloneqq \frac{\epsilon}{q(1-\epsilon)}$.

    (a) For $v \in \mathbb{S}^{d-1}$, we have by the same argument as the proof of~\eqref{Eq:CondVarBound} that
    \begin{align*}
        \Var(v^\top Z_1 \,|\, Z_1 \in \mathbb{R}^d) \leq (1 + \kappa^{2/r})v^\top\Sigma v.
    \end{align*}
    Therefore, writing $\Gamma \coloneqq \Cov(Z_1 \,|\, Z_1 \in \mathbb{R}^d) \in \mathcal{S}_{+}^{d \times d}$, we have 
    \begin{align}
        \| \Gamma\|_{\mathrm{op}} \leq (1 + \kappa^{2/r}) \|\Sigma\|_{\mathrm{op}} 
        \quad\text{and}\quad
        \tr(\Gamma) \leq (1 + \kappa^{2/r})\tr(\Sigma). \label{eq:trace-bound}
    \end{align}
    By Lemma~\ref{lemma:binomial-tail}(b), the event $\mathcal{E}_0 \coloneqq \{|\mathcal{D}| \geq nq(1-\epsilon)/2\}$ has $\mathbb{P}(\mathcal{E}_0) \geq 1-\delta/2$, since $nq(1-\epsilon)\geq 8\log(2/\delta)$.  Moreover, by~\eqref{eq:assumption-on-alg-heavy-tail-multivariate} and~\eqref{eq:trace-bound},
    \begin{align*}
        \mathbb{P}\biggl\{\bigl\|\hat{\theta}_n - \mathbb{E}(Z_1|Z_1\in\mathbb{R}^d)\bigr\|_2^2 > 2C(1+\kappa^{2/r}) \biggl(\frac{\tr(\Sigma)}{nq(1-\epsilon)} + \frac{\|\Sigma\|_{\mathrm{op}}\log(2/\delta)}{nq(1-\epsilon)}\biggr) \,\bigg|\, \mathcal{E}_0 \biggr\} \leq \frac{\delta}{2}.
    \end{align*}
    Thus, 
    \begin{align*}
        &\mathbb{P}\biggl\{\bigl\|\hat{\theta}_n - \mathbb{E}(Z_1|Z_1\in\mathbb{R}^d)\bigr\|_2^2 \leq 2C(1+\kappa^{2/r}) \biggl(\frac{\tr(\Sigma)}{nq(1-\epsilon)} + \frac{\|\Sigma\|_{\mathrm{op}}\log(2/\delta)}{nq(1-\epsilon)}\biggr) \biggr\}\\
        &\geq \mathbb{P}\biggl\{\bigl\|\hat{\theta}_n \!-\! \mathbb{E}(Z_1|Z_1\in\mathbb{R}^d)\bigr\|_2^2 \leq 2C(1\!+\!\kappa^{2/r}) \biggl(\frac{\tr(\Sigma)}{nq(1-\epsilon)} + \frac{\|\Sigma\|_{\mathrm{op}}\log(2/\delta)}{nq(1-\epsilon)}\biggr) \,\bigg|\, \mathcal{E}_0 \biggr\}\mathbb{P}(\mathcal{E}_0)\\
        &\geq 1-\delta.
    \end{align*}
    On the event that $\bigl\{\bigl\|\hat{\theta}_n - \mathbb{E}(Z_1|Z_1\in\mathbb{R}^d)\bigr\|_2^2 \leq 2C(1+\kappa^{2/r}) \bigl(\frac{\tr(\Sigma)}{nq(1-\epsilon)} + \frac{\|\Sigma\|_{\mathrm{op}}\log(2/\delta)}{nq(1-\epsilon)}\bigr)\bigr\}$, we have by Proposition~\ref{prop:bias-of-multivariate-mean}(a) that
    \begin{align*}
        \|\hat{\theta}_n &- \theta_0\|_2^2 \leq 2\bigl\|\hat{\theta}_n - \mathbb{E}(Z_1\,|\,Z_1\in\mathbb{R}^d)\bigr\|_2^2 + 2\bigl\|\mathbb{E}(Z_1\,|\,Z_1\in\mathbb{R}^d) - \theta_0\bigr\|_2^2\\
        &\leq 4C\frac{\tr(\Sigma) + \|\Sigma\|_{\mathrm{op}}\log(2/\delta)}{nq(1-\epsilon)} + 4C\frac{\mathbf{r}(\Sigma) + \log(2/\delta)}{nq(1-\epsilon)} \|\Sigma\|_{\mathrm{op}}\kappa^{2/r} \!+\! 2\|\Sigma\|_{\mathrm{op}} (\kappa^2 \wedge \kappa^{2/r})\\
        &\leq 8C\frac{\tr(\Sigma) + \|\Sigma\|_{\mathrm{op}}\log(2/\delta)}{nq(1-\epsilon)} + (5C+2)\|\Sigma\|_{\mathrm{op}} (\kappa^2 \wedge \kappa^{2/r}),
    \end{align*}
    where the final inequality follows by considering separately the cases $\kappa \leq 1$ and $\kappa > 1$, and in the second case noting that $\mathbf{r}(\Sigma) \leq nq(1-\epsilon)$ and $\log(2/\delta) \leq nq(1-\epsilon)/8$.
    
    \medskip(b) For $v \in \mathbb{S}^{d-1}$, we have by the same argument as in the proof of~\eqref{eq:sub-exponential-norm-bound} that conditional on $\{Z_1 \in \mathbb{R}^d\}$,
    \begin{align*}
        \bigl\| v^\top Z_1 - \mathbb{E}(v^\top Z_1 \,|\, Z_1 \in \mathbb{R}^d) \bigr\|_{\psi_1} \leq \sqrt{(1+\kappa)v^\top\Sigma v},
    \end{align*}
    so that $Z_1 \,|\, \{Z_1 \in \mathbb{R}^d\} \in \mathcal{P}_{d,\psi_1}\bigl( \mathbb{E}(v^\top Z_1 \,|\, Z_1 \in \mathbb{R}^d), (1+\kappa)\Sigma \bigr)$. By Lemma~\ref{lemma:binomial-tail}(b), the event $\mathcal{E}_0 \coloneqq \{|\mathcal{D}| \geq nq(1-\epsilon)/2\}$ satisfies $\mathbb{P}(\mathcal{E}_0) \geq 1-\delta/4$ since $nq(1-\epsilon) \geq 8\log(8/\delta)$. Moreover, writing
    \begin{align*}
        \mathcal{E}_2 \coloneqq \biggl\{ \bigl\|\hat{\theta}_n - \mathbb{E}(Z_1\,|\,Z_1\in\mathbb{R}^d)\bigr\|_2^2 \leq 48(1+\kappa) \cdot \frac{\tr(\Sigma) + \|\Sigma\|_{\mathrm{op}} \log(8/\delta)}{nq(1-\epsilon)}\biggr\},
    \end{align*}
    we have by Lemma~\ref{lemma:concentration-of-sample-mean-sub-exponential-vector} (a consequence of the PAC--Bayes lemma) that $\mathbb{P}\bigl(\mathcal{E}_2 \,\big|\, |\mathcal{D}| = s\bigr) \geq 1-\delta/4$ for $s\geq nq(1-\epsilon)/2$, so $\mathbb{P}(\mathcal{E}_0 \cap \mathcal{E}_2) \geq 1-\delta/2$. On $\mathcal{E}_0 \cap \mathcal{E}_2$, we have by Proposition~\ref{prop:bias-of-multivariate-mean}(b) that
    \begin{align*}
        \|\hat{\theta}_n - \theta_0\|_2^2 &\leq 2\bigl\|\hat{\theta}_n - \mathbb{E}(Z_1\,|\,Z_1\in\mathbb{R}^d)\bigr\|_2^2 + 2\bigl\|\mathbb{E}(Z_1\,|\,Z_1\in\mathbb{R}^d) - \theta_0\bigr\|_2^2\\
        &\lesssim \frac{\tr(\Sigma) + \|\Sigma\|_{\mathrm{op}} \log(8/\delta)}{nq(1-\epsilon)} + \|\Sigma\|_{\mathrm{op}}\kappa\cdot \frac{\mathbf{r}(\Sigma) + \log(8/\delta)}{nq(1-\epsilon)} + \|\Sigma\|_{\mathrm{op}}\kappa^2\\
        &\lesssim \frac{\tr(\Sigma) + \|\Sigma\|_{\mathrm{op}} \log(8/\delta)}{nq(1-\epsilon)} + \|\Sigma\|_{\mathrm{op}}\kappa^2, \numberthis \label{eq:E0-cap-E2}
    \end{align*}
    where the final inequality follows by considering separately the cases $\kappa \leq 1$ and $\kappa > 1$, and in the second case noting that $\mathbf{r}(\Sigma) \leq nq(1-\epsilon)$ and $\log(8/\delta) \leq nq(1-\epsilon)/8$.
    
    For the last term in the upper bound, we first consider the case where $r>1$. For $w \in \mathbb{R}^d$, we have by the same argument as in the proof of~\eqref{eq:log-mgf-bound} that
    \begin{align}
        \log\mathbb{E}\bigl\{ \exp\bigl(\lambda w^\top(Z_1 -\theta_0)\bigr) \bigm| Z_1 \in \mathbb{R}^d \bigr\} \leq \log(1+\kappa^2) + \log 2 + \bigl(2\lambda \sqrt{w^\top\Sigma w}\bigr)^{r/(r-1)},  \label{eq:log-mgf-f-lambda}
    \end{align}
    for all $\lambda>0$. Let $\beta\coloneqq \mathbf{r}(\Sigma)$, let $\mu$ denote the distribution of $\mathsf{N}_d(0,\beta^{-1}\Sigma)$ and for $u\in\Sigma^{1/2}\mathbb{S}^{d-1}$, let $\rho_u$ denote the conditional distribution of $Y$ given $\bigl\{\|Y-u\|_2 \leq 2\|\Sigma\|_{\mathrm{op}}^{1/2}\bigr\}$, where $Y\sim\mathsf{N}_d(u,\beta^{-1}\Sigma)$.  By Chebychev's inequality, 
    \[
    \mathbb{P}\bigl(\|Y-u\|_2 \geq 2\|\Sigma\|_{\mathrm{op}}^{1/2}\bigr) \leq \frac{\tr(\Sigma)}{4\beta \|\Sigma\|_{\mathrm{op}}} = \frac{1}{4}.
    \]
    Hence, by the third displayed equation of \citet[][p.~11]{zhivotovskiy2024dimension}, we have 
    \begin{align*}
        \mathrm{KL}(\rho_u,\mu) = \log \biggl( \frac{1}{\mathbb{P}\bigl(\|Y-u\|_2 \leq 2\|\Sigma\|_{\mathrm{op}}^{1/2}\bigr)} \biggr) + \frac{\beta}{2} \leq 2\log 2 + \frac{\mathbf{r}(\Sigma)}{2}.
    \end{align*}
    Fix $u \in \Sigma^{1/2}\mathbb{S}^{d-1}$, let $v\in\mathbb{R}^d$ be such that $\|v-u\|_2 \leq 2\|\Sigma\|_{\mathrm{op}}^{1/2}$, and for $\lambda>0$, define $f_{\lambda}:\mathbb{R}^d\times\mathbb{R}^d \to \mathbb{R}$ by $f_{\lambda}(x,y) \coloneqq \lambda y^\top\Sigma^{-1/2}(x-\theta_0)$. Then, since $\|v\|_2 \leq 3\|\Sigma\|_{\mathrm{op}}^{1/2}$, we have by~\eqref{eq:log-mgf-f-lambda} that 
    \begin{align*}
        \log \mathbb{E}_{Z\sim R}\bigl(e^{f_\lambda(Z,v)} \bigm| Z\in\mathbb{R}^d \bigr) \leq \log(2+2\kappa^2) + \bigl(6\lambda \|\Sigma\|_{\mathrm{op}}^{1/2}\bigr)^{r/(r-1)},
    \end{align*}
    so $\mathbb{E}_{\xi_u \sim \rho_u} \bigl\{\log \mathbb{E}_{Z\sim R}\bigl(e^{f_{\lambda}(Z,\xi_u)} \bigm| Z\in\mathbb{R}^d \bigr) \bigr\} \leq \log(2+2\kappa^2) + \bigl(6\lambda \|\Sigma\|_{\mathrm{op}}^{1/2}\bigr)^{r/(r-1)}$.
    Therefore, for $s\geq nq(1-\epsilon)/2$, by the PAC--Bayes lemma (Lemma~\ref{lemma:PAC-Bayes}), conditional on $|\mathcal{D}| = s$, we have with probability at least $1-\delta/4$ that
    \begin{align*}
        \biggl\| \frac{1}{|\mathcal{D}|}\sum_{i\in\mathcal{D}} Z_i &- \theta_0 \biggr\|_2 = \sup_{u\in\Sigma^{1/2}\mathbb{S}^{d-1}} \frac{1}{\lambda|\mathcal{D}|}\sum_{i\in\mathcal{D}} \mathbb{E}_{\xi_u \sim \rho_u} f_{\lambda}(Z_i, \xi_u)\\
        &\leq \inf_{\lambda>0} \biggl\{ \frac{\log(2+2\kappa^2)}{\lambda} + \bigl(6\|\Sigma\|_{\mathrm{op}}^{1/2}\bigr)^{r/(r-1)} \lambda^{1/(r-1)} + \frac{\mathbf{r}(\Sigma)/2 + 2\log(4/\delta)}{s\lambda} \biggr\}\\
        \overset{(i)}&{\leq} 12\|\Sigma\|_{\mathrm{op}}^{1/2} \biggl\{ \log(2+2\kappa^2) + \frac{\mathbf{r}(\Sigma)/2 + 2\log(4/\delta)}{s} \biggr\}^{1/r}\\
        \overset{(ii)}&{\leq} 12\|\Sigma\|_{\mathrm{op}}^{1/2} \bigl\{ \log(2+2\kappa^2) + 2\bigr\}^{1/r} \lesssim \|\Sigma\|_{\mathrm{op}}^{1/2} \log^{1/r}(2+2\kappa),
    \end{align*}
    where $(i)$ follows by choosing $\lambda = \frac{1}{6\|\Sigma\|_{\mathrm{op}}^{1/2}} \bigl\{ \log(2+2\kappa^2) + \frac{\mathbf{r}(\Sigma)/2 + 2\log(4/\delta)}{s} \bigr\}^{(r-1)/r}$ and $(ii)$ follows from the assumptions that $nq(1-\epsilon) \geq \mathbf{r}(\Sigma)$ and $\delta \geq 8\exp\bigl( -nq(1-\epsilon)/8 \bigr)$. Hence, there exists a universal constant $C_1 > 0$ such that the event 
    \begin{align}
        \mathcal{E}_3 \coloneqq \biggl\{ \biggl\| \frac{1}{|\mathcal{D}|}\sum_{i\in\mathcal{D}} Z_i - \theta_0 \biggr\|_2^2 \leq C_1\|\Sigma\|_{\mathrm{op}} \log^{2/r}(2+2\kappa) \biggr\}, \label{eq:def-E3}
    \end{align}
    satisfies $\mathbb{P}(\mathcal{E}_0 \cap \mathcal{E}_3) \geq 1-\delta/2$. Thus, on the event $\mathcal{E}_0 \cap \mathcal{E}_2 \cap \mathcal{E}_3$, which has probability at least $1-\delta$, we combine~\eqref{eq:E0-cap-E2} and~\eqref{eq:def-E3} to obtain the desired result for $r>1$.

    Finally, we consider the case where $r=1$. For $w \in \mathbb{R}^d$, we have by the same argument as the proof of~\eqref{eq:log-mgf-bound-r=1} that
    \begin{align*}
        \log\mathbb{E}\bigl\{ \exp\bigl(\lambda w^\top(Z_1 -\theta_0)\bigr) \, \big| \, Z_1 \in \mathbb{R}^d \bigr\} \leq \log(1+\kappa^2) + \log 2 + \bigl(2\lambda \sqrt{w^\top\Sigma w}\bigr)^2, 
    \end{align*}
    for $|\lambda| \leq \frac{1}{2}\|\Sigma\|_{\mathrm{op}}^{-1/2} \leq \frac{1}{2\sqrt{w^\top\Sigma w}}$. Hence, for $s\geq nq(1-\epsilon)/2$, by following the same proof as the $r>1$ case above, we deduce that, conditional on $|\mathcal{D}|=s$, we have with probability at least $1-\delta/4$ that
    \begin{align*}
        \biggl\| \frac{1}{|\mathcal{D}|}\sum_{i\in\mathcal{D}} Z_i \!-\! \theta_0 \biggr\|_2
        &\leq \inf_{\lambda \in (0,\frac{1}{2}\|\Sigma\|_{\mathrm{op}}^{-1/2}]} \biggl\{ \frac{\log(2\!+\!2\kappa^2)}{\lambda} \!+\! \bigl(6\|\Sigma\|_{\mathrm{op}}^{1/2}\bigr)^2 \lambda \!+\! \frac{\mathbf{r}(\Sigma)/2 \!+\! 2\log(4/\delta)}{s\lambda} \biggr\}\\
        \overset{(i)}&{\leq} 2\|\Sigma\|_{\mathrm{op}}^{1/2} \biggl\{ \log(2+2\kappa^2) + 9 + \frac{\mathbf{r}(\Sigma)/2 + 2\log(4/\delta)}{s} \biggr\}\\
        \overset{(ii)}&{\leq} 2\|\Sigma\|_{\mathrm{op}}^{1/2} \bigl\{ \log(2+2\kappa^2) + 11 \bigr\} \lesssim \|\Sigma\|_{\mathrm{op}}^{1/2} \log(2+2\kappa),
    \end{align*}
    where $(i)$ follows by choosing $\lambda = \frac{1}{2}\|\Sigma\|_{\mathrm{op}}^{-1/2}$, and $(ii)$ follows from the assumptions that $nq(1-\epsilon) \geq \mathbf{r}(\Sigma)$ and $\delta \geq 8\exp\bigl( -nq(1-\epsilon)/8 \bigr)$. Hence, there exists a universal constant $C_2 > 0$ such that the event 
    \begin{align}
        \mathcal{E}_4 \coloneqq \biggl\{ \biggl\| \frac{1}{|\mathcal{D}|}\sum_{i\in\mathcal{D}} Z_i - \theta_0 \biggr\|_2^2 \leq C_2 \|\Sigma\|_{\mathrm{op}} \log^2(2+2\kappa) \biggr\}, \label{eq:def-E4}
    \end{align}
    satisfies $\mathbb{P}(\mathcal{E}_0 \cap \mathcal{E}_4) \geq 1 - \delta/2$.  Thus, on the event $\mathcal{E}_0 \cap \mathcal{E}_2 \cap \mathcal{E}_4$, which has probability at least $1-\delta$, we combine~\eqref{eq:E0-cap-E2} and~\eqref{eq:def-E4} to obtain the desired result for $r=1$.
\end{proof}

\section{Proofs from Section~\ref{sec:regression-missing-response}}

\subsection{Proof of Lemma~\ref{lemma:beta-gamma-regular}}
\begin{proof}[Proof of Lemma~\ref{lemma:beta-gamma-regular}]
(a) Let $(v_m)$ be a sequence in $\mathbb{S}^{d-1}$ with 
\[
\mathbb{P}\bigl(|X_1^\top v_m| > \gamma\bigr) \searrow \inf_{v \in \mathbb{S}^{d-1}} \mathbb{P}\bigl(|X_1^\top v| > \gamma\bigr)
\]
as $m \rightarrow \infty$.  Then by compactness of $\mathbb{S}^{d-1}$, there exists a subsequence $(v_{m_k})$, as well as $v_* \in \mathbb{S}^{d-1}$, for which $v_{m_k} \rightarrow v_*$ as $k \rightarrow \infty$.  But then $|X_1^\top v_{m_k}| \stackrel{d}{\rightarrow} |X_1^\top v_*|$ as $k \rightarrow \infty$, so by, e.g., \citet[][Lemma~2.2]{van1998asymptotic},
\[
\mathbb{P}\bigl(|X_1^\top v_*| > \gamma\bigr) \leq \liminf_{k \rightarrow \infty} \mathbb{P}\bigl(|X_1^\top v_{m_k}| > \gamma\bigr) = \inf_{v \in \mathbb{S}^{d-1}} \mathbb{P}\bigl(|X_1^\top v| > \gamma\bigr).
\]
It follows that the infimum in the definition of $\beta$ is attained.  

If $\beta = 0$ for all $\gamma > 0$, then for every $\gamma > 0$ we can find $v_*(\gamma) \in \mathbb{S}^{d-1}$ with $\mathbb{P}\bigl(|X_1^\top v_*(\gamma)| > \gamma\bigr) = 0$.  Writing $v_m \coloneqq v_*(1/m)$, there exist integers $1 \leq m_1 < m_2 < \ldots$ and $v_{**} \in \mathbb{S}^{d-1}$ with $v_{m_k} \rightarrow v_{**}$ as $k \rightarrow \infty$.  Since $|X_1^\top v_{m_k}| - 1/m_k \stackrel{d}{\rightarrow} |X_1^\top v_{**}|$ as $k \rightarrow \infty$ we have by \citet[][Lemma~2.2]{van1998asymptotic} again that
\[
\mathbb{P}\bigl(|X_1^\top v_{**}| > 0\bigr) \leq \liminf_{k \rightarrow \infty} \mathbb{P}\biggl(|X_1^\top v_{m_k}| > \frac{1}{m_k}\biggr) = 0.
\]
But then, defining the hyperplane $H \coloneqq \{x \in \mathbb{R}^d:x^\top v_{**} = 0\}$, we have $P(H) = 1$.

\medskip

(b) The claim is equivalent to showing that there exists a universal constant $c>0$ such that if $\frac{d + \log(1/\delta)}{n} \leq c\beta^2$, then,  with probability at least $1-\delta$,
\begin{align*}
    \sup_{v\in\mathbb{S}^{d-1}} -\frac{1}{n} \sum_{i=1}^n \mathbbm{1}_{\{|X_i^\top v| > \gamma\}} \leq -2\beta.
\end{align*}
To establish this, let $\mathcal{H} \coloneqq \{x\mapsto -\mathbbm{1}_{\{|x^\top v| > \gamma\}} : v\in \mathbb{S}^{d-1}\}$.  Then 
\begin{align*}
    \sup_{v\in\mathbb{S}^{d-1}} -\frac{1}{n} \sum_{i=1}^n \mathbbm{1}_{\{|X_i^\top v| > \gamma\}} + 3\beta &\leq \sup_{v\in\mathbb{S}^{d-1}} \frac{1}{n} \sum_{i=1}^n \bigl\{-\mathbbm{1}_{\{|X_i^\top v| > \gamma\}} + \mathbb{P}(|X_i^\top v| > \gamma) \bigr\}\\
    &= \sup_{h\in\mathcal{H}} \frac{1}{n}\sum_{i=1}^n \bigl\{h(X_i) - \mathbb{E}h(X_i)\bigr\}\eqqcolon V, \numberthis \label{eq:beta-gamma-regular-bound-1}
\end{align*}
where the first inequality follows since $\mathbb{P}(|X_i^\top v| > \gamma) \geq 3\beta$ for all $v\in\mathbb{S}^{d-1}$. By the bounded differences inequality~\citep[e.g.,][Theorem 6.2]{boucheron2003concentration}, with probability at least $1 - \delta$,
\begin{align} \label{eq:beta-gamma-regular-bounded-diff}
     V \leq \mathbb{E}(V) + \sqrt{\frac{\log(1/\delta)}{2n}}.
\end{align}
For a collection $\mathcal{H}_1$ of binary-valued functions, we let $\mathrm{VC}(\mathcal{H}_1)$ denote its Vapnik--Chervonenkis dimension.  For $v \in \mathbb{R}^d$ and $b \in \mathbb{R}$, define $g_{v,b}:\mathbb{R}^d \rightarrow \mathbb{R}$ by $g_{v,b}(x) \coloneqq x^\top v + b$, and define the vector space $\mathcal{G} \coloneqq \{g_{v,b}: v\in\mathbb{R}^d,b\in\mathbb{R}\}$.  Now let $\mathcal{H'} \coloneqq \{x\mapsto -\mathbbm{1}_{\{g(x)>0\}} : g\in \mathcal{G}\}$, which by \citet[Proposition~4.20]{wainwright2019high} satisfies $\mathrm{VC}(\mathcal{H}') \leq \mathrm{dim}(\mathcal{G}) = d+1$. Then
\begin{align*}
    \mathcal{H} = \bigl\{x\mapsto -\mathbbm{1}_{\{g_{v,-\gamma}(x) > 0\} \cup \{g_{-v,-\gamma}(x) > 0\}}& : v\in\mathbb{S}^{d-1}\bigr\} \\
    &\subseteq \bigl\{x\mapsto -\mathbbm{1}_{\{g_1(x)>0\} \cup \{g_2(x)>0\}} : g_1,g_2\in\mathcal{G} \bigr\}.
\end{align*}
Hence, by \citet[Lemma~3.2.3]{blumer1989learnability}, we have $\mathrm{VC}(\mathcal{H}) \leq 4\log_2(6)\mathrm{VC}(\mathcal{H}') \leq 11d+11$.
We deduce by \citet[Theorem~8.3.23]{vershynin2018high} that there exists a universal constant $C_1 > 0$ such that $\mathbb{E}(V) \leq C_1\sqrt{\frac{d+1}{n}}$.  Thus, by~\eqref{eq:beta-gamma-regular-bound-1} and~\eqref{eq:beta-gamma-regular-bounded-diff} we conclude that there exists a universal constant $C_2>0$ such that with probability at least $1-\delta$,
\begin{align*}
    \sup_{v\in\mathbb{S}^{d-1}} -\frac{1}{n} \sum_{i=1}^n \mathbbm{1}_{\{|X_i^\top v| > \gamma\}} + 3\beta \leq \mathbb{E}(V) + \sqrt{\frac{\log(1/\delta)}{2n}} \leq C_2\sqrt{\frac{d+\log(1/\delta)}{n}} \leq \beta,
\end{align*}
where the final inequality follows by choosing $c \coloneqq 1/C_2^2$ and using the assumption that $\frac{d+\log(1/\delta)}{n} \leq c\beta^2$. This proves the claim.
\end{proof}

\subsection{Proof of Theorem~\ref{thm:gaussian-realisable-response}} \label{sec:proofs-regression-missing-response}
We begin with some preliminary lemmas. 
\begin{lemma}\label{lemma:uniform-dkw}
    Consider the setting of Theorem~\ref{thm:gaussian-realisable-response}. There exists a universal constant $C>0$ such that for $\delta\in(0,1]$, we have with probability at least $1-\delta$ conditional on $X_1=x_1,\ldots,X_n=x_n$ that
    \begin{align*}
        \sup_{\theta\in\mathbb{R}^d} d_{\mathrm{K}}^{\mathrm{sym}}(\hat{R}_{n,\theta}, R_{n,\theta}) \leq C\sqrt{\frac{d+\log(1/\delta)}{n}}.
    \end{align*}
\end{lemma}
\begin{proof}
First define
    \begin{align*}
        V \coloneqq \sup_{\theta\in\mathbb{R}^d} d_{\mathrm{K}}^{\mathrm{sym}}(\hat{R}_{n,\theta}, R_{n,\theta}) = \sup_{\theta\in\mathbb{R}^d} \sup_{A \in \mathcal{A}^{\mathrm{sym}}} \biggl| \frac{1}{n} \sum_{i=1}^n \mathbbm{1}_{\{Z_i - x_i^\top \theta \in A\}} - \frac{1}{n} \sum_{i=1}^n \tilde{R}_{i, \theta}(A) \biggr|.
    \end{align*}
    Then, by the bounded differences inequality~\citep[e.g.,][Theorem 6.2]{boucheron2003concentration}, with probability at least $1 - \delta$,
    \begin{align} \label{eq:uniform-dkw-bounded-diff}
        V \leq \mathbb{E}(V) + \sqrt{\frac{\log(1/\delta)}{2n}}.
    \end{align}
    Now define
    \begin{align*}
        \mathcal{G} \!\coloneqq \!\Bigl\{ g:\mathbb{R}^d \times \mathbb{R}_{\star} \to \mathbb{R} \text{ s.t. } g(x,z) \!=\! (z \!-\! x^\top\theta \!-\! t)\mathbbm{1}_{\{z\neq \star\}}\!\! +  \!\!\mathbbm{1}_{\{z= \star\}} \text{ for some } \theta\in\mathbb{R}^d,\, t\in\mathbb{R} \Bigr\},
    \end{align*}
    and define $\mathcal{H}_+ \coloneqq \bigl\{ (x,z) \mapsto \mathbbm{1}_{\{g(x,z) \leq 0\}} : g \in \mathcal{G}\bigr\}$ and $\mathcal{H}_- \coloneqq \bigl\{ (x,z) \mapsto \mathbbm{1}_{\{g(x,z) \leq 0\}} : g \in -\mathcal{G}\bigr\}$. Then 
    \begin{align*}
        V = \sup_{h \in \mathcal{H}_+ \cup \mathcal{H}_-} \biggl| \frac{1}{n} \sum_{i=1}^n \bigl\{h(x_i,Z_i) - \mathbb{E} h(x_i,Z_i)\bigr\} \biggr|.
    \end{align*}
    Since $\mathcal{G}$ is a vector space of functions with $\dim(\mathcal{G}) = d+1$, by \citet[Exercise~3.24(b)]{mohri2018foundations} and \citet[Proposition~4.20]{wainwright2019high}, we deduce that $\mathrm{VC}(\mathcal{H}_+ \cup \mathcal{H}_-) \leq \mathrm{VC}(\mathcal{H}_+) + \mathrm{VC}(\mathcal{H}_-) + 1 \leq 2\dim(\mathcal{G}) + 1 \leq 2d+3$. Therefore, applying~\citet[Theorem~8.3.23]{vershynin2018high} yields that $\mathbb{E}(V) \leq C'\sqrt{\frac{2d+3}{n}}$ for some universal constant $C'>0$. Combining this with~\eqref{eq:uniform-dkw-bounded-diff} proves the desired result.
\end{proof}

\begin{lemma} \label{lem:fixed-KS-linear-regression}
    Consider the setting of Theorem~\ref{thm:gaussian-realisable-response}, and assume that $\theta \neq \theta_0$.  Then, writing $a\coloneqq \frac{1}{2}\|\theta_0 - \theta\|_2 > 0$ and $b \coloneqq \frac{1}{2}\log\bigl( 1+ \frac{4(1-\beta q(1-\epsilon))}{\beta q(1-\epsilon)} \bigr)$, we have
    \begin{align*}
        d_{\mathrm{K}}^{\mathrm{sym}}\bigl(R_{n,\theta},\mathcal{R}_0^{\mathrm{Lin}}\bigr) \geq \beta q(1-\epsilon)\Phi\biggl( \frac{a\gamma}{\sigma} - \frac{2\sigma b}{a\gamma} \biggr) - \Phi\biggl( -\frac{a\gamma}{\sigma} - \frac{2\sigma b}{a\gamma} \biggr) \eqqcolon f_{\mathrm{K},b}(a),
    \end{align*}
 where $f_{\mathrm{K},b} : (0,\infty) \to (0,\infty)$ is strictly increasing and continuous.
\end{lemma}
\begin{proof}
    By Assumption~\ref{asm:fixed-design-regularity}, we may assume without loss of generality that there exists $\mathcal{T}_+ \subseteq [n]$ such that $|\mathcal{T}_+| \geq \beta n$ and $-x_i^\top (\theta_0 - \theta) \geq 2a\gamma$.  By Proposition~\ref{prop:univariate-realisability}, for $i\in\mathcal{T}_+$ and $t\in\mathbb{R}$, we have
    \begin{align*}
        \tilde{R}_{i,\theta}\bigl((-\infty,t]\bigr) \geq q(1-\epsilon)\Phi_{(0,\sigma)}\bigl( t - x_i^\top(\theta_0 - \theta) \bigr) &\geq q(1-\epsilon)\Phi_{(0,\sigma)}(t+2a\gamma) \\
        &= q(1-\epsilon)\Phi_{(-2a\gamma,\sigma)}(t).
    \end{align*}
    Moreover, by Proposition~\ref{prop:univariate-realisability} again for $R_0 \in \mathcal{R}_0^{\mathrm{Lin}}$ and $t\in\mathbb{R}$, we have $R_0\bigl( (-\infty,t] \bigr) \leq \Phi_{(0,\sigma)}(t)$. Therefore,
    \begin{align*}
        d_{\mathrm{K}}^{\mathrm{sym}}\bigl(R_{n,\theta} , \mathcal{R}_0^{\mathrm{Lin}}\bigr)
        &\geq \inf_{R_0\in \mathcal{R}_0^{\mathrm{Lin}}} \sup_{t \in \mathbb{R}} \biggl\{\frac{1}{n}\sum_{i=1}^n \tilde{R}_{i,\theta}\bigl( (-\infty,t] \bigr) - R_0\bigl( (-\infty,t] \bigr) \biggr\}\\
        &\geq \inf_{R_0\in \mathcal{R}_0^{\mathrm{Lin}}} \sup_{t \in \mathbb{R}}\, \biggl\{\frac{1}{n}\sum_{i\in\mathcal{T}_+} \tilde{R}_{i,\theta}\bigl( (-\infty,t] \bigr) - R_0\bigl( (-\infty,t] \bigr) \biggr\}\\
        &\geq \sup_{t\in\mathbb{R}}\, \bigl\{ \beta q(1-\epsilon)\Phi_{(-2a\gamma,\sigma)}(t) - \Phi_{(0,\sigma)}(t) \bigr\}\\
        &\geq \beta q(1-\epsilon)\Phi\biggl( \frac{a\gamma}{\sigma} - \frac{2\sigma b}{a\gamma} \biggr) - \Phi\biggl( -\frac{a\gamma}{\sigma} - \frac{2\sigma b}{a\gamma} \biggr) = f_{\mathrm{K},b}(a),
    \end{align*}
    where the final inequality follows by choosing $t = -\frac{2\sigma^2b}{a\gamma} - a\gamma$.  The function $f_{\mathrm{K},b}$ is continuous as a composition of continuous functions, and the fact that it is strictly increasing follows as in the proof of Lemma~\ref{lemma:one-dim-kolmogorov-distance-realisable-sets}, setting $(\epsilon, q)$ therein as $(\bar{\epsilon}, \bar{q})$, with $\bar{\epsilon} \coloneqq 1 - \beta q(1 - \epsilon)$ and $\bar{q} \coloneqq 1$.
\end{proof}

\begin{proof}[Proof of Theorem~\ref{thm:gaussian-realisable-response}]
    Let $\delta\in(0,1]$ and for the universal constant $C > 0$ from Lemma~\ref{lemma:uniform-dkw}, define the event
    \[
    \mathcal{E} \coloneqq \biggl\{\sup_{\theta\in\mathbb{R}^d} d_{\mathrm{K}}^{\mathrm{sym}}(\hat{R}_{n,\theta}, R_{n,\theta}) \leq C\sqrt{\frac{d+\log(1/\delta)}{n}}\biggr\}. 
    \]
    By Lemma~\ref{lemma:uniform-dkw}, satisfies $\mathbb{P}(\mathcal{E}\,|\,X_1=x_1,\ldots,X_n=x_n) \geq 1-\delta$, and from now on, we will work on the event $\mathcal{E}$.  Recalling that $R_{n,\theta_0}\in \mathcal{R}_0^{\mathrm{Lin}}$, we have
    \begin{align*}
        d_{\mathrm{K}}^{\mathrm{sym}}(\hat{R}_{n,\theta_0}, \mathcal{R}_0^{\mathrm{Lin}}) \leq d_{\mathrm{K}}^{\mathrm{sym}}(\hat{R}_{n,\theta_0}, R_{n,\theta_0}) \leq C\sqrt{\frac{d+\log(1/\delta)}{n}}.
    \end{align*}
    Moreover, if $\theta\in\mathbb{R}^d$ satisfies $d_{\mathrm{K}}^{\mathrm{sym}}(R_{n,\theta}, \mathcal{R}_0^{\mathrm{Lin}}) > 2C\sqrt{\frac{d+\log(1/\delta)}{n}}$, then 
    \begin{align*}
        d_{\mathrm{K}}^{\mathrm{sym}}(\hat{R}_{n,\theta}, \mathcal{R}_0^{\mathrm{Lin}}) &\geq d_{\mathrm{K}}^{\mathrm{sym}}(R_{n,\theta}, \mathcal{R}_0^{\mathrm{Lin}}) - d_{\mathrm{K}}^{\mathrm{sym}}(\hat{R}_{n,\theta}, R_{n,\theta})\\
        &> C\sqrt{\frac{d+\log(1/\delta)}{n}} \geq d_{\mathrm{K}}^{\mathrm{sym}}(\hat{R}_{n,\theta_0}, \mathcal{R}_0^{\mathrm{Lin}}),
    \end{align*}
   so $\hat{\theta}_n^{\mathrm{K}} \neq \theta$. Therefore, with $b$ and $f_{\mathrm{K},b}$ as defined in Lemma~\ref{lem:fixed-KS-linear-regression}, we deduce that with probability at least $1-\delta$,
    \begin{align*}
        \|\hat{\theta}^{\mathrm{K}}_n - \theta_0\|_2 &\leq \sup \biggl\{ \|\theta - \theta_0\|_2 : \theta\in\mathbb{R}^d ,\, d_{\mathrm{K}}^{\mathrm{sym}}(R_{n,\theta}, \mathcal{R}_0^{\mathrm{Lin}}) \leq 2C\sqrt{\frac{d+\log(1/\delta)}{n}} \biggr\}\\
        &\leq 2\inf \biggl\{ a>0 : \beta q(1-\epsilon)\Phi\biggl( \frac{a\gamma}{\sigma} - \frac{2\sigma b}{a\gamma} \biggr) - \Phi\biggl( -\frac{a\gamma}{\sigma} - \frac{2\sigma b}{a\gamma} \biggr)\\
        &\hspace{7cm}\geq 2C\sqrt{\frac{d+\log(1/\delta)}{n}} \biggr\}, \numberthis \label{eq:gaussian-realisable-response-eq1}
    \end{align*}
    where to obtain the second inequality, we note that by Lemma~\ref{lem:fixed-KS-linear-regression}, $d_{\mathrm{K}}^{\mathrm{sym}}(R_{n,\theta}, \mathcal{R}_0^{\mathrm{Lin}}) \geq f_{\mathrm{K},b}\bigl( \frac{\|\theta-\theta_0\|_2}{2} \bigr)$ and $f_{\mathrm{K},b}$ is a strictly increasing and continuous function.  Letting $a = \frac{3\sigma b}{\gamma\sqrt{\xi\log n}}$, we have by our assumption on $b$ that $2\sigma b/(a\gamma) - a\gamma/\sigma = \frac{2}{3}\sqrt{\xi\log n} - \frac{3b}{\sqrt{\xi\log n}} > 0$, so
    \begin{align*}
        \beta q(1&-\epsilon)\Phi\biggl( \frac{a\gamma}{\sigma} - \frac{2\sigma b}{a\gamma} \biggr) - \Phi\biggl( -\frac{a\gamma}{\sigma} - \frac{2\sigma b}{a\gamma} \biggr)\\
        \overset{(i)}&{\geq} \frac{\beta q(1-\epsilon)}{\bigl(-\frac{a\gamma}{\sigma} + \frac{2\sigma b}{a\gamma}\bigr) + \bigl(-\frac{a\gamma}{\sigma} + \frac{2\sigma b}{a\gamma}\bigr)^{-1}} \cdot \phi\biggl(-\frac{a\gamma}{\sigma} + \frac{2\sigma b}{a\gamma}\biggr) - \frac{1}{\frac{a\gamma}{\sigma} + \frac{2\sigma b}{a\gamma}} \cdot \phi\biggl( \frac{a\gamma}{\sigma} + \frac{2\sigma b}{a\gamma} \biggr)\\
        \overset{(ii)}&{\geq} \biggl( \frac{a\gamma}{\sigma} + \frac{2\sigma b}{a\gamma} \biggr)^{-1} \frac{1}{\sqrt{2\pi}} \biggl\{ \beta q(1-\epsilon) \exp\biggl( -\frac{a^2\gamma^2}{2\sigma^2} - \frac{2\sigma^2b^2}{a^2\gamma^2} + 2b \biggr) \\
        &\hspace{8cm} - \exp\biggl( -\frac{a^2\gamma^2}{2\sigma^2} -\frac{2\sigma^2b^2}{a^2\gamma^2} - 2b \biggr) \biggr\}\\
        \overset{(iii)}&{\geq} \biggl( \frac{a\gamma}{\sigma} + \frac{2\sigma b}{a\gamma} \biggr)^{-1} \frac{\sqrt{2}}{\sqrt{\pi}} \cdot \exp\biggl( -\frac{a^2\gamma^2}{2\sigma^2} - \frac{2\sigma^2b^2}{a^2\gamma^2}\biggr)\\
        \overset{(iv)}&{\geq} \frac{1}{\sqrt{\xi\log n}} \cdot n^{-\xi/4} \overset{(v)}{\geq} 2C\sqrt{\frac{d+\log(1/\delta)}{n}}
    \end{align*}
    where $(i)$ follows from the Mills ratio bound $\phi(x)/(x + x^{-1}) \leq \Phi(-x) \leq \phi(x)/x$ for $x > 0$; $(ii)$ follows since $\frac{1}{2} \leq b \leq \frac{\xi\log n}{9}$ implies $\bigl(-\frac{a\gamma}{\sigma} + \frac{2\sigma b}{a\gamma}\bigr) + \bigl(-\frac{a\gamma}{\sigma} + \frac{2\sigma b}{a\gamma}\bigr)^{-1} \leq \frac{a\gamma}{\sigma} + \frac{2\sigma b}{a\gamma}$; $(iii)$ follows by substituting the definition of $b$ and using the fact that $\beta q(1-\epsilon)\leq 1/2$; $(iv)$ follows since $b \leq \frac{\xi\log n}{30}$ implies $\frac{a^2\gamma^2}{\sigma^2} \leq \frac{\xi\log n}{100}$; and $(v)$ follows from the assumption that $n^{1-\xi} \geq C_1\bigl\{d+\log(1/\delta)\bigr\}$ with $C_1 \coloneqq 4C^2$ and using the fact that $x^{\xi/2} \geq \xi\log x$ for $x\in(0,\infty)$. Therefore, with probability at least $1-\delta$, we have by~\eqref{eq:gaussian-realisable-response-eq1} that
    \begin{align*}
        \|\hat{\theta}_n^{\mathrm{K}} - \theta_0\|_2^2 \leq \frac{36\sigma^2 b^2}{\gamma^2\xi\log n}\leq \frac{9\sigma^2 \log^2\bigl( 1+ \frac{4(1-\beta q(1-\epsilon))}{\beta q(1-\epsilon)} \bigr)}{\gamma^2 \xi \log \bigl(nq(1-\epsilon)\bigr)},
    \end{align*}
    as required.
\end{proof}

\section{Proof of Proposition~\ref{prop:adaptation}}
\begin{proof}[Proof of Proposition~\ref{prop:adaptation}]
    Let $\epsilon_1 \coloneqq \min \{\epsilon\in\mathcal{C} : \epsilon\geq\epsilon_{\star}\}$. For $\epsilon\in\mathcal{C}$, define the event $\mathcal{E}_{\epsilon} \coloneqq \bigl\{ \|\hat{\theta}_n(\epsilon,\delta') - \theta_0\|_2 \leq \phi(\epsilon,\delta') \bigr\}$ and let $\mathcal{E} \coloneqq \bigcap_{\epsilon\in\mathcal{C}: \epsilon\geq\epsilon_1} \mathcal{E}_{\epsilon}$. By~\eqref{eq:phi-def} and a union bound, we have $\mathbb{P}(\mathcal{E}) \geq 1-\delta$. On the event $\mathcal{E}$, we have
    \begin{align*}
        \theta_0 \in \bigcap_{\epsilon\in\mathcal{C}: \epsilon\geq\epsilon_1} B_2\bigl(\hat{\theta}_n(\epsilon,\delta'), \phi(\epsilon,\delta')\bigr),
    \end{align*}
    so by definition of $\epsilon_0$, we have $\epsilon_0 \leq \epsilon_1$.  Moreover, on $\mathcal{E}$, we have 
    \[
    B_2\bigl(\hat{\theta}_n(\epsilon_0,\delta'), \phi(\epsilon_0,\delta')\bigr) \bigcap B_2\bigl(\hat{\theta}_n(\epsilon_1,\delta'), \phi(\epsilon_1,\delta')\bigr) \neq \emptyset,
    \]
    so that
    \begin{align}
        \|\tilde{\theta}_n(\delta) - \theta_0\|_2 = \|\hat{\theta}_n(\epsilon_0,\delta') - \theta_0\|_2 &\leq \|\hat{\theta}_n(\epsilon_0,\delta') - \hat{\theta}_n(\epsilon_1,\delta')\|_2 + \|\hat{\theta}_n(\epsilon_1,\delta') - \theta_0\|_2\nonumber\\
        &\leq \bigl(\phi(\epsilon_0,\delta') + \phi(\epsilon_1,\delta')\bigr) + \phi(\epsilon_1,\delta') \leq 3\phi(\epsilon_1,\delta'), \label{eq:lepski-triangle-ineq}
    \end{align}
    where the final inequality follows since $\epsilon_0 \leq \epsilon_1$. If $\epsilon_{\star} < \epsilon_{\max} \cdot 2^{-\lceil\log_2 n\rceil}$, then $\epsilon_1 = \epsilon_{\max} \cdot 2^{-\lceil\log_2 n\rceil}\leq \epsilon_{\max}/n$, so by the assumption on $\phi$ and~\eqref{eq:lepski-triangle-ineq}, we have with probability at least $1-\delta$ that
    \begin{align*}
        \|\hat{\theta}_n(\epsilon_0,\delta') - \theta_0\|_2 \leq 3\phi(\epsilon_{\max}/n,\delta') \leq 3C \phi(0,\delta') \leq 3C \phi(2\epsilon_{\star}\wedge\epsilon_{\max},\delta').
    \end{align*}
    On the other hand, if $\epsilon_{\star} \geq \epsilon_{\max} \cdot 2^{-\lceil\log_2 n\rceil}$, then $\epsilon_1\leq 2\epsilon_{\star}\wedge\epsilon_{\max}$, so by the assumption on $\phi$ and~\eqref{eq:lepski-triangle-ineq}, we have with probability at least $1-\delta$ that
    \begin{align*}
        \|\hat{\theta}_n(\epsilon_0,\delta') - \theta_0\|_2 \leq 3\phi(2\epsilon_{\star}\wedge\epsilon_{\max},\delta').
    \end{align*}
    This completes the proof.
\end{proof}

\section{Auxiliary lemmas}
\label{sec:auxiliary}

If $\mathcal{Z}$ is a topological space, then we define the embedding $\phi_{\mathcal{Z}}: C_{\mathrm{b}}(\mathcal{Z}) \to \mathcal{M}(\mathcal{Z})^*$ by $\phi_{\mathcal{Z}}(f)(\mu) \coloneqq \mu(f)$.  If $\mathcal{Z}$ is a locally compact Hausdorff space, then a Borel measure $\mu$ on $\mathcal{Z}$ is \emph{regular} if $\mu(E) = \inf\{\mu(U): U \supseteq E, U \ \text{open}\}$ and $\mu(E) = \sup\{\mu(K): K \subseteq E, K \ \text{compact}\}$ for every Borel subset $E$ of $\mathcal{Z}$.  
\begin{lemma}
\label{Lemma:DualPair}
Let $\mathcal{Z}$ be a locally compact Hausdorff space in which every open set is $\sigma$-compact.  Then $\phi_\mathcal{Z}$ embeds $C_{\mathrm{b}}(\mathcal{Z})$ into a subspace of $\mathcal{M}(\mathcal{Z})^*$ that separates points. 
\end{lemma}
\begin{proof}
If $f, g \in C_{\mathrm{b}}(\mathcal{Z})$ and $\lambda_1,\lambda_2 \in \mathbb{R}$, then $\phi_{\mathcal{Z}}(\lambda_1 f + \lambda_2 g) = \lambda_1 \phi_{\mathcal{Z}}(f) + \lambda_2\phi_{\mathcal{Z}}(g)$, so $\phi_{\mathcal{Z}}$ embeds $C_{\mathrm{b}}(\mathcal{Z})$ into a subspace of $\mathcal{M}(\mathcal{Z})^*$.  

Let $\mu$ and $\mu'$ be two distinct measures in $\mathcal{M}(\mathcal{Z})$ and define $\nu \coloneqq \mu-\mu' \in \mathcal{M}(\mathcal{Z})$. By the the Jordan decomposition theorem \citep[Theorem~3.3]{folland1999real}, we can write $\nu = \nu_+ - \nu_-$ where $\nu_+, \nu_- \in \mathcal{M}_+(\mathcal{Z})$ are supported on disjoint measurable sets $P, N \subseteq \mathcal{Z}$ respectively.  Since $\nu\neq 0$, there exists a Borel set $B \subseteq \mathcal{Z}$ and $\epsilon > 0$ such that either $\nu_+(B\cap P) \geq \epsilon$ or $\nu_-(B\cap N) \geq \epsilon$. Without loss of generality, we assume the former.  By \citet[][Theorem~7.8]{folland1999real}, $\nu_+$ and $\nu_-$ are regular measures, so there exists a compact set $K \subseteq \mathcal{Z}$ and an open set $U \subseteq \mathcal{Z}$ such that $K\subseteq B\cap P\subseteq U$ and $\nu_+(U\setminus K) + \nu_-(U\setminus K) \leq \epsilon/2$. By Urysohn's lemma for locally compact Hausdorff spaces \citep[Lemma~4.32]{folland1999real}, there exists a continuous function $f:\mathcal{Z}\to [0,1]$ such that $f(K) = \{1\}$, $f(U^c) = \{0\}$. Observe that
\begin{align*}
\nu(f) & \geq \nu_+(K) - \nu_-(U\setminus P) \geq \nu_+(B\cap P) - \bigl(\nu_+(U\setminus K) + \nu_-(U\setminus K)\bigr) \geq \epsilon/2.
\end{align*}
Consequently, $f\in C_{\mathrm{b}}(\mathcal{Z})$ separates $\mu$ and $\mu'$ as desired.
\end{proof}
If $(X,\tau)$ and $(Y,\sigma)$ are topological spaces, we write $\tau \otimes \sigma$ for the product topology on the Cartesian product $X \times Y$, i.e.~$\tau \otimes \sigma$ is the coarsest topology for which the projections $(x,y) \mapsto x$ and $(x,y) \mapsto y$ are continuous. 
\begin{lemma}
\label{Lemma:ProductWeakTopology}
If $X$ and $Y$ are real vector spaces and  $X'$ and $Y'$ are subspaces of $X^*$ and $Y^*$, then the map $\iota:X'\times Y' \rightarrow (X\times Y)^*$ given by $\iota(f,g)(x,y) := f(x) + g(y)$ embeds $X' \times Y'$ as a subspace of $(X\times Y)^*$.  Furthermore, if $\tau(X; X')$, $\tau(Y; Y')$ and $\tau\bigl(X\times Y; \iota(X'\times Y')\bigr)$ denote the weak topologies generated by $X'$, $Y'$ and $\iota(X'\times Y')$ on $X$, $Y$ and $X\times Y$ respectively, then $\tau(X; X')\otimes \tau(Y; Y') = \tau\bigl(X\times Y; \iota(X'\times Y')\bigr)$.
\end{lemma}
\begin{proof}
To check that $\iota$ embeds $X'\times Y'$ as a subspace of $(X\times Y)^*$,  we only need to verify the bilinearity of the map $((x,y),(f,g))\mapsto f(x)+g(y)$ on $(X \times Y) \times (X' \times Y')$, which is true since $X,Y,X',Y'$ are vector spaces and $(X,X')$, $(Y,Y')$ are dual pairs.

For the second claim, let $\pi_X:X\times Y\to X$ and $\pi_Y:X\times Y\to Y$ be projection maps defined by $\pi_X(x,y) \coloneqq x$ and $\pi_Y(x,y) \coloneqq y$.  By the definition of the product topology, $\tau(X;X')\otimes \tau(Y;Y')$ is the coarsest topology on $X\times Y$ under which both $\pi_X$ and $\pi_Y$ are continuous. Also, $\tau\bigl(X\times Y;\iota(X'\times Y')\bigr)$ is the coarsest topology on $X\times Y$ under which $\iota(f,g)$ is continuous for all $f\in X'$ and $g\in Y'$. Hence the desired result is equivalent to the statement that for any topology $\mathcal{T}$ on $X\times Y$, the functions $\pi_X:(X\times Y, \mathcal{T})\to(X,\tau(X;X'))$ and $\pi_Y:(X\times Y, \mathcal{T})\to(Y,\tau(Y;Y'))$ are continuous if and only if $\iota(f,g):(X\times Y, \mathcal{T}) \to \mathbb{R}$ is continuous for all $(f,g)\in X'\times Y'$.

The `only if' direction is true since for any $(f,g)\in X'\times Y'$,  $\iota(f,g) = f\circ \pi_X + g\circ \pi_Y$ is the sum of compositions of continuous functions, and hence continuous.    For the `if' direction, we assume that $\iota(f,g)$ is continuous for all $(f,g)\in X'\times Y'$; by symmetry we only need to check that $\pi_X$ is continuous. Taking $g$ to be the zero map, we have $\iota(f,0)(x,y) = f(\pi_X(x,y))$, so $f\circ\pi_X$ is continuous for every $f$. Open sets in $(X,\tau(X;X'))$ are unions of sets in $\{f^{-1}(U): f\in X', \text{$U$ open in $\mathbb{R}$}\}$. Since $f\circ \pi_X$ is continuous, we have
    \[
    \pi_X^{-1}\bigl(f^{-1}(U)\bigr) = (f\circ \pi_X)^{-1}(U)
    \]
    is open for every $f\in X'$ and $U$ open in $\mathbb{R}$. Therefore, $\pi_X$ is continuous as desired, and this establishes the lemma.
\end{proof}

\begin{lemma}
\label{Lemma:ProductContinuity}
    Let $X,Y,Z$ be topological spaces and equip $Y\times Z$ with the product topology. Then $f:X\to Y$ and $g:X\to Z$ are continuous if and only if $h:x\mapsto \bigl(f(x),g(x)\bigr)$ is a continuous function from $X$ to $Y\times Z$. 
\end{lemma}
\begin{proof}
    By definition of the product topology, the projection maps $\pi_Y:Y\times Z \to Y$ and $\pi_Z:Y\times Z \to Z$ defined by $\pi_Y(y,z) \coloneqq y$ and $\pi_Z(y,z) \coloneqq z$ are continuous. This proves the `if' direction since $f = \pi_Y\circ h$ and $g = \pi_Z\circ h$. For the `only if' direction, we observe that open sets in $Y\times Z$ are unions of sets of the form $U\times V$ for $U$ open in $Y$ and $V$ open in $Z$. Since $f$ and $g$ are continuous, $h^{-1}(U\times V) = f^{-1}(U)\cap g^{-1}(V)$ is open in $X$, so $h$ is continuous as desired.
\end{proof}

Recall that if $A_1$ and $A_2$ are sets, then the \emph{disjoint union} of $A_1$ and $A_2$ is defined by $A_1 \sqcup A_2:=\{(a,1):a\in A_1\}\cup \{(a,2):a\in A_2\}$.  Moreover, if $(A_1,\tau_1)$ and $(A_2,\tau_2)$ are topological spaces, then $A_1 \sqcup A_2$ can be endowed with the \emph{disjoint union topology}, given by $\bigl\{(U_1\times\{1\}) \cup (U_2\times \{2\}):U_1\in\tau_1, \, U_2\in\tau_2\bigr\}$.  In the special case where $A_1$ and $A_2$ are disjoint subsets of a topological space $(\mathcal{X},\tau)$, the second argument of elements in $A_1 \sqcup A_2$ becomes redundant, so we can identify $A_1\sqcup A_2$ with $A_1 \cup A_2$, and we may write the disjoint union topology simply as $\{U_1 \cup U_2:U_1 \in A_1 \cap \tau,U_2 \in A_2 \cap \tau\}$.

\begin{lemma}\label{lem:projection-of-open-set-compact-set}
    Let $\mathcal{Z}_1, \ldots, \mathcal{Z}_d$ be topological spaces, and let $\mathcal{Z} \coloneqq \prod_{j=1}^d \mathcal{Z}_j$ be the product space equipped with the product topology. Let $S \subseteq [d]$ be non-empty and let $\mathcal{Z}_S \coloneqq \prod_{j\in S} \mathcal{Z}_j$ be the product space equipped with the product topology. 
    \begin{itemize}
        \item[(a)] If $U\subseteq \mathcal{Z}$ is open, then the set $U_S\coloneqq \{x_S : x\in U\}$ is open in $\mathcal{Z}_S$.
        \item[(b)] If $K\subseteq \mathcal{Z}$ is compact, then the set $K_S\coloneqq \{x_S : x\in K\}$ is compact in $\mathcal{Z}_S$.
    \end{itemize}
\end{lemma}
\begin{proof}
    (a) We can write $U = \bigcup_{i\in I} U^{(i)}$ for some index set $I$, where $U^{(i)} = \prod_{j=1}^d U^{(i)}_j$ and $U^{(i)}_j$ is open in $\mathcal{Z}_j$ for all $i\in I$, $j\in[d]$. Hence $U_S = \bigcup_{i\in I} U^{(i)}_S$ where $U^{(i)}_S = \prod_{j\in S} U^{(i)}_j$, so $U_S$ is open in $\mathcal{Z}_S$.

    \vspace{1em}
    \noindent (b) For any open cover $\{U^{(i)}_S\}_{i\in I}$ of $K_S$, define $U^{(i)} \coloneqq \{x\in\mathcal{Z} : x_S \in U^{(i)}_S\}$ for $i\in I$. Note that $U^{(i)}$ is open in $\mathcal{Z}$ for $i\in I$, as it is the pre-image of an open set under a projection map (which is continuous, by definition of the product topology). Thus, $\{U^{(i)}\}_{i\in I}$ is an open cover of $K$, which has a finite subcover $I_0 \subseteq I$ since $K$ is compact. Therefore, $\{U^{(i)}_S\}_{i\in I_0}$ is also a finite subcover of $K_S$, so $K_S$ is compact in $\mathcal{Z}_S$.
\end{proof}

\begin{lemma}
\label{Lemma:Preservation}
Let $\mathcal{X}_1$ and $\mathcal{X}_2$ be topological spaces.  
\begin{enumerate}[(a)]
\item If $\mathcal{X}_1$ and $\mathcal{X}_2$ are Hausdorff, then $\mathcal{X}_1 \times \mathcal{X}_2$ is Hausdorff in the product topology and $\mathcal{X}_1 \sqcup \mathcal{X}_2$ is Hausdorff in the disjoint union topology.
\item If $\mathcal{X}_1$ and $\mathcal{X}_2$ are locally compact, then $\mathcal{X} \times \mathcal{X}_2$ is locally compact in the product topology and $\mathcal{X}_1 \sqcup \mathcal{X}_2$ is locally compact in the disjoint union topology.
\end{enumerate}
\end{lemma}
\begin{proof}
(a) The first statement follows from \citet[Theorem~19.4]{munkrestopology}. 
 For the second statement, let $(x_1,j_1),(x_2,j_2) \in \mathcal{X}_1 \sqcup \mathcal{X}_2$ be distinct, where $j_1,j_2 \in \{1,2\}$, and for $\ell \in \{1,2\}$, we have $x_\ell \in \mathcal{X}_\ell$.  If $j_1 = j_2$, then the result follows from the fact that $\mathcal{X}$ and $\mathcal{X}_2$ are Hausdorff.  Otherwise, we can separate the two points using the open sets $\mathcal{X}_1\times \{1\}$ and $\mathcal{X}_2\times \{2\}$.

\medskip

\noindent (b) For the first statement, if $x_1 \in \mathcal{X}_1$ and $x_2 \in \mathcal{X}_2$, then we can find compact neighbourhoods $K_j \subseteq \mathcal{X}_j$ of $x_j$ for $j\in\{1,2\}$.  Then by Tychonov's theorem \citep[Theorem~37.3]{munkrestopology}, $K_1 \times K_2$ is a compact neighbourhood of $(x_1,x_2)$.  For the second statement, if $(x,j)\in \mathcal{X}_1 \sqcup\mathcal{X}_2$, then we can find a compact subset $K\subseteq \mathcal{X}_j$ containing $x$. Then $K\times \{j\}$ is a compact subset of $\mathcal{X}_1\sqcup\mathcal{X}_2$ containing $(x,j)$.
\end{proof}

\begin{lemma}
Let $\mathcal{X}_1,\ldots,\mathcal{X}_d$ be locally compact Hausdorff spaces, and let $\mathcal{X} \coloneqq \prod_{j=1}^d \mathcal{X}_j$.  Then $\mathcal{X}$, $\mathcal{X}_\star$ and $\mathcal{X} \times 2^{[d]}$ are locally compact Hausdorff spaces.  Moreover, if every open set in~$\mathcal{X}$ is $\sigma$-compact, then $\mathcal{X}_\star$ and $\mathcal{X} \times 2^{[d]}$ also have this property.   
\end{lemma}
\begin{proof}
The fact that $\mathcal{X}$ is a locally compact Hausdorff space follows from Lemma~\ref{Lemma:Preservation}.  Moreover, the singleton space $\{\star\}$ as well as the space $2^{[d]}$ endowed with the discrete topology are both locally compact Hausdorff spaces.  We observe that $\mathcal{X}$, $\mathcal{X}_\star$ and $\mathcal{X} \times 2^{[d]}$ can be generated from $\mathcal{X}_1,\ldots,\mathcal{X}_d$, $\{\star\}$, $2^{[d]}$ via a combination of product space and disjoint union operations. Hence the first result follows from Lemma~\ref{Lemma:Preservation}. 

    To check that every open set in $\mathcal{X}_{\star} = \bigsqcup_{S\in 2^{[d]}} \mathcal{X}^{(S)}$ is $\sigma$-compact, observe that for any open set $U \subseteq \mathcal{X}_{\star}$, we can write $U = \bigcup_{S \subseteq [d]} U^{(S)}$ where $U^{(S)} \coloneqq U \cap \mathcal{X}^{(S)}$. Therefore, it suffices to show that for every $S \subseteq [d]$, any open set $U\subseteq \mathcal{X}^{(S)}$ is $\sigma$-compact. Let $U_S \coloneqq \{a_S : a\in U\}$ and let $V \coloneqq \{x\in \mathcal{X} : x_S \in U_S\}$. Then $V$ is open in $\mathcal{X}$ since it is the pre-image of a projection of an open set, so we can write $V = \bigcup_{i=1}^\infty K(i)$, where $K(i)$ is compact in $\mathcal{X}$ for each $i$.  Moreover, $U_S = \bigcup_{i=1}^\infty K(i)_S$, and $K(i)_S$ is compact in $\mathcal{X}_S$ by Lemma~\ref{lem:projection-of-open-set-compact-set}(b). We claim that $K(i)^{(S)} \coloneqq \{z\in\mathcal{X}_\star : z_S\in K(i)_S,\, z_j=\star\; \forall j\notin S\}$ is compact.  To see this, consider any open cover $\{U(j)\}_{j\in J}$ of $K(i)^{(S)}$, where, without loss of generality, we assume that $U(j) \subseteq \mathcal{X}^{(S)}$ for all $j\in J$.  Writing $U(j)_{S} \coloneqq \{a_S : a\in U(j)\}$ for $j\in J$, we have by Lemma~\ref{lem:projection-of-open-set-compact-set}(a) that $\{U(j)_{S}\}_{j \in J}$ forms an open cover of $K(i)_S$.  We can therefore find a finite subcover $\{U(j)_S\}_{j\in J_0}$ of $K(i)_S$, so that $\{U(j)\}_{j\in J_0}$ forms a finite subcover of $K(i)^{(S)}$.  We deduce that $U = \bigcup_{i=1}^\infty K(i)^{(S)}$ is a countable union of compact sets. 

    To show that every open set in $\mathcal{X} \times 2^{[d]}$ is $\sigma$-compact, observe that since $2^{[d]}$ is finite, any open set in $\mathcal{X} \times 2^{[d]}$ is of the form $\bigcup_{S \subseteq [d]} \bigl\{U(S) \times S\bigr\}$, where $U(S)$ is open in $\mathcal{X}$.  Since each $U(S)$ is $\sigma$-compact and $S$ is finite (hence compact), it follows that each set $U(S) \times S$ is $\sigma$-compact, and hence $\bigcup_{S \subseteq [d]} \bigl\{U(S) \times S\bigr\}$ is $\sigma$-compact.
\end{proof}

\begin{lemma}\label{lemma:X-star-is-polish}
    If $(\mathcal{X}_1,\tau_1),\ldots,(\mathcal{X}_d,\tau_d)$ are Polish spaces, then the Cartesian product space $\mathcal{X}_{\star} \coloneqq \prod_{j=1}^d \mathcal{X}_{j,\star}$ equipped with the product topology is also a Polish space.
\end{lemma}
\begin{proof}
    A finite (or even countable) Cartesian product of Polish spaces is Polish \citep[][Proposition~3.3]{kechris2012classical}, so it suffices to show that $(\mathcal{X}_{j,\star},\tau_{j,\star})$ is Polish for each $j\in[d]$, where $\tau_{j,\star} \coloneqq \tau_j \cup \{A \cup \{\star\} : A\in \tau_j\}$. Now fix $j\in[d]$, and find a countable dense subset $\{x_n\}_{n=1}^\infty$ of $\mathcal{X}_j$. Then $\{\star\} \cup \{x_n\}_{n=1}^\infty$ is a countable dense subset of $\mathcal{X}_{j,\star}$, so $\mathcal{X}_{j,\star}$ is separable. Now find a metric $d$ on $\mathcal{X}_j$ such that $d$ generates the topology $\tau_j$ and $(\mathcal{X}_j,d)$ is complete. Define the standard bounded metric $\bar{d}$ by $\bar{d}(x,y) \coloneqq d(x,y) \wedge 1$ for $x,y\in\mathcal{X}_j$.  Then, by \citet[][Theorem~20.1]{munkrestopology}, $\bar{d}$ also generates the topology $\tau_j$. Define a metric $d'$ on $\mathcal{X}_{j,\star}$ by $d'(x,y)\coloneqq \bar{d}(x,y)$ for $x,y\in\mathcal{X}_j$, $d'(x,\star)\coloneqq 2$ for $x\in\mathcal{X}_j$, and $d'(\star,\star) \coloneqq 0$. Letting~$\tau_{j,\star}'$ denote the topology on $\mathcal{X}_{j,\star}$ generated by $d'$, we first show that $\tau_{j,\star}' = \tau_{j,\star}$. On the one hand, since $\{\star\} \in \tau_{j,\star}'$ and $\tau_j \subseteq \tau_{j,\star}'$, we have $\tau_{j,\star} \subseteq \tau_{j,\star}'$. On the other hand, if $x_0 \in \mathcal{X}_{j,\star}$, $r \geq 0$ and $A \coloneqq \{x\in\mathcal{X}_{j,\star} : d'(x,x_0) < r\}$ denotes an open ball in $\tau_{j,\star}'$, then
    \[
    A= \left\{ \begin{array}{ll} \mathcal{X}_{j,\star} & \mbox{if $r > 2$} \\ 
    \{x\in\mathcal{X}_j : \bar{d}(x,x_0) < r\} & \mbox{if $r \leq 2$ and $x_0 \in \mathcal{X}_j$}\\
    \{\star\} & \mbox{if $r \leq 2$ and $x_0 = \star$.}
    \end{array} \right.
    \]
    We deduce that $A \in \tau_{j,\star}$, so since such open balls generate $\tau_{j,\star}'$, we have $\tau_{j,\star}' \subseteq \tau_{j,\star}$. Hence, $d'$ generates the topology $\tau_{j,\star}$. Next, we show that $(\mathcal{X}_{j,\star}, d')$ is complete. Let $(z_n)_{n=1}^{\infty}$ be a Cauchy sequence in $\mathcal{X}_{j,\star}$, so there exists $N\in\mathbb{N}$ such that $d'(z_{n_1},z_{n_2}) \leq 1/2$ for all $n_1,n_2 \geq N$. Therefore, either $z_n = \star$ for all $n\geq N$ or $z_n \in \mathcal{X}_j$ for all $n\geq N$.  In the former case,  $z_n \to \star$ as $n \rightarrow \infty$.  In the latter case, $(z_n)_{n=N}^\infty$ is also a Cauchy sequence in $(\mathcal{X}_j, d)$ and hence by completeness of $(\mathcal{X}_j, d)$, it has a limit in $\mathcal{X}_j$. This shows that $(\mathcal{X}_{j,\star}, d')$ is complete and $d'$ generates the topology $\tau_{j,\star}$, so $(\mathcal{X}_{j,\star}, \tau_{j,\star})$ is completely metrisable. Therefore, $(\mathcal{X}_{j,\star}, \tau_{j,\star})$ is a Polish space, as required.
\end{proof}

The following lemma can be deduced from~\citet[][3.1(e), p.~95]{horn1990hadamard}, but for the convenience of the reader, we provide a short proof here.  
\begin{lemma}\label{lemma:operator-norm-of-hadamard-product}
    Let $A,B\in\mathcal{S}^{d\times d}$ and further suppose that $A$ is positive semi-definite.  Then $\|A\odot B\|_{\mathrm{op}} \leq \|A\|_{\infty} \|B\|_{\mathrm{op}}$.
\end{lemma}
\begin{proof}
The proof largely follows that of~\citet{horn1990hadamard}.  Since \begin{align*}
    \begin{pmatrix}
        A & A\\
        A & A
    \end{pmatrix} \in \mathcal{S}^{2d \times 2d} \quad\text{and}\quad \begin{pmatrix}
        \|B\|_{\mathrm{op}}I_d & B\\
        B & \|B\|_{\mathrm{op}}I_d
    \end{pmatrix} \in \mathcal{S}^{2d \times 2d}
\end{align*}
are both positive semi-definite, by the Schur product theorem \citep[Theorem 7.5.3(a)]{horn2012matrix}, their Hadamard product \begin{align*}
    \begin{pmatrix}
        \|B\|_{\mathrm{op}}(I_d \odot A) & A \odot B\\
        A \odot B & \|B\|_{\mathrm{op}}(I_d \odot A)
    \end{pmatrix}
\end{align*}
is also positive semi-definite. Hence, for any $v\in\mathbb{R}^d$, we have \begin{align*}
    0 &\leq \begin{pmatrix}
        v^\top & -v^\top
    \end{pmatrix} \begin{pmatrix}
        \|B\|_{\mathrm{op}}(I_d \odot A) & A \odot B\\
        A \odot B & \|B\|_{\mathrm{op}}(I_d \odot A)
    \end{pmatrix} \begin{pmatrix}
        v \\ -v
    \end{pmatrix}\\
    &= 2 \|B\|_{\mathrm{op}} v^\top(I \odot A) v - 2v^\top (A\odot B) v,
\end{align*}
so $\|A\odot B\|_{\mathrm{op}} \leq \|B\|_{\mathrm{op}} \|I \odot A\|_{\mathrm{op}} \leq \|A\|_{\infty} \|B\|_{\mathrm{op}}$.
\end{proof}

\begin{lemma}\label{lem:inverse-binomial-bounds}
        Let $Y \sim \mathsf{Bin}(n, q)$ for some $n \in \mathbb{N}$ and $q \in (0,1]$. Then 
    \begin{align*}
        \mathbb{E} \bigl( Y^{-1} \cdot \mathbbm{1}_{\{Y>0\}} \bigr) \leq \frac{2}{n q} \quad\text{and}\quad \mathbb{E}\bigl\{(Y+1)^{-2} \bigr\} \leq \frac{2}{n^2 q^2}.
    \end{align*}
\end{lemma}
\begin{proof}
    We have 
    \begin{align*}
        \mathbb{E}\bigl\{ (Y+1)^{-1} \bigr\} &= \sum_{y=0}^n (y+1)^{-1}\binom{n}{y} q^y (1-q)^{n-y}\\
        &= \sum_{y=0}^n \frac{1}{q(n+1)} \binom{n+1}{y+1} q^{y+1}(1-q)^{n-y}\\
        &= \frac{1}{q(n+1)} \sum_{y=1}^{n+1} \binom{n+1}{y} q^{y}(1-q)^{n+1-y}\leq \frac{1}{nq}.
    \end{align*}
    The first inequality in the statement then follows as $y^{-1} \leq 2(y + 1)^{-1}$ for all $y \geq 1$.  Similarly, we have 
    \begin{align*}
        \mathbb{E}\bigl\{ (Y+1)^{-1}(Y+2)^{-1} \bigr\} &= \sum_{y=0}^n (y+1)^{-1}(y+2)^{-1}\binom{n}{y} q^y (1-q)^{n-y}\\
        &= \sum_{y=0}^n \frac{1}{q^2(n+1)(n+2)} \binom{n+2}{y+2} q^{y+2}(1-q)^{n-y} \leq \frac{1}{n^2q^2}.
    \end{align*}
    The second inequality in the statement then follows since $(y+1)^{-2} \leq 2(y+1)^{-1}(y+2)^{-1}$ for all $y \geq 0$.  
\end{proof}

\begin{lemma}\label{lemma:binomial-tail}
    Suppose that $(B_i)_{i \in [n]} \overset{\mathrm{iid}}{\sim} \mathsf{Ber}(q)$. 
    \begin{itemize}
        \item[(a)] With probability at least $1 - \delta$, 
        \[
        \frac{1}{n} \sum_{i=1}^{n} B_i \leq 2q + \frac{\log(1/\delta)}{n}.
        \]
        \item[(b)] If $q\geq \frac{8\log(1/\delta)}{n}$, then with probability at least $1 - \delta$, 
        \[
        \frac{1}{n} \sum_{i=1}^{n} B_i \geq \frac{q}{2}.
        \]
    \end{itemize}
\end{lemma}
\begin{proof}
(a) By Bernstein's inequality \citep[][Theorem 2.10]{boucheron2003concentration}, we have with probability at least $1-\delta$ that
\begin{align*}
\frac{1}{n} \sum_{i=1}^{n} B_i &\leq q + \sqrt{\frac{2q(1-q)}{n}}\log^{1/2}(1/\delta) + \frac{1}{3n}\log(1/\delta) \\
&\leq \biggl(q^{1/2} + \frac{1}{\sqrt{2n}}\log^{1/2}(1/\delta)\biggr)^2 \leq 2q + \frac{1}{n}\log(1/\delta).
\end{align*}

(b) By the multiplicative Chernoff bound~\citep[][Theorem 2.3(c)]{McDiarmid1998} for the sum of independent Bernoulli random variables, we have
\begin{align*}
    \mathbb{P}\biggl( \frac{1}{n} \sum_{i=1}^{n} B_i \leq \frac{q}{2} \biggr) \leq \exp(-nq/8) \leq \delta,
\end{align*}
where the final inequality follows from the assumption that $q\geq \frac{8\log(1/\delta)}{n}$.
\end{proof}

\begin{lemma}\label{lemma:inclusion-of-psi-r-class}
    Let $0<r_1\leq r_2$.  Then $\mathcal{P}_{\psi_{r_2}}(\theta_0, \sigma^2) \subseteq \mathcal{P}_{\psi_{r_1}}(\theta_0, \sigma^2)$.
\end{lemma}
\begin{proof}
    Let $X\sim P \in \mathcal{P}_{\psi_{r_2}}(\theta_0, \sigma^2)$.  Then 
    \begin{align*}
        2\geq \mathbb{E}\biggl\{ \exp\biggl( \frac{|X - \theta_0|^{r_2}}{\sigma^{r_2}} \biggr) \biggr\} \geq \biggl[\mathbb{E}\biggl\{ \exp\biggl( \frac{|X - \theta_0|^{r_1}}{\sigma^{r_1}} \biggr) \biggr\}\biggr]^{r_2/r_1},
    \end{align*}
    by Jensen's inequality. Thus $\mathbb{E}\exp\bigl( |X - \theta_0|^{r_1} / \sigma^{r_1} \bigr) \leq 2$, so $P\in \mathcal{P}_{\psi_{r_1}}(\theta_0, \sigma^2)$.
\end{proof}

\begin{lemma} \label{lemma:MGF-bound}
    Let $r>1$, $\sigma>0$ and $X\sim P\in\mathcal{P}_{\psi_r}(0,\sigma^2)$. Then
    \begin{align*}
        \mathbb{E} \exp(\lambda X) \leq 2\exp\bigl\{ (\sigma\lambda)^{r/(r-1)} \bigr\},
    \end{align*}
    for all $\lambda>0$.
\end{lemma}
\begin{proof}
Young's inequality states that whenever $p,q > 1$ are such that $1/p + 1/q  =1$, we have $ab \leq a^p/p + b^q/q$ for all $a,b \geq 0$.  Hence
    \begin{align*}
        \lambda X \leq \lambda|X| \leq \frac{|X|^r}{r\sigma^r} + \frac{(\sigma\lambda)^{r/(r-1)}}{r/(r-1)} \leq \frac{|X|^r}{\sigma^r} + (\sigma\lambda)^{r/(r-1)}.
    \end{align*}
    Therefore,
    \begin{align*}
        \mathbb{E} \exp(\lambda X) \leq \mathbb{E} \bigl\{ \exp(|X|^r/\sigma^r) \bigr\} \cdot \exp\bigl\{ (\sigma\lambda)^{r/(r-1)} \bigr\} \leq 2\exp\bigl\{ (\sigma\lambda)^{r/(r-1)} \bigr\},
    \end{align*}
    as required.
\end{proof}

\begin{lemma}[PAC--Bayes lemma] \label{lemma:PAC-Bayes}
    Let $\mathcal{X}$ be a measurable space and let $X_1,\ldots,X_n \overset{\mathrm{iid}}{\sim} P \in \mathcal{P}(\mathcal{X})$. Let $\Xi \subseteq \mathbb{R}^d$ and $\mu \in \mathcal{P}(\Xi)$. Further let $f : \mathcal{X} \times \Xi \to \mathbb{R}$ be such that $\mathbb{E}_{X \sim P} (e^{f(X,\xi)}) < \infty$ for $\mu$-almost all $\xi\in\Xi$. Then, for every $\delta\in(0,1]$, we have with probability at least $1-\delta$ that
    \begin{align*}
        \sup_{\rho \in \mathcal{P}(\Xi):\rho \ll \mu} \biggl\{ \frac{1}{n} \sum_{i=1}^n \mathbb{E}_{\xi\sim\rho} f(X_i,\xi) - \mathbb{E}_{\xi\sim\rho} \log\bigl\{ \mathbb{E}_{X \sim P} &(e^{f(X,\xi)})\bigr\} \\
        &- \frac{\mathrm{KL}(\rho,\mu) + \log(1/\delta)}{n} \biggr\} \leq 0,
    \end{align*}
    where, for instance, $\mathbb{E}_{\xi\sim\rho} f(X_i,\xi) \coloneqq \int_{\Xi}  f(X_i,v) \, \mathrm{d}\rho(v)$.
    \end{lemma}
\begin{proof}
    See, for example, \citet[][Lemma~2.1]{zhivotovskiy2024dimension}.
\end{proof}

The following lemma provides a concentration result for the sample mean of independent and identically distributed sub-exponential random vectors. The proof strategy follows that of \citet[][Proposition~3.1]{zhivotovskiy2024dimension}, who considered the case $n=1$.
\begin{lemma} \label{lemma:concentration-of-sample-mean-sub-exponential-vector}
    Let $\theta_0 \in \mathbb{R}^d$, $\Sigma\in\mathcal{S}_{++}^{d\times d}$, $\delta\in(0,1]$ and $X_1,\ldots,X_n \overset{\mathrm{iid}}{\sim} P \in \mathcal{P}_{d,\psi_1}(\theta_0,\Sigma)$. Assume further that $\delta \geq 2e^{-n/3}$. Then with probability at least $1-\delta$,
    \begin{align*}
        \biggl\| \frac{1}{n} \sum_{i=1}^n X_i - \theta_0 \biggr\|_2^2 \leq 24 \cdot \frac{\tr(\Sigma) + \|\Sigma\|_{\mathrm{op}} \log(2/\delta)}{n}.
    \end{align*}
\end{lemma}
\begin{proof}
    Let $\beta\coloneqq 2\log(2/\delta)$, let $\mu$ denote the distribution of $\mathsf{N}_d(0,\beta^{-1}\Sigma)$ and for $u\in\Sigma^{1/2}\mathbb{S}^{d-1}$, let $\rho_u$ denote the conditional distribution of $Y$ given $\bigl\{\|Y-u\|_2 \leq \sqrt{2\beta^{-1}\tr(\Sigma)}\bigr\}$, where $Y\sim\mathsf{N}_d(u,\beta^{-1}\Sigma)$. By the computation of \citet[][p.~11]{zhivotovskiy2024dimension}, we have 
    \begin{align*}
        \mathrm{KL}(\rho_u,\mu) \leq \frac{\beta}{2} + \log 2 \leq 2\log\Bigl(\frac{2}{\delta}\Bigr).
    \end{align*} 
    Now, let $v \in \mathbb{R}^d$ be such that $\|v - u\|_2 \leq \sqrt{2\beta^{-1}\tr(\Sigma)}$, and for $\lambda \in \mathbb{R}$, define $f_\lambda:\mathbb{R}^d\times\mathbb{R}^d \to \mathbb{R}$ by $f_\lambda(x,y) \coloneqq \lambda y^\top\Sigma^{-1/2}(x-\theta_0)$. Then, for $X\sim P$ and  $\lambda \in \mathbb{R}$, we have 
    \[
    \bigl\| v^\top \Sigma^{-1/2}(X - \theta_0) \bigr\|_{\psi_1} \leq \|v\|_2 \leq \|\Sigma\|_{\mathrm{op}}^{1/2} + \sqrt{2\beta^{-1}\tr(\Sigma)} \eqqcolon R.
    \]
    It follows by \citet[][Lemma~2.5]{zhivotovskiy2024dimension} that
    \begin{align*}
        \log \mathbb{E}_{X\sim P} ( e^{f_{\lambda}(X,v)}) = \log \mathbb{E}_{X\sim P} ( e^{\lambda \cdot v^\top \Sigma^{-1/2}(X - \theta_0)}) \leq 4\lambda^2 R^2,
    \end{align*}
    for all $|\lambda| \leq \frac{1}{2R}$, so $\mathbb{E}_{\xi_u \sim \rho_u} \bigl\{ \log \mathbb{E}_{X\sim P} ( e^{f_{\lambda}(X,\xi_u)}) \bigr\} \leq 4\lambda^2 R^2$ for all $|\lambda| \leq \frac{1}{2R}$.  The PAC--Bayes lemma~(Lemma~\ref{lemma:PAC-Bayes}) then yields that with probability at least $1-\delta$,
    \begin{align*}
        \sup_{u \in \Sigma^{1/2}\mathbb{S}^{d-1}} \biggl\{ \frac{1}{n} \sum_{i=1}^n \mathbb{E}_{\xi_u \sim \rho_u} f_{\lambda}(X_i,\xi_u) - \mathbb{E}_{\xi_u \sim \rho_u} \bigl\{ \log &\mathbb{E}_{X\sim P} (e^{f_{\lambda}(X,\xi_u)}) \bigr\} \\
        &- \frac{\mathrm{KL}(\rho_u,\mu) + \log(1/\delta)}{n} \biggr\} \leq 0.
    \end{align*}
    Therefore, we deduce that with probability at least $1-\delta$,
    \begin{align*}
        \biggl\| \frac{1}{n} \sum_{i=1}^n X_i - \theta_0 \biggr\|_2 &= \sup_{u \in \Sigma^{1/2}\mathbb{S}^{d-1}} \frac{1}{n} \sum_{i=1}^n u^\top \Sigma^{-1/2} (X_i - \theta_0)\\
        &= \sup_{u \in \Sigma^{1/2}\mathbb{S}^{d-1}} \frac{1}{n\lambda} \sum_{i=1}^n \mathbb{E}_{\xi_u \sim \rho_u} f_{\lambda}(X_i,\xi_u) \\
        &\leq \inf_{\lambda \in [0,\frac{1}{2R}]} \biggl\{4\lambda R^2 + \frac{3\log(2/\delta)}{n\lambda}\biggr\}\\
        \overset{(i)}&{=} 2R\sqrt{\frac{3\log(2/\delta)}{n}} = 2\sqrt{\frac{3}{n}} \cdot \Bigl\{\sqrt{\tr(\Sigma)} + \sqrt{\|\Sigma\|_{\mathrm{op}} \log(2/\delta)}\Bigr\}.
    \end{align*}
    where $(i)$ follows by choosing $\lambda = \frac{1}{2R}\sqrt{\frac{3\log(2/\delta)}{n}}$, which is at most $\frac{1}{2R}$ since $\frac{3\log(2/\delta)}{n} \leq 1$ by assumption. The final conclusion follows by squaring both sides of the inequality above and using the inequality $(a+b)^2 \leq 2a^2 + 2b^2$ for $a,b \in \mathbb{R}$.
\end{proof}

\section{Background on disintegrations}\label{sec:disintegration}

Our definition of MAR relies on the decomposition of a probability measure on a product space into the marginal distribution on one coordinate and a family of conditional distributions on the other.  This can be achieved via the notion of disintegration.  Let $(\mathcal{X},\mathcal{A})$ and $(\mathcal{Y},\mathcal{B})$ be measurable spaces, and let $P$ be a probability measure on the product space $(\mathcal{X}\times \mathcal{Y}, \mathcal{A} \otimes \mathcal{B})$.  Further, let $\mu$ denote the marginal distribution of $P$ on $(\mathcal{X},\mathcal{A})$. We say that $(P_x)_{x \in \mathcal{X}}$ is a \emph{disintegration\index{disintegration} of~$P$ into conditional distributions on $\mathcal{Y}$} if 
\begin{enumerate}[(a)]
    \item $P_x$ is a probability measure on $(\mathcal{Y},\mathcal{B})$, for each $x \in \mathcal{X}$;
    \item $x \mapsto P_x(B)$ is an $\mathcal{A}$-measurable function, for every $B \in \mathcal{B}$;
    \item $P(A\times B) = \int_A P_x(B) \, d\mu(x)$ for all $A \in \mathcal{A}$ and $B \in \mathcal{B}$.
\end{enumerate}
In our setting, $P$ denotes the joint distribution of a random pair $(X,Y)$, taking values in $\mathcal{X}$ and $\mathcal{Y}$ respectively.  We interpret $P_x$ as the conditional distribution of $Y$ given $X=x$, even though it may be the case that the conditioning event has probability zero.  Going further, we also interpret $P_X$ as the conditional distribution of $Y$ given~$X$.  Indeed, we then have for all $A \in \mathcal{A}$ and $B \in \mathcal{B}$ that	\begin{align*}
	\mathbb{E} \bigl(P_X(B) \mathbbm{1}_A(X) \bigr) = \int_A P_x(B) \,d\mu(x) = P(A \times B) &= \mathbb{P}(X \in A,Y \in  B) \\
 &= \mathbb{E}\bigl(\mathbbm{1}_A(X)\mathbbm{1}_{B}(Y)\bigr),
	\end{align*}
so $\mathbb{P}(Y \in B \, | \, X) = \mathbb{E}\bigl(\mathbbm{1}_B(Y) \, | \, X) = P_X(B)$ almost surely.
The following result, which follows from \citet[][Theorems~10.2.1 and~10.2.2]{dudley2018real}, provides a sufficient condition for the existence of a disintegration and may be regarded as a generalisation of Fubini's theorem for probability measures on the product of Polish spaces.
\begin{theorem}\label{lemma:existenceOfDisintegrationFromDudley} Suppose that $(\mathcal{X},\mathcal{A})$ and $(\mathcal{Y},\mathcal{B})$ are Polish spaces with their corresponding Borel $\sigma$-algebras.  Let $P$ be a probability distribution on $(\mathcal{X}\times \mathcal{Y}, \mathcal{A} \otimes \mathcal{B})$, with $\mu$ denoting the marginal distribution of $P$ on $(\mathcal{X},\mathcal{A})$. Then there exists a disintegration $(P_x)_{x \in \mathcal{X}}$ of $P$ into conditional distributions on $\mathcal{Y}$ with the property that 
\begin{align*}
\int_{\mathcal{X}\times \mathcal{Y}}g(x,y)\,dP(x,y) = \int_{\mathcal{X}} \biggl(\int_{\mathcal{Y}} g(x,y) \,dP_x(y)\biggr) \,d\mu(x),
\end{align*}
for every $P$-integrable function $g: \mathcal{X}\times \mathcal{Y} \rightarrow \mathbb{R}$. Moreover, the disintegration $(P_x)_{x \in \mathcal{X}}$ of $P$  is unique in the sense that if there exists another disintegration $(\tilde{P}_x)_{x \in \mathcal{X}}$ of $P$ into conditional distributions on $\mathcal{Y}$, then $\tilde{P}_x = P_x$ for $\mu$-almost every $x \in \mathcal{X}$.
\end{theorem}
In order to apply this result in our missing data context, recall the random pair $(X,\Omega')$ taking values in $\mathcal{X} \times \{0,1\}^d$ from~\eqref{eq:MAR-law}.  For each $\omega \in \{0,1\}^d$, we assume the existence of disintegrations $(P_{x\ostar\omega})_{x \in \mathcal{X}}$ of the joint distribution of $(X \ostar \omega,\Omega')$ into conditional distributions on $\{0,1\}^d$ as well as $(P_x)_{x \in \mathcal{X}}$ of the joint distribution of $(X, \Omega')$ into conditional distributions on $\{0,1\}^d$.  The existence of these disintegrations is guaranteed by Theorem~\ref{lemma:existenceOfDisintegrationFromDudley} when $\mathcal{X}_j$ is a Polish space for each $j \in [d]$, because it then follows from Lemma~\ref{lemma:X-star-is-polish} and its proof that $\mathcal{X} \coloneqq \prod_{j=1}^d \mathcal{X}_j$ and $\mathcal{X}_\star \coloneqq \prod_{j=1}^d \mathcal{X}_{j,\star}$ are Polish.  Formally then, the condition $\mathbb{P}(\Omega' = \omega\,|\, X=x) = \mathbb{P}(\Omega' = \omega\,|\, X\ostar\omega=x\ostar\omega)$ in~\eqref{eq:MAR-law} means that $P_x(\omega) = P_{x\ostar\omega}(\omega)$.  In fact, since the MAR definition refers to a family of distributions of $X \ostar \Omega'$, we need these disintegrations for each possible joint distribution of $(X,\Omega')$ with $X \sim P$ and $\mathbb{P}(\Omega' = \bm{1}_S) = \pi(S)$ for all $S \subseteq[d]$ (such disintegrations are again guaranteed to exist by Theorem~\ref{lemma:existenceOfDisintegrationFromDudley} when $\mathcal{X}_j$ is a Polish space for each $j \in [d]$).  

\section{MCAR lower bounds for mean estimation} \label{sec:heterogeneous-mean}

Recall the definition of an $f$-divergence $\mathrm{Div}_f(\cdot,\cdot)$ from~\eqref{eq:f-divergence}.  Lemma~\ref{lem:equivalence-of-f-af} below relates the $f$-divergence of two MCAR distributions on $\mathcal{X}_\star$ to a notion of average $f$-divergence given in Definition~\ref{def:ADiv} below.  For probability measures $P, Q \in \mathcal{P}(\mathcal{X})$, we let, for $S \subseteq [d]$, $P_S$ and $Q_S$ denote their respective marginal distributions on $\mathcal{X}_S$.  
\begin{defn} \label{def:ADiv}
    Let $P, Q \in \mathcal{P}(\mathcal{X})$, let $\pi\in \mathcal{P}(2^{[d]})$ and let $f: (0, \infty) \rightarrow \mathbb{R}$ be a convex function with $f(1) = 0$. We define the \emph{average $f$-divergence} between $P$ and $Q$ with respect to $\pi$ to be
    \begin{align*}
        \mathrm{ADiv}_f(P, Q;\pi) \coloneqq \sum_{S \subseteq [d]} \pi(S) \cdot \mathrm{Div}_f\bigl(P_S, Q_S\bigr),
    \end{align*}
    where we adopt the convention that $\mathrm{Div}_f(P_S, Q_S) \coloneqq 0$ if $S = \emptyset$.
\end{defn} 
We will write $\mathrm{ATV}(\cdot,\cdot;\pi)$ and $\mathrm{AKL}(\cdot,\cdot;\pi)$ respectively for the average total variation distance and average Kullback--Leibler divergence with respect to $\pi$.  It is worth noting that the average total variation distance is a pseudo-metric but not necessarily a metric on $\mathcal{P}(\mathcal{X})$; indeed, we have $\mathrm{ATV}(P,Q;\pi) = 0$ whenever $P$ and $Q$ have the same marginal distributions on the support of $\pi$.  

The following lemma shows how an $f$-divergence between two MCAR distributions on $\mathcal{P}(\mathcal{X}_{\star})$ can be computed as an average $f$-divergence on $\mathcal{P}(\mathcal{X})$ in the sense of Definition~\ref{def:ADiv}.  
\begin{lemma}\label{lem:equivalence-of-f-af}
    Let $P, Q \in \mathcal{P}(\mathcal{X})$ and let $\pi \in \mathcal{P}(2^{[d]})$. Then
    \[
    \mathrm{Div}_f\bigl(\mathsf{MCAR}_{(\pi, P)}, \mathsf{MCAR}_{(\pi, Q)}\bigr) = \mathrm{ADiv}_f(P, Q; \pi).
    \]
\end{lemma}
\begin{proof}[Proof of Lemma~\ref{lem:equivalence-of-f-af}]
Recall the definition of $\mathcal{X}^{(S)}$ and $\mathcal{X}_S$ from Section~\ref{sec:notation-proofs}.  For $A \in\mathcal{B}(\mathcal{X}_\star)$, note that $A \cap \mathcal{X}^{(S)} \in \mathcal{B}(\mathcal{X}^{(S)})$ and define $(A \cap \mathcal{X}^{(S)})_S \coloneqq \bigl\{ x_S : x\in A \cap \mathcal{X}^{(S)} \bigr\} \in \mathcal{B}(\mathcal{X}_S)$.  Let $P^{(S)} \in \mathcal{P}(\mathcal{X}_{\star})$ be defined as $P^{(S)}(A) \coloneqq P_S \bigl((A \cap \mathcal{X}^{(S)})_S \bigr)$ for $A \in \mathcal{B}(\mathcal{X}_{\star})$, so that $P^{(S)}$ is supported on $\mathcal{X}^{(S)}$. 
For each $S \subseteq [d]$, we can apply the Lebesgue decomposition theorem to obtain the decomposition $P^{(S)} = P^{(S)}_{\mathrm{ac}} + P^{(S)}_{\mathrm{sing}}$ (with respect to $Q^{(S)}$).  Then, with respect to $\mathsf{MCAR}_{(\pi, Q)}$, 
\[
\bigl(\mathsf{MCAR}_{(\pi, P)}\bigr)_{\mathrm{ac}} = \biggl(\sum_{S \subseteq [d]} \pi(S) \cdot P^{(S)}\biggr)_{\mathrm{ac}} = \sum_{S \subseteq [d]} \pi(S) \cdot P^{(S)}_{\mathrm{ac}} 
\]
and
\[
\bigl(\mathsf{MCAR}_{(\pi, P)}\bigr)_{\mathrm{sing}} = \biggl(\sum_{S \subseteq [d]} \pi(S) \cdot P^{(S)}\biggr)_{\mathrm{sing}} = \sum_{S \subseteq [d]} \pi(S) \cdot P^{(S)}_{\mathrm{sing}}.
\]
Hence, since $\mathcal{X}_{\star} = \sqcup_{S \subseteq [d]} \mathcal{X}^{(S)}$,
\begin{align*}
&\mathrm{Div}_f\bigl(\mathsf{MCAR}_{(\pi, P)}, \mathsf{MCAR}_{(\pi, Q)}\bigr)\\
&= \int_{\mathcal{X}_{\star}} f\biggl(\frac{\mathrm{d}\sum_{S \subseteq [d]} \pi(S) P^{(S)}_{\mathrm{ac}}}{\mathrm{d}\sum_{S \subseteq [d]} \pi(S) Q^{(S)}}\biggr)  \sum_{S \subseteq [d]} \pi(S) \, \mathrm{d}Q^{(S)} + M_f \cdot \sum_{S \subseteq [d]} \pi(S) P^{(S)}_{\mathrm{sing}}(\mathcal{X}_\star) \\
&= \sum_{S \subseteq [d]}\pi(S) \int_{\mathcal{X}^{(S)}} f\biggl(\frac{\mathrm{d}P^{(S)}_{\mathrm{ac}}}{\mathrm{d} Q^{(S)}}\biggr) \, \mathrm{d}Q^{(S)} +  M_f \cdot \sum_{S \subseteq [d]}\pi(S) P^{(S)}_{\mathrm{sing}}(\mathcal{X}^{(S)}) = \mathrm{ADiv}_f(P,Q;\pi),
\end{align*}
as desired.
\end{proof}

Very often, it is convenient to apply Pinsker's inequality to total variation distances, in order to control them via (more tractable) Kullback--Leibler divergences.  We remark that in doing so directly to the left-hand side of Lemma~\ref{lem:equivalence-of-f-af}, we obtain 
\begin{align*} 
\mathrm{TV}\bigl(\mathsf{MCAR}_{(\pi, P)}, &\mathsf{MCAR}_{(\pi, Q)}\bigr) \leq \frac{1}{2^{1/2}} \cdot \mathrm{KL}^{1/2}\bigl(\mathsf{MCAR}_{(\pi, P)}, \mathsf{MCAR}_{(\pi, Q)}\bigr)\\
&= \frac{1}{2^{1/2}} \mathrm{AKL}^{1/2}(P, Q;\pi) = \frac{1}{2^{1/2}} \biggl\{\sum_{S \subseteq [d]} \pi(S) \cdot \mathrm{KL}\bigl(P_S, Q_S\bigr)\biggr\}^{1/2}.
\end{align*}
On the other hand, applying Pinsker's inequality to the right-hand side of Lemma~\ref{lem:equivalence-of-f-af} yields the bound
\begin{align*}
\mathrm{ATV}(P, Q;\pi) = \sum_{S \subseteq [d]} \pi(S) \cdot \mathrm{TV}\bigl(P_S, Q_S\bigr) \leq \frac{1}{2^{1/2}} \sum_{S \subseteq [d]} \pi(S) \cdot \mathrm{KL}^{1/2}\bigl(P_S, Q_S\bigr),
\end{align*}
which is an improvement, by Jensen's inequality.

We now state two lower bounds in the MCAR setting, beginning with the univariate setting.

\begin{prop}\label{prop:univariate-mcar-lb}
    Let $n \in \mathbb{N}$, $q \in (0, 1]$ and $\Theta \coloneqq \mathbb{R}$. 
    \begin{enumerate}
    \item[(a)] Let $\sigma>0$ and $\delta\in(0, 1/4]$. Then, writing $\mathcal{P}_{\theta} \coloneqq \bigl\{\mathsf{MCAR}_{(q,\mathsf{N}(\theta,\sigma^2))}^{\otimes n}\bigr\}$, we have
    \begin{align*}
    \mathcal{M}\bigl(\delta,\mathcal{P}_{\Theta},|\cdot|^2\bigr) 
    \begin{cases}
         \geq \dfrac{\sigma^2 \log(1/\delta)}{20nq} \quad&\text{if }\delta\geq \dfrac{(1-q)^n}{2}\\
         = \infty \quad&\text{if }\delta< \dfrac{(1-q)^n}{2}.
    \end{cases}
    \end{align*}
    \item[(b)] Let $K>0$ and $\delta\in(0,1/4]$. Then, with $\mathcal{P}_{\mathrm{b}}(\theta, K)$ as in~\eqref{eq:distributions-with-bounded-support} and writing $\mathcal{P}_{\theta} \coloneqq \bigl\{\mathsf{MCAR}_{(q,P)}^{\otimes n} : P\in\mathcal{P}_{\mathrm{b}}(\theta,K)\bigr\}$, we have
    \begin{align*}
        \mathcal{M}\bigl(\delta, \mathcal{P}_{\Theta}, | \cdot |^2\bigr) \begin{cases}
            \geq \dfrac{K^2 \log(1/\delta)}{80nq} \quad&\text{if } \delta\geq \exp(-nq/2)\\
            = \infty \quad&\text{if } \delta < \dfrac{(1-q)^n}{2}.
        \end{cases} 
    \end{align*}
    \end{enumerate}
\end{prop}
\begin{proof}
    (a) Let $\theta_1 \coloneqq 0$ and $\theta_2 \coloneqq \sigma \sqrt{\frac{1}{nq} \log\bigl( \frac{1}{4\delta(1-\delta)} \bigr)}$. By Lemma~\ref{lem:equivalence-of-f-af}, we have
    \begin{align*}
        \mathrm{KL}\bigl( \mathsf{MCAR}^{\otimes n}_{(\pi, \mathsf{N}(\theta_1,\sigma^2))}, \mathsf{MCAR}^{\otimes n}_{(\pi, \mathsf{N}(\theta_2,\sigma^2))}\bigr) &= nq \cdot \mathrm{KL}\bigl( \mathsf{N}(\theta_1,\sigma^2), \mathsf{N}(\theta_2,\sigma^2)\bigr)\\
        &= \frac{1}{2} \log\biggl( \frac{1}{4\delta(1-\delta)} \biggr) < \log\biggl( \frac{1}{4\delta(1-\delta)} \biggr).
    \end{align*}
    Therefore, by \citet[Corollary~6 and Theorem~4]{ma2024high}, we deduce that for $\delta\in(0,1/4]$,
    \begin{align*}
        \mathcal{M}\bigl(\delta, \mathcal{P}_{\Theta}, | \cdot |^2\bigr) \geq \biggl( \frac{\theta_1 - \theta_2}{2} \biggr)^2 = \frac{\sigma^2\log\bigl(\frac{1}{4\delta(1-\delta)}\bigr)}{4nq} \geq \frac{\sigma^2 \log(1/\delta)}{20nq}.
    \end{align*}
    Moreover, for any $\theta_1,\theta_2\in\mathbb{R}$, we have
    \begin{align*}
        \mathrm{TV}\bigl(&\mathsf{MCAR}_{(q,\mathsf{N}(\theta_1,\sigma^2))}^{\otimes n}, \mathsf{MCAR}_{(q,\mathsf{N}(\theta_2,\sigma^2))}^{\otimes n}\bigr)\\
        &= \sup_{A \in \mathcal{B}(\mathbb{R}_{\star}^n) \setminus \{\star\}^n} \Bigl\{\mathsf{MCAR}_{(q,\mathsf{N}(\theta_1,\sigma^2))}^{\otimes n}(A) - \mathsf{MCAR}_{(q,\mathsf{N}(\theta_2,\sigma^2))}^{\otimes n}(A)\Bigr\} \leq 1-(1-q)^n,
    \end{align*}
    where both steps follow since $\mathsf{MCAR}_{(q,\mathsf{N}(\theta_1,\sigma^2))}^{\otimes n}(\{\star\}^n) = \mathsf{MCAR}_{(q,\mathsf{N}(\theta_2,\sigma^2))}^{\otimes n}(\{\star\}^n) = (1-q)^n$.  Therefore, by \citet[Lemma~5]{ma2024high}, we have that $\mathcal{M}(\delta,\mathcal{P}_{\Theta},|\cdot|^2) \geq (\theta_1-\theta_2)^2/4$ for $\delta< \frac{(1-q)^n}{2}$. The claim follows since $\theta_1,\theta_2$ were arbitrary.

    \medskip
    (b) Define $P_1,P_2 \in \mathcal{P}(\mathbb{R})$ by 
    \[
    P_1(\{x\}) \coloneqq \begin{cases} \frac{1}{2} \quad &\text{ if } x = 0\\
    \frac{1}{2} \quad &\text{ if } x = K
    \end{cases} \quad \text{and} \quad P_2(\{x\}) \coloneqq \begin{cases}
    \frac{1-a}{2} \quad &\text{ if } x = 0\\
    \frac{1 + a}{2} \quad &\text{ if } x=K,
    \end{cases}
    \]
    where $a\coloneqq \sqrt{\frac{1}{nq}\log\bigl(\frac{1}{4\delta(1-\delta)}\bigr)} \leq \sqrt{\frac{\log(1/\delta)}{nq}} \leq 1/\sqrt{2}$ for $\delta \in [e^{-nq/2}, 1/4]$. 
    Let $\theta_1 \coloneqq \mathbb{E}_{P_1}(X) = K/2$ and $\theta_2 \coloneqq \mathbb{E}_{P_2}(X) = (1+a)K/2$, so that $P_{\ell} \in \mathcal{P}_{\mathrm{b}}(\theta_\ell, K)$ for $\ell\in\{1,2\}$. Moreover, by Lemma~\ref{lem:equivalence-of-f-af},
    \begin{align*}
        \mathrm{KL}\bigl(\mathsf{MCAR}^{\otimes n}_{(q,P_1)}, \mathsf{MCAR}^{\otimes n}_{(q,P_2)}\bigr) = nq \mathrm{KL}(P_1,P_2) &= \frac{nq}{2} \log\biggl( \frac{1}{1-a^2} \biggr)\\
        &< nqa^2 = \log\biggl( \frac{1}{4\delta(1-\delta)} \biggr),
    \end{align*}
    where the inequality follows because 
    $\log\bigl(\frac{1}{1-x^2}\bigr) < 2x^2$ for $x \in (0,1/\sqrt{2}]$. Hence, by \citet[Corollary~6 and Theorem~4]{ma2024high}, we deduce that for $\delta\in[e^{-nq/2},1/4]$,
    \begin{align*}
        \mathcal{M}\bigl(\delta, \mathcal{P}_{\Theta}, | \cdot |^2\bigr) \geq \biggl( \frac{\theta_1 - \theta_2}{2} \biggr)^2 \geq \frac{K^2 \log(1/\delta)}{80nq}.
    \end{align*}
    Now let $\theta\in\mathbb{R}$, $P_1' \coloneqq \mathsf{Unif}[0,K]$ and $P_2' \coloneqq \mathsf{Unif}[\theta, \theta + K]$. Then by the same argument as in part (a), we have
    \begin{align*}
        \mathrm{TV}\bigl( \mathsf{MCAR}_{(q,P_1')}^{\otimes n}, \mathsf{MCAR}_{(q,P_2')}^{\otimes n} \bigr) \leq 1-(1-q)^n.
    \end{align*}
    Therefore, by \citet[Lemma~5]{ma2024high}, we have that $\mathcal{M}(\delta,\mathcal{P}_{\Theta},|\cdot|^2) \geq \theta^2/4$ for $\delta< \frac{(1-q)^n}{2}$. The claim follows since $\theta_1,\theta_2$ were arbitrary.
\end{proof}

Our next proposition lower bounds the minimax quantile for mean estimation in the multivariate Gaussian setting when the covariance matrix is diagonal. 
\begin{prop} \label{prop:arb-mean-MCAR-lb}
    Let $\delta \in (0, 1/4]$, $\Sigma = (\Sigma_{jk})_{j,k \in [d]} \in \mathcal{S}^{d \times d}_{++}$ be diagonal, $\pi\in\mathcal{P}(2^{[d]})$, and let $P_{\theta} \coloneqq \mathsf{N}(\theta,\Sigma)$ for $\theta\in\mathbb{R}^d$. Then, writing $\mathcal{P}_{\theta} \coloneqq \bigl\{ \mathsf{MCAR}_{(\pi,P_{\theta})}^{\otimes n} \bigr\}$, we have
    \begin{align*}
        \mathcal{M}\bigl(\delta, \mathcal{P}_{\Theta}, \| \cdot \|_2^2\bigr) \gtrsim \frac{\tr\bigl(\Sigma^{\mathrm{IPW}}\bigr)}{n} +\frac{\| \Sigma^{\mathrm{IPW}} \|_{\mathrm{op}} \log(1/\delta)}{n}.
    \end{align*}
\end{prop}

\begin{proof}
We consider two separate constructions to capture each of the terms in the lower bound.  For the first, let $\mathcal{V} \coloneqq \{0, 1\}^{d}$ and for each $v = (v_1,\ldots,v_d)^\top \in \mathcal{V}$, set $\theta_v = (\theta_{v,1},\ldots,\theta_{v,d})^\top \coloneqq a \odot v$, where $a = (a_1,\ldots,a_d)^\top \in \mathbb{R}^d$ is given by $a_j \coloneqq \frac{4}{3}\sqrt{\Sigma_{jj}/(n q_j)}$ for $j\in[d]$. Define $\Theta_0 \coloneqq \{\theta_v : v\in\mathcal{V}\}$, which has diameter $D \coloneqq \frac{4}{3}\sqrt{\tr(\Sigma^{\mathrm{IPW}})/n}$.  For any $v,v' \in \mathcal{V}$ that differ only in their $j$th coordinates, we have by Pinsker's inequality and Lemma~\ref{lem:equivalence-of-f-af} that
    \begin{align*}
    \mathrm{TV}\bigl( \mathsf{MCAR}_{(\pi, P_{\theta_v})}^{\otimes n}, \mathsf{MCAR}_{(\pi, P_{\theta_{v'}})}^{\otimes n}\bigr) &\leq \biggl\{\frac{n}{2} \cdot \mathrm{KL}\bigl( \mathsf{MCAR}_{(\pi, P_{\theta_v})}, \mathsf{MCAR}_{(\pi, P_{\theta_{v'}})}\bigr)\biggr\}^{1/2}\\
    &= \biggl\{\frac{n}{2} \sum_{S\subseteq [d]} \pi(S) \cdot \mathrm{KL}\bigl((P_{\theta_v})_S, (P_{\theta_{v'}})_S \bigr)\biggr\}^{1/2}\\
    &= \biggl\{\frac{n}{2} \sum_{S\subseteq [d]} \pi(S) \cdot \sum_{k\in S} \frac{(\theta_{v,k} - \theta_{v',k})^2}{2\Sigma_{kk}} \biggr\}^{1/2}\\
    &= \biggl\{\frac{n}{4} \sum_{S\subseteq [d] : j\in S} \pi(S) \cdot \frac{a_j^2}{\Sigma_{jj}} \biggr\}^{1/2} = \frac{2}{3}.
    \end{align*}
    Therefore, by Assouad's Lemma \citep[e.g.,][Lemma 23]{ma2024high}, 
    \begin{align*}
        \inf_{\hat{\theta}_n \in \hat{\Theta}_{n}} \sup_{\theta_0\in\Theta_0} \mathbb{E}_{\mathsf{MCAR}_{(\pi,P_{\theta_0})}^{\otimes n}} \bigl( \|\hat{\theta}_n - \theta\|_2^2 \bigr) \geq \frac{4\tr(\Sigma^{\mathrm{IPW}})}{27n}.
    \end{align*}
    Applying \citet[Theorem~8]{ma2024high}, with $\epsilon=3/40$ therein, we deduce that for $\delta\in(0,1/15]$,
    \begin{align*}
        \mathcal{M}_-\bigl(\delta,\mathcal{P}_{\Theta},\|\cdot\|_2^2\bigr) \geq \frac{\tr(\Sigma^{\mathrm{IPW}})}{100n}.
    \end{align*}
    We then apply \citet[Theorem~4 and Proposition~9]{ma2024high}, with $A=k=2$ therein, to deduce that for $\delta\in(0,1/4]$,
    \begin{align} \label{eq:mcar-minimax-quantile-term1}
        \mathcal{M}\bigl(\delta,\mathcal{P}_{\Theta},\|\cdot\|_2^2\bigr) \geq \frac{\tr(\Sigma^{\mathrm{IPW}})}{2^6 \cdot 3^2 \cdot 5^2 \cdot n}.
    \end{align}
    Our second construction involves just two distributions.  Let $j_0 \coloneqq \sargmax_{j \in [d]} \Sigma_{jj}/q_j$ and set $\theta_1 \coloneqq 0$, $\theta_2 \coloneqq \sqrt{\frac{\Sigma_{j_0 j_0}}{n q_{j_0}} \log\bigl(\frac{1}{4\delta(1-\delta)}\bigr)}\, e_{j_0}$.  Then by Lemma~\ref{lem:equivalence-of-f-af}, 
\[
\mathrm{KL}\bigl(\mathsf{MCAR}_{(\pi, P_{\theta_1})}^{\otimes n}, \mathsf{MCAR}_{(\pi, P_{\theta_{2}})}^{\otimes n}\bigr) = n \cdot \mathrm{AKL}\bigl(P_{\theta_1}, P_{\theta_2}; \pi \bigr) = \frac{1}{2} \log\biggl(\frac{1}{4\delta(1-\delta)}\biggr).
\]
By~\citet[][Theorem 4 and Corollary 6]{ma2024high}, we have for $\delta \in (0, 1/4]$ that 
\begin{align} \label{eq:mcar-minimax-quantile-term2}
\mathcal{M}\bigl(\delta, \mathcal{P}_{\Theta}, \| \cdot \|_2^2\bigr) \geq \frac{\| \Sigma^{\mathrm{IPW}} \|_{\mathrm{op}} \log(1/\delta)}{20n}.
\end{align}
    Finally, combining~\eqref{eq:mcar-minimax-quantile-term1} and~\eqref{eq:mcar-minimax-quantile-term2} yields the desired result.
\end{proof}

\end{document}